\documentclass[11pt]{article} % use larger type; default would be 10pt

\usepackage[utf8]{inputenc} % set input encoding (not needed with XeLaTeX)

\usepackage{geometry} % to change the page dimensions
\usepackage[all]{xy}
\usepackage{authblk}
\usepackage[hidelinks]{hyperref}
\usepackage{graphicx} % support the \includegraphics command and options
\usepackage{amsfonts,amsthm,amsmath,amssymb,esint}
 \usepackage{slashed}
\usepackage{stmaryrd}
\usepackage{enumerate}
\usepackage{tikz}
\usetikzlibrary{matrix,arrows}
\usepackage{bbm}
\usepackage{mathbbol}

\newcommand{\C}{\mathbb{C}}
\newcommand{\Z}{\mathbb{Z}}

\newcommand{\V}{\mathbb{V}}
\newcommand{\D}{\mathbb{D}}

\newcommand{\J}{\mathbb{J}}
\newcommand{\N}{\mathbb{N}}
\newcommand{\K}{\mathbb{K}}
\newcommand{\F}{\mathbb{F}}
\newcommand{\G}{\mathbb{G}}

\newcommand{\B}{\mathbb{B}}
\newcommand{\M}{\mathbb{M}}
\newcommand{\U}{\mathbb{U}}

\newcommand{\E}{\mathbb{E}}
\newcommand{\ket}{\rangle}
\newcommand{\bra}{\langle}
\newcommand{\T}{\mathbb{T}}
\newcommand{\Pb}{\mathbb{P}}
\newcommand{\Sb}{\mathbb{S}}

\newcommand{\Li}{\mathcal{L}}
\newcommand{\Ni}{\mathcal{N}}
\newcommand{\Ci}{\mathcal{C}}
\newcommand{\Ai}{\mathcal{A}}
\newcommand{\Bi}{\mathcal{B}}

\newcommand{\Hi}{\mathcal{H}}

\newcommand{\Mi}{\mathcal{M}}
\newcommand{\Ki}{\mathcal{K}}

\newcommand{\Ei}{\mathcal{E}}

\newcommand{\Fi}{\mathcal{F}}
\newcommand{\Oi}{\mathcal{O}}
\newcommand{\Si}{\mathcal{S}}
\newcommand{\Ti}{\mathcal{T}}
\newcommand{\Ii}{\mathcal{I}}
\newcommand{\Gi}{\mathcal{G}}

\newcommand{\sing}{\operatorname{sing}}

\newcommand{\GH}{\mathfrak{H}}
\newcommand{\GE}{\mathfrak{E}}
\newcommand{\GF}{\mathfrak{F}}
\newcommand{\GG}{\mathfrak{G}}
\newcommand{\GK}{\mathfrak{K}}

\newcommand{\GI}{\mathfrak{I}}

\newcommand{\Gm}{\mathfrak{m}}

\newcommand{\id}{\operatorname{id}}
\newcommand{\Aut}{\operatorname{Aut}}

\newcommand{\Tr}{\operatorname{Tr}}
\newcommand{\tr}{\operatorname{tr}}

\newcommand{\SU}{\operatorname{SU}}

\newcommand{\Un}{\operatorname{U}}

\newcommand{\Spec}{\operatorname{Spec}}

\newcommand{\Mn}{\operatorname{M}}

\newcommand{\bone}{\mathbf{1}}

\newtheorem{thm}{Theorem}[section]
\newtheorem{Lemma}[thm]{Lemma}
\newtheorem{prop}[thm]{Proposition}
\newtheorem{cor}[thm]{Corollary}

\numberwithin{equation}{section}

\theoremstyle{definition}
\newtheorem{dfn}[thm]{Definition}

\newtheorem{Remark}[thm]{Remark}
\newtheorem{Example}[thm]{Example}
\newtheorem{Conjecture}[thm]{Conjecture}

\newtheorem{Question}[thm]{Question}

\newtheorem*{Ack}{Acknowledgment}

\newcommand{\Ker}{\operatorname{Ker}}

\newcommand{\Ran}{\operatorname{Ran}}

\newcommand{\GM}{\mathfrak{M}}

\newcommand{\GeL}{\operatorname{GL}}

\newcommand{\End}{\operatorname{End}}

\newcommand{\pd}{\partial}

\newcommand{\rank}{\operatorname{rank}}

\newcommand{\bvee}{\mbox{\Large $\vee$}}

\newcommand{\vol}{\operatorname{vol}}

\newcommand{\Gr}{\operatorname{Gr}}

\newcommand{\FS}{\operatorname{FS}}
\newcommand{\CD}{\operatorname{CD}}

\newcommand{\gr}{\operatorname{gr}}
\newcommand{\qgr}{\operatorname{qgr}}

\newcommand{\coh}{\operatorname{coh}}
\newcommand{\Proj}{\operatorname{Proj}}

\newcommand{\cspan}{\overline{\operatorname{span}}}
\newcommand{\Irrep}{\operatorname{Irrep}}

\usepackage{booktabs} % for much better looking tables
\usepackage{array} % for better arrays (eg matrices) in maths
\usepackage{paralist} % very flexible & customisable lists (eg. enumerate/itemize, etc.)
\usepackage{verbatim} % adds environment for commenting out blocks of text & for better verbatim
\usepackage{subfig} % make it possible to include more than one captioned figure/table in a single float
\usepackage{fancyhdr} % This should be set AFTER setting up the page geometry
\evensidemargin=\oddsidemargin
\usepackage{sectsty}
\allsectionsfont{\sffamily\mdseries\upshape} % (See the fntguide.pdf for font help)
\usepackage[nottoc,notlof,notlot]{tocbibind} % Put the bibliography in the ToC
\usepackage[titles,subfigure]{tocloft} % Alter the style of the Table of Contents

\title{Quantization of Yang--Mills metrics on holomorphic vector bundles%Eigenvector bundles in quantization of homogeneous Kähler manifolds
}
\author{Andreas Andersson}
\affil{\small Email: andane@chalmers.se}
\affil{\footnotesize Chalmers University of Technology, Mathematical Sciences, Maskingränd 2, 412 58 Gothenburg, Sweden
}
\affil{\footnotesize Mathematics Subject Classification 2010 -- Primary: 47L80%Algebras of specific types of operators (Toeplitz, integral, pseudodifferential, etc.)
; Secondary: 47B10%Operators belonging to operator ideals (nuclear, $p$-summing, in the Schatten--von Neumann classes, etc.)
, 32L10  % Sheaves and cohomology of sections of holomorphic vector bundles, general results
}

\begin{document}
\maketitle
\abstract
We investigate quantization properties of Hermitian metrics on holomorphic vector bundles over homogeneous compact Kähler manifolds. This allows us to study operators on Hilbert function spaces using vector bundles in a new way. We show that Yang--Mills metrics can be quantized in a strong sense and for equivariant vector bundles we deduce a strong stability property which supersedes Gieseker-stability. We obtain interesting examples of generalized notions of contractive, isometric, and subnormal operator tuples which have geometric interpretations related to holomorphic vector bundles over coadjoint orbits. 

\tableofcontents

\section{Introduction}

\subsection{Quantization}
It is well established that every smooth projective variety $\M\subset\C\Pb^{n-1}$ can be quantized in the sense that there is a sequence $\GH_\bullet=(\GH_m)_{m\in\N_0}$ of finite-dimensional Hilbert spaces such that the matrix algebra $\Bi(\GH_m)$ of all bounded operators on $\GH_m$ becomes arbitrarily close, as a $C^*$-algebra, to the $C^*$-algebra $C^0(\M)$ of continuous functions on $\M$ when $m$ goes to infinity. As a vector space one has $\GH_m=H^0(\M;\Li^m):=$ the space of global holomorphic sections of the $m$th power of a positive line bundle $\Li$ over $\M$. If we fix a Kähler form $\omega$ on $\M$ in the class $c_1(\Li)$ then the choice of inner product on $\GH_m$ is a choice of Hermitian metric $h_m$ on $\Li$ and the relation between $h_m$ and $\omega$ is not arbitrary as $m$ gets large. If $\omega$ has constant scalar curvature then the quantization is slightly more well-behaved than in general. The strongest possible quantization (a ``regular'' quantization) can obtained only when $\M$ is a coadjoint orbit $\G/\K$ under some Lie group $\G$ and $\omega$ is $\G$-invariant. Here in precise terms a regular quantization means a quantization where two kinds of natural positive maps (Toeplitz and covariant symbol maps) are both unital. In \cite{An6} we took an operator-theoretic and operator-algebraic approach to quantization. The physical motivation for this is outlined in \cite{An4, An5}.

In this paper we shall study the quantization of Hermitian metrics on holomorphic vector bundles over a coadjoint orbit $\M=\G/\K\subset\C\Pb^{n-1}$. The case $\M=\C\Pb^{n-1}$ is already interesting. Inspired by noncommutative geometry, and in particular \cite{Hawk1, Hawk2}, a quantization of a vector bundle $\Ei$ will be a sequence of modules $\Bi(\GH_m,\GE_m)$ over the matrix algebras $\Bi(\GH_m)$ such that when $m$ is getting large $\Bi(\GH_m,\GE_m)$ becomes arbitrarily close to the $C^0(\M)$-module $\Gamma^0(\M;\Ei)$ of global continuous sections of $\Ei$. Actually there is more to it. When we quantized $C^0(\M)$ we had inner products whose associated norms approximated the norm on $C^0(\M)$. A Hermitian metric $h^E$ on the vector bundle $\Ei$ is the same thing as a $C^0(\M)$-valued inner product on $\Gamma^0(\M;\Ei)$; when the $C^0(\M)$-module $\Gamma^0(\M;\Ei)$ is endowed with such a $C^0(\M)$-valued inner product it is called a Hilbert $C^*$-module and we denote it by $\Gamma^0(\M;\Ei,h^E)$. Therefore we would like to have a Hilbert $C^*$-module structure on $\Bi(\GH_m,\GE_m)$ (or, what is the same, an inner product on the Hilbert space $\GE_m$) which approximates $\Gamma^0(\M;\Ei,h^E)$ in the limit $m\to\infty$. Observe that we have fixed the inner product on $\GH_m$ and therefore the $C^*$-structure on $\Bi(\GH_m)$ and the Hilbert $C^*$-module structure on $\Bi(\GH_m,\GE_m)$, so that we are looking for the possibility to quantize the Hermitian metric $h^E$ \emph{with respect to} the given quantization $\GH_\bullet$ of the manifold $\M$. 

In \cite{An6} we used a particular choice of quantization $\GH_\bullet$ of $\M$ which gives $C^0(\M)$ as a kind of inductive limit of the matrix algebras $\Bi(\GH_m)$ provided by unital completely positive maps $\iota_{m,m+1}:\Bi(\GH_m)\to\Bi(\GH_{m+1})$. On the Hilbert space $\GH_\N:=\bigoplus_{m\in\N_0}$ acts a $C^*$-algebra $\Ti_\GH^{(0)}$ of generalized Toeplitz operators with symbol in $C^0(\M)$, and one has a short exact sequence of $C^*$-algebras
\begin{equation}\label{Toeplexaxt}
0\to\bigoplus_{m\in\N_0}\Bi(\GH_m)\to\Ti_\GH^{(0)}\overset{\varsigma}{\to}C^0(\M)\to 0
\end{equation}
which is split by a positive linear Toeplitz-type map $\breve{\varsigma}:C^0(\M)\to\Ti_\GH^{(0)}$. The map $\varsigma$ is the adjoint of the Toeplitz map and called the covariant Berezin \textbf{symbol} map. The operators in $\Ti_\GH^{(0)}$ preserve the $\N_0$-grading on $\GH_\N$ and its elements can therefore be regarded as sequences $(A_m)_{m\in\N_0}$ with $A_m\in\Bi(\GH_m)$. The ideal $\bigoplus_{m\in\N_0}\Bi(\GH_m)$ consists of those compact operators on $\GH_\N$ which preserve the grading. 

This is the quantization of the manifold which we will fix, and we will describe the quantization of vector bundles, as defined above, as follows. For any $N\in\N$ we can extend $\breve{\varsigma}$ to a map on the algebra $C^0(\M)\otimes\Mn_N(\C)$ of $N\times N$-matrices with entries in $C^0(\M)$ by applying $\breve{\varsigma}$ to each entry; we use the same notation $\breve{\varsigma}$ for all $N$. If $P^E$ is an idempotent in the algebra $C^0(\M)\otimes\Mn_N(\C)$ for some $N$, in which case we say that $P^E$ is an idempotent \textbf{over} $C^0(\M)$, then $P^E$ defines a continuous vector bundle $\Ei$ over $\M$. Indeed, since $P^E$ is continuous the rank of $P^E(x)$, say $r$, is the same for all $x\in\M$, and we can view $P^E$ as a topological embedding of $\M$ in the Grassmannian manifold $\Gr_r(\C^N)$ of rank-$r$ projections acting on $\C^N$. We then obtain a vector bundle $\Ei$ over $\M$ by pulling back the universal vector bundle over $\Gr_r(\C^N)$ via $P^E$. Another way of viewing is that $P^E$ defines a projective $C^0(\M)$-module
$$
\Gamma^0(\M;\Ei,P^E):=P^E(C^0(\M)\otimes\C^N)
$$
which by Swan's theorem determines a $C^0$ vector bundle $\Ei$ uniquely up to isomorphism. Having presented a vector bundle $\Ei$ using an idempotent $P^E$, a quantization of $\Ei$ is an idempotent $P_E$ over $\Ti_\GH^{(0)}$ which \textbf{lifts} $P^E$ in the sense that
$$
\varsigma(P_E)=P^E.
$$
Since $\breve{\varsigma}$ splits the Toeplitz short exact sequence \eqref{Toeplexaxt}, this means that there is a compact operator $K_E$ such that
$$
P_E=\breve{\varsigma}(P^E)+K_E.
$$
If we let $P_{E,m}$ be the component of $P_E$ in $\Bi(\GH_m)\subset\Bi(\GH_\N)$ then we can indeed define vector spaces $\GE_m$ via the identification
\begin{equation}\label{PEmide}
P_{E,m}(\Bi(\GH_m)\otimes\C^N)=\Bi(\GH_m,\GE_m)
\end{equation}
of left $\Bi(\GH_m)$-modules. By a purely algebraic argument we can show that the $C^0(\M)$-module is obtained as the set-theoretical inductive limit of the $\Bi(\GH_m)$-modules $\Bi(\GH_m,\GE_m)$. Such a lift $P_E$ always exists (see \S\ref{liftsec}). Since we have fixed the $C^*$-structure on $\Bi(\GH_m)$, the quantization is really with respect to $\GH_\bullet$. If we take $P^E$ to be a projection (meaning a selfadjoint idempotent) then $\breve{\varsigma}(P^E)$ is a positive operator, $P_E$ must be a projection and $\GE_m$ becomes a Hilbert space under the canonical inner product obtained from the identification \eqref{PEmide}. And an inner product on the vector space $\GE_m$ is the same datum as a $\Bi(\GH_m)$-valued inner product on $\Bi(\GH_m,\GE_m)$.
The standard $C^0(\M)$-valued inner product on $C^0(\M)\otimes\C^N$ gives a $C^0(\M)$-valued inner product on the module $\Gamma^0(\M;\Ei,P^E)$ (i.e. a Hermitian metric on $\Ei$) when $P^E$ is a projection. Moreover, up to isomorphism every Hermitian vector bundle $\Ei$ is obtained like this for some projection $P^E$. Therefore, we shall typically only consider a vector bundle $\Ei$ up to smooth or $C^0$ isomorphism and refer to a projection $P^E$ with $\Gamma^0(\M;\Ei)\cong P^E(C^0(\M)\otimes\C^N)$ as a metric on $\Ei$.

However, the $\Bi(\GH_m)$-valued inner product on $\Bi(\GH_m,\GE_m)$ may not approximate the given Hermitian metric $P^E$ on $\Ei$. 
We would like to have some better compatibility between the Hilbert $C^*$-modules $\Bi(\GH_m,\GE_m)$ for different $m$'s or, what is the same, a better behavior of the sequence $\GE_\bullet$ with respect to the sequence $\GH_\bullet$. That $\breve{\varsigma}(P^E)$ itself is a projection, so that we could take $P_E$ to be just $\breve{\varsigma}(P^E)$, happens only if all Chern classes of $\Ei$ vanish. So this is a too strong requirement. As mentioned, we shall in this paper assume that $\M$ is a homogeneous space $\G/\K$ under some Lie group $\G$. This ensures that the Toeplitz map $\breve{\varsigma}$ is unital. The Toeplitz operators are then precisely the fixed points under an explicit unital completely positive map $\Psi:\Bi(\GH_\N)\to\Bi(\GH_\N)$. Therefore the Toeplitz operators will also be referred to as $\Psi$-\textbf{harmonic} operators (cf. the theory of noncommutative Poisson boundaries \cite{Iz1, INT1}). So we have
$$
\Psi(\breve{\varsigma}(P^E))=\breve{\varsigma}(P^E),
$$
and we said that it is too strong to assume that the lift $P_E$ is fixed by $\Psi$. As a second best thing we could look for a lift $P_E$ which is \textbf{superharmonic} under $\Psi$,
\begin{equation}\label{PEsupharmon}
\Psi(P_E)\leq P_E.
\end{equation}
We shall investigate the geometric meaning of such a lift. Let us first look at its meaning in operator theory. This is in fact very easy to describe. The Hilbert space $\GH_\N$ is a reproducing kernel Hilbert space of analytic functions on a complex-analytic submanifold $\B$ of the unit ball $\B^n$ in $\C^n$. More precisely, if $\V\subset\C^n$ denotes the spectrum of the graded algebra $\Ai:=\bigoplus_{m\in\N_0}H^0(\M;\Li^m)$, so that $(\V\setminus\{0\})/\C^\times=\M$, then
$$
\B:=\V\cap\B^n.
$$
The reproducing kernel of $\GH_\N$ has the complete Nevanlinna--Pick property. Let us give two important characterizations of this property. The first is that $\GH_\N$ can be identified with a subspace of the Drury--Arveson space $H^2_n$, whose tuple of multiplication operators by the coordinate functions $z_1,\dots,z_n$ on $\B^n$ is the universal model for pure row contractions. This means that the multiplication tuple $S=(S_1,\dots,S_n)$ on $\GH_\N$, defined by
$$
(S_\alpha\psi)(w):=w_\alpha\psi(w),\qquad\forall\psi\in\GH_\N,\ w\in\B,
$$
satisfies the row contraction property $\sum^n_{\alpha=1}S_\alpha S_\alpha^*=\bone-p_0$, where $p_0$ is the projection onto the constant functions, and is pure in the sense that $\mathrm{SOT-}\lim_{p\to\infty}\Phi^p(\bone)=0$, where $\Phi$ is the contractive completely positive map $\Phi(X):=\sum^n_{\alpha=1}S_\alpha XS_\alpha^*$ on $\Bi(\GH_\N)$. When restricted to the algebra $\prod_m\Bi(\GH_m)$ of grading-preserving operators the map $\Phi$ is the adjoint of $\Psi$ with respect to the state $\sum_m\phi_m$, where $\phi_m$ is the tracial state on $\Bi(\GH_\N)$. Explicitly,
$$
\Psi(X)=\sum^n_{\alpha=1}T_\alpha^*XT_\alpha,\qquad\forall X\in\Bi(\GH_\N)
$$
with $T=(T_1,\dots,T_n)$ acting just as $S$ but weighted to ensure $\Psi(\bone)=\bone$. The operator tuples $S$ and $T$ have the same set of invariant graded subspaces and the same set of coinvariant graded subspaces. The $\Psi$-superharmonicity condition \eqref{PEsupharmon} means precisely that the range $\GE_\N$ of $P_E$ as an operator on $\GH_\N\otimes\C^N$ is invariant under the operators $S_1^*,\dots,S_n^*$. This leads us to the second characteristic property of the complete Nevanlinna--Pick space $\GH_\N$. Namely, that \eqref{PEsupharmon} is equivalent to the existence of a matrix $\Theta_E$ of analytic multipliers of $\GH_\N$ such that
\begin{equation}\label{introBeurl}
P_E=\bone-M_{\Theta_E}M_{\Theta_E}^*,
\end{equation}
where $M_{\Theta_E}$ is the operator of multiplication by $\Theta_E$. Note that $M_{\Theta_E}$ must then be a partial isometry. We know from \cite[Thm. 4.3]{GRS2} or \cite[Thm. 6.1]{BhSa1} that $\Theta_E$ has an $L^\infty$ extension to the boundary $\Sb$ of $\B$% (still assuming here $\varsigma(P_E)$ to be continuous) 
such that $\Theta_E(\zeta)$ is a partial isometry for almost every $\zeta\in\Sb$. Therefore \eqref{introBeurl} is regarded as a multivariable analogue of the Beurling representation for model spaces on the unit disk $\D$, although in the one-variable setting the graded situation is trivial. A multivariable Beurling representation holds also for not necessarily grading-preserving projections $P_E\in\Bi(\GH_\N\otimes\C^N)$ with \eqref{PEsupharmon} but we will focus on the graded ones here.

We shall see that \eqref{PEsupharmon} guarantees that $P_E$ has entries in a von Neumann algebra of Toeplitz operators with $L^\infty$ symbols. We will also show that if $P_E$ has entries in the Toeplitz $C^*$-algebra $\Ti_\GH^{(0)}$ then the real-analytic matrix-valued function
$$
\varsigma_\B(P_E)(v):=\bone-\Theta_E(v)\Theta_E(v)^*,\qquad\forall v\in\B
$$
has a continuous extension to a projection-valued function on $\Sb$, which is equivariant under the action on $\Sb$ by the circle group $\Un(1)$ and hence descends to the projective variety $\M=\Sb/\Un(1)$. We shall see that the latter is precisely the symbol $\varsigma(P_E)$ of the $\Psi$-superharmonic projection $P_E$. Thus, starting from a $\Psi$-superharmonic projection $P_E$ over $\Ti_\GH^{(0)}$ we obtain in a canonical way a continuous vector bundle over $\M$, namely the vector bundle $\Ei$ defined by the projection $\varsigma(P_E)$ over $C^0(\M)$. The compression of the shift on $\GH_\N\otimes\C^N$ to the range of $P_E$ is again a pure row contraction $S_E$. The rank of the vector bundle $\Ei$ is precisely the so-called Arveson curvature of $S_E$ which is studied e.g. in \cite{Arv7a, Arv7b, GRS1, GRS2}. In the present paper we shall investigate further the geometric properties of $\Ei$ and their relation to the operator tuple $S_E$.

There is a more well-studied way of associating a vector bundle to an operator tuple such as $S_E$, provided that one can show that $S_E$ satisfies certain conditions. Cowen and Douglas invented an approach to classify operator tuples with uncountably many eigenvalues \cite{CoDo1, CoDo2}. If $\Omega$ is a subset of $\C^n$ and $r\geq 1$ is an integer, a tuple $W=(W_1,\dots,W_n)$ of commuting operators on a Hilbert space $\Hi$ is said to belong to the \textbf{Cowen--Douglas class} $B_r(\Omega)$ if
\begin{enumerate}[(i)]
\item{$\Ran(W-w\bone)$ is a closed subspace of $\Hi$ for all $w\in\Omega$,}
\item{$\dim\Ker(W-w\bone)=r$ for all $w\in\Omega$, where $\Ker(W-w\bone):=\bigcap_{\alpha=1}^n\Ker(W_\alpha-w_\alpha\bone)$, and}
\item{$\bvee_{w\in\Omega}\Ker(W-w\bone)=\Hi$.}
\end{enumerate}
%Cowen--Douglas operators are thus semi-Fredholm operators with an extra ``analyticity'' condition (iii). 

If $W$ is in class $B_r(\Omega)$ then the family of Hilbert subspaces $\Ker(W-w\bone)$ parameterized by $w\in\Omega$ can be given the structure of form an Hermitian holomorphic vector bundle over $\Omega$. 
It is not hard to show that when $S$ is the shift on $\GH_\N$ as before then the backward shift $S^*$ is in class $B_1(\B)$, with the domain $\B:=\B^n\cap\Spec\Ai$ mentioned above. The associated holomorphic line bundle over $\B$ is isomorphic to the trivial one but its Hermitian metric is still quite interesting. 

In view of our quantization problem it is natural to look at graded subspaces $\GE_\N$ of $\GH_\N\otimes\GE_0$ for some finite-dimensional Hilbert space $\GE_0$. Compressing the shift $S$ on $\GH_\N\otimes\GE_0$ down to the subspace $\GE_\N$ we obtain a commutative operator tuple $S_E$ which we also refer to as the shift on $\GE_\N$. When the projection $P_E$ onto $\GE_\N$ is $\Psi$-superharmonic, which happens iff $\GE_\N$ is invariant under $S^*$, we call $\GE_\N$ a graded \textbf{quotient module}. We will prove:
\begin{thm}[Theorem \ref{bigCDversuslocallyfree}]
Let $\GE_\N\subset\GH_\N\otimes\GE_0$ be a graded quotient module and suppose that the projection $P_E$ onto $\GE_\N$ has entries in the Toeplitz $C^*$-algebra $\Ti_\GH^{(0)}$. Then the operator tuple $S_E$ is in class $B_{r_E}(\B\setminus\{0\})$, where
$$
r_E=\lim_{m\to\infty}\frac{\dim\GE_m}{\dim\GH_m}
$$
is the Arveson curvature of the pure row contraction $S_E$.
\end{thm}
%The vector bundle $\Ei_{\rm CD}$ over $\B\setminus\{0\}$ with fiber over $v$ given by $\Ker(S^*_E-\bar{v}\bone)$ turns out to be smoothly isomorphic to the pullback to $\B\setminus\{0\}$ of the vector bundle over $\M=(\B\setminus\{0\})/\D^\times$ defined by the projection $\varsigma(P_E)$. 

Typically $\Ei_{\rm CD}$ does not extend to a vector bundle over all of $\B$, i.e. it is necessary to remove the origin.
By \cite[Thm. 3.3]{DKKS2}, a quotient module $\GE_\N=\Ker\Theta^*_E$ of $\GH_\N\otimes\C^N$ is in the Cowen--Douglas class $B_r^*(\B)$ for some $r$ if $M_{\Theta_E}$ has a left inverse (see also \cite[Cor. 4.4]{DFS1}). From Theorem \ref{bigCDversuslocallyfree} one can then deduce obstructions to the existence of a left inverse $M_{\Theta_E}$: Such a left inverse cannot exist if the ``Serre sheaf'' of the graded $\Ai$-module $\GE_\N$ (see \S\ref{alggradmodsec}) is not locally free on all of $\B$. The left-invertibility of $M_{\Theta_E}$ is a kind of corona condition (see \cite{Doug1}). 
%By Theorem \ref{bigCDversuslocallyfree} $M_{\Theta_E}$ cannot have a left inverse unless 

%However, $\Theta_E$ is a partial isometry only if $\GH^E_\N$ is reducing. In this case $\Ei_{\rm CD}$ is holomorphically trivial. % Anyway, it seems like their corona condition (left invertibility) is equivalent to having $\Ei_\V$ having a removable singularity at 0. 

%Müller says in \cite[\S II]{Mull2} that ``it is not possible to expect the punctured neighbourhood theorem for $n\geq 2$''. We shall see that when $\GE_\N$ is a graded quotient module whose Serre sheaf is a stable vector bundle over $\M$ then the Cowen--Douglas sheaf $\Ei_{\rm CD}$ of $\GE_\N$ is locally free on $\B\setminus\{0\}$. 

%So when the symbol $\varsigma(P_E)$ of a superharmonic projection $P_E$ is continuous, 

\subsection{Yang--Mills metrics}

We just discussed the vector bundles associated with a $\Psi$-superharmonic projection $P_E$ over $\Ti_\GH^{(0)}$. If we go back to our starting point, with an arbitrary projection $P^E$ over $C^0(\M)$, it turns out that there does not exist a $\Psi$-superharmonic lift of $P^E$. But we shall prove:
\begin{thm}[Lemma \ref{coinvYMlemma}, Theorem \ref{YMCDthm}]\label{intrYMthm}
Let $P^E$ be a projection over $C^\infty(\M)$ defining a Hermitian Yang--Mills vector bundle $\Ei$ over $\M$. Then the $\Psi$-superharmonic projection $\Ran\breve{\varsigma}(P^E)$ is a lift of $P^E$:
$$
\varsigma(\Ran\breve{\varsigma}(P^E))=P^E.  
$$
Moreover, the Cowen--Douglas sheaf $\Ei_{\rm CD}$ of $\GH^E_\N:=\Ran\breve{\varsigma}(P^E)$ is analytically isomorphic over $\B\setminus\{0\}$ to the pullback $\Ei_{\B\setminus\{0\}}$ of $\Ei$; as Hermitian holomorphic vector bundles we have
\begin{equation}\label{Hermfactor}
\Ei_{\rm CD}=\Oi_{\rm CD}\otimes\Ei_{\B\setminus\{0\}}
\end{equation}
where $\Ei_{\B\setminus\{0\}}$ is endowed with the Hermitian metric given by pullback of $P^E$ to $\B\setminus\{0\}$.
\end{thm}
The Yang--Mills equation under consideration is
$$
\tr_\omega\Theta^E=\mu(\Ei)\bone_E,
$$
where $\omega$ is the Kähler form in class $c_1(\Li)$ associated with the $\G$-invariant volume form, $\Theta^E$ is the Chern curvature of $P^E$ and the given holomorphic structure on $\Ei$, and $\tr_\omega:\Ai^2(\M)\otimes\End\Gamma^\infty(\M;\Ei)\to\End\Gamma^\infty(\M;\Ei)$ is the operator of taking trace against $\omega$ (with $\Ai^2(\M)$ denoting the space of differential 2-forms on $\M$).

%The connection between the Yang--Mills condition and the existence of a $\Psi$-superharmonic lift can be understood as follows. 

The Yang--Mills condition depends on the Kähler metric $\omega$, which for our coadjoint orbit $\M=\G/\K$ is encoded in the quantization $\GH_\bullet$. The difference between the Toeplitz operator $\breve{\varsigma}(P^E)$ and its range projection $P_E$ is a measure on how well $P^E$ can be quantized with respect to the chosen quantization $\GH_\bullet$ of the manifold. We have said that $\breve{\varsigma}^{(m)}(P^E)$ cannot itself be expected to be a projection for all $m$ unless the vector bundle $\Ei$ defined by $P^E$ has trivial Chern character. The next best thing would be that $\breve{\varsigma}^{(m)}(P^E)$ is a scalar multiple of its range projection $P_{E,m}$ for each $m$. This holds for very special metrics:
\begin{prop}[Proposition \ref{balancedeqprop}]
Suppose that $P^E$ is a projection over $C^\infty(\M)$ defining an irreducible $\G$-equivariant Hermitian vector bundle $\Ei$ over $\M=\G/\K$. Then for all $m\gg 0$ we have
\begin{equation}\label{balsinglePE}
\breve{\varsigma}^{(m)}(P^E)=c_{E,m}P_{E,m}
\end{equation}
with the scalar
$$
c_{E,m}:=\frac{n_m\rank\Ei}{\chi(\Ei(m))}
$$
where $\chi(\Ei(m))$ is the Euler characteristic of the vector bundle $\Ei(m):=\Li^m\otimes\Ei$.
\end{prop}
Recall that for $m$ large enough, $\chi(\Ei(m))$ is just the dimension of the vector space $H^0(\M;\Ei(m))$ of global holomorphic sections of $\Ei(m)$. 

\subsection{Balanced metrics}
For more general vector bundles $\Ei$, which are not $\G$-equivariant, there does not exist a metric $P^E$ satisfying  \eqref{balsinglePE} for each $m\gg 0$. But suppose that we have a sequence $(B^E_m)_{m\gg 0}$ of projections over $C^\infty(\M)$ which all define vector bundles smoothly isomorphic to $\Ei$ and such that
\begin{equation}\label{balsinglePE}
\breve{\varsigma}^{(m)}(B^E_m)=c_{E,m}P_{E,m},
\end{equation}
where $P_{E,m}$ is a projection in $\Bi(\GH_m)\otimes\Mn_N(\C)$ for some $N$. If we want the $B^E_m$'s to give a quantization of $\Ei$ then we need to require that the $B^E_m$'s for different $m$ are related in some way. One way would be to require the projection $P_E:=\sum_mP_{E,m}$ to belong to the $C^*$-algebra $\Ti_\GH^{(0)}\otimes\Mn_N(\C)$, so that the symbol $\varsigma(P_E)$ is continuous. Let us instead assume that $\Ei$ admits a holomorphic structure. Then the vector space
$$
E_\N:=\bigoplus_{m\in\N_0}H^0(\M;\Ei(m))
$$
is a graded module over the graded algebra $\Ai:=\bigoplus_{m\in\N_0}H^0(\M;\Li^m)$. Up to finite-dimensional vector spaces we may assume that $E_\N$ for some $N$ is the quotient of $\Ai\otimes\C^N$ by some graded submodule $I_\N$. Then a natural condition on the $B^E_m$'s is to require that $P_E$ is the projection of $\GH_\N\otimes\C^N$ onto the orthogonal complement $\GE_\N$ of $I_\N$. Indeed, the graded $\Ai$-module structure on $\GE_\N$ defined by the compressed shift $S_E$ is isomorphic to that on $E_\N$. The subspace $\GE_\N\subset\GH_\N\otimes\C^N$ is invariant under $S^*$, so $P_E$ is $\Psi$-superharmonic.  

The Toeplitz map $\breve{\varsigma}^{(m)}$ has an explicit expression in terms of $\omega$ and the Fubini--Study metric $\FS(\GH_m)$ of the line bundle $\Li^m$. To find the meaning of \eqref{balsinglePE} we will make use of frame theory for Hilbert $C^*$-modules (see \cite{FrLa3} for background). A frame for a Hilbert $C^*$-module will be referred to as a \textbf{$C^*$-frame} in order to distinguish it from the usual frames for Hilbert spaces, which we will also need (see \cite{Chri1, HaLa1} for background on these). 
As will be discussed in more detail in \S\ref{balasec}, the condition \eqref{balsinglePE} says that there is a Parseval $C^*$-frame for the Hilbert $C^0(\M)$-module $\Gamma^0(\M;\Ei(m),\FS(\GH_m)\otimes B^E_m)$ which is at the same time an orthonormal basis for the vector space $H^0(\M;\Ei(m))$ endowed with the $L^2$-inner product of $\omega$ and the metric $\FS(\GH_m)\otimes B^E_m$ on $\Ei(m)$. Briefly, \eqref{balsinglePE} says that $\FS(\GH_m)\otimes B^E_m$ is a \textbf{balanced} metric on $\Ei(m)$ in the sense of \cite{Wang1}. One can view the balance of $\FS(\GH_m)\otimes B^E_m$ as saying that the Hilbert $C^*$-struture on $\Gamma^0(\M;\Ei(m),\FS(\GH_m)\otimes B^E_m)$ (i.e. the metric $\FS(\GH_m)\otimes B^E_m$) is completely encoded in the finite-dimensional Hilbert space $H^0(\M,\omega;\Ei(m),\FS(\GH_m)\otimes B^E_m)$. 

Balanced metrics are related to Yang--Mills metrics via Wang's theorem, here reformulated using the Toeplitz maps:
\begin{Lemma}[{\cite{Wang2}}]\label{Wanglemma}
Let $P^E$ be a metric on a holomorphic vector bundle $\Ei$ over $\M$. Then $P^E$ is Yang--Mills if and only if there exists a sequence $(B^E_m)_{m\gg 0}$ of metrics on $\Ei$ such that the metric $\FS(\GH_m)\otimes B^E_m$ on $\Ei(m)$ is balanced and
$$
\lim_{m\to\infty}B^E_m=P^E
$$
in the topology of $C^\infty(\M)$.
\end{Lemma}
The connection between the Yang--Mills condition and the existence of a $\Psi$-superharmonic lift stated in Theorem \ref{intrYMthm} can now be better understood: If $P^E$ is a Yang--Mills then 
$P^E$ is the limit of some projections $B^E_m$ satisfying \eqref{balsinglePE}. We shall see that $P_{E,m}$ for $m\gg 0$ coincides with the range projection of $\breve{\varsigma}^{(m)}(P^E)$ and therefore the convergences $\lim_{m\to\infty}B^E_m=P^E$ and $\lim_{m\to\infty}c^E_m=1$ give that $\breve{\varsigma}(P^E)$ differs from its range projection modulo compacts.

\subsection{Into Hardy space}
Given a metric $P^E$ on a holomorphic vector bundle $\Ei(m)$ we can also look at the Hardy space $H^0(\Sb,\omega;P^E)$ of the pullback of $P^E$ to the circle bundle $\Sb=\pd\B$. It is natural to ask for a relation between $H^0(\Sb,\omega;P^E)$ and the quotient module $\GE_\N:=\Ran\breve{\varsigma}(P^E)$. 
First observe that the Hardy space $H^0(\Sb,\omega;P^E)$ is invariant under the tuple of operator on $L^2(\Sb,\omega;P^E)$ acting by multiplication with the coordinate functions $Z=(Z_1,\dots,Z_n)$. The multiplication tuple $V_E$ on $H^0(\Sb,\omega;P^E)$ is therefore subnormal. In fact, $V_E$ is a spherical isometry in the sense that 
$$
V_E^*V_E:=\sum_{\alpha=1}^nV_{E,\alpha}^*V_{E,\alpha}=\bone,
$$ 
and this algebraic relation directly implies subnormality by \cite[Prop. 2]{Atha3}. In contrast, the shift $S_E$ on $\GE_\N$ is not subnormal, i.e. it does not have a normal extension. Even if we take the spherical isometry $T$ on $\GH_\N\otimes\C^N$ and consider its compression $T_E$ to the coinvariant subspace $\GE_\N$, the operator $T_E^*T_E$ is typically not the identity on $\GE_\N$. But while $S_E$ belongs to something that one could call a multivariable analogue of the class $C_{\cdot 0}$ of contractions, $T_E$ is a contraction of class $C_{1\cdot}$ if we assume that the projection $P_E$ has entries in the Toeplitz algebra $\Ti_\GH^{(0)}$. And for contractions $T_E$ of class $C_{1\cdot}$ there is a canonical way of turning $T_E$ into an isometry by applying a similarity transformation, namely the Fredholm operator
$$
A_E:=\lim_{m\to\infty}\sum_{|\mathbf{k}|=m}T_{E,\mathbf{k}}^*T_{E,\mathbf{k}},
$$
where $T_{E,\mathbf{k}}:=T_{E,k_1}\cdots T_{E,k_m}$. The inverse of $A_E^{-1}$ of $A_E$ is well-defined outside a finite-dimensional subspace, and we may ignore this subspace for present purposes. The tuple $A_E^{1/2}T_EA_E^{-1/2}$ is a spherical isometry and in particular subnormal. 
% Since $P^E$ is $C^0$, the projection $P_E=\Ran\breve{\varsigma}(P^E)$ has entries in $\Ti_\GH^{(0)}$.
\begin{thm}[Theorem \ref{backtoHardythmgen}]\label{introbktoHardy}
The operator tuple $A_E^{1/2}T_EA_E^{-1/2}$ on $\GE_\N$ is unitarily equivalent to the multiplication tuple $V_E$ on $H^0(\Sb,\omega;P^E)$. 
\end{thm}
Another way of viewing this is to say that the positive invertible operator $A_E$ can be used to change the inner product on $\GE_\N$ to that of $H^0(\Sb,\omega;P^E)$. Now observe that $A_E$ is just the restriction to $\GE_\N$ of the Toeplitz operator $\breve{\varsigma}(P^E)$. Let us discuss the geometric interpretation of $\breve{\varsigma}(P^E)$. We have already mentioned that $\breve{\varsigma}^{(m)}(P^E)$ equals $P_{E,m}$ only if there is a Parseval frame for the Hilbert $C^0(\M)$-module $\Gamma^0(\M;\Ei(m),\FS(\GH_m)\otimes P^E)$ which is also a Parseval frame for the Hilbert space $H^0(\M,\omega;\Ei(m),\FS(\GH_m)\otimes P^E)$. In fact $\breve{\varsigma}^{(m)}(P^E)$ can be identified with the frame operator of a Parseval frame for the Hilbert $C^0(\M)$-module $\Gamma^0(\M;\Ei(m),\FS(\GH_m)\otimes P^E)$ regarded as a frame for the Hilbert space $H^0(\M;\Ei(m),\FS(\GH_m)\otimes P^E)$. For this reason it is interesting to note that the diagonal of the \textbf{Szegö kernel} (aka Bergman kernel), studied e.g. in \cite{Catl1, MaMa3}, is in the present context the frame operator $\Sigma^{E(m)}$ of a Parseval frame for the Hilbert space $H^0(\M;\Ei(m),\FS(\GH_m)\otimes P^E)$ regarded as a frame for the Hilbert $C^0(\M)$-module $\Gamma^0(\M;\Ei(m),\FS(\GH_m)\otimes P^E)$. Thus finding out how $\breve{\varsigma}^{(m)}(P^E)$ differs from the identity operator on $\GE_m\cong H^0(\M,\omega;\Ei(m),\FS(\GH_m)\otimes P^E)$ is analogous to finding the error between $\Sigma^{E(m)}$ and the identity endomorphism. The latter is given in the Szegö expansion \cite[Thm. 5.2]{Wang2} 
$$
\frac{\chi(\Ei(m))}{\vol(\M,\Li)\rank\Ei}\Sigma^{E(m)}=m^d\bone_E+m^{d-1}(\tr_\omega\Theta^E-s_\omega/2)\bone_E+O(m^{d-2}),
$$
where $s_\omega$ is the scalar curvature of the Kähler metric $\omega$. For $\breve{\varsigma}^{(m)}(P^E)$ the expansion would be one of operators on the Hilbert space $\GE_m$. Since $A_E$ is the limit of $\Psi^p(P_E)$ as $p$ goes to infinity, a first approximation to the compact operator $P_E-A_E$ is given by
\begin{align*}
(\id-\Psi)(P_E)&=\sum^n_{\alpha=1}[P_E,T_\alpha]^*[P_E,T_\alpha]
\\&=\sum^n_{\alpha=1}[T_{E,\alpha}^*,T_{E,\alpha}]+\sum^n_{\alpha=1}[T_\alpha^*,T_\alpha],
\end{align*}
which may be viewed as an operator ``second fundamental form'', while $\sum^n_{\alpha=1}[T_{E,\alpha}^*,T_{E,\alpha}]$ and $\sum^n_{\alpha=1}[T_\alpha^*,T_\alpha]$ are operator analogues of the mean curvature $\tr_\omega\Theta^E$ and the scalar curvature $s_\omega$ respectively. For $P_E-A_E$ itself we have:
\begin{thm}[Theorem \ref{VEsymbprop}]\label{hiddSzegintrothm}
In the setting of Theorem \ref{introbktoHardy} there is an injective completely positive map $\varsigma_{V_E}^{(m)}$ the space of operators on $\GE_m$ into $\End\Gamma^0(\M;\Ei)$ such that
$$
\varsigma^{(m)}_{V_E}(P_{E,m}-A_{E,m})=m^{-1}(\tr_\omega\Theta^E-\mu(\Ei))\bone_E+O(m^{-2}).
$$
\end{thm}
Here the curvature $\Theta^E$ is that of the Chern connection of $P^E$ and the holomorphic structure on $\Ei$ coming from the graded $\Ai$-module underlying $\GE_\N$, which by Theorem \ref{introbktoHardy} is the same as the holomorphic structure on $\Ei$ that we started with if we take $P^E$ to be real-analytic. 

%\textbf{So what is this saying?} Which holomorphic structure is $\Theta^E$ the Chern curvature for? Should $\tr_\omega\Theta^E-\mu(\Ei)\bone_E$ not be zero if $\varsigma(P_E)$ is always a Yang--Mills metric? No because the Hardy space uses the graded $\Ai$-module $E_\N$ underlying $\GH^E_\N$ and if $E_\N$ is not stable then $\varsigma(P_E)$ will be Yang--Mills for another holomorphic structure. 

\subsection{The nature of $\varsigma(P_E)$}

So for a projection $P^E$ which does not define a Yang--Mills metric, what is the meaning of the symbol $\varsigma(P_E)$ of the superharmonic projection $P_E:=\Ran\breve{\varsigma}(P^E)$? The Uhlenbeck--Yau theorem says that a holomorphic vector bundle $\Ei$ admits a Yang--Mills metric if and only if $\Ei$ is slope-stable \cite{Koba1, UhYa1}. Any holomorphic vector bundle $\Ei$ has a filtration by slope-semistable subsheaves, and each of these slope-semistable subsheaves has a filtration by slope-stable subsheaves. Taking successive quotients of the members of these filtrations and summing up one obtains a torsionfree sheaf $\Gr(\Ei)$ which is a direct sum of slope-stable vector bundles (see \cite[\S2.1]{Jaco2} for details). Each of these slope-stable summand of $\Gr(\Ei)$ admits a Yang--Mills metric, and the direct sum of these metrics will be referred to as the Yang--Mills metric on $\Gr(\Ei)$. We prove:
\begin{thm}[Theorem \ref{GrvarsigmaPEthm}]
Let $\Ei$ be a holomorphic vector bundle over $\M$ and suppose that $\Gr(\Ei)$ is locally free. Let $\GE_\N$ be the quotient module obtained by completing $E_\N:=\bigoplus_{m\in\N_0}H^0(\M;\Ei(m))$ in some embedding into $\GH_\N\otimes\C^N$, and let $P_E$ be the projection onto $\GE_\N$. Then the projection $\varsigma(P_E)$ defines $\Gr(\Ei)$ as smooth vector bundle, and $\varsigma(P_E)$ is the Yang--Mills metric on $\Gr(\Ei)$.
\end{thm}
We conjecture that in general, when $\Gr(\Ei)$ is merely torsionfree, $\varsigma(P_E)$ still gives a singular Yang--Mills metric on $\Gr(\Ei)$. To prove that one would need a generalization of Lemma \ref{Wanglemma} to singular Yang--Mills metrics and balanced metrics on torsionfree sheaves. 

When $\Gr(\Ei)$ is locally free the assumptions of Theorem \ref{hiddSzegintrothm} thus holds. The term $m^{-1}(\tr_\omega\Theta^E-\mu(\Ei))\bone_E$ in the statement of Theorem \ref{hiddSzegintrothm} need not be zero even though $\varsigma(P_E)$ is a Yang--Mills on $\Gr(\Ei)$, because $\Theta^E$ is the Chern connection for $\varsigma(P_E)$ and the holomorphic structure on $\Ei$, and not the holomorphic structure on $\Gr(\Ei)$. 

%\begin{Question}
%If $P_E$ is a superharmonic projection such that $\varsigma(P_E)$ is continuous, does it follow that $\varsigma(P_E)$ is Yang--Mills on the holomorphic vector bundle it defines? Well it does follow that $\Ei_{\rm CD}$ is locally free so it is isomorphic to the Serre sheaf $\Ei$ of the graded $\Ai$-module $E_\N$ underlying the image $\GH^E_\N$ of $P_E$. Thus $\Ei$ is also locally free. And so Theorem \ref{GrvarsigmaPEthm} gives that $\varsigma(P_E)$ is Yang--Mills on $\Gr(\Ei)$. But yeah, there is no reason to believe that the first term in the hidden Szegö expansion vanishes. 
%\end{Question}

\subsection{Guo-stability}

A slightly weaker notion of stability than slope-stability is Gieseker-stability. Also this notion of stability has been characterized by Wang using balanced metrics \cite{Wang1}. Namely, a holomorphic vector bundle $\Ei$ is Gieseker-stable if and only if for all $m\gg 0$ there exists a balanced metric on $\Ei(m)$. The distinction between Gieseker- and slope-stability is thus the convergence of the balanced metrics. We will give an operator-theoretic proof of half of Wang's theorem (see Theorem \ref{HalfWangthm}) and we will consider the failure of the convergence of the balanced metrics in Proposition \ref{GiesSOTlimprop}.   

Recall that a holomorphic vector bundle $\Ei$ over $\M$ is called \textbf{Gieseker-semistable} if for all proper analytic quotients $\Ei\to\Fi\to 0$ one has
\begin{equation}\label{usualGieseker}
\frac{\chi(\Fi(m))}{\rank\Fi}\geq\frac{\chi(\Ei(m))}{\rank\Ei},\qquad\forall m\gg 0.
\end{equation}
If we observe that $\rank\Fi/\rank\Ei=\lim_{l\to\infty}\chi(\Fi(l))/\chi(\Ei(l))$ then we can rewrite \eqref{usualGieseker} as
$$
\frac{\chi(\Fi(m))}{\chi(\Ei(m))}\geq\lim_{l\to\infty}\frac{\chi(\Fi(l))}{\chi(\Ei(l))},\qquad\forall m\gg 0.
$$
It is known \cite[Remark 2.2]{ACKi1} that \eqref{usualGieseker} is equivalent to 
\begin{equation}\label{usualGiesekerasGuo}
\frac{\chi(\Fi(m))}{\chi(\Ei(m))}\geq\frac{\chi(\Fi(l))}{\chi(\Ei(l))},\qquad\forall l\gg m\gg 0.
\end{equation}
We say that a holomorphic vector bundle $\Ei$ is \textbf{Guo-semistable} for each proper analytic quotient sheaf $\Fi$ the condition \eqref{usualGiesekerasGuo} holds for \emph{all} $l\geq 0$ (not just for $l$ sufficiently large compared to $m$):
\begin{equation}\label{Guointro}
\frac{\chi(\Fi(m))}{\chi(\Ei(m))}\geq\frac{\chi(\Fi(l))}{\chi(\Ei(l))},\qquad\forall l\geq m\gg 0.
\end{equation}
As usual we replace semistable by \textbf{stable} when strict inequality holds in \eqref{Guointro} for all $\Fi$. We use this terminology because it was shown in \cite[Prop. 2.3]{Guo3} that the trivial line bundle on $\C\Pb^{n-1}$ is Guo-stable in this sense. We shall prove:
\begin{thm}[Theorem \ref{finitelevelGies}]
Let $\Ei$ be an irreducible $\G$-equivariant vector bundle over the coadjoint orbit $\M=\G/\K$. Then $\Ei$ is Guo-stable. 
\end{thm}
We thereby see that the operator-theoretic result as it is stated in \cite[Prop. 2.3]{Guo3} extends to a much wider range of reproducing kernel Hilbert spaces.

In the end of the paper we will discuss some natural questions and open problems building on this work that would be interesting to investigate in the future.  

%The long-term goal with this research is to characterize slope-stability and Gieseker-stability using the operator tuples $S_E$, and to generalize this to nonholomorphic vector bundles where no sensible notion of stability is known. In that way we hope to prove that for any smooth projective variety $(\M,\Li)$ the smooth vector bundles with this special property on $S_E$ generate the topological $K$-theory group $K^0(\M)$.

\begin{Ack}
%I thank Robert Berman, Daniel Persson, and Magnus Goffeng for discussions on the topic of the paper. I thank Erlend Fornæss-Wold, Tuyen Troung, and Håkan Samuelsson Kalm for discussions about \S\ref{extvectsec}. I thank Mats Andersson, Bo Berndtsson, and Martin Sera for various discussions about holomorphic vector bundles. I thank Ramiz Reza for discussions about Cowen--Douglas operators. I thank Ulrik Enstad for discussions about frames and for comments on the paper. 
I thank Robert Berman, Daniel Persson, Magnus Goffeng, Erlend Fornæss-Wold, Tuyen Troung, Håkan Samuelsson Kalm, Mats Andersson, Bo Berndtsson, Martin Sera, Ramiz Reza, Dennis Eriksson, and Ulrik Enstad for discussions on the topic of the paper.
This work was initiated when I was a postdoc at the University of Oslo, supported by ERC (grant 307663-NCGQG).
\end{Ack}

\subsection{Notation}
If $\Ei$ is a holomorphic vector bundle on a smooth projective variety $\M$ then we denote by $H^0(\M;\Ei)$ the vector space of global holomorphic sections of $\Ei$. 
If $P^{E(m)}$ is a Hermitian metric on $\Ei(m)$ then we denote by $H^0(\M,\omega;\Ei(m),P^{E(m)})$, or just $H^0(\omega,P^{E(m)})$, the vector space $H^0(\M;\Ei)$ endowed with the inner product
\begin{equation}\label{GKEinnerprod}
\bra\phi|\psi\ket_{\omega,P^{E(m)}}:=\frac{\chi(\Ei(m))}{\rank\Ei}\omega((\phi|\psi)_{P^{E(m)}}),\qquad\forall\phi,\psi\in\Gamma^0(\M;\Ei(m)).
\end{equation}
%As usual, if $h^E_m$ is the pullback metric of the smooth embedding of $\Ei(m)$ into a Grassmannian given by a projection $P^{E,m}$, we sometimes identify $P^{E,m}$ with the Hermitian metric $h^E_m$ and so we write $H^0_{\bar{\pd}^E}(\M,\omega;\Ei(m),P^{E,m})$, or just $H^0_{\bar{\pd}^E}(\omega,P^{E,m})$, for the same Hilbert space $H^0_{\bar{\pd}^E}(\M,\omega;\Ei(m),h^E_m)$. 

The Euler characteristic of a coherent analytic sheaf $\Ei$ over $\M$ is denoted by $\chi(\Ei)$ and defined as the integer
$$
\chi(\Ei):=\sum^d_{p=0}(-1)^p\dim H^p(\M;\Ei),
$$
here $H^p(\M;\Ei)$ is the $p$th sheaf cohomology group of the $\Oi_\M$-module $\Ei$. If $\Ei$ is a holomorphic vector bundle then $\chi(\Ei(m))$ depends only on the isomorphism class of the topological vector bundle underlying $\Ei$. Indeed, the Hirzebruch--Riemann--Roch theorem says that $\chi(\Ei)$ is the pairing of the fundamental class of $\M$ with the Chern character of $\Ei$ wedged with the Todd class of the tangent bundle of $\M$.
We can take the latter formula as the definition of $\chi(\Ei)$ for an arbitrary smooth vector bundle $\Ei$ which need not admit any holomorphic structure.

With $\Li=\Oi_\M(m)$ a fixed very ample line bundle on $\M$, we set 
$$
\Ei(m):=\Li^m\otimes\Ei
$$
for each $m\in\N_0$. The \textbf{Hilbert polynomial} of a smooth vector bundle $\Ei$ is the polynomial $\N_0\ni m\to\chi(\Ei(m))$. %Sometimes we shall refer to the sequence of integers $\chi(\Ei(m))$ for $m\gg 0$ as the ``polytail'' of $\Ei$. 
While $\chi(\Ei(m))$ for a given $m$ only depends on the topological structure of $\Ei(m)$, the choice of line bundle $\Li$ is affected by the holomorphic structure of $\M$ since we want $\Li$ to be very ample. 

We will often write $c_{E,m}$ for the constant
$$
c_{E,m}:=\frac{n_m\rank\Ei}{\chi(\Ei(m))}
$$
where as always $n_m:=\dim\GH_m$ is the Hilbert polynomial of the trivial line bundle $\Oi_\M$.  

Let $\omega:C^0(\M)\to\C$ be a state, i.e. a functional with $\omega(\bone)=1$. If $P^E$ is a projection in $C^0(\M)\otimes\Mn_N(\C)$ then we denote by $L^2(\omega,P^E)$ the completion of $\Gamma^0(\M;P^E):=P^E(C^0(\M)\otimes\C^N)$ in the $L^2$-inner product of $P^E$ and $\omega$,
$$
\bra\phi|\psi\ket_{L^2(\omega,P^E)}:=\frac{\chi(\Ei(m))}{\rank\Ei}\omega((\phi|\psi)_{\Gamma^0(\M;P^E)}),\qquad\forall\phi,\psi\in\Gamma^0(\M;P^E).
$$

%Pulling $P^E$ back to a $\Un(1)$-equivariant function on $\Sb$ we can also consider $\Gamma^0(\Sb;P^E):=P^E(C^0(\Sb)\otimes\C^N)$ and the associated $L^2$-space $L^2(\Sb,\omega;P^E)$. We have decompositions 
%$$
%\Gamma^0(\Sb;P^E)=\bigoplus_{k\in\Z}\Gamma^0(\M;\FS(\GH_m)\otimes P^E)
%$$
%and $L^2(\Sb,\omega;P^E)=\bigoplus_{k\in\Z}L^2(\omega;\FS(\GH_m)\otimes P^E)$. 

\begin{Remark}[$\omega$ on matrices]\label{frameomegaremark}
For an element $B$ of $L^\infty(\M,\omega)\otimes\Bi(\Hi)$ for some separable Hilbert space $\Hi$ we can canonically define an operator $\omega(B)\in\Bi(\Hi)$ by defining $\omega(f\otimes X):=\omega(f)X$ on simple tensors $f\otimes X$ and extending $\omega$ by $\C$-linearity. A choice of basis (or more generally a countable Parseval frame) for $\Hi$ gives a representation of $B$ as an $L^\infty(\M,\omega)$-valued matrix (see \cite{Bala1}), and the above definition just means that we apply $\omega$ to each entry in such a matrix. So even if the Parseval frame has many more element than the dimension of $\Hi$, applying $\omega$ to each entry in the frame matrix of $B$ gives the same as if we apply $\omega$ to each entry in a matrix of size $\dim\Hi$ representing the action of $B$ in an orthonormal basis for $\Hi$.

In particular, $\omega$ commutes with taking the trace over $\Hi$,
$$
\omega(\Tr_\Hi(B))=\Tr_\Hi\omega(B),\qquad\forall B\in L^\infty(\M,\omega)\otimes\Bi(\Hi).
$$
Sometimes we write $(\omega\otimes\id)(B)$ for $\omega(B)$ for clarity when $B$ is in $L^\infty(\M,\omega)\otimes\Bi(\Hi)$, with $\id$ standing for the identity map on $\Bi(\Hi)$.
\end{Remark}

\section{Multivariable operator theory of $\G/\K$}\label{multsection}

\subsection{Preliminaries}
Let $n\geq 2$ be an integer. In this section we recall from \cite[\S6]{An6} how to associate a graded quotient $\GH_\N$ of $H^2_n$ to a compact matrix group
$$
\G\subset\Un(n).
$$
%The construction in \cite[\S6]{An6} works more generally for quantum groups but we shall not need that here.

\subsubsection{The first-row algebra}

Throughout the rest of the paper, $\G$ is a compact matrix group, i.e. $\G$ is a closed subgroup of the group $\Un(\GH)$ of unitray transformations of some finite-dimensional Hilbert space $\GH$. The $C^*$-algebra $C^0(\G)$ of continuous functions on $\G$ is generated by the matrix coefficients of a unitary matrix $u\in\Bi(\GH)\otimes C^0(\G)$. Set $n:=\dim(\GH)$ and fix an orthonormal basis $e_1,\dots,e_n$ of $\GH$ so that $\GH\cong\C^n$, and let $u_{\alpha,\beta}$ be the matrix coefficients of $u$ in this basis.  

\begin{dfn} The \textbf{first-row algebra} of $\G$ is the $C^*$-algebra $C^0(\Sb)$ generated by the first row $Z_1:=u_{1,1},\dots,Z_n:=u_{1,n}$. This defines the homogeneous space $\Sb$. 
\end{dfn}
Since $\G$ is a Lie group, $\Sb$ is a smooth manifold. 

There is a $\Z$-grading on $C^0(\Sb)$ obtained by letting the $Z_\alpha$'s have degree $1$ while their adjoints are given degree $-1$. We write the decomposition into spectral subspaces for the corresponding $\Un(1)$-action as
$$
C^0(\Sb)=\overline{\bigoplus_{k\in\Z}C^0(\Sb)^{(k)}}^{\|\cdot\|}.
$$
\begin{dfn} 
We define the homogeneous space $\G/\K$ as the manifold corresponding to the $C^*$-subalgebra of fixed points in $C^0(\Sb)$ for the $\Un(1)$-action:
$$
C^0(\G/\K):=C^0(\Sb)^{(0)}.
$$
\end{dfn}
It is clear that $C^0(\G/\K)$ is generated by the $n^2$ elements $\{Z_\alpha^*Z_\beta\}_{\alpha,\beta=1}^n$.

\begin{Example}
If $\G=\Un(n)$ is the whole unitary group then $\Sb$ is the unit sphere $\Sb^{2n-1}$ in $\C^n$ while
 $\G/\K$ is the complex projective space $\C\Pb^{n-1}$. Here we obtain $\K=\Un(1)\times\SU(n-1)$. 
\end{Example}
By the above example we see that, in general, $\Sb\subset\Sb^{2n-1}$ and
$$
\G/\K\subset\C\Pb^{n-1}.
$$
Since $\G$ is a Lie group, the space $\Sb$, and hence also $\G/\K=\Sb/\Un(1)$, is a smooth manifold and the action of $\G$ on $C^0(\G/\K)$ restricts to an action on the subalgebra $C^\infty(\G/\K)$ of smooth functions. 
The space $\Sb$ is a smooth principal $\Un(1)$-bundle over the smooth manifold $\G/\K$.

We denote by $\Irrep\G$ the set of equivalence classes of irreducible unitray representations of $\G$. We choose a representative $\GH_\lambda$ for each $\lambda\in\Irrep\G$, i.e. $\GH_\lambda$ is a (necessarily finite-dimensional) Hilbert space which carries an irreducible representation of $\G$ in class $\lambda$. 

An important part in the theory of compact groups is the \textbf{Peter--Weyl decomposition} \cite[Cor. 9.14]{Seg1} 
\begin{equation}\label{PeterforCG}
C^0(\G)=\bigoplus_{\lambda\in\Irrep\G}\Bi(\GH_\lambda)^*
\end{equation}
of the vector space underlying the $C^*$-algebra $C^0(\G)$, where $\Bi(\GH_\lambda)^*$ denotes the dual of $\Bi(\GH_\lambda)$ with respect to the trace. Recall also that $C^0(\G)$ has a unique state $\omega$ (the \textbf{Haar state}) which is invariant under the left and right translation action of $\G$ on $C^0(\G)$. The completion $L^2(\G)$ of $C^0(\G)$ in the inner product defined by the Haar state decomposes into irreducibles as well, because of \eqref{PeterforCG} and the $\G$-invariance of $\omega$.

\begin{Lemma}{\cite[Lemma 6.18]{An6}}\label{ergonfirstrow}
The first-row algebra $C^0(\Sb)$ carries an ergodic action of $\G$ which contains every irreducible representation of $\G$ with multiplicity one. The unique $\G$-invariant state on $C^0(\Sb)$ is the restriction to $C^0(\Sb)$ of the Haar state $\omega$ on $C^0(\G)$. 
\end{Lemma}
%\begin{proof} We define a left action of $\G$ on $C^0(\Sb)$ by restriction of the action by left translation of $\G$ on $C^0(\G)$. The Peter--Weyl decomposition \eqref{PeterforCG} of $C^0(\G)$ gives the decomposition
%$$
%C^0(\Sb)\cong\bigoplus_{\lambda\in\mathrm{Irrep}\G}\GH_\lambda,
%$$
%and the left translation restricts to the irreducible $\G$-representation on each $\GH_\lambda$.  
%\end{proof}
We write $\omega$ also for the $\G$-invariant states on $C^0(\Sb)$ and $C^0(\G/\K)$.

We have already seen that the topological space $\G/\K$ is contained in $\C\Pb^{n-1}$. Now observe that the algebra 
$$
\Ai:=\mathrm{Alg}(Z_1,\dots,Z_n)
$$
generated by $Z_1,\dots,Z_n$ is the quotient of the polynomial algebra $\C[z_1,\dots,z_n]$ by some homogeneous ideal. Therefore $\Ai$ is the homogeneous coordinate ring of some projective variety 
$$
\M:=\Proj(\Ai)\subset\C\Pb^{n-1}.
$$
That is, if $\Li=\Oi_\M(1)$ denotes the restriction to $\M$ of the hyperplane bundle on $\C\Pb^{n-1}$ then
$$
\Ai=\bigoplus_{m\in\N_0}H^0(\M;\Li^m).
$$
The completion $\GH_\N$ of $\Ai$ in the inner product of the symmetric Fock space $\GH^{\vee\N}$ is a graded quotient module. The elements of $\GH_\N$ are analytic functions on the manifold 
$$
\B:=\B^n\cap\V,\qquad\text{where }\V:=\Spec\Ai.
$$
It is shown in \cite{An6} that the $C^*$-algebra $C^0(\M)$ of continuous functions on $\M$ is the inductive limit of the finite-dimensional matrix algebras $\Bi(\GH_m)$ as $m$ goes to infinity, and that $C^0(\M)$ and $C^0(\G/\K)$ coincide in such a way that the generators $Z_1,\dots,Z_n$ of $C^0(\Sb)$ become the homogeneous coordinates on $\M$,
$$
\M=\G/\K.
$$
Thus $\G/\K$ is given the structure of a complex projective variety. 

The state $\omega:C^0(\M)\to\C$ defines a unique Kähler 2-form, also denoted by $\omega$, in the cohomology class $c_1(\Li)$ via
$$
\omega(f)=\frac{1}{\vol(\M,\Li)}\int_\M f(x) e^{\omega(x)},\qquad\forall f\in C^0(\M),
$$
where $\vol(\M,\Li)=\int_\M e^{\omega(x)}=\lim_{m\to\infty}\dim\GH_m/m^{\dim\M}$ is the volume of $(\M,\Li)$.

\subsubsection{Haar orthogonality relations}\label{Harrelsec}

Recall that the \textbf{Haar orthogonality relations} \cite[Thm. 9.7(iii)]{Seg1} say that if $\GK$ is an irreducible representation of $\G$ and $e_1,\dots,e_{n_\GK}$ is an orthonormal basis for $\GK$ then
$$
\delta_{\alpha\beta}=\bra e_\alpha|e_\beta\ket_{\GK}=(\dim\GK)\omega(f_\alpha^*f_\beta),
$$
where $f_\alpha\in C^0(\G)$ is the function $f_\alpha(a):=\bra e_\alpha|a\cdot e_\alpha\ket$. 
%It follows that if $(\xi_j)_{j\in\J}$ is a Parseval frame for $\GK$, i.e. if $\sum_{j\in\J}|\xi_j\ket\bra\xi_j|=\bone$, then
%$$
%\bra \xi_j|\xi_k\ket_{\GK}=\frac{1}{\dim\GK}\omega(f_j^*f_k)
%$$
%with $f_j(a):=\bra\xi_j|a\cdot\xi_j\ket$.
That is, in addition to the $C^0$ Peter--Weyl decomposition 
$$
C^0(\G)=\bigoplus_{\kappa\in\operatorname{Irrep}\G}\Bi(\GK_\kappa)
$$
(as vector spaces) one has that the Haar state $\omega$ restricts to the normalized trace on $\Bi(\GK_\kappa)$, viz. the $L^2$ Peter--Weyl decomposition
$$
L^2(\G,\omega)=\bigoplus_{\kappa\in\operatorname{Irrep}\G}\Bi(\GK_\kappa)
$$
(as Hilbert spaces). The first-row algebra takes the form
$$
C^0(\Sb)=\bigoplus_{\kappa\in\operatorname{Irrep}\G}\GK_\kappa,
$$ 
where the irreducible representations $\GH_m$ appear as special cases of the $\GK_\kappa$'s, namely as the subspaces spanned by the products of $m$ elemens of the generating set $\{Z_1,\dots,Z_n\}$. To obtain an arbitrary irreducible representation $\GK_\kappa$ one has to use also products with elements of $\{Z_1^*,\dots,Z_n^*\}$. 

For the special case of $\GK_\kappa=\GH_m$ for $m\in\N_0$, the Haar orthogonality relations give the following:
\begin{Lemma}\label{Haarfirstrolemma}
%Denote by $\omega$ also the restriction of the Haar state to the first-row algebra $C^0(\Sb)$. 
Let $e_1,\dots,e_n$ be an orthonormal basis for $\GH_1$. 
Then for all $m\in\N_0$ and all $\mathbf{j},\mathbf{k}\in\F_n^+$ with $|\mathbf{j}|=m=|\mathbf{k}|$ we have
$$
\omega(Z_\mathbf{j}Z_\mathbf{k}^*)=\frac{\bra e_\mathbf{j}|p_me_\mathbf{k}\ket}{\dim\GH_m}.
$$
\end{Lemma}
Here $\F_n^+$ denotes the set of multi-indices $\mathbf{k}=k_1\cdots k_m$ of finite length $|\mathbf{k}|:=m\in\N_0$, and $p_m:\GH^{\otimes m}\to\GH_m$ is the orthogonal projection. 
%\begin{proof}
%By definition $\GH_m$ is the a $\G$-subrepresentation of the tensor product representation $\GH^{\otimes m}$. 
%The family $(e_\mathbf{k})_{|\mathbf{k}|=m}$ is an orthonormal basis for $\GH^{\otimes m}$. Therefore, $(p_me_\mathbf{k})_{|\mathbf{k}|=m}$ is a Parseval frame for the subspace $\GH_m\subset\GH^{\otimes m}$. The Haar orthogonality relations say that the inner product on $\GH_m$ is, up to a factor of $n_m$, that of $\GH_m$ as a subspace of $L^2(\Sb,\omega)$. This gives the statement. 
%\end{proof}

\subsection{Subnormality and spherical expansivity}\label{spexpsec}
We are given our coadjoint orbit $\M=\G/\K\subset\C\Pb^{n-1}$ and the associated graded quotient $\GH_\N$ of the symmetric Fock space $\GH^{\vee\N}$. The shift on $\GH^{\vee\N}$ can be compressed to the subspace $\GH_\N$ to give a tuple $S=(S_1,\dots,S_n)$ of mutually commuting operators. If $e_1,\dots,e_n$ denotes an orthonormal basis for $\GH_1$, this means
$$
S_\alpha\psi=p_m(e_\alpha\otimes\psi),\qquad\forall \psi\in\GH_\N,\ \alpha\in\{1,\dots,n\},
$$
where $p_m:\GH^{\otimes m}\to\GH_m$ is the orthogonal projection, and if the vectors in $\GH_\N$ are identified with analytic functions on $\B$, then each $S_\alpha$ becomes a multiplication operator:
$$
(S_\alpha\psi)(w):=w_\alpha\psi(w),\qquad\forall\psi\in\GH_\N,\ w\in\B.
$$
We shall also refer to $\GH_\N$ as the Fock space and call $S$ the \textbf{shift} on $\GH_\N$. 

It can be helpful to view the shift $S$ as a quantization of the generating tuple $Z=(Z_1,\dots,Z_n)$ of the $C^*$-algebra $C^0(\Sb)$. Indeed, the graded algebra $\Ai\subset C^0(\Sb)$ generated by $Z_1,\dots,Z_n$ is isomorphic to the graded algebra generated by $S_1,\dots,S_n$. While the $Z_\alpha$'s commute with their adjoints, $[S_\alpha^*,S_\beta]$ is nonzero. 
One has
$$
SS^*:=\sum_{\alpha=1}^nS_\alpha S_\alpha^*=\bone-p_0,
$$
where $p_0\in\Bi(\GH_\N)$ denotes the projection onto the 1-dimensional subspace $\GH_0$ spanned by the constant functions. This is a remnant of the sphere condition $Z^*Z=ZZ^*=\bone$ satisfied by $Z$. 
So what about the operator $S^*S:=\sum_{\alpha=1}^nS_\alpha^*S_\alpha$? 
\begin{Lemma}\label{squareofSlemma}
 The shifts $S_1,\dots,S_n$ satisfy
\begin{equation}\label{formulaforSmodulus}
\sum_{\alpha=1}^nS_\alpha^*S_\alpha=\sum_{m\in\N_0}\frac{\dim\GH_{m+1}}{\dim\GH_m}p_m.
\end{equation}
\end{Lemma}
\begin{proof}
Let $\omega$ be the restriction of the Haar state on $C^0(\G)$ to the subalgebra $C^0(\Sb)$ and let $L^2(\Sb,\omega)$ be the GNS Hilbert space of $\omega$. Let $H^0(\Sb,\omega)$ be the closure of $\Ai$ in $L^2(\Sb,\omega)$. The operators of multiplication by $Z_1,\dots,Z_n$ on $L^2(\Sb,\omega)$ leave the subspace $H^0(\Sb,\omega)$ invariant. Denote by $T_1,\dots,T_n$ the restrictions of the multiplication operators $Z_1,\dots,Z_n$ to $H^0(\Sb,\omega)$ and let $P$ be the orthogonal projection of $L^2(\Sb,\omega)$ onto $H^0(\Sb,\omega)$. Since the representation of $C^0(\G)$ is a $*$-homomorphism, the multiplication operators satisfy $\sum^n_{\alpha=1}Z_\alpha^*Z_\alpha=\sum^n_{\alpha=1}Z_\alpha Z_\alpha^*=\bone$ and hence
\begin{equation}\label{Tisoneisom}
\sum^n_{k=1}T_k^*T_k=P\sum^n_{k=1}Z_k^*Z_k\big|_{H^0(\Sb,\omega)}=\bone.
\end{equation}
%Recall that $\GH_m$ is the irreducible representation of $\G$ spanned by the $Z_\mathbf{k}$'s with $|\mathbf{k}|=m$, endowed with the inner product of $\GH^{\otimes m}$. This is the same as the inner product on $\GH_m$ regarded as a subspace of the Fock space $\GH_\N$. Therefore, by the Haar orthogonality relations, 
By Lemma \ref{Haarfirstrolemma} the inner product $\bra\cdot|\cdot\ket$ on Fock space $\GH_\N$ is a simple scaling of that of $L^2(\Sb,\omega)$,
$$
\bra\phi|\psi\ket_{L^2(\Sb,\omega)}=\frac{1}{n_m}\bra\phi|\psi\ket,\qquad \forall\phi,\psi\in\Ai_m=\mathrm{span}\{Z_\mathbf{k}|\ \mathbf{k}\in\F_n^+,\ |\mathbf{k}|=m\}.
$$
It follows that the tuple $T_1,\dots,T_n$ is unitarily equivalent to the operator tuple $\tilde{T}_1,\dots,\tilde{T}_n$ on the Fock space $\GH_\N$ defined by
$$
\tilde{T}_\alpha|_{\GH_m}:=\sqrt{\frac{n_m}{n_{m+1}}}S_\alpha|_{\GH_m},\qquad\forall m\in\N_0,\ \alpha,\beta\in\{1,\dots,n\}.
$$
if $S_1,\dots,S_n$ are the standard shifts on $\GH_\N$. From \eqref{Tisoneisom} we get
$$
S_\alpha=|S|\tilde{T}_\alpha,\qquad\forall \alpha\in\{1,\dots,n\}.
$$
with $|S|:=\sqrt{\sum_{k=1}^nS_k^*S_k}$. The formula \eqref{formulaforSmodulus} then follows from \eqref{Tisoneisom} and the definition of $\tilde{T}_1,\dots,\tilde{T}_n$.
\end{proof}
We thus see that the Haar orthogonality relations ensure that the tuple $S=(S_1,\dots,S_n)$ is a simple quasi-affine transform of the spherical isometry $T=(T_1,\dots,T_n)$ acting on the Hardy-type space $H^0(\Sb,\omega)$. Recall that $T$ being a spherical isometry means $T^*T:=\sum^n_{\alpha=1}T_\alpha^*T_\alpha=\bone$, and that by \cite[Prop. 2]{Atha3} this is equivalent to saying that $T$ is subnormal with normal extension having joint spectrum in $\Sb^{2n-1}$ (in the present case the normal spectrum is $\Sb\subset\Sb^{2n-1}$).  

In the rest of the paper we will not distinguish between $\tilde{T}$ and $T$, so $T$ will sometimes be regarded as an operator tuple acting on $\GH_\N$. 

\begin{cor}\label{commisposcor}
The operator tuple $S=(S_1,\dots,S_n)$ is a \textbf{spherical expansion}, i.e. $\sum^n_{\alpha=1}S_\alpha^*S_\alpha\geq\bone$. Equivalently (since $\sum^n_{\alpha=1}S_\alpha S_\alpha^*=\bone-p_0$), the operator
$$
[S^*,S]:=\sum^n_{\alpha=1}[S^*_\alpha,S_\alpha]
$$
is positive.
\end{cor}
\begin{Question}
%Let $\GH_\N$ be a graded quotient module of $H^2_n$. Is the operator $[S^*,S]:=\sum^n_{\alpha=1}[S_\alpha^*,S_\alpha]$ positive? Our feeling is that the answer is positive and, using $\Tr([S^*,S]p_m)=\dim\GH_{m+1}-\dim\GH_m$, that would simplify the proof of Theorem \ref{scalaressthm} even further.
We thus have $[S^*,S]\geq 0$ when the quotient module $\GH_\N$ comes from a coadjoint orbit. Does that hold for a general quotient module? One could also ask whether each of the operators $S_\alpha$ is hyponormal, i.e. whether $[S_\alpha^*,S_\alpha]\geq 0$ holds for all $\alpha\in\{1,\dots,n\}$. The hyponormality of each $S_\alpha$ appears a bit optimistic even for coadjoint orbits (although it is true for $\GH_\N=H^2_n$ \cite[\S5]{Arv6c}). 
\end{Question}

The scalar curvature of the Kähler metric associated with the state $\omega$ on the coadjoint orbit $\M=\G/\K$ is a constant function $s_\omega=\underline{s_\omega}\bone$. For any quotient module $\GH_\N$ of $H^2_n$, i.e. for any projective variety $\M\subset\C\Pb^{n-1}$ and any Kähler metric in the class $c_1(\Li)$ of the line bundle $\Li=\Oi_\M(1)$, the average scalar curvature $\underline{s_\omega}:=\omega(s_\omega)$ appears, by Lemma \ref{squareofSlemma} and Hirzebruch--Riemann--Roch, as a contribution to the traces $\Tr([S^*,S]p_m)$,
$$
\phi_m([S^*,S])=\frac{n_{m+1}-n_m}{n_m}=m^{-1}\underline{s_\omega}/2+O(m^{-1}).
$$
From this behavior of the traces one may expect that $[S^*,S]$ is a quantization of the scalar curvature $s_\omega$ in some sense.

\subsection{Schatten-class membership}
In \cite{An6} it was shown that the commutators $[S_\alpha^*,S_\beta]$ of the shift operators are compact. Here we give a simpler proof in our special case of coadjoint orbits, and we also obtain sharp estimates for membership in the Schatten classes $\Li^p$:
\begin{thm}
Let $n\in\N$, let $\GH_\N=\bigoplus_{m\in\N_0}\GH_m$ be the graded quotient module of the Drury--Arveson space $H^2_n$ associated with a coadjoint orbit $\M=\G/\K\subset\C\Pb^{n-1}$, let $d:=\dim_\C\M$, and let $S_1,\dots,S_n$ be the compressions to $\GH_\N$ of the shift operators on $H^2_n$. Then for all $\alpha,\beta\in\{1,\dots,n\}$ we have
$$
[S_\alpha^*,S_\beta]\in\Li^p\iff p>d+1.
$$ 
\end{thm}
\begin{proof}
Recall that Corollary \ref{commisposcor} says that the operator $[S^*,S]:=\sum^n_{\alpha=1}[S^*_\alpha,S_\alpha]$ is positive. 
%Since $[S_\alpha^*,S_\alpha]\geq 0$ (i.e. $S_\alpha$ is ``hyponormal'') for each $\alpha\in\{1,\dots,n\}$, the operator $\Lambda$ is in $\Li^p$ if and only if $[S_\alpha^*,S_\alpha]$ is in $\Li^p$ for all $\alpha\in\{1,\dots,n\}$. 
We shall use the fact that $[S_\alpha^*,S_\beta]$ is in $\Li^p$ for all $\alpha,\beta\in\{1,\dots,n\}$ if and only if the operator $[S^*,S]$ is in $\Li^p$ (see \cite[Thm. 4.3]{Arv8} for a proof).% if we use a straightforward relation between $\Lambda$ and the ```defect operator'' $\bone-\sum^n_{\alpha=1}S_\alpha S_\alpha^*$ which is implicit in the formulas below).

Since $\Tr(p_m)=\dim\GH_m$, for $m\geq 1$ we have from Lemma \ref{squareofSlemma} that
$$
\Tr([S^*,S]p_m)=\dim\GH_{m+1}-\dim\GH_m,
$$
and $\N_0\ni m\to \dim\GH_{m+1}-\dim\GH_m$ is a polynomial of degree $d-1$ (because $\N_0\ni m\to\dim\GH_m\in\N$ is a polynomial of degree $d$). So for the normalized trace $\phi_m(\cdot):=\Tr(\ \cdot\  p_m)/\Tr(p_m)$ we have
$$
\phi_m([S^*,S])=O(m^{-1}).
$$
Thus the largest (and only) eigenvalue of $[S^*,S]p_m$ grows as $O(m^{-1})$. The eigenvalue of $[S^*,S]^pp_m$ then grows as $O(m^{-p})$. So we have $\Tr([S^*,S]^p)\sim\sum_{m\in\N}(\dim\GH_m)/m^p<\infty$ if and only if 
$$
\Tr([S^*,S]^p)\sim\sum_{m\in\N}\frac{1}{m^{p-d}}<\infty,
$$
which is the case if and only if $p>d+1$.
\end{proof}

\subsection{$(d+1)$-isometries}

\subsubsection{Background}\label{disombacksec}

For the moment let $S=(S_1,\dots,S_n)$ be an arbitrary commutating tuple of opertors on a Hilbert space $\Hi$ and consider the map $\Phi_*:\Bi(\Hi)\to\Bi(\Hi)$ defined by $\Phi_*(X):=\sum^n_{\alpha=1}S_\alpha^*XS_\alpha$. For each $p\in\N_0$ we define %(cf. \cite[\S2]{GlRi1}, \cite[\S2]{HoMa1})
$$
B_p(S):=(\id-\Phi_*)^p(\bone).%=\sum_{|\mathbf{k}|=q}(-1)^q{p\choose q}\frac{q!}{\mathbf{k}!}S_\mathbf{k}^*S_\mathbf{k}.
$$
%In \cite[\S2]{GlRi1} these operators are denoted by $P_p(S)$ but we will reserve this notation for the operators
%$$
%P_p(S):=(-1)^pB_p(S),
%$$
%which is how the operators $P_p(S)$ are defined in \cite[\S2]{HoMa1}. 
%We have the tautological formula \cite[Eq. (2.1)]{GlRi1}
%\begin{equation}\label{Bprecurs}
%B_p(S)=(\id-\Phi_*)(B_{p-1}(S)),
%\end{equation}
%as well as the relation \cite[Prop. 3.1]{HoMa1}
%$$
%\|(P_p(S)-P_{p-1}(S))\xi\|^2=\sum^n_{\alpha=1}\|P_p(S)S_\alpha\xi\|^2,\qquad\forall \xi\in\Hi.
%$$
With the convention ${m\choose p}:=0$ for $p>m$ we have \cite[Lemma 2.2]{GlRi1} 
\begin{equation}\label{PhimtoBp}
\Phi_*^m(\bone)%=\sum_{p\in\N_0}{m\choose p}P_p(S)
=\sum_{p\in\N_0}(-1)^p{m\choose p}B_p(S).
\end{equation}
\begin{dfn}[{\cite[\S2]{GlRi1}}]
Let $q\in\N$. A commuting operator tuple $S=(S_1,\dots,S_n)$ is a $q$-\textbf{isometry} if
$$
B_q(S)=0.
$$
%and in this case we always have $B_{q-1}(S)\geq 0$ \cite[Prop. 2.3]{GlRi1}. 
A $q$-isometry $S$ is \textbf{strict} if $B_{q-1}(S)\ne 0$. 
\end{dfn}
Using the tautological relation $B_p(S)=B_{p-1}(S)-\Phi_*(B_{p-1}(S))$ we can equivalently say that a tuple $S$ is a $q$-isometry if the operator $B_{q-1}(S)$ is a fixed point of $\Phi_*$,
$$
\Phi_*(B_{q-1}(S))=B_{q-1}(S).
$$

By \eqref{PhimtoBp} (see also \cite[Thm. 3.1]{HoMa1}), $S$ is a $q$-isometry if and only if there exists a degree-$(q-1)$ polynomial $\chi_S(m)=C_{q-1}m^{q-1}+C_{q-2}m^{q-2}+\cdots+C_0$ with operator coefficients $C_p\in\Bi(\Hi)$ such that
$$
\chi_S(m)=\Phi_*^m(\bone),\qquad\forall m\in\N_0.
$$
A 1-isometry is what is usually called a \textbf{spherical isometry}.

\subsubsection{Result}

%One can view the decomposition 
%$$
%S=T\Phi^*(\bone)^{1/2}=T(S^*S)^{1/2}
%$$
%as a ``polar decomposition" of the 1-expansion $S$. The fact that $|S|:=(S^*S)^{1/2}=\Phi^*(\bone)^{1/2}$ acts as a scalar on each component $\GH_m$ is equivalent to the fact that the Cauchy dual $S^\natural=|S|^{-2}S$ of $S$ is a commutative operator tuple. %$S$ commutes with its Cauchy dual $S^\natural=|S|^{-2}S$. 
%This can be viewed as a ``constrained" version of \cite[Thm. 1.10]{ChKu2}, \cite[Thm. 1.3]{ChKu3}, where by ``constrained" we mean that $S$ satisfy the relations of the embedding $\M\hookrightarrow\C\Pb^{n-1}$ with $\M$ not necessarily equal to $\C\Pb^{n-1}$. 

%The same thing is discussed in \cite[Thm. A.1]{Arv9}: the Hardy shift $T$ can be viewed as a weighted shift $\Delta S$ acting on $\Ai_\GH\subset\GH_\N$. This weighted shift $\Delta S$ is precisely $|S|^{-2}S$, the Cauchy dual of the Drury--Arveson shift. 

\begin{thm}\label{Cauchydualthm}
The $n$-tuple $S=(S_1,\dots,S_n)$ on $\GH_\N$ is a $(d+1)$-isometry. % and completely hyperexpansive. %, while its Cauchy dual $S^\natural=(S^*S)^{-1}S$ is jointly subnormal (or equivalently, completely hypercontractive). %In particular, $S$ is a 2-expansion, so that (see \cite[Lemma 3.5]{Chav3}) $S^\natural$ is a compact perturbation of $S$, i.e. there exists an $n$-tuple $K$ of compact operators such that
%$$
%S^\natural=S+K.
%$$
\end{thm}
\begin{proof}
%The tuple $S$ is quasisimilar to a 1-isometry, since $\Phi^*(\rho)=\rho$ for $\rho=\sum_mp_m/\Tr(p_m)$. The same should be true for $S^*$ since $\Phi(\rho)=\rho$ as well. Well, not really we have more like $\Phi(\bone)=\bone-|\Omega\ket\bra\Omega|$, but then how can $\Phi^\dagger$ intertwine $\phi_m$'s? 

Recall that by Lemma \ref{squareofSlemma}, $\sqrt{S^*S}$ is the central operator $\sum_mw_mp_m$ determined by the weight sequence $(w_m)_{m\in\N_0}$ given by
$$
w_m:=\sqrt{\frac{n_{m+1}}{n_m}}.
$$
Recall also that $n_m$ is a polynomial in $m\in\N_0$ of degree $d:=\dim_\C\M$. Motivated by the proof of \cite[Prop. 3.2]{BMN2}, \cite[Thm. 1]{AbLe1} we can deduce that $S$ is a $(d+1)$-isometry.

First recall that to check the $(d+1)$-isometric property we need to consider the powers $(\Phi^r_*(\bone))_{r\in\N_0}$ of the map $\Phi_*$ applied to the identity. We have
%$$
%\Phi_*^{(l-m)}(\bone)p_m=\frac{n_l}{n_m}p_m.
%$$
%Setting $r=l-m$ we rewrite this as
$$
\Phi^r_*(\bone)p_m=\frac{n_{m+r}}{n_m}p_m. %=w_m^2\cdots w_{m+r}^2p_m.
$$
Since
$$
B_p(S)=\sum^p_{r=0}(-1)^r{p\choose r}\Phi^r_*(\bone),
$$
%$$
%B_{d+1}(S)=\sum^{d+1}_{r=0}(-1)^rp{d+1\choose r}\Phi^r_*(\bone),
%$$
the tuple $S$ is a $q$-isometry iff
$$
0=\sum^q_{r=0}(-1)^r{q\choose r}\Phi^r_*(\bone).
$$
This is to say $0=\sum^q_{r=0}(-1)^r{q\choose r}\frac{n_{m+r}}{n_m}p_m$ for all $m\in\N_0$, which is equivalent to 
$$
0=\sum^q_{r=0}(-1)^r{q\choose r}n_{m+r},\qquad\forall m\in\N_0
$$
and holds iff $n_m$ is a polynomial in $m$ of degree $\leq q-1$.
\end{proof}

%It is well-known that the $(d+1)$-isometric property gives a constraint to the various parts of the Taylor spectrum. Let $\Sb\subset\Sb^{2n-1}$ be the preimage of $\M=\G/\K$ under the quotient map $\Sb^{2n-1}\to\C\Pb^{n-1}$ associated with the circle action on the unit sphere in $\C^n$.
%\begin{cor}
%The joint approximate point spectrum of $S$ is contained in $\Sb$. 
%\end{cor}
%\begin{proof}
%From \cite[Lemma 3.2]{GlRi1} we get that $\sigma_{\rm ap}(S)\subset\Sb^{2n-1}$. Now $S_1,\dots,S_n$ satisfy the relations of the homogeneous coordinates restricted to $\M$. The latter are precisely the coordinates on $\Sb$.  
%\end{proof}
There are other properties of $S$ encoded in the operators $B_p(S)$. We expect that when $\GH_\N$ is a coadjoint orbit, as assumed here, $S$ is a complete hyperexpansion, i.e.
$$
B_p(S)\leq 0,\qquad\forall p\in\N,
$$
but we do not know. When we replace $S$ by the 1-isometry $T$ we have:
\begin{prop}
Let $S=|S|T$ be the polar decomposition of the shift compressed to $\GH_\N$. Then $T$ is completely hypercontractive, i.e.
$$
B_p(T)\geq 0,\qquad\forall p\in\N.
$$
\textbf{But $B_p(T)=0$ for all $p$ since $T$ is a 1-isometry!!}
\end{prop}
\begin{proof}
By definition $T$ is a 1-isometry, hence subnormal with spectrum in $\mathrm{cl}\B\subset\mathrm{cl}\B^n$. The result is now given by \cite[Prop. 3.4]{ChCu1}. 
\end{proof}
As a special case of Proposition \ref{zerotraceprop} later in the paper we have for each $p\in\N_0$ that
$$
\Tr(B_p(S)p_m)%=\Tr((\id-\Phi_*)^p(P_{E,m}))
=\sum^p_{r=0}(-1)^r{p\choose r}n_{m+r},
$$
and in particular
$$
\Tr(B_1(S)p_m)=n_{m+1}-n_m
$$
%Note also that
%$$
%B_1(S)=\sum_{m\in\N_0}\Big(\frac{\Tr(p_{m+1})}{\Tr(p_m)}-1\Big)p_m\in\Bi(\GH_\N),
%$$
%and 
%$$
%\frac{\Tr(p_{m+1})}{\Tr(p_m)}-1=O(m^{-1}).%=s_\omega m^{-1}+O(m^{-2}).
%$$
For $p\geq d+1$ we have, from the fact that $n_m$ is a polynomial of degree $d$, that
$$
\Tr(B_p(S))=0.
$$
Since $B_1(S)p_m$ is a scalar we obtain recursively from the relation $B_p(S)=(\id-\Phi_*)(B_{p-1}(S))$ that $B_p(S)p_m$ is a scalar also for $p\geq 1$, for each $m\in\N_0$. This gives another proof of the vanishing $B_p(S)=0$ for $p\geq d+1$.

%\begin{Question}
%\begin{enumerate}[(i)]
%\item{If $S^{(d)}$ the maximal $d$-isometric piece of a $(d+1)$-isometry $S$, does it have $B_p(S^{(d)})$’s as the coefficient of the Hilbert polynomial of $\Y\cap\M$ where $\Y$ is a hyperplane in $\C\Pb^{n-1}$?}
%\item{}
%\item{}
%\end{enumerate}
%\end{Question}

%\subsection{Commutant lifting}

%Generalizing \cite[Lemma 2.5]{DaJe1} (the case of $\M=\C\Pb^{n-1}$) we have:
%\begin{prop}
%Let $\GH_\N\subset H^2_n$ be the quotient module associated with a coadjoint orbit $\M=\G/\K\subset\C\Pb^{n-1}$, and let $T=(T_1,\dots,T_n)$ be the canonical 1-isometry on the corresponding Hardy space $\GK_\N=H^0(\Sb)\subset L^2(\Sb)$. Suppose that $X\in\Bi(H^0(\Sb))$ satisfies
% $$
%\sum^n_{\alpha=1}T_\alpha^*XT_\alpha=X.
%$$
%Then there exists an $Y\in\Bi(L^2(\Sb))$ with $\|Y\|=\|X\|$, 
%$$
%\sum^n_{\alpha=1}Z_\alpha^*YZ_\alpha=Y,
%$$
%and $X$ is the compression of $Y$ to the subspace $H^0(\Sb)\subset L^2(\Sb)$.

%Similarly, if $\Hi_\Psi$ denotes the minimal Stinespring representation of the unital completely positive map $\Psi:\Bi(\GH_\N)\to\Bi(\GH_\N)$ then for any $X\in\Bi^\infty$ there exists an $Y\in\Bi(\Hi_\Psi)$ with $\|Y\|=\|X\|$, 
%$$
%?=Y,
%$$
%and $X$ is the compression of $Y$ to the subspace $\GH_\N\subset\Hi_\Psi$.
%\end{prop}
%\begin{proof}
%Both statements are special cases of Prunaru's.
%\end{proof}

\subsection{SOT-Toeplitz operators}

\subsubsection{The $L^\infty$ Toeplitz algebra $C^*(\breve{\varsigma}(L^\infty(\Sb)))$}
Again we consider the polar decomposition $S=|S|T$ of the shift. By Lemma \ref{formulaforSmodulus} the tuples $S$ and $T$ have the same invariant graded subspaces. The fact that $T$ is a spherical isometry says that the completely positive map $\Psi:\Bi(\GH_\N)\to\Bi(\GH_\N)$ defined by
$$
\Psi(X):=\sum^n_{\alpha=1}T_\alpha^*XT_\alpha,\qquad\forall X\in\Bi(\GH_\N)
$$
is unital.
Let $\Bi(\GH_\N)^\Psi$ be the fixed-point set of the map $\Psi$. We say that an operator in $\Bi(\GH_\N)^\Psi$, i.e. an operator $X$ with $\Psi(X)=X$, is $\Psi$-\textbf{harmonic}. And if $X\in\Bi(\GH_\N)$ satisfies $\Psi(X)\leq X$ then we say that $X$ is \textbf{$\Psi$-superharmonic}. If $X$ is $\Psi$-superharmonic and $\mathrm{SOT-}\lim_{p\to\infty}\Psi^p(X)=0$ then $X$ is called a $\Psi$-\textbf{potential} (or \textbf{pure} $\Psi$-superharmonic). 

By \cite[Thm. 3.3]{GaKu1} or \cite[Thm. 3.1]{Pop7}, if $X$ is $\Psi$-superharmonic then it has a \textbf{Riesz decomposition} into the sum of a unique $\Psi$-harmonic operator $X_1$ and a unique $\Psi$-potential $X_2$,
$$
X=X_1+X_2.
$$

%In \cite{An6} we only discussed the $\Un(1)$-invariant part 
%$$
%\Bi^\infty:=\Bi(\GH_\N)^\Psi\cap\Gamma_b
%$$
%of $\Bi(\GH_\N)^\Psi$, and we showed that $\Bi^\infty$ is completely isometrically isomorphic to $L^\infty(\M)$ via a Toeplitz-type map 
%$$
%\breve{\varsigma}:L^\infty(\M)\to\Bi^\infty.
%$$
%We will see that $\breve{\varsigma}$ extends to a map
%$$
%\breve{\varsigma}:L^\infty(\Sb)\to\Bi(\GH_\N)^\Psi
%$$
%which is again a complete isometry. 
Consider the von Neumann algebra generated by the operators in $\Bi(\GH_\N)^\Psi$ (the $L^\infty$ \textbf{Toeplitz algebra})
$$
\Li:=C^*(\Bi(\GH_\N)^\Psi). %,\qquad \Ni:=C^*(\Bi^\infty)=\Li\cap\Gamma_b.
$$
The following lemma can then be deduced directly from \cite[Thm. 1.2]{Prun1}:
\begin{Lemma}\label{LinftyToeplwhole}
There is a short exact sequence of $C^*$-algebras
\begin{equation}\label{vNSES}
0\to\Si\Ci(\Li)\to\Li\overset{\varsigma}{\to}L^\infty(\Sb)\to 0
\end{equation}
with a unital completely positive splitting 
$$
\breve{\varsigma}:L^\infty(\Sb)\to\Li
$$
whose image equals $\Bi(\GH_\N)^\Psi$. Here $\Si\Ci(\Li)=\Ker\varsigma$ is the semicommutator ideal in $\Li$, i.e. the two-sided ideal generated by the operators $[\breve{\varsigma}(f),\breve{\varsigma}(g)):=\breve{\varsigma}(f)\breve{\varsigma}(g)-\breve{\varsigma}(fg)$ with $f,g\in L^\infty(\M)$. So we have a direct sum of vector spaces
$$
\Li=\breve{\varsigma}(L^\infty(\Sb))+\Si\Ci(\Li)=\Bi(\GH_\N)^\Psi+\Si\Ci(\Li).
$$
Endowed with the Choi--Effros multiplication the operator system $\Bi(\GH_\N)^\Psi$ becomes a von Neumann algebra isomorphic via $\breve{\varsigma}$ to $L^\infty(\Sb)$,
$$
\mathrm{SOT-}\lim_{m\to\infty}\Psi^m(\breve{\varsigma}(f)\breve{\varsigma}(g))=\breve{\varsigma}(fg),\qquad\forall f,g\in L^\infty(\Sb).
$$
For every $X\in\Li$ there is a unique $\Psi$-harmonic operator $\breve{\varsigma}(f_X)\in\Bi(\GH_\N)^\Psi$ such that
\begin{equation}\label{SOTToeplinLi}
\mathrm{SOT-}\lim_{m\to\infty}\Psi^m(X)=\breve{\varsigma}(f_X)
\end{equation}
and the map $\varsigma$ can be described as
$$
\varsigma(X)=f_X.
$$
If $U:\GH_\N\to H^0(\Sb,\omega)$ is the unitary which intertwines the multiplication tuple $T$ on the Hardy space $H^0(\Sb,\omega)$ with the weighted shift $\tilde{T}=U^{-1}TU$ on $\GH_\N$ as in the proof of Lemma \ref{squareofSlemma} then 
%(which is the tuple of multiplication operators by the generators $Z_1,\dots,Z_n$ of $C^0(\Sb)$), the map $\breve{\varsigma}$ can be described as 
$$
\breve{\varsigma}(f)=U^{-1}(Pf|_{H^0(\Sb,\omega)})U,\qquad \forall f\in L^\infty(\Sb)
$$
where $P$ is the orthogonal projection of $L^2(\Sb,\omega)$ onto $H^0(\Sb,\omega)$ and we identify $L^\infty(\Sb)$ in its $*$-representation on $L^2(\Sb,\omega)$. 
\end{Lemma}
%\begin{proof}
%The completely positive map $X\to\Psi(X)=\sum^n_{\alpha=1}T_\alpha^*XT_\alpha$ is unital and the $T_\alpha$'s mutually commute. The whole lemma can then be deduced directly from \cite[Thm. 1.2]{Prun1}.
%We therefore get the short exact sequence \eqref{vNSES} from \cite[Thm. 1.2]{Prun1} and the fact that $T$ is unitarily equivalent to the multiplication tuple on $H^0(\Sb,\omega)$ (recall the proof of Lemma \ref{squareofSlemma}). 
%\end{proof}
By analogy with \cite{BaHa1, Fein1} we say that an operator $X\in\Bi(\GH_\N)$ is \textbf{SOT-asymptotic Toeplitz} if the limit $\mathrm{SOT-}\lim_{m\to\infty}\Psi^m(X)$ exists. In this case $\mathrm{SOT-}\lim_{m\to\infty}\Psi^m(X)$ is fixed by $\Psi$ so it must be of the form $\breve{\varsigma}(f_X)$ for some $f_X\in L^\infty(\Sb)$. Then $f_X$ is the \textbf{symbol} of $X$.
Lemma \ref{LinftyToeplwhole} says that every element of $\Li$ is SOT-asymptotic Toeplitz. 
\begin{prop}\label{SOTvsprop}
The SOT-asymptotic Toeplitz symbol map $\varsigma^{\rm SOT}$ coincides with $\varsigma$ when restricted to $\Li$. 
\end{prop}
\begin{proof}
We have $\Li=\breve{\varsigma}(L^\infty(\Sb))+\Ker\varsigma$ and the conditional expectation onto $\breve{\varsigma}(L^\infty(\Sb))$ just kills $\Ker\varsigma$. Thus $\Ker\varsigma$ is the set of elements $C$ of $\Li$ with $\varsigma^{\rm SOT}(C)=0$.
\end{proof}
\begin{cor}
Suppose that $\breve{\varsigma}(f)$ is a Toeplitz operator with SOT-Toeplitz symbol zero,
$$
\mathrm{SOT-}\lim_{m\to\infty}\Psi^m(\breve{\varsigma}(f))=0.
$$
Then $f=0$. 
\end{cor}
Note that $\Ker\varsigma^{\rm SOT}$ is not contained in $\Li$ however, and in particular there are SOT-asymptotic Toeplitz operators on $\GH_\N$ which do not belong to $\Li$. 

\begin{prop}\label{LiBeurlingprop}
Let $\Mi_\GH$ the WOT-closed algebra generated by $\bone,S_1,\dots,S_n$ (this is the multiplier algebra of the Hilbert function space $\GH_\N$) and let $\Mi_\GH\Mi_\GH^*$ be the set of operators of the form $\xi\eta^*$ with $\xi,\eta\in\Mi_\GH$. Then  
\begin{equation}\label{factorablevssuperharm}
\overline{\operatorname{span}}^{\rm WOT}_\C\Mi_\GH\Mi_\GH^*=\operatorname{span}_\C\{X\in\Bi(\GH_\N)|\ X\geq 0,\ \Phi(X)\leq X\}
\end{equation}
and this is a von Neumann algebra equal to $\Li$,
\begin{equation}\label{LiBeurling}
\Li=\overline{\operatorname{span}}^{\rm WOT}_\C\Mi_\GH\Mi_\GH^*.
\end{equation}
\end{prop}
\begin{proof}
Since $\Phi$ is pure we have (by \cite[Prop. 1.6]{Arv7c} or \cite[Cor. 3.10]{Pop7}) that an operator $X\in\Bi(\GH_\N)$ satisfies $X\geq 0$ and $\Phi(X)\leq X$ if and only if $X$ is ``factorable" in the sense that 
$$
X=LL^*
$$
with some $\Ai$-module map $L$ from $\GH_\N\otimes\GM$ and $\GH_\N$, for some separable Hilbert space $\GM$. Suppose then that $X$ is factorable like this. The fact that $L$ is a module map means precisely that if $(e_j)_{j\in\J}$ is an orthonormal basis for $\GM$ then the vectors
$$
\xi_j:=L(\bone\otimes e_j),\qquad\forall j\in\J
$$ 
belong to $\Mi_\GH$. Moreover, we have
\begin{equation}\label{BeurlingX}
X=\mathrm{SOT-}\sum_{j\in\J}\xi_j\xi_j^*
\end{equation}
so that $X$ belongs to $\overline{\operatorname{span}}^{\rm WOT}_\C\Mi_\GH\Mi_\GH^*$. Conversely, for $X\in\overline{\operatorname{span}}^{\rm WOT}_\C\Mi_\GH\Mi_\GH^*$ with $X\geq 0$ we have an expansion \eqref{BeurlingX} and this gives $\Phi(X)\leq X$ since $\Phi(\bone)\leq\bone$ gives
$$
\Phi\Big(\sum_{j\in\J}\xi_j\xi_j^*\Big)=\sum_{\alpha=1}^n\sum_{j\in\J}S_\alpha\xi_j\xi_j^*S_\alpha^*=\sum_{j\in\J}\xi_j\Phi(\bone)\xi_j^*\leq\sum_{j\in\J}\xi_j\xi_j^*.
$$
So Equality \eqref{factorablevssuperharm} holds.

We next prove that \eqref{LiBeurling} holds. The quotient $\Li/\Ker\varsigma$ is isomorphic to $L^\infty(\Sb)$, which is commutative. Therefore every element of $\Li$ can be normally ordered modulo $\Ker\varsigma$, i.e. every element of $\Li$ belongs to $\overline{\operatorname{span}}^{\rm WOT}_\C\Mi_\GH\Mi_\GH^*$ up to some term in $\Ker\varsigma$.  Now observe that the set $\Mi_\GH^*\Mi_\GH$ of antinormally ordered products of multipliers is precisely the set of fixed points under $\Psi$, since  $\Psi(\bone)=\bone$. 
%Since 
%$$
%\Mi_\GH^*\Mi_\GH=\Bi(\GH_\N)^\Psi\subset C^*(\Bi(\GH_\N)^\Psi)=\Li,
%$$
%we get $\overline{\operatorname{span}}^{\rm WOT}_\C\Mi_\GH\Mi_\GH^*\subset\Li$.
Therefore 
$$
\Mi_\GH^*\Mi_\GH=\Bi(\GH_\N)^\Psi=\Ran\breve{\varsigma}=(\Ker\varsigma)^\perp,
$$
and so $\Ker\varsigma$ is contained in $\overline{\operatorname{span}}^{\rm WOT}_\C\Mi_\GH\Mi_\GH^*$. Therefore $\overline{\operatorname{span}}^{\rm WOT}_\C\Mi_\GH\Mi_\GH^*$ is an algebra, and thus a von Neumann algebra which contains $\Bi(\GH_\N)^\Psi$, whence it must equal $\Li=C^*(\Bi(\GH_\N)^\Psi)$.
\end{proof}

\begin{Remark}[Uniform Toeplitz operators]
Again by analogy with \cite{Fein1} we say that an operator $X\in\Bi(\GH_\N)$ is \textbf{uniform asymptotic Toeplitz} if the limit $\lim_{m\to\infty}\Psi^m(X)$ exists in the norm topology on $\Bi(\GH_\N)$. As a generalization of pre-existing proofs in the literature (cf. \cite[Lemma 3.1]{CuLe1}) we can deduce that every compact operator $K$ is uniform asymptotic Toeplitz with symbol
$$
f_K=0.
$$
First observe that, since the polynomials are dense in $\GH_\N$ and since every compact operator is a norm-limit of finite-rank operators, it sufficies to show that
$$
\lim_{m\to\infty}\|\Psi^m(|\xi\ket\bra\eta|)\|=0.
$$
for any rank-1 operator $K=|\xi\ket\bra\eta|$ where $\xi,\eta\in\Ai\subset\GH_\N$ are polynomials. Next, $\Psi^m(K)$ is a finite sum of operators of the form $T_\mathbf{k}^*KT_\mathbf{k}$ for multi-indices $\mathbf{k}\in\F_n^+$ with $|\mathbf{k}|=m$. Now
$$
T_\mathbf{k}^*|\xi\ket\bra\eta|T_\mathbf{k}=|T_\mathbf{k}^*\xi\ket\bra T_\mathbf{k}^*\eta|.
$$
If $\deg\xi=m_0$ then $T_\mathbf{k}^*\xi=0$ for all $\mathbf{k}\in\F_n^+$ with $|\mathbf{k}|>m_0$. Thus, for all $m\geq\min\{\deg\xi,\deg\eta\}$ we obtain $\Psi^m(|\xi\ket\bra\eta|)=0$, as desired.
%But are there any other uniform asymptotic Toeplitz operators?
\end{Remark}
From \cite[Thm. 1.2]{Prun1} and \cite{An6} we also obtain a short exact sequence of $C^*$-algebras
\begin{equation}\label{vNSES}
0\to\Ki(\GH_\N)\to\Ti_\GH\overset{\varsigma}{\to}C^0(\Sb)\to 0
\end{equation}
where $\Ki(\GH_\N)$ is the ideal of compact operators on the Fock space $\GH_\N$, and this sequence is also linearly split by the Toeplitz map $\breve{\varsigma}$. That is, every element of $\Ti_\GH$ is of the form $\breve{\varsigma}(f)+K$ with $f\in C^0(\Sb)$ and $K$ compact. This shows that every element of $\Ti_\GH$ is uniform asymptotic Toeplitz. 
\begin{Remark}[Fixed points]
As mentioned, the $(d+1)$-isometric property of $S$ can be expressed as $\Phi_*(B_d(S))=B_d(S)$. Each operator $B_p(S)$ is compact. While $\Psi$ has no compact fixed points, $\Phi_*$ thus has. 
\end{Remark}

%\subsubsection{The $L^\infty$ Toeplitz algebra $\Ni=C^*(\breve{\varsigma}(L^\infty(\M)))$}

If $X$ is an operator on $\GH_\N$ that preserves the grading then one has $X=\sum_{m\in\N_0}X_mp_m$ with $X_m\in\Bi(\GH_m)$ for each $m$. The algebra $\Gamma_b$ of grading-preserving bounded operators on $\GH_\N$ can therefore be identified as 
$$
\Gamma_b=\prod_{m\in\N_0}\Bi(\GH_m).
$$
It is a von Neumann algebra admitting a finite trace $\sum_{m\in\N_0}\phi_m$. The grading on $\GH_\N$ gives rise to an action of $\Un(1)$ on $\Bi(\GH_\N)$ whose fixed-point subalgebra is given by $\Gamma_b$. 

In \cite{An6} we only discussed the $\Un(1)$-invariant part 
$$
\Bi^\infty:=\Bi(\GH_\N)^\Psi\cap\Gamma_b
$$
of the operator system $\Bi(\GH_\N)^\Psi$, and we showed that $\Bi^\infty$ is completely isometrically isomorphic to $L^\infty(\M)$ via an explictly Toeplitz-type map 
$$
\breve{\varsigma}:L^\infty(\M)\to\Bi^\infty.
$$
The notation for the map $\breve{\varsigma}:L^\infty(\Sb)\to\Bi(\GH_\N)$ appearing in Lemma \ref{SOTToeplinLi} was chosen because its defining property \eqref{SOTToeplinLi} shows that it restricts to the Toeplitz map $\breve{\varsigma}:L^\infty(\M)\to\Gamma_b$. Therefore there is also no confusion if we refer to $\breve{\varsigma}:L^\infty(\Sb)\to\Bi(\GH_\N)$ as the Toeplitz map. This is in accordance with \cite{Prun1, Prun2} and the single-variable theory of ``generalized Toeplitz operator". %, but observe that $T_1,\dots,T_n$ are only \emph{unitarily equivalent} to compressions of the generatrors of a $*$-representation of $C^0(\Sb)$ to an \emph{invariant} subspace. Thus the ``Toeplitz" terminology is indeed in a slightly generalized sense now that we have already fixed a $*$-structure on $C^0(\Sb)$.

%Let $\Bi^\infty$ be the fixed-point set of the unital completely positive map $\Psi:\Gamma_b\to\Gamma_b$. We shall now look at the $C^*$-algebra $\Ni$ generated by $\Bi^\infty$ and its relation to $\Psi$-superharmonic operators. Note that $\Ni$ is the subalgebra of elements in $\Li$ which preserve the grading on $\GH_\N$. 
Let $\Ni:=C^*(\Bi^\infty)$ be the $C^*$-algebra generated by the operators in $\Bi^\infty$. Then $\Ni$ is a von Neumann subalgebra of $\Li$ which we call the \textbf{$L^\infty$ Toeplitz core}. 
Taking the ``$\Un(1)$-invariant part'' of Proposition \ref{LiBeurlingprop} we obtain a presentation of $\Ni$ in terms of factorable or superharmonic elements,
$$
\Ni=\overline{\operatorname{span}}^{\rm WOT}_\C(\Mi_\GH\Mi_\GH^*)\cap\Gamma_b=\operatorname{span}_\C\{X\in\Gamma_b|\ X\geq 0,\ \Phi(X)\leq X\}.
$$
With these extra facts the following is implicit in \cite{An6}:
\begin{cor}\label{LinftyToepl}
The von Neumann algebra $\Ni:=C^*(\Bi^\infty)$ fits into the short exact sequence of $C^*$-algebras
\begin{equation}\label{vNSEScore}
0\to[\Bi^\infty,\Bi^\infty)\to\Ni\overset{\varsigma}{\to}L^\infty(\M)\to 0
\end{equation}
where the kernel $[\Bi^\infty,\Bi^\infty)=\Ker(\varsigma)$ is the semicommutator ideal in $\Ni$. The Toeplitz map $\breve{\varsigma}:L^\infty(\M)\to\Ni$ gives a completely positive splitting, so that
$$
\Ni=\Bi^\infty+[\Bi^\infty,\Bi^\infty),
$$
and $\varsigma$ is the adjoint of the Toeplitz map $\breve{\varsigma}:L^\infty(\M)\to\Gamma_b$ with respect to the inner product of $L^2(\M,\omega)$ and the inner product on $\Gamma_b$ defined by the state $\sum_m\phi_m$.
%Moreover, Cesàro summation is not needed in the formula for the conditional expectation $\Psi:\Ni\to\Bi^\infty$
\end{cor}
%\begin{proof}
%The defining property \eqref{SOTToeplinLi} of the map $\breve{\varsigma}:\Li\to L^\infty(\Sb)$ shows that it restricts to the Toeplitz map $\breve{\varsigma}:L^\infty(\M)\to\Gamma_b$ which we defined in \cite{An6} and which equals the adjoint of the covariant symbol map $\varsigma:\Ni\to L^\infty(\M)$. Therefore we get the exact sequence \eqref{vNSEScore} by taking the $\Un(1)$-invariant part of \eqref{LinftyToeplwhole}. 
%\end{proof}

\subsubsection{Projections onto vector-valued quotient modules}
We now look at graded subspaces $\GH^E_\N$ of the graded Hilbert space $\GH_\N\otimes\C^N$ for some integer $N\in\N$. For ease of notation we write
$$
S_\alpha:=S_\alpha\otimes\bone_N
$$
for the diagonal representation of the shift operators on $\GH_\N\otimes\C^N$. We shall characterize coinvariance of $\GH^E_\N$ in terms of the orthogonal projection $P_E$ of $\GH_\N\otimes\C^N$ onto $\GH^E_\N$. It will often be convenient to regard $P_E$ and other operators on $\GH_\N\otimes\C^N$ as $N\times N$-matrices with entries in $\Bi(\GH_\N)$. Then we can apply the maps $\Phi$ and $\Psi$ entrywise to such operators. Thus we define
$$
\Phi(X):=(\Phi\otimes\id_N)(X),\qquad\forall X\in\Bi(\GH_\N)\otimes\Mn_N(\C)=\Bi(\GH_\N\otimes\C^N)
$$
and so on. Recall that $\Gamma_b\subset\Bi(\GH_\N)$ is the von Neumann algebra of grading-preserving operators on $\GH_\N$.

%\begin{prop}
%Every $\Psi$-superharmonic operator $X$ on $\GH_\N\otimes\C^N$ has entries in $\Ni$. 
%\end{prop}
%Thus, while an arbitrary operator $C$ on $\GH_\N\otimes\C^N$ with $\mathrm{SOT-}\lim_{m\to\infty}\Psi^m(C)=0$ need not have entries in $\Ni$, if we in addition require $\Psi(C)\leq C$ then $C$ must be over $\Ni$.
%\begin{proof}
%The proof would be based on showing that $\Phi(X)\geq X$ is equivalent to $\Psi(X)\leq X$. 
%\end{proof}
If $\phi$ is a bounded multiplier of $\GH_\N$ then we denote by $M_\phi$ the associated bounded operator on $\GH_\N$,
$$
(M_\phi\psi)(v):=\phi(v)\psi(v),\qquad\forall\psi\in\GH_\N,\ v\in\B.
$$
\begin{Lemma}\label{PEinNilemma}
For a graded projection $P_E$ acting on $\GH_\N\otimes\C^N$, the following are equivalent:
\begin{enumerate}[(a)]
%\item{$P_E$ is over $\Ni$.}
\item{$\bone-P_E$ has a Beurling factorization, i.e. there is a multiplier $\Theta_E:\GH_\N\otimes\ell^2(\N_0)\to\GH_\N\otimes\C^N$ such that
$$
\bone-P_E=M_{\Theta_E}M_{\Theta_E}^*=\mathrm{SOT-}\sum_{j\in\N}M_{\phi_j}M_{\phi_j}^*
$$
with $\phi_j\in\Ai_{m_j}\otimes\C^N$ a homogeneous polynomial for each $j\in\N$ (where $\phi_j=0$ is allowed).}
\item{$\Phi(\bone-P_E)\leq\bone- P_E$, in which case $\mathrm{SOT-}\lim_{m\to\infty}\Phi^m(\bone-P_E)=0$.}
\item{The image $\GH^E_\N$ of $P_E$ is invariant under $S^*$.}
\item{$\Psi(P_E)\leq P_E$.}
\item{$P_E=\breve{\varsigma}(P^E)+C_E$ where $P^E$ is a projection over $L^\infty(\M)$ and $C_E$ is a $\Psi$-potential.}
\end{enumerate}
In this case $C_E$ has entries in $\Ni$ (since $P_E$ has, by (a)).
\end{Lemma}
%Note that any operator $C$ with $\Psi(C)\leq C$ and $\mathrm{SOT-}\lim_{m\to\infty}\Psi^m(C)=0$ can be added to $\breve{\varsigma}(P^E)$ to produce a
\begin{proof}
(a)$\iff$(b) is shown in \cite[Prop. 1.6]{Arv7c} or \cite[Cor. 3.10]{Pop7}, using $\mathrm{SOT-}\lim_{m\to\infty}\Phi^m(\bone)=0$. More precisely, that the $\phi_j$'s can be taken to be homogeneous polynomials is in \cite[Lemma 2.2]{Zhao1} (the proof there generalizes immediately to arbitrary $N$; see also \cite[Cor. 2 to Prop. 8.13]{Arv7b} for $N=1$). 
%Seems to be clear if we can take the $\phi_j$'s to be eigenvectors of the core operator because then they will belong to $E_m\subset\Ai_m\otimes\C^N$. ). 

If $\Phi(\bone-P_E)\leq\bone- P_E$ then $\GH^E_\N$ is invariant under $S^*$ by \cite[Lemma 4.1]{Pop7} (using $\Phi(\bone)\leq\bone$), so (c)$\implies$(b). 

Conversely, if $\GH^E_\N$ is coinvariant then we have $\Phi(\bone-P_E)\leq\bone- P_E$ (see \cite[Eq. (2.4]{Arv7c}). %Since $\Phi$ is unital outside the vacuum this gives $\Phi(P_E)\geq P_E$. 
So (b)$\implies$(c). 

%Conversely, if $\GH^E_\N$ is coinvariant then, since $\Phi$ is unital outside the vacuum, we get that $\Phi(P_E)\leq P_E$ by \cite[Cor. 4.4]{Pop7}. So (b)$\implies$(c) (this fact was stated in \cite[Eq. (2.4]{Arv7c} with a proof which only uses that $\Phi$ is contractive. Then why does Popescu not state it for any contractive map? Because he only found a proof which uses unitality or because Arveson is wrong? Need to check Kubrusly for the one-variable case of invarant subspaces of a contraction. Arveson's proof looks ok though. So only the characterization of reducing subspaces really needs unitality). Note however that $\Psi$ is contractive iff it is unital, so the coadjoint-orbit assumption is needed.

Since $\Psi(\bone)=\bone$ we have the equivalence of (c) and (d) by \cite[Cor. 4.4]{Pop7}. 

Combining, we have shown that (a)--(c) is equivalent to (d). Assuming that this holds, i.e. that $\Psi(P_E)\leq P_E$, we obtain a Riesz decomposition $P_E=\breve{\varsigma}(P^E)+C_E$ where $\breve{\varsigma}(P^E)$ is in $\Bi^\infty$ and $C_E$ is a $\Psi$-potential. By (a) we have $P_E$ over $\Ni$ and since $\Ni=\Bi^\infty+\Ker\Psi^\infty$ we get that both $\breve{\varsigma}(P^E)$ and $C_E$ are over $\Ni$ as well. Thus $P^E$ is over $L^\infty(\M)$. Conversely, if (e) holds then $\Psi(P_E)\leq P_E$ and so we are done.
\end{proof}
\begin{Remark}
For the equivalence of (a) and (c), see also \cite[Cor. 3.3]{Sark3}. 
\end{Remark}

In the one-variable setting ($n=1$), every submodule of $H^2_1=H^0(\T)$ is the range of an inner multiplier, by the Beurling theorem. There is a Beurling decomposition of every submodule of $H^2_n$ also for $n\geq 2$. There are no nontrivial graded submodules of $H^0(\T)$, so there is no surprise that when we look at graded submodules we encounter new phenomena, namely that the multipliers in the Beurling decomposition can be chosen to be homogeneous polynomials. 

In the same fashion we obtain the ungraded version of Lemma \ref{PEinNilemma}:
\begin{Lemma}\label{PEinLilemma}
For a projection $P_E$ on $\GH_\N\otimes\C^N$, the following are equivalent:
\begin{enumerate}[(a)]
%\item{$P_E$ is over $\Ni$.}
\item{$\bone-P_E$ has a Beurling factorization, i.e.
$$
\bone-P_E=\mathrm{SOT-}\sum_{j\in\N}\phi_j\phi_j^*
$$
with $\phi_j\in\Mi_\GH$ a multiplier for each $j\in\N$ (where $\phi_j=0$ is allowed).}
\item{$\Phi(\bone-P_E)\leq\bone- P_E$, in which case $\mathrm{SOT-}\lim_{m\to\infty}\Phi^m(\bone-P_E)=0$.}
\item{The image of $P_E$ is invariant under $S^*$.}
\item{$\Psi(P_E)\leq P_E$.}
\item{$P_E=\breve{\varsigma}(P^E)+C_E$ where $P^E$ is a projection over $L^\infty(\Sb)$ and $C_E$ is a $\Psi$-potential.}
\end{enumerate}
In this case $C_E$ has entries in $\Li$ (since $P_E$ has, by (a)).
\end{Lemma}
\begin{Remark}
For the Fock space $\GH_\N$ associated with an arbitrary projective variety $\M$ one has $\Psi(\bone)\leq\bone$ iff $\Psi(\bone)=\bone$. So in case $\Psi$ is not unital, the superharmonicity property $\Psi(P_E)\leq P_E$ is not equivalent to the coinvariance property $\Phi(\bone-P_E)\geq\bone-P_E$. 
\end{Remark}

If we let $P_E$ be a $\Psi$-superharmonic projection and write its Riesz decomposition as
$$
P_E=\breve{\varsigma}(P^E)+C_E
$$
then $C_E$ is a potential. The ``charge'' of the potential $C_E$ is, in the terminology of \cite[\S3]{GaKu1}, given by
$$
X_E:=(\id-\Psi)(C_E)=(\id-\Psi)(P_E)\geq 0,
$$
i.e.
$$
C_E=\mathrm{SOT-}\sum_{m\in\N_0}\Psi^m(X_E)=\mathrm{SOT-}\sum_{m\in\N_0}(\Psi^m-\Psi^{m+1})(P_E).
$$
We can see $(\Psi^m-\Psi^{m+1})(P_E)$ as the $(m+1)$th order obstruction of $P_E$ from being harmonic. So $C_E$ contains not only the first-order obstruction $X_E=(\id-\Psi)(P_E)$, although this is the leading term. 

The potential $C_E$ need not have the same range projection as its charge $X_E$. However, there is always a $\Psi$-summable element $\tilde{X}_E$ whose range projection is equal to that of $C_E$ \cite[Lemma 4.1]{GaKu1}.

%\begin{prop}
%Let $P^E$ be a projection over $L^\infty(\M)$ and suppose that there is a pure $\Psi$-superharmonic operator $X$ on $\GH_\N\otimes\C^N$ such that 
%$$
%\breve{\varsigma}(P^E)\leq X.
%$$
%Then $\breve{\varsigma}(P^E)=0$. 
%\end{prop}
%\begin{proof}
%Since $\breve{\varsigma}(\bone)=\bone$ we have $\|\breve{\varsigma}(P^E)\|\leq 1$. The result is now a particular case of \cite[Cor. 3.4]{GaKu1}.
%\end{proof}

\subsubsection{Vector-valued essential normality}

Following Barría--Halmos \cite{BaHa1} we observe that since both $\sum^n_{\alpha=1}S_\alpha^*S_\alpha-\bone$ and $\sum^n_{\alpha=1}S_\alpha^*S_\alpha-\bone$ are compact (i.e. $S$ is essentially a spherical unitary), the essential commutant of $S$ can be described as
$$
\{S\}^{\rm ec}=\{X\in\Bi(\GH_\N)|\ (\id-\Psi)(X)\in\Ki(\GH_\N)\}.
$$
Since the essential commutant is a $C^*$-algebra which contains the set $\Bi(\GH_\N)^\Psi$ of all $\Psi$-harmonic elements, it must contain $\Li=C^*(\Bi(\GH_\N)^\Psi)$,
$$
\Li\subset\{S\}^{\rm ec}.
$$
%\begin{Warning}
%The essential commutant $\{S\}^{\rm ec}$ discussed by Barría--Halmos is much larger than the essential commutant $\Mi^{\rm ec}_\GH$ of \emph{all} analytic Toeplitz operators studied e.g. in \cite{Davi5, DEE1}. The latter contains only compact perturbations of \emph{analytic} Toeplitz operators whose Hankel operators are compact. 
%\end{Warning}
This observation leads to:
\begin{thm}\label{essvectthm}
Suppose that $P_E$ is a projection acting on $\GH_\N\otimes\C^N$ with $\Psi(P_E)\leq P_E$ and preserving the $\N_0$-grading. Then the shift tuple $S_E$ on the graded quotient module $\GH^E_\N:=\Ran P_E$ is essentially normal.
\end{thm}
\begin{proof}
Write
$$
(\id-\Psi)(P_E)|_{\GH^E_\N}%=(\id-\Psi_E)(\bone)
=\bone-\sum_{m\in\N_0}\frac{n_m}{n_{m+1}}\sum^n_{\alpha=1}S_{E,\alpha}^*S_{E,\alpha} p_m=\bone-\Lambda^{-1} S^*_ES_E
$$
where $\Lambda:=\sum_{m\in\N_0}\frac{n_{m+1}}{n_m}p_m$ does not depend on the projection $P_E$. % (Note that this works only when $P_E$ commutes with $p_m$, as we assume here).
By Proposition \ref{LiBeurlingprop}, the assumptions of the theorem imply that $P_E$ is a projection in $\Ni\otimes\Mn_N(\C)$. The entries of $(\id-\Psi)(P_E)|_{\GH^E_\N}$ are compact since $P_E$ has entries in $\Ni\subset\{S\}^{\rm ec}$. Since $\Gamma_0:=\Ki(\GH_\N)\cap\Gamma_b$ is an ideal in $\Gamma_b$ and $\Lambda$ is bounded, we have
$$
S^*_ES_E-\Lambda=\Lambda(\id-\Psi(P_E)|_{\GH^E_\N}\in\Gamma_0\otimes\Mn_N(\C).
$$
Since $S_ES_E^*=\bone-P_{E,0}$ is a finite-rank perturbation of the identity operator % (here we are crucially using the assumption that $\GH_\N^E$ is graded) 
and $n_m$ is a polynomial of degree $d$ we easily get that  $S_ES_E^*-\Lambda$ is in $\Li^p$ iff $p>d+1$, so in particular $S_ES_E^*-\Lambda$ is compact. This gives
$$
S^*_ES_E-S_ES^*_E=(S^*_ES_E-\Lambda)-(S_ES_E^*-\Lambda)\in\Gamma_0\otimes\Mn_N(\C)
$$
as desired. 
\end{proof}
The proof relies on the assumption that $P_E$ is graded, since only then does $P_E$ commute with $p_m$ for all $m$. And indeed there are ungraded counterexamples to essential normality \cite[\S4.1]{GRS1}.

%\begin{Remark}
%The above should interpreted as saying that the quantum $\tr_\omega\Theta^E$, namely $[S^*_E,S_E]$, is becoming 
%\end{Remark}

%\begin{Idea}
%Extend to correspondence between quotient module and projections over $\Li$ with $\Psi(P_E)\leq P_E$. Same proof gives essential normality. Presumably Schatten estimates $p>d+1$ work only in the graded case, since it is related to Hilbert polynomials. 
%\end{Idea}

\begin{Example}
 If $N=1$ then graded submodules of $\GH_\N$ are in bijection with homogeneous ideals in the coordinate ring $\Ai$. The quotient modules $\GH^E_\N$ of $\GH_\N$, which  are also quotients of $H^2_n$, are thus orthogonal complements of ideals and equal the completions of the coordinate rings of analytic subvarieties $\M_E\subset\M$. If we write the projection $P_E$ of $\GH_\N$ onto such a quotient module $\GH^E_\N$ as $P_E=\breve{\varsigma}(P^E)+C_E$ then we get
$$
\omega(P^E)=\lim_{m\to\infty}\phi_m(P_{E,m})=\lim_{m\to\infty}\frac{\dim\GH^E_m}{n_m}=0
$$ 
  unless $\dim\M_E=\dim\M$ (i.e. unless $\M_E=\M$) since $\dim\GH^E_m$ is a polynomial in $m$ of degree $\dim\M_E$. Thus, if $\M_E$ is not all of $\M$, 
  $$
 \breve{\varsigma}(P^E)=0.
 $$
In other words $P_E$ is in the kernel of the symbol map $\varsigma$. %This is related to the fact that the structure sheaf $\Oi_{\M_E}$ of $\M_E$ is not torsionfree as $\Oi_\M$-module, so that it does not correspond to any nonzero projection over $L^\infty(\M)$. 
 We see that the kernel of $\varsigma$ is much larger than the ideal of compact operators. We also see that the projection $P^I$ onto the orthogonal complement $\GH^I_\N$ of $\GH^E_\N$ has symbol $\varsigma(P_I)=\varsigma(\bone-P_E)=\bone$. So the projection $\varsigma(P_I)$ does not recover the ideal sheaf $\Ii$ of $\M_E$, even if $\Ii$ is locally free. This is not surprising since, even if $\Ii$ is locally free, the subsheaf $\Ii\subset\Oi_\M$ is not a subbundle. 
 
%The coinvariant projections $P_E$ with nonzero Toeplitz part $\breve{\varsigma}(P^E)$ give only (torsionfree) quotients $\Ei$ of the trivial vector bundle. It does not give subsheaves. That is why the ideal sheaf $\Ii_{\M_E}$ of projective variety $\M_E\subset\M$ does not give a quotient module. It does after tensoring with a sufficiently large power of $\Li$, but then one gets a quotient module of $\GH_\N\otimes\C^N$ for some $N\geq 2$.  
 
%\textbf{Weird:} If $\Ii_D\subset\Oi_\M$ is the ideal sheaf of a divisor then it is locally free. So it can be defined by a projection in $C^\infty(\M)$, or? Cannot be. No the point is that $\Ii_D$ is not a \emph{subbundle} but merely a locally free subsheaf.
\end{Example}

\begin{Remark}[Generalization to general $\M$]% gives nothing new for essential normality]
Submodules of $\GH_\N\otimes\C^N$ are semi-invariant under the backward Drury--Arveson shift $M^*$: they are submodules of a quotient module. The projection onto such a subspace is of the form $P_E-P_F$ where both $P_E$ and $P_F\leq P_E$ are coinvariant. Since Theorem \ref{essvectthm} says that $M$-coinvariant projections are essentially normal we get that all submodules of vector-valued complete Nevanlinna--Pick-modules are essentially normal. So for essential normality one gets nothing new by dropping the assumption about having a coadjoint orbit. 
\end{Remark}

If $\Psi(X)\leq X$ and we write the Riesz decomposition of $X$ as $X=\breve{\varsigma}(f_X)+C_X$ then we get
$$
(\id-\Psi)(X)=(\id-\Psi)(C_X)\geq 0.
$$
Hence, if $X$ is over $\Ni$, the pure $\Psi$-superharmonic part always satisfies
$$
(\id-\Psi)(C_X)\in\Gamma_0.
$$
But even in the case of a graded projection $X=P_E$ it need not be that $C_X$ itself is compact. %In case $P_E$ is over $\Ti_\GH^{(0)}$ this is clear, and we are working on a geometric proof to show that in this case $C_E$ is in fact in $\Li^p$ for all $p>d+1$ (where $d$ is the complex dimension of $\M$). %We conjecture that also this property of $C_E$ extends to arbirary projection $P_E$ with $\Psi(P_E)\leq P_E$.  

\section{Cowen--Douglas bundles of quotient modules}

\subsection{Setup}

\subsubsection{The reference Hermitian line bundle}\label{reflinesec}
Recall that $\B\subset\B^n$ is defined as the zero set in the unit ball of the ideal in $\C[z_1,\dots,z_n]$ corresponding to the embedding $\M\hookrightarrow\C\Pb^{n-1}$. In other words, if $\Ai$ is the homogeneous coordinate ring of $\M$ then $\B=\Spec(\Ai)\cap\B^n$. The Fock space $\GH_\N$ is a graded completion of $\Ai$ and its elements are analytic functions on $\B$.

In the following we write
$$
\Ker(S^*-v\bone):=\bigcap_{\alpha=1}\Ker(S_\alpha^*-v_\alpha\bone)\subset\GH_\N
$$
for the subspace of joint eigenvectors of $S_1^*,\dots,S_n^*$ with joint eigenvalue $v=(v_1,\dots,v_n)\in\B$. We know that $\Ker(S^*-v\bone)$ is nonzero for each $v\in\B$,
$$
\sigma_p(S^*)=\B.
$$
Indeed, %suppose that $\psi\in\GH_\N$ is orthogonal to $\Ker(S^*-\bar{v}\bone)$. Then 
%$$
%0=\bra\psi|k_v\ket=\psi(v).
%$$
$\Ker(S^*-v\bone)$ is 1-dimensional and spanned by the reproducing vector $K_{\bar{v}}$ (or its normalized version $k_{\bar{v}}:=K_{\bar{v}}/\|K_{\bar{v}}\|$) at $\bar{v}$,
$$
\Ker(S^*-v\bone)=\C k_{\bar{v}}.
$$
Since $\dim\Ker(S^*-v\bone)=1$ for all $v\in\B$ and the map $\B\ni v\to k_{\bar{v}}$ is holomorphic, the vector spaces
$$
\Oi_{\rm CD}(v):=\Ker(S^*-v\bone)
$$
form a holomorphic line bundle over $\B$ which we denote by $\Oi_{\rm CD}$ and call the \textbf{Cowen--Douglas bundle} of the Hilbert module $\GH_\N$ (or of the operator tuple $S$). 
We have a global holomorphic section $v\to k_{\bar{v}}$ of $\Oi_{\rm CD}$ which vanishes nowhere. So $\Oi_{\rm CD}$ is a trivial holomorphic line bundle (it is not algebraic though, since $k_{\bar{v}}$ is not algebraic):
$$
\Oi_{\rm CD}\cong\Oi_\B.
$$
Sometimes we will only need the restriction of $\Oi_{\rm CD}$ to $\B\setminus\{0\}$, and we denote this restriction again by $\Oi_{\rm CD}$.
\begin{dfn}
The \textbf{Cowen--Douglas projection} of $\GH_\N$ is the projection %$\CD(\GH_\N)$ over $C^0(\B\setminus\{0\})$ 
$$
\CD(\GH_\N)\in C^0(\B\setminus\{0\})\otimes\Bi(\GH_\N)
$$
which defines the line bundle $\Oi_{\rm CD}$, i.e. 
$$
\CD(\GH_\N)(v):=|k_{\bar{v}}\ket\bra k_{\bar{v}}|=\text{projection onto }\Ker(S^*-v\bone).
$$
\end{dfn}

For $\psi\in\GH_\N$ and $v\in\B$ we have $\bra K_{\bar{v}}|\psi\ket=\overline{\psi(\bar{v})}$ and hence
$$
(\CD(\GH_\N)\psi)(v)=\CD(\GH_\N)(v)\psi=\|K_{\bar{v}}\|^{-1}\overline{\psi(\bar{v})}k_{\bar{v}}.
$$
Since $\|K_{\bar{v}}\|^{-1}\overline{\psi(\bar{v})}k_{\bar{v}}$ belongs to the subspace $\Ker(S^*-v\bone)$ of $\GH_\N$ for each $v\in\B$, the function $\B\ni v\to\CD(\GH_\N)(v)\psi\in\GH_\N$ is a section
$$
\CD(\GH_\N)\psi\in\Gamma^0(\B;\Oi_{\rm CD})
$$
of the line bundle $\Oi_{\rm CD}$. Moreover, the section $\CD(\GH_\N)\psi$ is holomorphic because both $v\to k_{\bar{v}}$ and $\psi$ are holomorphic. 

If $\CD(\GH_\N)\psi$ is the zero section then $\psi$ belongs to $\Big(\bvee_{v\in\B}\Ker(S^*-v\bone)\Big)^\perp=\{0\}$. 
Thus $\GH_\N$ embeds into the space of global holomorphic sections of $\Oi_{\rm CD}$. 

%\begin{Lemma}
%Every global holomorphic section of $\Oi_{\rm CD}$ is of the form $v\to\xi(v)$ for a unique $\xi\in\GH_\N$. That is, the space of global holomorphic sections of $\Oi_{\rm CD}$ identifies with $\GH_\N$. 
%\end{Lemma}
%\begin{proof}
%The above map 
%$$
%\GH_\N\ni\xi\to\CD(\GH_\N)\xi\in\Oi_{\rm CD}(\B)
%$$ 
%is injective due to the fact that $\GH_\N=\cspan\{k_v|\ v\in\B\}$. 
%the set $\{\xi(v)\in\Oi_{\rm CD}(v)|\ \xi\in\GH_\N,\ v\in\B\}=\bigcup_{v\in\B}\Oi_{\rm CD}(v)$ is a dense subspace of $\GH_\N$.
%Indeed, this gives
%$$
%\{0\}=\Big(\bvee_{v\in\B}\Ker(S^*-v\bone)\Big)^\perp=\bvee_{v\in\B}\Ran(S-v\bone)
%$$
%and $\bvee_{v\in\B}\Ran(S-v\bone)$ is evidently the kernel of the map $\GH_\N\ni\xi\to\CD(\GH_\N)(v)\to H^0(\B;\Oi_{\rm CD})$. 
 %by \cite[Lemma 4.9]{Wern1}. 

%Surjectivity is not considered by Wernet. In fact, $\GH_\N$ cannot contain all holomorphic sections since not every holomorphic function on $\B$ has finite $\GH_\N$-norm. So what about the graded components $\GH_m$? 

%An alternative approach is to look at $\CD(\GH_\N)$ as giving an embedding of $\B$ into projective space $\Pb[\GH_\N]$. Then $\Oi_{\rm CD}$ becomes the restriction to the submanifold $\CD(\GH_\N)(\B)$ of the hyperplane bundle on $\Pb[\GH_\N]$, and hence its space of holomorphic sections identifies with $\GH_\N$. Well, that would be the case if $\GH_\N$ were finite-dimensional. The Beltitas should have the infinite-dimensional version.  
%\end{proof}
If we identify the projective space $\Pb[\GH_\N]$ with the manifold of rank-1 projections acting on $\GH_\N$ then we can regard $\CD(\GH_\N)$ as a mapping
$$
\CD(\GH_\N):\B\to\Pb[\GH_\N],
$$
which is a holomorphic embedding since $k_v$ is never zero. If $(\psi_j)_{j\in\N}$ is any orthonormal basis for $\GH_\N$ then from $\bra \psi_j|K_v\ket=\overline{\psi_j(v)}$ we see that one can express the coherent vectors as
$$
k_v=\|K_v\|^{-1}\sum_{j\in\N}\overline{\psi_j(v)}\psi_j,\qquad\forall v\in\B.
$$
Therefore, whatever orthonormal basis $(\psi_j)_{j\in\N}$ for $\GH_\N$ chosen for the identification of $\Pb[\GH_\N]$ with the infinite-dimensional projective space $\C\Pb^\infty$, the embedding $\CD(\GH_\N)$ can therefore be described as %(Note that we are not embedding into the projective space of the complex conjugate Hilbert space since we are just taking complex conjugates of scalars.)
$$
\CD(\GH_\N)(v)=[\overline{\psi_1(\bar{v})}:\overline{\psi_2(\bar{v})}:\cdots],
$$
where we denote by $[z_1:z_2:\cdots]$ the homogeneous coordinates on $\C\Pb^\infty$. 

In the literature on Kähler geometry it is more common to consider the slightly different holomorphic embedding 
$$
\FS(\GH_\N):\B\to\Pb[\GH_\N],\qquad \FS(\GH_\N)(v):=[\psi_1(v):\psi_2(v):\cdots].
$$
The notation $\FS(\GH_\N)$ makes sense since $\FS(\GH_\N)$ depends only on the Hilbert space and not the choice of orthonormal basis; $\FS$ stands for ``Fubini--Study''. We shall not need to distinguish between $\FS(\GH_\N)(v)$ and $\CD(\GH_\N)(v)$ and we identify them as the same projection acting on $\GH_\N$.

More generally if $\Hi$ is the Hilbert space of global holomorphic sections of some globally generated vector bundle then we denote by $\FS(\Hi)$ the embedding of the base manifold into the Grassmannian defined in the same way as above by an orthonormal basis for $\Hi$. 

As a special case, the embedding $\M\subset\C\Pb^{n-1}$ is correspond to the projection $\FS(\GH_1)$, after we have the identification $\GH_1=\C^n$ so that $\Pb[\GH_1]=\C\Pb^{n-1}$. 
We are also interested in the embeddings $\FS(\GH_m)$ of $\M$ into $\Pb[\GH_m]$ where $\GH_m$ as before is the vector space $H^0(\M;\Li^m)$ of global holomorphic sections of the line bundle $\Li^m$ endowed with the inner product of the tensor product $(\GH_1)^{\otimes m}$. 
We can regard $\FS(\GH_m)$ as a projection in the $C^*$-algebra $C^0(\M)\otimes\Bi(\GH_m)$. 
Any choice of  Parseval frame $(\psi_j)_{j\in\J}$ for $\GH_m$ gives an expansion of $\FS(\GE_m)$ as
$$
\FS(\GH_m)=\sum^{n_m}_{j,k\in\J}\FS(\GH_m)_{j,k}|\psi_j\ket\bra\psi_k|
$$
with coefficients $\FS(\GH_m)_{j,k}\in C^0(\M)$ given by
$$
\FS(\GH_m)_{j,k}(x)=\psi_j(x)\overline{\psi_k(x)},\qquad\forall x\in\M.
$$
For the particular choice $(\psi_j)_{j\in\J}=(Z_\mathbf{k})_{|\mathbf{k}|=m}$ we obtain
\begin{equation}\label{FSframe}
\FS(\GH_m)=\sum_{|\mathbf{j}|=m=|\mathbf{k}|}Z_\mathbf{j}Z_\mathbf{k}^*\otimes S_\mathbf{j} S_\mathbf{k}^*p_m
\end{equation}
where $S_\mathbf{j}:=S_{j_1}\cdots S_{j_m}$ for a multi-index $\mathbf{j}=j_1\cdots j_m$. 
The $m$-fold tensor product of the projection $\FS(\GH_1)$ is given by
$$
\FS(\GH_1)^{\otimes m}=\sum_{|\mathbf{j}|=m=|\mathbf{k}|}Z_\mathbf{j}Z_\mathbf{k}^*\otimes|e_\mathbf{j}\ket\bra e_\mathbf{k}|,
$$
where $e_1,\dots,e_n$ is an orthonormal basis for $\GH_m$ and $e_\mathbf{j}:=e_{j_1}\otimes\cdots\otimes e_{j_m}$. 
Since the $Z_\alpha$'s satisfy the relations of the ideal defining $\M$, we can express this as
$$
\FS(\GH_1)^{\otimes m}=\sum_{|\mathbf{j}|=m=|\mathbf{k}|}Z_\mathbf{j}Z_\mathbf{k}^*\otimes |p_me_\mathbf{j}\ket\bra p_me_\mathbf{k}|.
$$
and be regarded as a function on $\M$ with values in the subalgebra $\Bi(\GH_m)\subset\Bi(\GH^{\otimes m})$. Now $|p_me_\mathbf{j}\ket\bra p_me_\mathbf{k}|=S_\mathbf{j} S_\mathbf{k}^*p_m$. Thus the Fock inner product on $\GH_m$ is precisely the one such that $\FS(\GH_1)^{\otimes m}$ is naturally identified with $\FS(\GH_m)$. 

If we write $v=r\zeta\in\B$ with $r\in(0,1)$ and $\zeta\in\Sb$ then we have
$$
k_v=(1-|v|^2)^{1/2}\sum_{\mathbf{j}\in\F_n^+}\bar{v}_\mathbf{j}S_\mathbf{j}\Omega=(1-r^2)^{1/2}\sum_{m\in\N_0}r^m\sum_{|\mathbf{j}|=m}\bar{\zeta}_\mathbf{j}S_\mathbf{j}\Omega.
$$
Thus
  $$
\CD(\GH_\N)(r\zeta)=|k_{r\bar{\zeta}}\ket\bra k_{r\bar{\zeta}}|=(1-r^2)\sum_{m\in\N_0}r^{2m}\sum_{|\mathbf{j}|=m=|\mathbf{k}|}\zeta_\mathbf{j}\bar{\zeta}_\mathbf{k}S_\mathbf{j} S_\mathbf{k}^*.
$$
While $\GH_m\cap\C k_v=\{0\}$, the projection $p_mk_v$ of $k_v$ onto $\GH_m$ spans a 1-dimensional subspace of $\GH_m$; the projection onto this subspace is 
$$
\CD(\GH_m)(v):=\frac{|p_m k_{\bar{v}}\ket\bra p_mk_{\bar{v}}|}{\|p_m k_v\|^2}.
$$ 
It depends only on the class $[v]$ of $v$ in the quotient $\M=(\B\setminus\{0\})/\D^\times$. We write $k_{[v]}^{(m)}:=p_mk_v/\|p_mk_v\|$, so that $\CD(\GH_m)(x)=|k_{\bar{x}}^{(m)}\ket\bra k_{\bar{x}}^{(m)}|$ for all $x\in\M$. We have $\|p_m k_v\|^2=(1-|v|^2)^{-1}|v|^{2m}$. For all $\psi\in\GH_m$ the reproducing property of $K_v$ gives
$$
\bra k_x^{(m)}|\psi\ket=\psi(x),
$$
where $x\to\psi(x)$ is the function on $\M$ induced by the homogeneous degree-$m$ polynomial $\psi$. 
Therefore, for any Parseval frame $(\psi_j)_{j\in\J}$ for $\GH_m$ we can expand the $\Bi(\GH_m)$-factor of $\CD(\GH_m)$ (cf. \cite{Bala1}) to obtain
$$
\CD(\GH_m)=\sum_{j,k\in\J}\CD(\GH_m)_{j,k}|\psi_j\ket\bra \psi_k|
$$
with coefficients $\CD(\GH_m)_{j,k}\in C^0(\M)$ given by
$$
\CD(\GH_m)_{j,k}(x)=\bra\psi_j|k_{\bar{x}}^{(m)}\ket\bra k_{\bar{x}}^{(m)}|\psi_k\ket=\overline{\psi_j(\bar{x})}\psi_k(\bar{x}),\qquad\forall x\in\M.
$$
Thus $\CD(\GH_m)$ coincides with $\FS(\GH_m)$. If we take the Parseval frame for $\GH_m$ given by $(\psi_j)_{j\in\J}=(e_\mathbf{k})_{|\mathbf{k}|=m}$ where $e_1,\dots,e_n$ is the standard basis for $\GH_1=\C^n$ then we see that
 $$
\CD(\GH_\N)(r\zeta)=\sum_{m\in\N_0}r^{2m}\sum_{|\mathbf{j}|=m=|\mathbf{k}|}\FS(\GH_m)([\zeta]).
$$

%\begin{Update}
%Probaby $\CD(\GH_\N)$ is not a pullback metric in any sense, since it is not $\D^\times$-equivariant. What is interesting to not is just that it expands like
%$$
%\CD(\GH_\N)(r\zeta):=(1-r^2)\sum_{m\in\N_0}r^{2m}\sum_{|\mathbf{j}|=m=|\mathbf{k}|}\FS(\GH_1)(\zeta)^{\otimes m}\in\Bi(\GH_\N),
%$$
%and $\FS(\GH_1)^{\otimes m}=\FS(\GH_m)$. It is thus a special kind of metric.  
%\end{Update}

\subsubsection{Higher-rank Cowen--Douglas bundles}
%In this section we are not using the assumption that $\M$ is of the form $\G/\K$, i.e. the result holds for all quotients $\GH_\N$ of $H^2_n$ as long as the projective variety $\M$ associated with $\GH_\N$ is nonsingular.

Throughout this section, $\GH^E_\N$ is a graded quotient module of $\GH_\N\otimes\C^N$ and $P_E$ is the projection onto $\GH^E_\N$. The backward shift $S^*$ on $\GH_\N\otimes\C^N$ restricts to a row contraction $S_E^*$ on the Hilbert subspace $\GH^E_\N$. We shall study the commuting operator tuple $S_E^*$ with the Cowen--Douglas approach. This amounts to looking at the family of eigenspaces of $S^*_E$,
$$
\Ei_{\rm CD}(v):=\Ker(S_E^*-v\bone),\qquad\forall v\in\B,
$$
and how they vary with $v$. If $S_E^*$ were in the Cowen--Douglas class $\Bi_r(\B)$ for some $r\geq 1$ then the $\Ei_{\rm CD}(v)$'s would be the fibers of a holomorphic vector bundle on $\B$. However, we shall see that this is too strong an assumption if we want to use operator theory to study vector bundles over $\M$. The ``correct'' condition for projective geometry is instead to ask for membership in the Cowen--Douglas class $\Bi_r(\B\setminus\{0\})$, i.e. we have to allow $\Ei_{\rm CD}$ to be singular at $0$. %And we shall see that for \emph{any} $\GH^E_\N$ the $\Ei_{\rm CD}(v)$'s define a coherent analytic sheaf on $\B\setminus\{0\}$ which descends to a coherent analytic sheaf on $\M=(\B\setminus\{0\})/\D^\times$. 

%\subsubsection{The space $\E_{\rm CD}$}

The eigenvectors of the backward shift $S^*$ on $\GH_\N\otimes\C^N$ are of the form $k_v\otimes\xi$ with $v\in\B$ and $\xi\in\C^N$,
$$
S^*(k_v\otimes\xi)=\bar{v}k_v\otimes\xi.
$$
The joint spectrum of the tuple $S^*$ is equal to the joint point spectrum, which is $\sigma(S^*)=\sigma_p(S^*)=\B$. The multiplicity of each eigenvalue is equal to $N$. 

Since $S_E^*$ is just the restriction of $S^*$, the vector spaces $\Ker(S_E^*-v\bone)=\Ran(S_E-\bar{v}\bone)^\perp$ is a subspace of $k_{\bar{v}}\otimes\C^N$ for each $v\in\B$, and so finite-dimensional. In particular, $\Ran(S_E-{\bar{v}}\bone)$ is a closed subspace of $\GH^E_\N$. Clearly the eigenvectors of $S_E^*$ span $\GH^E_\N$. So the condition that $S_E^*$ is in class $\Bi_r(\B\setminus\{0\})$ is the same as asking for the consancy of the function $\dim\Ker(S_E^*-v\bone)$ in $v\in\B\setminus\{0\}$.

Since $\GH^E_\N$ is an $S^*$-invariant subspace, it is of the form
$$
\GH^E_\N=\overline{\mathrm{span}}\{k_v\otimes\xi|\ (v,\xi)\in\E\}
$$
for some closed subset $\E\subset \B\times\C^N$. But since the coherent vectors $k_v$ are not orthogonal, there are many such sets $\E$. We will will adopt a special notation for the largest possible choice of $\E$, namely
$$
\E_{\rm CD}:=\{(v,\xi)\in\B\times\C^N|\ S_E^*(k_{\bar{v}}\otimes\xi)=vk_{\bar{v}}\otimes\xi\}.
$$ 
Define
$$
\E_{\rm CD}(v):=\{\xi\in\C^N|,\ (v,\xi)\in\E_{\rm CD}\},
$$
so that $\E_{\rm CD}=\bigsqcup_{v\in\B}\E_{\rm CD}(v)$. %Recall that 
%$$
%\Ran\mathbf{P}_E(v)=\Ker(S_E-v\bone)=\mathrm{span}\{ P_E(k_v\otimes\xi)|\ \xi\in\C^N\}.
%$$
Since $\Ker(S_E-v\bone)\subset\GH^E_\N$ we must thus have $\Ei_{\rm CD}(v)=\mathrm{span}\{k_{\bar{v}}\otimes\xi|\ \xi\in\E_{\rm CD}(v)\}$. Let us record this fact:
\begin{prop}\label{tensorprop}
For each $v\in\B$ there is a vector space $\E_{\rm CD}(v)\subset\C^N$ such that
$$
\Ei_{\rm CD}(v)=\C k_{\bar{v}}\otimes\E_{\rm CD}(v).%\mathrm{span}\{k_w|\ w\in\E(v)\}.
$$
\end{prop}
%If $\Ei_{\rm CD}$ is locally free then for each $\alpha\in\{1,\dots,n\}$ we have an identification
%$$
%\E_{\rm CD}(v)=\C^r,\qquad \forall v\in\U_\alpha.
%$$
 Let $\Theta_E:\B\times\ell^2(\N_0)\to\B\times\C^N$ be any multiplier such that the associated multiplication operator $M_{\Theta_E}\in\Bi(\GH_\N\otimes\C^N)$ has range equal to the orthogonal complement of $\GH^E_\N$. 
The adjoint operator $M_{\Theta_E}^*$ acts on coherent vectors as
\begin{equation}\label{MThetastareq}
M_{\Theta_E}^*(k_v\otimes\xi)=k_v\otimes\Theta_E(v)^*\xi.
\end{equation}
So we have $M_{\Theta_E}^*(k_v\otimes\xi)=0$ if and only if $\Theta_E(v)^*\xi=0$. On the other hand, the relation $\GE_\N=\Ker M_{\Theta_E}^*$ gives
$$
k_{\bar{v}}\otimes\xi\in\Ker M_\Theta^*\iff (v,\xi)\in\E_{\rm CD},
$$
Thus \eqref{MThetastareq} gives
$$
\E_{\rm CD}(v)=\Ker\Theta_E^*(\bar{v}),
$$
and, for each $v\in\B$,
$$
\Ker(S_E^*-v\bone)=\C k_{\bar{v}}\otimes\Ker\Theta_E^*(\bar{v})
$$ 
(cf. \cite[Remark 1.2]{KwTr1}).

The analyticity of $\Theta_E$ and the finite-dimensionality of $\E_{\rm CD}(v)$ for each $v$ ensure that the $\E_{\rm CD}(v)$'s form the holmorphic linear space of a coherent analytic sheaf over $\B$ in the sense of \cite[\S1.6]{Fisc1}. Therefore also the $\Ei_{\rm CD}(v)$'s form a coherent analytic sheaf $\Ei_{\rm CD}$ over $\B$. We call $\Ei_{\rm CD}$ the \textbf{Cowen--Douglas sheaf} of the quotient module $\GH^E_\N$.

\subsection{Algebraic aspects}

\subsubsection{Graded $\Ai$-modules and coherent sheaves}\label{alggradmodsec}

Let $\Ai=\bigoplus_{m\in\N_0}H^0(\M;\Li^m)$ be the graded coordinate ring of the embedded variety $\M\subset\C\Pb^{n-1}$. The Fock space $\GH_\N$ is the completion of $\Ai$ in the inner product of $\GH^{\vee\N}$ and $H^0(\M;\Li^m)$ is the vector space underlying the Hilbert space $\GH_m$ for each $m\in\N_0$. 

%The following is trivial to show:
%\begin{Remark}
%Let $\GI\subset\GH_\N\otimes\C^N$ be an invariant subspace under $\Ai$ and let $\GQ$ be any vector-space complement of $\GI$. Then the isomorphism between $\GQ$ and the quotient vector space $(\GH_\N\otimes\C^N)/\GI$ given by
%$$
%\GQ\oplus\GI\ni\psi+\phi\to\psi+\GI
%$$
%intertwines the $\Ai$-actions
%$$
%f\cdot(\psi+\GI):=f\psi+\GI,\qquad\forall f\in\Ai,\ \psi\in\GQ
%$$
%and
%$$
%f\cdot\psi:=Q(f\psi),\qquad\forall f\in\Ai,\ \psi\in\GQ,
%$$
%where $Q$ is the oblique projection with $\Ran Q=\GQ$ and $\Ker Q=\GI$. In the case $\GQ$ is the orthogonal complement, i.e. when $\GQ$ is coinvariant under $\Ai$, we have $Q^*=Q$. 
%(Note that such a compression action appears in Sarason's theorem but there the projection is required to be selfadjoint, which is why only the case $\GE=\GI^\perp$ arises, i.e. only coinvariant subspaces give $\Ai$-actions if we compress with the orthogonal projection onto that subspace).
%\end{Remark}

Given a quotient module $\GH^E_\N$ of $\GH_\N\otimes\C^N$ as before, define
$$
E_\N:=\GH^E_\N\cap(\Ai\otimes\C^N).
$$
Then $E_\N$ is the graded $\Ai$-module whose completion in the inner product of $\GH_\N\otimes\C^N$ is equal to $\GH^E_\N$. 
Note that $E_\N$ is isomorphic to a quotient of $\Ai\otimes\C^N$ and thus finitely generated. 

Recall that for each $\Ai$-module one can associate in a canonical fashion an algebraic sheaf on $\V:=\Spec\Ai$ \cite{Serr2}: 
\begin{dfn}
Let $E$ be an $\Ai$-module. The \textbf{Serre sheaf} of $E$ is the $\Oi_\V$-module
$$
\Ei_\V:=E\otimes_\Ai\Oi_\V.
$$
If $E=E_\N$ is a graded $\Ai$-module we can also define an $\Oi_\M$-module $\Ei$ (also referred to as the Serre sheaf of $E_\N$) by
$$
\Ei:=E_\N\otimes_\Ai\Oi_\M.
$$
We shall denote by $\Ei_{\V\setminus\{0\}}$ the restriction of $\Ei_\V$ to $\V\setminus\{0\}$. 
\end{dfn}
%By Nakayama's lemma $\dim\Ei_\B(v)$ is the minimal number of generators of the $\Oi_v$-module $E_\N\otimes_\Ai\Oi_v$, which is finite as $E_\N$ is finitely generated.

If we for $\alpha\in\{1,\dots,n\}$ let $\U_\alpha\subset\V$ be the open set where the coordinate function $Z_\alpha\in\Ai$ is nonzero then $\Oi_\V(\U_\alpha)=\Ai_{Z_\alpha}$ is the localization of the ring $\Ai$ at $Z_\alpha$. So 
$$
\Ei_\V(\U_\alpha)=E\otimes_\Ai\Oi_\V(\U_\alpha)=E\otimes_\Ai\Ai_{Z_\alpha}=E_{Z_\alpha}
$$
is the the module of fractions of $E$ with denominator $Z_\alpha$. If we denote by $\U_\alpha\subset\M$ also the projectivization of $\U_\alpha\subset\V\setminus\{0\}$ then $\Oi_\M(\U_\alpha)=\Ai_{(Z_\alpha)}$ is the homogeneous localization of the ring $\Ai$ at $Z_\alpha$. So the Serre sheaf $\Ei$ on $\M$ of a graded module $E_\N$ can be described as
$$
\Ei(\U_\alpha)=E_\N\otimes_\Ai\Oi_\M(\U_\alpha)=E_\N\otimes_\Ai\Ai_{(Z_\alpha)}=(E_\N)_{(Z_\alpha)},
$$
i.e. by taking homogeneous localizations of the graded module $E_\N$. 

Let $E_\N$ be a graded $\Ai$-module with Serre sheaves $\Ei$ and $\Ei_{\V\setminus\{0\}}$ on $\M$ and $\V\setminus\{0\}$ respectively, and consider the quotient map $\pi:\V\setminus\{0\}\to\M$. Evidently 
$$
\pi^{-1}(E_\N\otimes_\Ai\Oi_\M)\otimes_{\pi^{-1}\Oi_\M}\Oi_{\V\setminus\{0\}}=E_\N\otimes_\Ai\Oi_{\V\setminus\{0\}}.
$$
That is,
$$
\pi^*\Ei=\Ei_{\V\setminus\{0\}}.
$$
%Conversely, for any coherent $\Oi_\M$-module $\Ei$ one can associate a finitely generated graded $\Ai$-module, namely
%$$
%E_\N:=\bigoplus_{m\in\N_0}H^0(\M;\Ei(m)).
%$$ 

\begin{Lemma}[{\cite{Serr2}}]\label{Serrelemma}
The Serre sheaf of a (graded) $\Ai$-module $E$ is a coherent as $\Oi_\V$-module (or $\Oi_\M$-module) if and only if $E$ is finitely generated. The module $E$ identifies with the module of global holomorphic sections of its Serre sheaf, 
$$
H^0(\V;\Ei_\V)=E.
$$
If $E_\N$ is a graded $\Ai$-module then modulo finite-dimensional $\Ai$-modules we have
%Then $\Ei_\V|_{\V\setminus\{0\}}$ is the pullback to $\V\setminus\{0\}$ of a coherent $\Oi_\M$-module $\Ei$ with
$$
H^0(\V\setminus\{0\};\Ei_{\V\setminus\{0\}})=\bigoplus_{m\in\N_0}H^0(\M;\Ei(m))\cong E_\N
$$
as graded $\Ai$-modules. Conversely, if we start with a coherent $\Oi_\M$-module $\Ei$ then the Serre sheaf of graded $\Ai$-module $E_\N:=\bigoplus_{m\in\N_0}H^0(\M;\Ei(m))$ is isomorphic to $\Ei$ as $\Oi_\M$-module.
\end{Lemma}
Thus replacing $E_\N$ by $\tilde{E}_\N:=\bigoplus_{m\in\N_0}H^0(\M;\Ei(m))$ we get a Serre sheaf $\tilde{\Ei}$ on $\M$ which is isomorphic to $\Ei$. However, $\tilde{\Ei}_\V:=\tilde{E}_\N\otimes_\Ai\Oi_\V$ can differ from $\Ei_\V$ at the origin $0\in\V$ where the finite-dimensional distinction between $E_\N$ and $\tilde{E}_\N$ is still significant.

%\begin{Notation}
%The restrictions $\Ei_\V|_{\V\setminus\{0\}}$ and $\Ei_\V|_{\B\setminus\{0\}}$ will also be denoted by $\Ei$, and it will be clear from the context which of the manifolds $\V\setminus\{0\}$, $\B\setminus\{0\}$, and $\M$ we are viewing $\Ei$ as a sheaf over.
%\end{Notation}
The Abelian category $\coh\M$ of coherent algebraic sheaves on $\M$ can be identified with the quotient category 
$$
\qgr(\Ai)=\gr(\Ai)/\operatorname{tors}(\Ai),
$$
where $\gr(\Ai)$ is the category of of finitely generated graded $\Ai$-modules and $\operatorname{tors}(\Ai)$ is the subcategory of  modules %$E_\N\in\gr(\Ai)$ such that
%$$
%\Gm^m\psi=0,\qquad\forall \psi\in E_\N
%$$
which are finite-dimensional as vector spaces over $\C$ \cite[\S59]{Serr2}. %The Serre sheaf $\Ei$ of $E_\N\in\gr(\Ai)$ can thus be viewed as an element of $\qgr(\Ai)$. 
The quotient functor $\gr(\Ai)\to\qgr(\Ai)$ is exact, while the global section functor $\qgr(\Ai)\ni\Ei\to\bigoplus_mH^0(\M;\Ei(m))\in\gr(\Ai)$ is only left-exact. Thus if $0\to\Ii\to\Ei\to\Fi\to 0$ is a short exact sequence of quasicoherent sheaves then we get an exact sequence $0\to I_\N\to E_\N\to F_\N$ of graded $\Ai$-modules by applying the global section functor. So even if $\Fi$ is globally generated, i.e. if we have a surjection $\Oi_\M\otimes\C^N\to\Fi\to 0$ for some $N$, we cannot conclude that $F_\N:=\bigoplus_mH^0(\M;\Fi(m))$ is a graded quotient of $\Ai\otimes\C^N$. There is however a graded $\Ai$-module quotient $\Ai\otimes\C^N\to\tilde{F}_\N\to 0$ with $\tilde{F}_m=F_m$ for $m\gg 0$ and the Serre sheaf of $\tilde{F}_\N$ equals $\Fi$. 

If $\GE_\N$ is a graded quotient module of $\GH_\N\otimes\C^N$ then the shift $S_E$ on $\GE_\N$ satisfies
$$
S_ES_E^*:=\sum^n_{\alpha=1}S_{E,\alpha}S_{E,\alpha}^*=P_E\sum^n_{\alpha=1}S_\alpha S_\alpha^*\big|_{\GE_\N}=P_E(\bone-p_0\otimes\bone_N)|_{\GE_\N}%=\sum_{m\in\N}P_{E,m},
$$
so that $S_ES_E^*$ restricted to $\GE_m$ equals the identity operator on $\GE_m$ for all $m\ne 0$. As in \cite{An6} (where $N=1$) one can use this fact to construct explicit isometric embeddings (cf. Eq. \eqref{iotaE})
\begin{equation}\label{modsubprod}
\GE_l\hookrightarrow\GE_m\otimes\GH_{l-m},\qquad\forall l\geq m\geq 0.
\end{equation}
 The ``subproduct'' structure \eqref{modsubprod} is a generalization of the subproduct property of $\GH_\bullet$, which reads
 $$
 \GH_l\hookrightarrow\GH_m\otimes\GH_{l-m},\qquad\forall l\geq m\geq 0.
 $$
 
%If $\Ei_\V$ is the Serre sheaf of a graded $\Ai$-module $E_\N$ then 
%$$
%H^0(\V;\Ei_\V)\cong E_\N
%$$
%as $\Ai$-modules% \cite[\S49]{Serr2}
%, but this is not necessarily an isomorphism of \emph{graded} $\Ai$-modules. On the other hand, from $\Ei_{\V\setminus\{0\}}=\pi^*\Ei$ we have
%$$
%H^0(\V\setminus\{0\};\Ei_{\V\setminus\{0\}})\cong \bigoplus_{m\in\N_0}H^0(\M;\Ei(m))
%$$
%as graded $\Ai$-modules. Recall \cite[\S13.5]{GoWe1} that the graded $\Ai$-module $E_\N$ is called \textbf{saturated} if
%$$
%E_\N\cong \bigoplus_{m\in\N_0}H^0(\M;\Ei(m)).
%$$
%In general there is an error given only by a finite-dimensional vector space, so that we always have
%$$
%E_m=H^0(\M;\Ei(m)),\qquad\forall m\gg 0.
%$$
%There are always many coherent algebraic sheaves on $\V$ which restricts to $\Ei_{\V\setminus\{0\}}$ on $\V\setminus\{0\}$, but they differ in their global sections only by such a finite-dimensional vector space. 
For an arbitrary coherent $\Oi_\M$-module $\Ei$ it is not necessarily the case that $\bigoplus_{m\in\N_0}H^0(\M;\Ei(m))$ is a graded quotient of $\Ai\otimes\C^N$; in this case from \eqref{modsubprod} we have canonical embeddings of vector spaces
\begin{equation}\label{zeroregular}
H^0(\M;\Ei(l))\hookrightarrow H^0(\M;\Ei(m))\otimes H^0(\M;\Li^{l-m}),\qquad\forall l\geq m\in\N_0,
\end{equation}
and this is a characteristic of sheaves $\Ei$ which are regular in the sense of Castelnuovo--Mumford \cite[Thm. 1.8.3]{Laza1}. These are in particular globally generated. %In general we have the weaker version of \eqref{zeroregular} which only involves large enough $m$. %When \eqref{zeroregular} holds for all $m\geq m_0$, i.e. if $\Ei(m_0)$ is $0$-regular, then we say that $\Ei$ is $m_0$-\textbf{regular}. 

\subsubsection{Serre sheaf versus Cowen--Douglas sheaf}

Let $\GE_\N$ be a graded quotient of the standard Hilbert module $\GH_\N\otimes\GE_0$ for some finite-dimensional Hilbert space $\GE_0$. Let $S_E$ be the shift on $\GE_\N$. As mentioned, the vector spaces $\Ker(S_E^*-v\bone)$ form a coherent analytic sheaf $\Ei_{\rm CD}$. Let $\Gm_v$ be the ideal of functions in $\Ai$ vanishing at $v$. 
The vector space $\Ker(S_E^*-v\bone)$ is linearly isomorphic to the annihilator of $\Ran(S_E-v\bone)$ in the dual space of $\GE_\N$, and the latter is linearly isomorphic to $(\GE_\N/\Gm_v\GE_\N)^*$, so
$$
\Ker(S_E^*-v\bone)\cong(\GE_\N/\Gm_v\GE_\N)^*.
$$

Let $E_\N$ be the graded $\Ai$-module whose completion equals $\GE_\N$, i.e. $E_\N=\GE_\N\cap(\Ai\otimes\GE_0)$. The fibers of the Serre sheaf $\Ei_\B$ are given by 
$$
\Ei_\B(v)=E_\N\otimes_\Ai\C_v\cong E_\N/\Gm_vE_\N.
$$
In general $\GE_\N/\Gm_v\GE_\N$ and $E_\N/\Gm_vE_\N$ may not be isomorphic. %Note that $\Ei_\B=E_\N\otimes_\Ai\Oi_\B$, while
%$$
%\Ei_{\rm CD}=\GE_\N\otimes_\Ai\Oi_\B.
%$$
In this section we compare the Cowen--Douglas sheaf of $\GE_\N$ with the Serre sheaf of $E_\N$. %Note that $\Ei_\B$ is a quotient of $\Oi_\B\otimes\GE_0$ whereas $\Ei_{\rm CD}$ is a subsheaf of $\Oi_\B\otimes\GE_0$. Ok, no since the total space of $\Ei_{\rm CD}$ is the sub-thing. The quotients $\GE_\N/\Gm_v\GE_\N$ do not define a holomorphic linear subspace of $\B\times\GE_0$, so they do not define a quotient sheaf. 

Let $\GI_\N$ be the orthogonal complement of $\GE_\N$; thus $\GI_\N$ is a graded submodule of $\GH_\N\otimes\GE_0$ and equals the completion of a graded $\Ai$-submodule $I_\N$ of $\Ai\otimes\GE_0$. 

\begin{Lemma}[{cf. \cite[p. 1690]{Fang4}}]
%Let $\GI_\N$ be a Hilbert submodule of $\GH_\N\otimes\GE_0$. 
Define the \textbf{fiber space} over $v\in\B$ of the Hilbert module $\GI_\N$ to be the vector space 
$$
\GI_\N(v):=\{\psi(v)\in\GE_0|\ \psi\in\GI_\N\}.
$$
Then the map
\begin{equation}\label{eviso}
\GI_\N/(\GI_\N\cap(\Gm_v\GH_\N\otimes\GE_0))\to\GI_\N(v)
\end{equation}
induced by evaluation of functions at $v$ is an isomorphism, and %if $\GE_\N$ denotes the orthogonal complement of $\GI_\N$ in $\GH_\N\otimes\GE_0$ then 
we have a short exact sequence %\cite[cf. Eq. (4.11)]{Fang4}
\begin{equation}\label{FangSESifGle}
0\to\GI_\N(v)^*\to\Oi_{\rm CD}(v)\otimes\GE_0\to\Ei_{\rm CD}(v)\to 0
\end{equation}
of vector spaces. %, where $\Ei_{\rm CD}(v):=\Ker(S_E^*-\bar{v}\bone)\cong\GE_\N/\Gm_v\GE_\N$. 
\end{Lemma}
\begin{proof}
Let $v\in\B$ and let $e_v:\GH_\N\otimes\GE_0\to\GE_0$ be the evaluation at $v$, i.e. $e_v(\psi):=\psi(v)$. Clearly the restriction of $e_v$ to $\GI_\N$ is onto the fiber space $\GI_\N(v)$. Moreover, if $\psi(v)=0$ then $\bra\psi|k_v\otimes\xi\ket_{\GH_\N\otimes\GE_0}=\bra\xi|\psi(v)\ket_{\GE_0}=0$ so $\psi$ belongs to $\Ker(S^*-\bar{v}\bone)^\perp\otimes\GE_0=\Gm_v\GH_\N\otimes\GE_0$. Hence the map \eqref{eviso} is an isomorphism.

The short exact sequence
$$
0\to\GI_\N/(\GI_\N\cap(\Gm_v\GH_\N\otimes\GE_0))\to(\GH_\N/\Gm_v\GH_\N)\otimes\GE_0\to\GE_\N/\Gm_v\GE_\N\to 0
$$
then gives \eqref{FangSESifGle}. 

%Define a vector subspace $\E_{\rm CD}(v)$ of $\GE_0$ by writing $\Ker(S_E^*-v\bone)=k_{\bar{v}}\otimes\E_{\rm CD}(v)$. 
%The vector space $\E_{\rm CD}(w)$ is a subspace of $\GE_0$ and hence comes with a natural inner product. 
%The orthogonal complement of $\E_{\rm CD}(v)$ in $\GE_0$ is given by the fiber space $\GI_\N(\bar{v})$. Indeed, we have
%\begin{align*}
%\GE_0\ominus\E(v)&=\{\xi\in\GE_0|\ \bra\psi|k_{\bar{v}}\otimes\xi\ket_{\GH_\N\otimes\GE_0}=0\text{ for all }\psi\in\GI_\N\}
%\\&=.
%\end{align*}
%and $\bra\psi|k_{\bar{v}}\otimes\xi\ket_{\GH_\N\otimes\GE_0}=\bra\xi|\psi(\bar{v})\ket_{\GE_0}$. This gives the short exact sequence \eqref{FangSESifGle}. 
\end{proof}
Let $\Ii$ be the Serre sheaf of $I_\N$, so that $\rank\Ii=\lim_{m\to\infty}\dim\GI_m/\dim\GH_m$. From \cite[Thm. 1.2]{GRS1} or \cite[Lemma 16]{Fang4} we have, for each $v\in\B$, 
$$
\dim\GI_\N(v)=\rank\Ii\implies \dim(\GI_\N/\Gm_v\GI_\N)=\rank\Ii.
$$

 For $\phi\in\GI_\N$ we have $\phi(v)=0$ if and only if $\phi$ belongs to $\GI_\N\cap\Gm_v(\GH_\N\otimes\GE_0)$, so if and only if we can write 
 $$
 \phi=\sum^n_{\alpha=1}(z_\alpha-v_\alpha)\psi_\alpha
 $$  
 with $\psi_1,\dots,\psi_n\in\GH_\N\otimes\GE_0$. By definition, $\GI_\N$ is \textbf{Gleason solvable} at $v$ if and only if for each $\phi\in\GI_\N\cap\Gm_v(\GH_\N\otimes\GE_0)$ we can take the $\psi_\alpha$'s to belong to $\GI_\N$. Since $\GE_\N$ is not necessarily invariant under $S_\alpha-v_\alpha\bone$, there could at the same time be possible to choose $\psi_\alpha$ to not belong to $\GI_\N$. But we see that $\GI_\N$ is Gleason solvable at $v$ if and only if the inclusion
\begin{equation}\label{Gleasiff}
\Gm_v\GI_\N\subset\GI_\N\cap\Gm_v(\GH_\N\otimes\GE_0).
\end{equation}
is an equality. 
We have $\dim\GI_\N(v)<\rank\Ii$ iff \eqref{Gleasiff} is a proper inclusion.

Let us now look at the algebraic analogues of the above. %Let $I_\N$ be the graded $\Ai$-module underlying $\GI_\N$. 
We have a short exact sequence 
  $$
0\to I_\N/(I_\N\cap(\Gm_v\otimes\GE_0))\to\Ai/\Gm_v\otimes\GE_0\to\Ei_\B(v)\to 0.
 $$
The vector space $I_\N/(I_\N\cap(\Gm_v\otimes\GE_0))$ is isomorphic to $I_\N(v):=\{\phi(v)|\ \phi\in I_\N\}$ %since $\Ker e_v=I_\N\cap\Gm_v(\Ai_\N\otimes\GE_0)$ 
so there is also a short exact sequence 
 $$
0\to I_\N(v)\to\Ai/\Gm_v\otimes\GE_0\to\Ei_\B(v)\to 0.
 $$
 Thus $\Ei_\B$ is locally free at $v$ iff $\dim I_\N(v)$ does not drop from its maximal value $\rank\Ii$. %That is, iff $\Gm_vI_\N=I_\N\cap\Gm_v(\Ai_\N\otimes\GE_0)$. 

\begin{prop}\label{SerrevsCDthm}
Let $\GE_\N$ be a graded quotient of $\GH_\N\otimes\GE_0$ with underlying graded $\Ai$-module $E_\N$, and suppose that the Serre sheaf $\Ei$ of $E_\N$ is locally free. Then the Cowen--Douglas sheaf $\Ei_{\rm CD}$ of $\GE_\N$ is locally free on $\B\setminus\{0\}$ and isomorphic to the dual of the pullback $\Ei_{\B\setminus\{0\}}$ of $\Ei$ to $\B\setminus\{0\}$, 
$$
\Ei_{\rm CD}\cong\Ei_{\B\setminus\{0\}}^*.
$$ 
\end{prop}
\begin{proof}
Since $\Ei_{\B\setminus\{0\}}$ is locally free, $\dim I_\N(v)=\rank\Ii$ for all $v\in\B\setminus\{0\}$. The inclusion $I_\N(v)\subset\GI_\N(v)$ is thus an equality for each $v$, so that $\Ei_{\rm CD}$ is locally free on $\B\setminus\{0\}$ as well. The
algebraic vector bundle $\bigcup_{v\in\B\setminus\{0\}}I_\N/\Gm_vI_\N$ is isomorphic to $\bigcup_{v\in\B\setminus\{0\}}I_\N(v)$. Since equality holds in \eqref{Gleasiff} for each $v\in\B\setminus\{0\}$, the vector bundle $\bigcup_{v\in\B\setminus\{0\}}I_\N(v)$ is analytically isomorphic to the holomorphic vector bundle $\bigcup_{v\in\B\setminus\{0\}}\GI_\N/\Gm_v\GI_\N$ via the maps induced by the evaluation homomorphisms $e_v:\GI_\N\to\GE_0$. 
\end{proof}

We shall later see that when $\Ei_{\rm CD}$ is locally free it is $\D^\times$-equivariant (up to a factor of $\Oi_{\rm CD}$), just as $\Ei_{\B\setminus\{0\}}$. When $\Ei_{\rm CD}$ and $\Ei_{\B\setminus\{0\}}$ are isomorphic, this need not be by a $\D^\times$-equivariant isomorphism however. %, and this is important because in some special cases the isomorphism is in fact $\D^\times$-equivariant. 
\begin{Remark}[Germ model]
Another sheaf associated to a submodule $\GI_\N$ is studied in \cite{BMP1}, namely the sheaf $\Ii_{\rm BMP}$ whose stalk at $v\in\B$ is given by
$$
\Ii_{\text{BMP},v}:=\Big\{\sum_{j=1}^k(\phi_j|_{\{v\}})\psi_j\Big|\ \phi_1,\dots,\phi_k\in\GI_\N,\ \psi_1,\dots,\psi_k\in\Oi_v,\ k\in\N\Big\}.
$$
The sheaf $\Ii_{\rm BMP}$ coincides with the ``germ model'' of $\GI_\N$ in the sense of Cheng--Fang \cite[\S4]{ChFa1}. The Cowen--Douglas sheaf $\GE_\N\otimes_\Ai\Oi_\B$ is instead the restriction to $\B\subset\V$ of the sheaf model of \cite[\S4]{ChFa1},
$$
\GE_\N\otimes_\Ai\Oi_\V=(\GE_\N\otimes\Oi_\V)/((S_E\otimes\bone-\bone\otimes Z)(\GE_\N\otimes\Oi_\V))=\Oi_\V(\GE_\N)/(S_E-Z\bone)\Oi_\V(\GE_\N).
$$

By definition $\Ii_{\rm BMP}$ is a subsheaf of $\Oi_\B\otimes\GE_0$. It coinides with the image under the map (cf. \cite[Eq. (1.1)]{BMP1})
$$
\Ii_{\rm CD}=\GI_\N\otimes_{\Oi(\B)}\Oi_\B\to(\GH_\N\otimes\GE_0)\otimes_{\Oi(\B)}\Oi_\B=\Oi_\B\otimes\GE_0.
$$
As in \cite[Eq. (1.3)]{BMP1} this gives a surjection of analytic sheaves %(\textbf{but Ivan, why?})
$$
\Ii_{\rm CD}\to\Ii_{\rm BMP}\to 0
$$
and surjections on fibers
$$
\Ker(S_I^*-v\bone)\to\Ii_{\rm BMP}(v)\to 0.
$$
Thus, as soon as the dimensions of the fibers coincide the two sheaves will be analytically isomorphic. %So we need just assume the Gleason property. 
Since $\dim\Ker(S_I^*-v\bone)\geq\dim\Ii_{\rm BMP}(v)\geq\rank\Ii$, the sheaves $\Ii_{\rm CD}$ and $\Ii_{\rm BMP}$ are thus analytically isomorphic over $\B\setminus\sing_{\rm C}(\GI_\N)$. 

%\textbf{What} is it \cite{BMP1} really prove for a submodule of the form $\GI_\N=[I_\N]$? Well they prove that $\dim\Ker(S_I^*-v\bone)=\dim\Ii_{\rm BMP}(v)$ for all $v$, what would give that $\Ii_{\rm CD}$ and $\Ii_{\rm BMP}$ are analytically isomorphic over $\B$. However, it is unclear if they assume $\dim\GE_0$ and/or assume that $\GI_\N$ is in their class $\GB_1(\B)$. 

\end{Remark}

\begin{Remark}[Spanning eigenvectors]
For an arbitrary graded submodule $\GI_\N$ we have $\Gm_v\GI_\N\subset\GI_\N\cap\Gm_v(\GH_\N\otimes\GE_0)$. This implies
\begin{equation}\label{spanningGMv}
\bigcap_{v\in\U}\Gm_v\GI_\N=\{0\}
\end{equation}
for all open subsets $\U\subset\B$. Indeed, for $\phi\in\bigcap_{v\in\U}\Gm_v\GI_\N$ we have $\phi(v)=0$ for all $v\in\U$ and so $\phi$ is zero on all of $\B$ by the identity theorem of holomorphic functions \cite[\S 0.6]{Kaup1}. 

We can rewrite \eqref{spanningGMv} as
\begin{equation}\label{spanningeigen}
\bvee_{v\in\U}\Ker(S_I^*-v\bone)=\GI_\N,
\end{equation}
which says that the eigenvectors of $S_I^*$ span the whole space $\GI_\N$. This does not imply that $\GI_\N$ is coinvariant, since typically no eigenvectors of $S_I^*$ are of the form $k_v\otimes\xi$ for some $(v,\xi)\in\B\times\GE_0$.  

%For a quotient module $\GE_\N$ the condition $\psi\in\Ran(S_E-v\bone)$ does not imply $\psi(v)=0$. Therefore we cannot conclude $\bvee_{v\in\U}\Ker(S_E^*-v\bone)=\GE_\N$ using the identity theorem. \textbf{Can it fail?} We know that there is a space $\E_{\rm CD}$ such that $\Ker(S^*_E-v\bone)=k_v\otimes\E_{\rm CD}(v)$. If we start with some total space $\E$ to define $\GE_\N$ then we obtain $\E\subset\E_{\rm CD}$. This seems to guarantee $\bvee_{v\in\U}\Ker(S_E^*-v\bone)=\GE_\N$. Indeed, as long as $\GE_\N$ is spanned by eigenvectors of $S^*$ it will also be spanned by eigenvectors of $S_E^*$. And we know that $\GH_\N\otimes\GE_0$ is spanned by eigenvectors of $S^*$ and hence so will each $S^*$-invariant subspace. So somehow coinvariance ensures that each $\psi\in\Ran(S_E-v\bone)$ satisfies $\psi(v)=0$.

\end{Remark}

\begin{Remark}[Reducing subspaces and subbundles]\label{reducrem}
Let $S$ be the shift on $\GH_\N\otimes\GE_0$ and let $S_I$ and $S_E$ be its compression to $\GI_\N$ and $\GE_\N$. 
The equality 
 \begin{equation}\label{redusubb}
\Ker(S^*-w\bone)=\Ker(S_I^*-w\bone)\oplus\Ker(S_E^*-w\bone),\qquad\forall w\in\B\setminus\{0\}.
\end{equation} 
can fail dramatically. For instance, if $\GI_\N\subset\M$ is the closure of an ideal in $\Ai$ defining an analytic subvariety $\E$ of $\V$ then for $v\in\E$ we have $\dim\Ker(S^*-v\bone)=1=\Ker(S^*_E-v\bone)$ while \cite[Cor. 2.12]{BMP1}
$$
\dim\Ker(S^*-v\bone)=n-\dim_\C\E.
$$
This comes from the failure of $\GI_\N$ to be Gleason solvable at $v$. We claim that $\Ker(S_I^*-v\bone)$ is a subspace of  $\Ker(S^*-v\bone)$ if and only if $\GI_\N$ is reducing. To see this, define the matrix-valued kernels $\mathbf{K}^I$ and $\mathbf{K}^E$ by
$$
\mathbf{K}^I(v,w)\xi:=(P_I(k_w\otimes\xi))(v),\qquad \mathbf{K}^E(v,w)\xi:=(P_E(k_w\otimes\xi))(v).
$$
%By definition, the matrix-valued kernel $\mathbf{K}(z,w)\xi:=k_w(z)\otimes\xi$ then decomposes into $\mathbf{K}=\mathbf{K}^I+\mathbf{K}^E$. 
For all $\xi$ in the subspace $\E(w)\subset\GE_0$ we have
 $$
S^*\mathbf{K}^E(\cdot,w)\xi=\bar{w}\mathbf{K}(\cdot,w)\xi.
$$
%However, $k_w\otimes\xi$ can have a nonzero component in $\GE_\N$ even if $k_w\otimes\xi\notin\GE_\N$. That is, %$\mathbf{K}^E(\cdot,w)\xi:=P_E(k_w\otimes\xi)$ can be nonzero also for $(w,\xi)\notin\E$. But we have the reproducing property
%$$
%\bra f|\mathbf{K}^I(\cdot,w)\xi\ket_{\GI_\N}=\bra f(w)|\xi\ket_{\GE_0},\qquad\forall f\in\GI_\N,\ \xi\in\GE_0.
%$$
%As remarked in \cite[\S2]{ChDo1}, this implies that 
%$$
%S^*_{I,\alpha}\mathbf{K}^I(\cdot,w)\xi=\bar{w}_\alpha\mathbf{K}^I(\cdot,w)\xi.
%$$
%This happens thus despite the fact that 
On the other hand,
\begin{equation}\label{KInoteigen}
S^*_{\alpha}\mathbf{K}^I(\cdot,w)\xi=(\bar{w}_\alpha+[S^*_{\alpha},P_I])\mathbf{K}^I(\cdot,w)\xi.
\end{equation}
%i.e. the extra application of $P_I$ kills the error term which arises because $\GI_\N$ is not coinvariant under $S$. 
Therefore, while $P_I(k_w\otimes\xi)$ is in the kernel of $S_I^*-\bar{w}\bone$ for all $\xi\in\GE_0$, it is not necessarily in the kernel of $S^*-\bar{w}\bone$. 
Indeed, \eqref{KInoteigen} shows that $\Ker(S_I^*-w\bone)\subset\Ker(S^*-w\bone)$ only if $[S^*_{\alpha},P_I]=0$ for all $\alpha\in\{1,\dots,n\}$. The latter is to say that $\GI_\N$ is a reducing subspace under $S$. So \eqref{redusubb}, which says that we have a holomorphic direct sum $\Oi_{\rm CD}\otimes\GE_0=\Ii_{\rm CD}\oplus\Ei_{\rm CD}$, holds if and only if $\GI_\N$ is reducing. 
\end{Remark}

%\begin{Remark}[Localization]
%For $k\in\N$ and $v\in\B$ denote by $\Gm_v^k$ the ideal in $\Ai$ whose elements are linear combinations of products of $k$ elements from the ideal $\Gm_v$. If $\GI_\N$ is Gleason solvable at $v$ then $\Gm_v^k\GI_\N$ is the subspace of $\phi\in\GI_\N$ such that $\phi$ vanishes to order $k$ at $v$, i.e. $\phi$ and all its derivative up to order $k$ vanish at $v$,
%$$
%\Gm_v^k\GI_\N=\{\phi\in\GI_\N|\ \phi(v)=\cdots=\phi^{(k-1)}(v)=0\}.
%$$ 
%The compression of $S_E$ to the coinvariant subspace $\GI_\N\ominus\Gm_v^k\GI_\N$ is studied in \cite{ChDo2} and called the ``$k$th-order localization'' of $S_E$ at $v$.  
%\end{Remark}

\subsection{Extension and boundary values}\label{bvaluessec}

\subsubsection{Abel convergence: $\varsigma$ versus $\varsigma_\B$}\label{Abelsumsec}

Since $\GH_\N$ has a reproducing vector $k_v$ at each $v\in\B$ we can associate in a standard fashion \cite{Bere2} to each operator $A\in\Bi(\GH_\N)$ a function
$$
\varsigma_\B(A)(v):=\bra k_v|Ak_v\ket,\qquad\forall v\in\B,
$$
which we call the \textbf{$\B$-Berezin symbol} of $A$. It can extended to a function on $\bar{\B}\times\B$ which is holomorphic in the first variable and antiholomorphic in the second variable,
$$
\varsigma_\B(A)(v,w)=\frac{\bra k_v|Ak_w\ket}{\bra k_v|k_w\ket},\qquad\forall v,w\in\B,
$$
but we will usually just consider the diagonal values. More generally, for $A\in\Bi(\GH_\N\otimes\GE_0)$ for some Hilbert space $\GE_0$ we define
$$
\varsigma_\B(A)(v):=(\Tr_{\GH_\N}\otimes\id)(A(|k_v\ket\bra k_v|\otimes\bone_{\GE_0})),\qquad\forall v\in\B,
$$
where $|k_v\ket\bra k_v|$ is the rank-1 projection onto the subspace $\C k_v$. Thus $\varsigma_\B(A)$ is a $\Bi(\GE_0)$-valued function on $\B$.

In \cite{Kara4} the Berezin symbol for the unit disk was used to prove a theorem of Abel, namely that if a sequence $(a_m)_{m\in\N_0}$ of complex numbers is convergent to $a\in\C$ then 
$$
a=\lim_{r\to 1^-}\sum_{m\in\N_0}a_mr^m.
$$
Here we shall use Abel's theorem to show that the Berezin symbol map $\varsigma$ gives the boundary limits of the $\B$-Berezin symbol map $\varsigma_\B$ when restricted to the algebra of grading-preserving Toeplitz operators with continuous symbol.

 We shall use some facts from \cite{An6}, namely that $C^0(\M)$ is the norm closure of a union of subspaces $\varsigma^{(m)}(\Bi(\GH_m))$ where $\varsigma^{(m)}:\Bi(\GH_m)\to C^0(\M)\subset L^2(\M,\omega)$ is the adjoint of the Toeplitz map $\breve{\varsigma}^{(m)}:L^\infty(\M)\to\Bi(\GH_m)$. The characteristic property of $\varsigma^{(m)}$ is that it maps normally ordered products of shift operators directly to their classical limits,
  $$
 \varsigma^{(m)}(p_mS_\mathbf{j}S_\mathbf{k}^*|_{\GH_m})=Z_\mathbf{j}Z_\mathbf{k}^*,
 $$
 whenever $|\mathbf{j}|=m=|\mathbf{k}|$. So if we express an operator $A\in\Bi(\GH_m)$ as a matrix $(A_{\mathbf{j},\mathbf{k}})_{|\mathbf{j}|=m=|\mathbf{k}|}$ in the frame $(p_me_\mathbf{k})_{|\mathbf{k}|=m}$, which is to say that we write $A=\sum_{|\mathbf{j}|=m=|\mathbf{k}|}A_{\mathbf{j},\mathbf{k}}S_\mathbf{j}S_\mathbf{k}^*|_{\GH_m}$, then 
\begin{equation}\label{varsigmachar}
 \varsigma^{(m)}(A)=\sum_{|\mathbf{j}|=m=|\mathbf{k}|}A_{\mathbf{j},\mathbf{k}}Z_\mathbf{j}Z_\mathbf{k}^*.
\end{equation}
We can extend $\varsigma^{(m)}(A)$ to a $\Un(1)$-equivariant function on $\Sb$, which takes the simple form
 \begin{equation}\label{varsigmacharonS}
 \varsigma^{(m)}(A)(\zeta)=\sum_{|\mathbf{j}|=m=|\mathbf{k}|}A_{\mathbf{j},\mathbf{k}}\zeta_\mathbf{j}\zeta_\mathbf{k}^*,\qquad\forall\zeta\in\Sb.
\end{equation}
We shall now compare $\varsigma^{(m)}(A)$ with $\varsigma_\B(A)$. For all $\mathbf{j},\mathbf{k}\in\F_n^+(m)$ and $v\in\B$ we have 
\begin{align*}
\varsigma_\B(S_\mathbf{j}S_\mathbf{k}^*p_m)(v)&=\bra S_\mathbf{j}^*k_v|S_\mathbf{k}^*p_mk_v\ket=(1-|v|^2)\bra\Omega|\Omega\ket v_\mathbf{j}v_\mathbf{k}^*=(1-|v|^2)v_\mathbf{j}v_\mathbf{k}^*
\\&=(1-r^2)r^{2m}\zeta_\mathbf{j}\zeta_\mathbf{k}^*
%\\&=(1-|v|^2)\frac{1}{m!}\|v\|^{2m}(Z_\mathbf{j}Z_\mathbf{k}^*)(v)
%\\&=(1-|v|^2)\frac{1}{m!}\|v\|^{2m}\varsigma^{(m)}(S_\mathbf{j}S_\mathbf{k}^*p_m)(v),
\end{align*}
where we write $v=r\zeta$ with $r\in(0,1)$ and $\zeta\in\Sb$. So for a finite-rank grading preserving operator $A\in\Bi(\GH_m)\subset\Bi(\GH_\N)$ we get, using the formula \eqref{varsigmachar}, that
\begin{equation}\label{finiterankBere}
\varsigma_\B(A)(r\zeta)=(1-r^2)r^{2m}\varsigma^{(m)}(A)(\zeta),\qquad\forall r\zeta\in\B.
\end{equation}

%Therefore, since $k_v$ is an eigenvector of $S_\mathbf{k}^*$ with eigenvalue $\bar{v}_\mathbf{k}$, 
%for all $r\in(0,1)$ and all $\zeta\in\Sb$ we have
%\begin{equation}\label{howobv}
%\varsigma_\B(A)(r\zeta)=(1-r^2)\sum_{m\in\N_0}r^{2m}\varsigma^{(m)}(A_m)(\zeta).
%\end{equation}
 
 In the following lemma we are not using the assumption that $\M$ is a coadjoint orbit. 
\begin{Lemma}\label{AbelsumLemma}
Let $A=(A_m)_{m\in\N_0}$ be an operator in the Toeplitz core $\Ti_\GH^{(0)}\subset\prod_{m\in\N_0}\Bi(\GH_m)$ and regard its symbol $\varsigma(A)\in C^0(\M)$ as a $\Un(1)$-equivariant function on $\Sb$. Then 
$$
\varsigma(A)(\zeta)=\lim_{r\to 1^-}\varsigma_\B(A)(r\zeta)=\lim_{r\to 1^-}(1-r^2)\sum_{m\in\N_0}r^{2m}\varsigma^{(m)}(A_m)(\zeta),\qquad\forall\zeta\in\Sb.
$$
In the special case of a Toeplitz operator $A=\breve{\varsigma}(f)$ with $f\in C^0(\M)$ we have
$$
f_\Sb(\zeta)=\varsigma(\breve{\varsigma}(f))(\zeta)=\lim_{r\to 1^-}(1-r^2)\sum_{m\in\N_0}r^{2m}n_m\sum_{|\mathbf{j}|=m=|\mathbf{k}|}\omega(Z_\mathbf{j}^*Z_\mathbf{k}f)\zeta_\mathbf{j}\zeta_\mathbf{k}^*,\qquad\forall\zeta\in\Sb.
$$
\end{Lemma}
\begin{proof}
Let $\zeta\in\Sb$ be given. Since the sequence $(\varsigma^{(m)}(A_m))_{m\in\N_0}$ converges in the norm of $C^0(\M)$ to the function $\varsigma(A)\in C^0(\M)$, the sequence $(\varsigma^{(m)}(A_m)(\zeta))_{m\in\N_0}\in c_b(\N_0)$ converges to the complex number $\varsigma(A)(\zeta)$. By Abel's theorem, we get
$$
\varsigma(A)(\zeta)=\lim_{m\to\infty}\varsigma^{(m)}(A_m)(\zeta)=\lim_{r\to 1^-}(1-r^2)\sum_{m\in\N_0}r^{2m}\varsigma^{(m)}(A_m)(\zeta).%=\lim_{r\to 1^-}\varsigma_\B(A)(r\zeta).
$$
The proof is complete by the formula \eqref{finiterankBere}.

%But again, since $\lim_{m\to\infty}\varsigma^{(m)}(\breve{\varsigma}^{(m)}(f)=f$, the operator of multiplication by the function $f=f(r\zeta)$ and the operator of multiplication by the function $\varsigma_\B(\breve{\varsigma}(f))(r\zeta)=(1-r^2)\sum_{m\in\N_0}\varsigma^{(m)}(\breve{\varsigma}^{(m)}(f)(\zeta)$ on the Hardy space $H^0(\Sb)$ differ by a compact operator. Hence they have the same boundary limit. Well ok, these multiplication operators do not preserve the Hardy space, so it does not work. 
\end{proof}
Using the expression \eqref{FSframe} for $\FS(\GH_m)$ we can rewrite the formula \eqref{varsigmachar} for the symbol $\varsigma^{(m)}(A)$ of an operator $A\in\Bi(\GH_m)$ as
$$
\varsigma^{(m)}(A)(x)=\Tr(\FS(\GH_m)(x)A),\qquad\forall x\in\M.
$$
Since the Toeplitz map $\breve{\varsigma}^{(m)}$ is the adjoint of $\varsigma^{(m)}$ with respect to $\omega$ and $\phi_m$, we obtain for each $f\in C^0(\M)$ the formula
\begin{equation}\label{Toeplexplic}
\breve{\varsigma}^{(m)}(f)=n_m(\omega\otimes\id)(\FS(\GH_m)(f\otimes p_m)),
\end{equation}
which will be useful later in the paper. 
\begin{Remark}
Let $\Hi$ be an reproducing kernel Hilbert space of functions on a set $\B$, and suppose that $\B$ sits inside a topological space and has nonempty topological boundary $\Sb=\pd\B$. Denote by $k_v$ the normalized reproducing kernel of $\Hi$. Then $\Hi$ is \textbf{standard} if the sequence $(k_{v_m})_{m\in\N}$ converges weakly to zero for every sequence $(v_m)_{m\in\N}$ of points in $\B$ converging to a point in $\Sb$. For the space $\Hi=\GH_\N$ we have
$$
\bra k_v|f\ket=(1-|v|^2)^{1/2}f(v),\qquad\forall f\in\GH_\N,
$$
so it is clear that $\lim_{|v|\to 1^-}\bra k_v|f\ket=0$ for all $f\in\Ai$. %(since every multiplier has boundary function in $L^\infty(\Sb)$). 
Since $\Ai$ is dense in $\GH_\N$ we get $\lim_{|v|\to 1^-}\bra k_v|\psi\ket=0$ for all $\psi\in\GH_\N$.
Thus the reproducing kernel Hilbert space $\GH_\N$ is standard. 
It follows that every compact operator $C$ on $\GH_\N$ has $\B$-Berezin symbol vanishing on the boundary,
$$
\lim_{r\to 1^-}\varsigma_\B(C)(r\zeta)=0,\qquad\forall\zeta\in\Sb.
$$
For an operator $A$ in $\Ti_\GH$ we have $A=\breve{\varsigma}(\varsigma(A))+C$ with $C\in\Gamma_0$, and hence 
$$
\varsigma(A)=\lim_{m\to\infty}\varsigma^{(m)}(Ap_m)=\lim_{m\to\infty}\varsigma^{(m)}(\breve{\varsigma}^{(m)}(\varsigma(A))).
$$
since $\varsigma\circ\breve{\varsigma}=\id$. 
Therefore
$$
\lim_{r\to 1^-}\varsigma_\B(A)(r\zeta)=\lim_{r\to 1^-}\varsigma_\B(\breve{\varsigma}(\varsigma(A)))(r\zeta),\qquad\forall\zeta\in\Sb.
$$
%It is therefore enough to prove Lemma \ref{AbelsumLemma} in case of a Toeplitz operator $A=\breve{\varsigma}(f)$ with $f\in C^0(\M)$. As we saw, the proof is the same for arbitrary $A\in\Ti_\GH^{(0)}$ however. 
\end{Remark}
%\begin{Remark}
%We can also use the essential normality of $S$ to conclude as in \cite[Proof of Thm. 4.3]{GRS2} that
%$$
%\lim_{v\to|\zeta|}\|fk_v\|=|f(\zeta)|,\qquad\forall f\in\Ai,\ \zeta\in\Sb.
%$$
%\end{Remark}

%Note that every element of $\Ai_\GH$ has a continuous boundary limit on $\Sb$. Hence so has every element of $\Ti_\GH=\overline{\mathrm{span}}\Ai_\GH\Ai_\GH^*$. So we can extend the above result to operators which do not necessarily preserve the grading on $\GH_\N$. The claim is thus that for $X\in\Ti_\GH$, if we take a Beurling presentation $X=\sum_{j,k\in\N_0}\xi_j\xi_k^*$ and identify an element $\xi_j\in\Ai_\GH\subset\Ti_\GH$ by the function $\xi_j\in\Ai_\GH\subset\Ai(\B)$ with which it shifts, then the boundary value of $(1-r^2)X(r\zeta)$ is precisely $\varsigma(X)(\zeta)$. But isn't $ \Ai_\GH\Ai_\GH^*$ just the set of elements $\xi_j\xi_j^*$, so that we only need $j=k$ but possibly some complex coefficients in front? Should be. So it should be enough to look at $\xi\xi^*$ with $\xi\in\Ai_\GH$. But if $X(r\zeta)$ has a finite limit as $r\to 1^-$ then $(1-r^2)X(r\zeta)$ should go to zero. 

Recall that the unique $\G$-invariant state $\omega$ on $C^0(\M)$ coincides with the limit $\omega=\lim_{m\to\infty}\phi_m$ of the normalized traces $\phi_m:\Bi(\GH_m)\to\C$. Moreover, $\omega$ extends to the unique $\G$-invariant state $\omega_\Sb$ on $C^0(\Sb)$, which coincides with the normalized surface measure when $\Sb$ is regarded as the boundary of a domain $\B$ in $\C^n$. Using these facts we have an $L^\infty$ version of Lemma \ref{AbelsumLemma}:
\begin{prop}\label{SOTAbelsumLemma}
Let $A=(A_m)_{m\in\N_0}$ be an operator in the $L^\infty$ Toeplitz core $\Ni\subset\prod_{m\in\N_0}\Bi(\GH_m)$ and regard its symbol $\varsigma(A)\in L^\infty(\M)$ as a $\Un(1)$-equivariant function on $\Sb$. Then 
$$
\varsigma(A)(\zeta)=\lim_{r\to 1^-}\varsigma_\B(A)(r\zeta)=\lim_{r\to 1^-}(1-r^2)\sum_{m\in\N_0}r^{2m}\varsigma^{(m)}(A_m)(\zeta)
$$
for $\omega_\Sb$-almost all $\zeta\in\Sb$. 
\end{prop}
%\begin{proof}
%Since $\varsigma_\B(A)$ is a bounded analytic function it has radial boundary values almost everywhere. We have
%\begin{align*}
%\omega(\lim_{r\to 1^-}\varsigma_\B(A)(r\ \cdot\ ))&=\lim_{r\to 1^-}(1-r^2)\sum_{m\in\N_0}r^{2m}\omega(\varsigma^{(m)}(A_m)
%\\&=\lim_{r\to 1^-}(1-r^2)\sum_{m\in\N_0}r^{2m}\phi_m(A_m)=\lim_{m\to\infty}\phi_m(A_m)
%\\&=\omega(\varsigma(A)),
%\end{align*}
%so the functions $\varsigma(A)$ and $\lim_{r\to 1^-}\varsigma_\B(A)(r\ \cdot\ )$ coincide almost everywhere. 
%\end{proof}

We saw in Proposition \ref{SOTvsprop} that the SOT-asymptotic Toeplitz symbol map $\varsigma^{\rm SOT}$ also coincides with $\varsigma$ when restricted to $\Ni$. So:
%Next we want to show that the SOT-asymptotic Toeplitz symbol map $\varsigma^{\rm SOT}$ also coincides with $\varsigma$ when restricted to $\Ni$. But that is easy since $\Ni=\Bi^\infty+\Ker\varsigma$ and the conditional expectation onto $\Bi^\infty$ just kills $\Ker\varsigma$, which can thus be identified with the set of elements $C$ of $\Ni$ with $\varsigma^{\rm SOT}(C)=0$. 
\begin{cor}[Asymptotic Toeplitz symbols versus boundary limits]
The map
$$
\Ni\ni X\to \lim_{r\to 1^-}\varsigma_\B(X)(r\ \cdot)\in L^\infty(\M)
$$
coincides with $\varsigma^{\rm SOT}$.
\end{cor}

\subsubsection{Extension of vector bundles}\label{extvectsec}

Denote by $\C^\times=\GeL(1,\C)$ the multiplicative group of nonzero complex numbers and consider the semigroup $\D^\times:=\{\lambda\in\C^\times|\ 0<|\lambda|<1\}$. As before we denote by $\Un(1)$ the circle group (the unitary group of dimension 1), identified with the subgroup of $\C^\times$ consisting of complex numbers of modulus 1. 

In this subsection we work with an arbitrary smooth projective variety $\M\subset\C\Pb^{n-1}$. Let $\pi:\C^n\setminus\{0\}\to\C\Pb^{n-1}$ be the natural surjection associated with the $\C^\times$-action on $\C^n\setminus\{0\}$, and set 
$$
\V:=\pi^{-1}(\M)\cup\{0\},\qquad \B:=\V\cap\B^n,\qquad \Sb:=\pd\B=\V\cap\Sb^{2n-1},
$$
where $\B^n\subset\C^n$ is the unit ball and $\Sb^{2n-1}=\pd\B^n$ is the unit sphere. The manifold $\M$ can be obtained by quoting out actions on the manifolds $\V\setminus\{0\}$, $\B\setminus\{0\}$, and $\Sb$:
$$
\M=(\V\setminus\{0\})/\C^\times=(\B\setminus\{0\})/\D^\times=\Sb/\Un(1).
$$

If $\Ei$ is an $\Oi_\M$-module then its inverse image under $\pi$, denoted by $\pi^{-1}\Ei$ and defined on open subsets $\U\subset\V$ by 
$$
(\pi^{-1}\Ei)(\U):=\Ei(\pi(\U)),
$$
is a $\pi^{-1}\Oi_\M$-module.  The sheaf $\pi^*\Ei$ defined on $\V\setminus\{0\}$ by
$$
\pi^*\Ei:=\pi^{-1}\Ei\otimes_{\pi^{-1}\Oi_\M}\Oi_{\V\setminus\{0\}}
$$
is an $\Oi_{\V\setminus\{0\}}$-module called the pullback (or analytic inverse image) of $\Ei$ under $\pi$. 

The space of global holomorphic sections of the structure sheaf $\Oi_{\V\setminus\{0\}}$ is isomorphic to the coordinate ring $\Ai$ of $\M$. The subsheaf 
$$
\pi^{-1}\Oi_\M\subset\pi^*\Oi_\M=\Oi_{\V\setminus\{0\}}
$$
has no nonconstant global holomorphic sections. Denoting by $(\Oi_{\V\setminus\{0\}})^{\C^\times}$ the $\C^\times$-invariant part of the structure sheaf $\Oi_{\V\setminus\{0\}}$ we have 
$$
(\Oi_{\V\setminus\{0\}})^{\C^\times}=\pi^{-1}\Oi_\M.
$$
So for any $\Oi_\M$-module $\Ei$ we have
$$
(\pi^*\Ei)^{\C^\times}=(\pi^{-1}\Ei\otimes_{\pi^{-1}\Oi_\M}\Oi_{\V\setminus\{0\}})^{\C^\times}=\pi^{-1}\Ei.
$$
This gives
$$
(\pi_*\pi^*\Ei)^{\C^\times}:=\pi_*(\pi^*\Ei)^{\C^\times}=\Ei.
$$
\begin{dfn}
An $\Oi_{\V\setminus\{0\}}$-module $\Ei_{\V\setminus\{0\}}$ is $\C^\times$-\textbf{equivariant} if $\C^\times$ acts on $\Ei_{\V\setminus\{0\}}$ compatibly with the $\Oi_{\V\setminus\{0\}}$-module structure on $\Ei_{\V\setminus\{0\}}$ and the $\C^\times$-action on $\Oi_{\V\setminus\{0\}}$. 
\end{dfn}
Since $\C^\times$ acts freely on $\V\setminus\{0\}$, an $\Oi_{\V\setminus\{0\}}$-module $\Ei_{\V\setminus\{0\}}$ is $\C^\times$-equivariant if and only if $\Ei_{\V\setminus\{0\}}$ equals $\pi^*\Ei$ for some $\Oi_\M$-module $\Ei$ \cite[Prop. 4.2]{KKT1}. %See [math.stackexchange.com/q/1990550].

%We want to show that if $\Ei_{\V\setminus\{0\}}$ is any $\C^\times$-equivariant $\Oi_{\V\setminus\{0\}}$-module then $\Ei_{\V\setminus\{0\}}$ equals $\pi^*\Ei$ for some $\Oi_\M$-module $\Ei$. In the quasicoherent case one could argue that $H^0(\V\setminus\{0\};\Ei_{\V\setminus\{0\}})$ is a graded $\Ai$-module iff $\Ei$ is $\C^\times$-equivariant, maybe. That would be interesting per se. Both of these statements seem to be equivalent to having the stalks of $\Ei_{\V\setminus\{0\}}$ generated by homogeneous polynomials, i.e. we have a generating set which forms a $\pi^{-1}\Oi_\M$-module.  

\begin{Lemma}\label{vectextlemma}
A coherent analytic sheaf $\Ei_{\B\setminus\{0\}}$ on $\B\setminus\{0\}$ is $\Un(1)$-equivariant if and only if there exists a coherent analytic sheaf $\Ei$ on $\M$ such that
$$
\Ei_{\B\setminus\{0\}}=\pi^*\Ei|_{\B\setminus\{0\}}.
$$
In this case $\Ei$ is unique, and $\Ei$ is locally free iff $\Ei_{\B\setminus\{0\}}$ is locally free.
\end{Lemma}
\begin{proof}
We use that the complex Lie group $\C^\times$ is the complexification of the compact Lie group $\Un(1)$. The manifold $\V\setminus\{0\}$ is the ``complexification'' $\C^\times\cdot\Sb$ of the $\Un(1)$-space $\Sb$ in the sense of \cite{Hein1}. As a special case of \cite{HaHe1}, we obtain that $\Ei_{\B\setminus\{0\}}$ has a unique extension to a $\C^\times$-equivariant coherent analytic sheaf $\Ei_{\V\setminus\{0\}}$ on $\V\setminus\{0\}$. Since $(\V\setminus\{0\})/\C^\times=\M$, we have $\Ei_{\V\setminus\{0\}}=\pi^*\Ei$ for some coherent analytic sheaf $\Ei$ on $\M$, as asserted. The uniqueness of $\Ei$ follows from the uniqueness of $\Ei_{\V\setminus\{0\}}$. The result about locally free sheaves is also in \cite{HaHe1}. 
\end{proof}
Every $C^0$ vector bundle over $\B$ admits a unique holomorphic structure by the Oka principle. However, this holomorphic structure is not $\D^\times$-equivariant in general, since otherwise every $C^0$ vector bundle over $\M$ would admit a holomorphic structure (which is not true). The above lemma says that $\D^\times$-equivariance of the holomorphic structure is the same as $\Un(1)$-equivariance.

Any $\Un(1)$-equivariant vector bundle $\Ei_\B$ on $\B\setminus\{0\}$ therefore also admits a Hermitian metric which is the pullback of a Hermitian metric on the induced bundle $\Ei$ on $\M$. Note however that $\Ei_\B$ also admits Hermitian metrics which are not $\D^\times$-equivariant (even if they are $\Un(1)$-equivariant). 

The relevance of this discussion for the present paper is that the Cowen--Douglas sheaf $\Ei_{\rm CD}$ of a graded quotients $\GE_\N$ of $\GH_\N\otimes\C^N$ is easily seen to be $\Un(1)$-equivariant and so by Lemma \ref{vectextlemma} it descends to coherent sheaves $\M$. It comes with a Hermitian metric, the Cowen--Douglas metric, which is also $\Un(1)$-equivariant but not always $\D^\times$-equivariant (i.e. not always a pullback of a metric on the induced bundle over $\M$). For the reference space $\GH_\N$, even though $\CD(\GH_\N)$ is not $\D^\times$-equivariant it defines a line bundle $\Oi_{\rm CD}$ which is isomorphic to the pullback of a line bundle on $\M$, viz. $\Oi_{\rm CD}=\Oi_{\B\setminus\{0\}}=\pi^*\Oi_\M|_{\B\setminus\{0\}}$.

%The vector spaces $\E_{\rm CD}(v)$ in Proposition \ref{tensorprop}

%\begin{prop}
%Let $\Ei_{\rm CD}$ be the Cowen--Douglas sheaf of a graded quotient module $\GE_\N$. Then there is a coherent $\Oi_\M$-module $\Ei$ such that
%$$
%\Ei_{\rm CD}\cong\pi^*\Ei|_{\B\setminus\{0\}}
%$$
%as analytic sheaves over $\B\setminus\{0\}$. 
%\end{prop}
%\begin{proof}

%\end{proof}

\subsubsection{The reproducing kernel}\label{repkernelsec}

Let $K^{H^n}(z,w)=(1-\bra w,z\ket)^{-1}$ be the reproducing kernel for the Drury--Arveson space $H^2_n$ and let $K(z,w)=K_w(z)$ be the reproducing kernel for the quotient module $\GH_\N$. If $P$ is the orthogonal projection of $H^2_n$ onto the subspace $\GH_\N$ then we get
$$
K_w(z)=(PK^{H^n}_w)(z)=\bra K_z|PK^{H^n}_w\ket=\bra K_z^{H^n}|PK^{H^n}_w\ket.
$$
Recall that $\GH_\N=\overline{\operatorname{span}}\{K_w|\ w\in\B\}\subset H^2_n$. So $PK_w^{H^n}=K_w^{H^n}$ for $w\in\B$. 
For $(z,w)\in\B\times\B$ we therefore get
$$
K_w(z)=K^{H^n}_w(z)=(1-\bra w,z\ket)^{-1}.
$$
Since we usually regard $\GH_\N$ as a space of functions on the subset $\B\subset\B^n$, we see that the reproducing kernel for $\GH_\N$ is just the restriction of the kernel for $H^2_n$. The projected kernel $PK^{H^n}_w$ however makes sense for all $w\in\B^n$ as a function on all of $\B^n$.

The orthogonal complement of $\GH_\N$ is not invariant under the backward shifts $S^*_1,\dots,S_n^*$, so it cannot be equal to $\overline{\operatorname{span}}\{K_v|\ v\in\B^n\setminus\B\}$. Therefore $PK_w$ can be nonzero also for $w\in\B^n\setminus\B$, and $PK_w\ne K_w$ always happens for such $w$'s. 
%$$
%K_w(z)=(PK^{H^n}_w)(z)=\bra K_z|PK^{H^n}_w\ket=\bra K_z^{H^n}|PK^{H^n}_w\ket=\begin{cases}
%    (1-\bra w,z\ket)^{-1},& \text{if } (z,w)\in\B\times\B\\
 %   0,              & \text{otherwise}
%\end{cases}.
%$$

The Beurling factorization $P=\bone-M_\Theta M_\Theta^*$ gives a formula for the extension of $K_w$ from $\B$ to the whole unit ball $\B^n$,
$$
K_w(z)=\frac{1-\Theta(z)\Theta(w)^*}{1-\bra w,z\ket},\qquad\forall z,w\in\B^n.
$$

Consider now the more general case of a vector-valued quotient module $\GH^E_\N$ of $\GH_\N\otimes\C^N$. The Beurling representation of $P_E=\bone-M_{\Theta_E}M_{\Theta_E}^*$ shows that the reproducing kernel $K^E_w(z)$ for $\GH^E_\N$ has the form 
$$
K^E(z,w)=\frac{\bone-\Theta_E(z)\Theta_E(w)^*}{\bone-\bra w,z\ket\bone}\in\Mn_N(\C).
$$
%just as in the cases of scalar-valued complete Nevanlinna--Pick kernels, but the multiplier $\Theta_E$ is now matrix-valued.
Using $M_{\Theta_E}^*(k_v\otimes\xi)=k_v\otimes\Theta_E(v)^*\xi$ we see that the numerator is precisely the $\B$-Berezin symbol of $P_E$ as defined in \S\ref{Abelsumsec},
$$
\frac{K^E(z,w)}{K(z,w)}=\bone-\Theta_E(v)\Theta_E(w)^*=%\bra k_v|(\bone-M_{\Theta_E}M_{\Theta_E}^*)k_w\ket=\bra k_v|P_Ek_w\ket=
\varsigma_\B(P_E)(v,w).
$$
%Here we used that $M_{\Theta_E}^*k_v=\Theta_E(v)^*$ and that 
%$$
%\bra k_v|k_w\ket=\|K_v\|^{-2}\|K_w\|^{-2}K_v(w)=\frac{K_v(w)}{K(v,v)K(w,w)}.
%$$
%On the diagonal in $\B\times\B$ we have
%$$
%K^E(v,v)=\|K_v\|^{-1}\varsigma_\B(P_E)(v),\qquad\forall v\in\B.
%$$
The function $\varsigma_\B(P_E)(v,w)$ has been studied for quotient modules of various reproducing kernel Hilbert spaces under the names ``core function'' and ``defect function''. It was observed in \cite{Arv7b, Cheng1, Fang5, GRS1} that the boundary value of the restriction of the core function to the diagonal exists as an element of $L^\infty(\Sb)\otimes\Mn_N(\C)$ and is idempotent. This boundary function is sometimes called the ``Arveson curvature function'' of the quotient module. %We have seen here that it coincide with the covariant symbol $\varsigma(P_E)$ of the projection $P_E$.

If $P_E$ is a projection over $\Ti_\GH^{(0)}$ then $P^E:=\varsigma(P_E)$ is a projection over $C^0(\M)$ which thus defines a continuous vector bundle $\Ei$ on $\M$. Applying Lemma \ref{AbelsumLemma} we see that 
$$
\varsigma_\Sb(P_E)(\zeta):=\lim_{r\to 1^-}\varsigma_\B(P_E)(r\zeta),\qquad\forall\zeta\in\Sb
$$
defines a continuous projection-valued function on $\Sb$ which is $\Un(1)$-equivariant and descends to $\M=\Sb/\Un(1)$ as the projection
$$
\varsigma(P_E)=\lim_{m\to\infty}\varsigma^{(m)}(P_{E,m})\in C^0(\M)\otimes\Mn_N(\C).
$$
Note that $\varsigma_\B(P_E)$ is always real-analytic. Only the continuity of its boundary value requires $P_E$ to be over $\Ti_\GH^{(0)}$. But $\varsigma_\B(P_E)$ itself is not a projection in general. %and the projection-valued part $\lim_{l\to\infty}\varsigma_\B(P_E)^l$ of $\varsigma_\B(P_E)$ is merely $L^\infty$ in general; it is $C^0$ if and only if $P_E$ is over $\Ti_\GH^{(0)}$.

Note also that $\varsigma^{(m)}(P_{E,m})$ is a matrix over the image of $\Bi(\GH_m)$ under the symbol map $\varsigma^{(m)}$, and therefore $\varsigma^{(m)}(P_{E,m})$ is real-algebraic and in particular $C^0$. This does not rely on $P_E$ having entries in $\Ti_\GH^{(0)}$. %If $\varsigma^{(m)}(P_{E,m})$ were a projection then it would therefore always define a smooth vector bundle over $\M$. 
\begin{cor}
Let $\GE_\N$ be a graded quotient module and let $P_E$ be the projection onto $\GE_\N$. Then the boundary value of the diagonal of the core function (i.e. the Arveson curvature function) of $\GE_\N$ coincides with the covariant symbol $\varsigma(P_E)=\mathrm{SOT-}\lim_m\varsigma^{(m)}(P_{E,m})$ of the projection $P_E$.
 \end{cor}
 
\begin{Remark}[Arveson curvature]
Let $\Ei$ be the Serre sheaf of the graded quotient module $\GE_\N$. Its rank is by definition the leading coefficient of the Hilbert polynomial $\N\ni m\to\chi(\Ei(m))=\dim\GE_m$ of the globally generated analytic sheaf $\Ei$.
The projection $\varsigma(P_E)$ defines a sheaf of $\Ci_\M^0$-modules with the same rank as $\Ei$. Indeed one gets
\begin{equation}\label{Arvcurvlim}
(\omega\otimes\Tr_N)(\varsigma(P_E))=\lim_{m\to\infty}(\phi_m\otimes\Tr_N)(P_{E,m})=\lim_{m\to\infty}\frac{\dim\GE_m}{n_m}=\rank\Ei,
\end{equation}
The integer $\rank\Ei$ coincides with the ``Arveson curvature'' of the pure row contraction $S_E$; see \cite{Arv7a, Arv7b, GRS2}. 
\end{Remark}

\subsection{The Cowen--Douglas projection}

%\subsubsection{The Cowen--Douglas projection for $\GE_\N$}
Let $\GE_\N\subset\GH_\N\otimes\C^N$ be a quotient module. % with underlying graded $\Ai$-module $E_\N$. %We have seen that the Cowen--Douglas sheaf $\Ei_{\rm CD}$ of $\GE_\N$ coincides with the Serre sheaf of $E_\N$ restricted to $\B\subset\V$. We shall from now on for simplicity assume that $\Ei_{\rm CD}$ is locally free on $\B\setminus\{0\}$. Note that $\Ei_{\rm CD}$ denotes the \emph{sheaf associated with the linear space} $v\to\Ker(S_E^*-\bar{v}\bone)$ over $\B\setminus\{0\}$. And $\Ei_{\rm CD}$ will be considered as a sheaf on $\B\setminus\{0\}$ (and not on $\B$ unless specified differently) or identified with the induced sheaves on $\Sb$ or $\M$.

%We write $\Ei_\B$ for the analytic coherent sheaf underlying $\Ei_{\rm CD}$ and then $\Ei$ for the vector bundle on $\M$ induced by $\Ei_\B$. So we write $\Ei_{\rm CD}$ only when we assume the Hermitian structure is given. This makes it clear that $\Ei$ does not inherit any Hermitian structure from $\Ei_{\rm CD}$ (unless $\varsigma(P_E)$ defines $\Ei$).

%Here we describe the tensor structure of the Cowen--Douglas metric and the relation between $\varsigma_\B(P_E)$, $B^E_m$, $\FS(\GE_m)$, the Cowen--Douglas metric $\FS(\GE_\N)$ and their boundary limits. 

 \begin{dfn}\label{masterEdfn}
 The \textbf{Cowen--Douglas projection} of $\GE_\N$ is the projection $\CD(\GE_\N)\in L^\infty(\B\setminus\{0\})\otimes\Bi(\GE_\N)$ defined by
 $$
\CD(\GE_\N)(v):=\text{projection onto }\Ker(S_E^*-v\bone),\qquad\forall v\in\B\setminus\{0\}.
 $$
 \end{dfn}
 The sheaf $\Ei_{\rm CD}$ is thus locally free over $\B\setminus\{0\}$ if and only if $\CD(\GE_\N)$ belongs to the subalgebra $C^0(\B\setminus\{0\})\otimes\Bi(\GE_\N)$, in which case one could regard $\CD(\GE_\N)$ as a continuous (in fact real-analytic) map from $\B\setminus\{0\}$ into the Grassmannian of $r$-planes in $\GE_\N$, where $r:=\rank\Ei_{\rm CD}$. Then $\Ei_{\rm CD}$ is the vector bundle obtained by pulling back the universal rank-$r$ vector bundle over the Grassmannian using the map $\CD(\GE_\N)$, and the Hermitian metric on $\Ei_{\rm CD}$ is obtained by pulling back the universal metric on the universal bundle via $\CD(\GE_\N)$.

Recall that
  $$
\Ker(S_E^*-v\bone)=\C k_v\otimes\E_{\rm CD}(v)
 $$
 where $\E_{\rm CD}(v)$ is a $\rank\Ei_{\rm CD}$-dimensional subspace of $\C^N$ for each $v\in\B\setminus\{0\}$. Therefore we have a factorization
  $$
\CD(\GE_\N)(v)=\CD(\GH_\N)(v)\otimes\Pi^E(v)
 $$
 where $\Pi^E(v)\in\Mn_N(\C)$ is a projection onto the subspace $\E_{\rm CD}(v)\subset\C^N$. The dimension of $\E_{\rm CD}(v)$ is $\geq\rank\Ei_{\rm CD}$. %The factor of $\CD(\GE_\N)$ in $\Bi(\GE_\N)$ is not grading-preserving, i.e. it does not belong to $\Gamma_b$. 
 
Denote by $P_E$ the projection of $\GH_\N\otimes\GE_0$ onto $\GE_\N$. Note that $\CD(\GE_\N)$ equals $\CD(\GH_\N\otimes\GE_0)\wedge P_E$, where for two projections $P$ and $Q$ acting on the same Hilbert space we denote by $P\wedge Q$

Similarly, for each $m\in\N_0$ we can define
\begin{equation}\label{FSasINF}
\CD(\GE_m):=\CD(\GH_m\otimes\GE_0)\wedge P_{E,m}.
\end{equation}
%Assuming that $\CD(\GE_\N)$ is continuous may not guarantee that $\CD(\GE_m)$ is continuous. 
In this subsection we shall look at the geometric meaning of $\CD(\GE_m)$ and $\CD(\GE_\N)$ and their relation to the symbols $\varsigma^{(m)}(P_{E,m})$ and $\varsigma_\B(P_E)$.

\subsubsection{Some alternating projections}

Recall \cite[Problem 122]{Halm1} that if $P$ and $Q$ are two projections acting on a Hilbert space then powers of the compression $PQP$ converges to the infimum $P\wedge Q$ of $P$ and $Q$,
$$
P\wedge Q=\lim_{p\to\infty}(PQP)^p.
$$
This is useful for us because then we can compare the infimum $P\wedge Q$ with the compression $PQP$. And for our choices of $P$ and $Q$ the compression $PQP$ is going to equal $\CD(\GH_m)\otimes\varsigma^{(m)}(P_{E,m})$:
\begin{prop}\label{altprojprop}
Define a projection $P^E_m\in L^\infty(\M)\otimes\Mn_N(\C)$ by writing %(\textbf{Need to show first that our assumption that $\Ei_{\rm CD}$ is locally free on $\B\setminus\{0\}$ implies that $B^E_m$ is $C^0$ for $m\gg 0$})
$$
\CD(\GE_m)=\CD(\GH_m)\otimes P^E_m.
$$
Then for each $x\in\M$ we have
\begin{equation}\label{normaltproj}
P^E_m(x)=\lim_{p\to\infty}\varsigma^{(m)}(P_{E,m})(x)^p
\end{equation}
in the norm of $\Mn_N(\C)$. 
\end{prop}
\begin{proof}
%If we let $B^{E(m)}_1(v)$ be the compression of $|k_v^{(m)}\ket\bra k_v^{(m)}|\otimes\bone_N$ by $P_{E,m}$ then we have (see \cite[Problem 122]{Halm1})
%$$
%B^{E(m)}(v)=\lim_{l\to\infty}B^{E(m)}_l(v),
%$$
%where $B^{E(m)}_l(v):=B^{E(m)}_1(v)^l$ is the $l$th power of the positive operator $B^{E(m)}_1(v)$. 

%From \cite[Lemma 9.38]{Deut1} we get that $\mathbf{P}_E(v)$ is always the norm-limit of the operators $(\mathbf{P}(v)P_E\mathbf{P}(v))^p$, since $\mathbf{P}(v)$ has image of finite dimension $N$. Similarly, $(\mathbf{P}^{(m)}\otimes\bone_N)\wedge P_{E,m}$ is always the norm-limit of the operators $(\mathbf{P}^{(m)}(v)P_{E,m}\mathbf{P}^{(m)}(v))^p$. 
%As a side remark, not however that
%$$
%\lim_{p\to\infty}(\mathbf{P}(v)P_{E,m}\mathbf{P}(v))^p=0
%$$
%since $\C k_v\cap\GE_m=\{0\}$. 

Fix $x\in\M$ and let $A:=\CD(\GH_m\otimes\C^N)(x)P_{E,m}\CD(\GH_m\otimes\C^N)$ be the compression of $P_{E,m}$ to the subspace $\C k_x^{(m)}\otimes\C^N$. Note that $\C k_x^{(m)}\otimes\C^N$ has finite dimension $N$. With that in mind we get from \cite[Lemma 9.38]{Deut1} that the positive operators $A^p$ converge in norm to $\CD(\GH_m\otimes\C^N)(x)\wedge P_{E,m}$ as $p$ goes to inifnity. %(Since $\mathbf{P}(v)$ has image of finite dimension $N$.)
Observe that
$$
\CD(\GH_m\otimes\C^N)(x)P_{E,m}\CD(\GH_m\otimes\C^N)(x)=|k_x^{(m)}\ket\bra k_x^{(m)}|\otimes\bra k_x^{(m)}|P_Ek_x^{(m)}\ket=\CD(\GH_m)(x)\otimes\varsigma^{(m)}(P_{E,m})(x).
$$
The proof is then complete, since by definition we have $\CD(\GE_m)=\CD(\GH_m\otimes\C^N)\wedge P_{E,m}$. 
\end{proof}
%In Lemma \ref{bigCDversussymbol} we will show that continuity of $\CD(\GE_\N)$ implies that of $\varsigma(P_E)$. The converse is an open problem. For that reason continuity of $\CD(\GE_\N)$ on $\B\setminus\{0\}$ is in general not easily characterized on the level of the $\CD(\GE_m)$'s. But we have a result for $m\gg 0$:

\begin{Remark}[Another compression]
Instead of the compression $\CD(\GH_m\otimes\GE_0)(x)P_{E,m}\CD(\GH_m\otimes\GE_0)(x)$ we can instead look at $P_{E,m}\CD(\GH_m\otimes\GE_0)(x)P_{E,m}$. The powers of this operator also converge to $\CD(\GE_m)(x)$. For all $r\zeta\in\B$ we have 
\begin{align*}
P_E\CD(\GH_\N)(r\zeta)|_{\GH^E_\N}&=P_E(|k_{r\zeta}\ket\bra k_{r\zeta}|\otimes\bone_N)|_{\GH^E_\N}
\\&=(1-r^2)P_E\sum_{m\in\N_0}r^{2m}\sum_{|\mathbf{j}|=m=|\mathbf{k}|}\zeta_\mathbf{j}\zeta_\mathbf{k}^*(S_\mathbf{j}S_\mathbf{k}^*\otimes\bone_N)\big|_{\GH^E_\N}
\\&=(1-r^2)\sum_{m\in\N_0}r^{2m}\sum_{|\mathbf{j}|=m=|\mathbf{k}|}\zeta_\mathbf{j}\zeta_\mathbf{k}^*S_{E,\mathbf{j}}S_{E,\mathbf{k}}^*
\\&=(1-r^2)\sum_{m\in\N_0}r^{2m}\sum_{|\mathbf{j}|=m=|\mathbf{k}|}\varsigma^{(m)}(S_\mathbf{j}S_\mathbf{k}^*)(\zeta)S_{E,\mathbf{j}}S_{E,\mathbf{k}}^*\in\Bi(\GE_\N).
\end{align*}
Let us apply $\omega$ to elements of $L^\infty(\M)\otimes\Bi(\GH_\N\otimes\C^N)$, producing elements of $\Bi(\GH_\N\otimes\C^N)$. Then clearly
$$
\omega(P_{E,m}\CD(\GH_m\otimes\C^N)P_{E,m})=P_{E,m}\omega(\CD(\GH_m\otimes\C^N))P_{E,m}=\frac{1}{n_m}P_{E,m}.
$$
We have $\breve{\varsigma}^{(m)}\varsigma^{(m)}=\id$ so $(\breve{\varsigma}^{(m)}\varsigma^{(m)})(P_{E,m})=P_{E,m}$ holds trivially. Therefore
\begin{align*}
n_m\omega(\CD(\GH_m\otimes\C^N)P_{E,m}\CD(\GH_m\otimes\C^N))&=n_m\omega(\varsigma^{(m)}(P_{E,m})\otimes\FS(\GH_m))=P_{E,m}
\\&=n_m\omega(P_{E,m}\CD(\GH_m\otimes\C^N)P_{E,m}).
\end{align*}
That is, the two choices of compressions have the same $\omega$-integrals. %Let us now look at 
%$$
%(P_{E,m}\CD(\GH_m\otimes\GE_0)P_{E,m})(P_{E,m}\CD(\GH_m\otimes\GE_0)P_{E,m})=P_{E,m}(\CD(\GH_m\otimes\GE_0)P_{E,m}\CD(\GH_m\otimes\GE_0))P_{E,m}.
%$$
%This is just the compression by $P_{E,m}$ of $\varsigma^{(m)}(P_{E,m})\otimes\FS(\GH_m)$. But is that compression really doing anything? Yes of course. 
\end{Remark}

\begin{Lemma}\label{PEvsCDlemma}
Let $\GE_\N\subset\GH_\N\otimes\C^N$ be a graded quotient module and suppose that the projection $P_E$ onto $\GE_\N$ has entries in $\Ti_\GH^{(0)}$. Then $\CD(\GE_m)$ is $C^0$ for all $m\gg 0$. That is, $\lim_{p\to\infty}\varsigma^{(m)}(P_{E,m})^p$ exists in $C^0(\M)$ for all $m\gg 0$.
\end{Lemma}
\begin{proof}
%Since $\iota_{m,l}(P_{E,m})\geq P_{E,l}$ for all $l\geq m\geq 0$ and $\|\varsigma^{(m)}(P_{E,m})\|\leq\|P_{E,m}\|=1$, we can write $\varsigma^{(m)}(P_{E,m})=B^E_m\oplus C^E_m$ where $B^E_m$ and $C^E_m$ are matrices over $L^\infty(\M)$ and $B^E_m$ is a projection. 
Since $\|\varsigma^{(m)}(P_{E,m})\|\leq\|P_{E,m}\|=1$, we can write $\varsigma^{(m)}(P_{E,m})=P^E_m\oplus C^E_m$ where $C^E_m$ is a positive operator of norm $\|C^E_m\|\leq 1$.

Moreover,
$$
P_E\in\Ti_\GH^{(0)}\otimes\Bi(\GE_0)\iff\lim_{m\to\infty}C^E_m=0,
$$
i.e. we have uniform convergence of the $P^E_m$'s to $\varsigma(P_E)$ iff $P_E$ is over $\Ti_\GH^{(0)}$. %If $\varsigma(P_E)$ is $C^0$ then the projections $B^E_m$ converge to the projection $\varsigma(P_E)$ in supremum norm. 
So assume that $P_E$ is over $\Ti_\GH^{(0)}$. Then for $m\gg 0$ we have $\|\varsigma(P_E)-P_m^E\|<1$. This gives that $\varsigma(P_E)(x)$ and $P^E_m(x)$ are unitarily equivalent for all $x\in\M$ \cite[Prop. 5.2.6]{Ols}; thus the rank of $P^E_m$ is constantly equal to $\rank\Ei$, which is to say that $P^E_m$ is continuous. And $C^E_m$ is $C^0$ iff $P^E_m$ is $C^0$. 
If $C^E_m$ is $C^0$ then we have $\lim_{p\to\infty}(C^E_m)^p=0$ uniformly by Dini's theorem. 
\end{proof}

\subsubsection{The boundary limit of $\Pi^E$}

We shall now investigate how far $\varsigma^{(m)}(P_{E,m})$ is from a projection in the case $P_E$ is $\Psi$-superharmonic.
Define a projection $\Pi^E\in L^\infty(\B\setminus\{0\})\otimes\Mn_N(\C)$ by writing %(\textbf{Need to show first that our assumption that $\Ei_{\rm CD}$ is locally free on $\B\setminus\{0\}$ implies that $B^E_m$ is $C^0$ for $m\gg 0$})
$$
\CD(\GE_\N)=\CD(\GH_\N)\otimes\Pi^E.
$$
%Recall that $\Pi^E$ denotes the projection such that $\CD(\GE_\N)=\CD(\GH_\N)\otimes\Pi^E$. 
In the same way as Proposition \ref{altprojprop} one deduces that for each $v\in\B$ we have
\begin{equation}\label{bignormaltproj}
\Pi^E(v)=\lim_{p\to\infty}\varsigma_\B(P_E)(v)^p.
\end{equation}
\begin{prop}\label{limitPiEcor}
For almost every $\zeta\in\Sb$ we have
$$
\lim_{r\to 1^-}\Pi^E(r\zeta)=\lim_{m\to\infty}P^E_m([\zeta])=\lim_{m\to\infty}\varsigma^{(m)}(P_{E,m})([\zeta])=\varsigma(P_E)([\zeta]).
$$
The function $\lim_{r\to 1^-}\Pi^E(r\zeta)$ descends to a projection over $L^\infty(\M)$ which concides %almost everywhere 
with $\varsigma(P_E)$. If $P_E$ is over $\Ti_\GH^{(0)}$ then $\lim_{r\to 1^-}\Pi^E(r\ \cdot)$ coincides with $\varsigma(P_E)$ as a projection over $C^0(\M)$. %So $\Pi^E$ will have a $C^0$ boundary limit iff $\varsigma(P_E)$ is $C^0$. 
\end{prop}
\begin{proof}
%Recall that $\Pi^E$ is the limit of the powers of $\varsigma_\B(P_E)$. 
Since $\lim_{r\to 1^-}\varsigma_\B(P_E)(r\zeta)$ is a projection $\varsigma(P_E)([\zeta])$ for almost every $\zeta\in\Sb$ we have from \eqref{bignormaltproj} that
$$
\lim_{r\to 1^-}(\Pi^E-\varsigma_\B(P_E))(r\zeta)=0.
$$
Similarly, since $P^E_m$ is the limit of the powers of $\varsigma^{(m)}(P_{E,m})$ we have that $\lim_{m\to\infty}P^E_m$ exist and equals $\lim_{m\to\infty}\varsigma^{(m)}(P_{E,m})=\varsigma(P_E)$ as element of $L^\infty(\M)\otimes\Mn_N(\C)$.  
\end{proof}

Consider the vector space
$$
\E_m(v):=\{\xi\in\GE_0|\ k_v^{(m)}\otimes\xi\in\GE_m\},
$$
also characterized by
$$
\C k_v^{(m)}\otimes\E_m(v)=\CD(\GE_m)(v)\GE_m=(\C k_v^{(m)}\otimes\GE_0)\cap\GE_m,
$$
i.e. $\E_m(v)=\Ran P^E_m([v])$. The grading on $\GE_\N$ gives $k_v^{(m)}\otimes\E(v)\subset\GE_m$ so we have
$$
\E(v)\subset\E_m(v).
$$
%The fact that $\GE_\N$ is graded says that $\bvee_{v\in\B}\C k_v^{(m)}\otimes\E_m(v)$ equals $\GE_m=\bvee_{v\in\B}\C k_v^{(m)}\otimes\E(v)$ but it does not imply that $\E_m(v)=\E(v)$ for each $v$. 
%The inclusion can be proper since we can have
 %$$
 %p_mP_E(k_v\otimes\xi)=p_mk_v\otimes\xi
% $$
%even if $P_E(k_v\otimes\xi)\ne k_v\otimes\xi$. This happens when $p_mk_v\otimes\xi$ is in $\GE_m$ but $p_lk_v\otimes\xi$ is not in $\GE_l$ for some $l$. For instance, it holds that $p_0(k_v\otimes\xi)=\bone\otimes\xi$ for all $v$, so that even for $(v,\xi)\notin\E$ we have $p_0(k_v\otimes\xi)\in\GE_\N$. 
We have equality $\E(v)=\E_m(v)$ iff the dimensions are equal. So when $P^E_m$ is continuous then so is $\Pi^E$, and they coincide. 

Note that $k_0^{(m)}\otimes\xi=p_m(\bone\otimes\xi)=0$ for all $m\ne 0$ and all $\xi\in\GE_0$, so $\E_m(0)=\GE_0$ holds for all $m$. We typically only consider $P^E_m$ for $v\ne 0$ however.

\begin{thm}\label{bigCDversuslocallyfree}
Suppose that $P^E:=\varsigma(P_E)$ is $C^0$. Then $P^E$ is real-analytic and $\Ei_{\rm CD}$ is locally free on $\B\setminus\{0\}$. In fact, for large enough $m$ we have% $\FS(\GE_m)=\CD(\GE_m)$ and 
$$
P^E_m=P^E,
$$
and $\Pi^E$ is the pullback of $P^E$ to $\B\setminus\{0\}$,
$$
\Pi^E(r\zeta)=P^E([\zeta]).
$$
\end{thm}
\begin{proof}
The coinvariance property $\iota_{m,l}(P_{E,m})\geq P_{E,l}$ gives that 
$$
\varsigma^{(m)}(P_{E,m})=(\iota_{m,l}(P_{E,m}))_{l\geq m}+\Gamma_0=P^E+C^E_m
$$
for some $C^E_m\in L^\infty(\M)\otimes\Bi(\GE_0)$ with $\|C^E_m\|\leq 1$ and $P^EC^E_m=0=C^E_mP^E$ (here we use that $\|\varsigma^{(m)}(X)\|\leq\|X\|$ for all $X$). By assumption we have norm-convergence $\lim_{m\to\infty}C^E_m=0$ so $\|C^E_m\|$ is strictly less that 1 for $m$ large. So for $m\gg 0$,
$$
P^E_m:=\lim_{p\to\infty}\varsigma^{(m)}(P_{E,m})^p=P^E,
$$
and thus $P^E_m$ is $C^0$.  
As remarked before the lemma, when $P^E_m$ is continuous we have $P^E_m=\Pi^E$. So $\Pi^E$ is $C^0$ and for $v\ne 0$ the projection $\Pi^E(v)$ depends only on the coset $[v]\in\M=(\B\times\{0\})/\D^\times$. Since $\lim_{r\to 1^-}\Pi^E(r\zeta)=P^E([\zeta])$, we obtain $\Pi^E(r\zeta)=P^E([\zeta])$ for all $\zeta$. 
%\begin{Question}
%Is $\Pi^E$ automatically real-analytic when $C^0$? Yes, the family $\Ker(S_E^*-v\bone)$ varies analytically. 
%\end{Question}
Since $\Pi^E$ is automatically real-analytic when $C^0$ we see that the same is true for $P^E$. 
\end{proof}

\begin{cor}
If $\varsigma(P_E)$ is $C^0$ then we have a factorization of Hermitian vector bundles
$$
\Ei_{\rm CD}=\Oi_{\rm CD}\otimes\Ei_{\B\setminus\{0\}}
$$
where $\Ei_{\B\setminus\{0\}}$ is the Hermitian vector bundle over $\B\setminus\{0\}$ defined by the pullback of $\varsigma(P_E)$ to $\B\setminus\{0\}$. 
\end{cor}
In this corollary, the holomorphic structure on $\Ei_{\rm CD}$ gives a holomorphic structure on $\Ei_{\B\setminus\{0\}}$. 
And since $\varsigma_\B(P_E)$ is $\Un(1)$-equivariant, so is $\Pi^E=\lim_{p\to\infty}\varsigma_\B(P_E)^p$ and hence the holomorphic structure on $\Ei_{\B\setminus\{0\}}$ is $\Un(1)$-equivariant. By Lemma \ref{vectextlemma} this means that there is a unique holomorphic structure on the smooth vector bundle $\Ei$ defined by $\varsigma(P_E)$ which pulls back to that of $\Ei_{\B\setminus\{0\}}$. % Since $\varsigma(P_E)$ is real-analytic, this must be the Levi-Civita holomorphic structure.

\begin{Remark}[Characteristic function isometric a.e. on $\Sb$]
Recall the Beurling factorization $P_E=\bone-\Theta_E^*\Theta_E$ where $\Theta_E$ is the characteristic function of the pure finite-rank row contraction $S_E$. It is shown in \cite[Thm. 4.3]{GRS2} and \cite[Thm. 6.1]{BhSa1} that $\Theta_E$ becomes a partial isometry a.e. at the boundary $\Sb$. Thus Proposition \ref{limitPiEcor} is not surprising. Indeed, every element of the tensor algebra $\Ai_\GH$ has a continuous extension to $\Sb$ and so it is easy to see that if $P_E$ has entries in $\Ti_\GH=\overline{\operatorname{span}}\Ai_\GH\Ai_\GH^*$ (i.e. when $\varsigma(P_E)$ is continuous) then %we have $\lim_m\|C^E_m\|=0$ in Proposition \ref{projpluscompactsymbol}, and 
$\Theta_E(r\zeta)$ becomes a partial isometry as $r\to 1^-$ for \emph{every} point $\zeta$ on $\Sb$.
\end{Remark}

Note that $\Ran\varsigma^{(m)}(P_{E,m})$ is not equal to $P^E_m$, and $\lim_{m\to\infty}\Ran\varsigma^{(m)}(P_{E,m})$ need not equal $\varsigma(P_E)$. Similarly, the boundary limit of $\Ran\varsigma_\B(P_E)$ need not equal $\varsigma(P_E)$.

We shall see in later sections that even if $\varsigma(P_E)$ is continuous and even if we assume $\Ei_{\rm CD}$ to be locally free on $\B\setminus\{0\}$, the vector bundle over $\M$ defined by $\varsigma(P_E)$ need not be $C^0$-isomorphic to $\Ei_{\rm CD}$ when pulled back to $\B\setminus\{0\}$.

%\begin{Remark}[Reducing case]
%Note that $\GE_\N$ is reducing iff $\GE_m=\GH_\N\otimes\C^r$ for some $r\leq N$ so that $P_{E,m}=p_m\otimes\bone_r$. So in the reducing case $C^E_m=0$ we get $\omega(\CD(\GE_m))=\omega(\CD(\GH_m)\otimes\varsigma(P_{E,m}))=P_{E,m}/n_m$, and
% $$
%(\phi_m\otimes\Tr_N)(\CD(\GE_m))=\frac{1}{n_m}(\Tr_{\GH_m}\otimes\Tr_N)(P_{E,m})=r.
%$$
%In this case the limit in \eqref{Arvcurvlim} is attained at finite $m$. 
%\end{Remark}

 \subsubsection{Geometric interpretation of $\GE_m$}\label{versussec}
Let $\Ei$ be a holomorphic vector bundle over $\M$. 
Let $m$ be large enough so that $\Ei(m)$ is globally generated and let $\GE_m$ be $H^0(\M;\Ei(m))$ endowed with some inner product. In \S\ref{reflinesec} we defined a projection $\FS(\GE_m)$ in $C^0(\M)\otimes\Bi(\GE_m)$ geometrically. Let us now give a more algebraic definition. %It will define a smooth vector bundle isomorphic to $\Ei(m)$. 

Since $\GE_m$ is a Hilbert space we have a standard $C^0(\M)$-valued inner product on the $C^0(\M)$-module $C^0(\M)\otimes\GE_m$, defined on simple tensors by
$$
(f\otimes\xi|g\otimes\eta)_{C^0(\M)\otimes\GE_m}:=f^*g\bra\xi|\eta\ket_{\GE_0},\qquad\forall f,g\in C^0(\M),\ \xi,\eta\in\GE_m.
$$
Equivalently, this means that
\begin{equation}\label{dfnofinn}
(\phi|\psi)_{C^0(\M)\otimes\GE_m}(x)=\bra\phi(x)|\psi(x)\ket_{\GE_m},\qquad\forall \phi,\psi\in C^0(\M)\otimes\GE_m.
\end{equation}
\begin{dfn}
Define $\FS(\GE_m)$ to be the projection acting on $C^0(\M)\otimes\GE_m$ whose range is isomorphic as $C^0(\M)$-module to $\Gamma^0(\M;\Ei(m))$ and the subspace $\FS(\GE_m)(\bone\otimes\GE_m)\subset\Gamma^0(\M;\Ei(m))$ identifies with $H^0(\M;\Ei(m))$. 
\end{dfn}
This uniquely determines $\FS(\GE_m)$. Indeed, we obtain
$$
\psi(x)=\FS(\GE_m)(x)\psi,\qquad\forall \psi\in\GE_m,\ x\in\M.
$$
Therefore, if $(\psi_j)_{j\in\J}$ is any Parseval frame for the Hilbert space $\GE_m$ then we can expand $\FS(\GE_m)(x)$ as
\begin{equation}\label{FSformulaimportantetwo}
\FS(\GE_m)(x)=\sum_{j,k\in\J}\bra\psi_j(x)|\psi_k(x)\ket_{\GE_m}|\psi_j\ket\bra\psi_k|\in\Bi(\GE_m).
\end{equation}
By \eqref{dfnofinn} this is equivalent to saying that any Parseval frame $(\psi_j)_{j\in\J}$ for $\GE_m$ is a Parseval $C^*$-frame for $\FS(\GE_m)$. This property thus characterizes $\FS(\GE_m)$. Note that $\FS(\GE_m)$ exists precisely when $\Ei(m)$ is globally generated (as we assume here) because that is when there exists a holomorphic $C^*$-frame for $\Gamma^0(\M;\Ei(m))$. 
\begin{prop}\label{geomintprop}
Let $\Ei$ be aglobally generated holomorphic vector bundle over $\M$, let $\GE_0$ be an inner product on $H^0(\M;\Ei)$ and let $\FS(\GE_0)$ be the projection over $C^0(\M)$ with image $\Gamma^0(\M;\E)$ and with a Parseval $C^*$-frame given by an orthonormal basis for $\GE_0$. Assume that $m$ is large enough so that $\Ei(m)$ is Castelnuovo--Mumford regular (see \eqref{zeroregular}). Let $\GE_m$ be $H^0(\M;\Ei(m))$ endowed with the inner product of $\GH_m\otimes\GE_0$. Then $\FS(\GH_m)\otimes\FS(\GE_0)\in C^0(\M)\otimes\Bi(\GH_m\otimes\GE_0)$ can be regarded as an element of $C^0(\M)\otimes\Bi(\GE_m)$ and it has a Parseval $C^*$-frame given by an orthonormal basis for $\GE_m$, i.e.
$$
\FS(\GH_m)\otimes\FS(\GE_0)=\FS(\GE_m).
$$
\end{prop}
\begin{proof}
We can dispense with the tensor products in the $C^0(\M)$-factor of $\FS(\GH_m)\otimes\FS(\GE_0)$, and since $\psi Z_\mathbf{k}$ is a holomorphic section of $\Ei(m)$ for all $\mathbf{k}\in\F_n^+(m)$ and $\psi\in\GE_0$ we obtain that $\FS(\GH_m)\otimes\FS(\GE_0)$ is in $C^0(\M)\otimes\Bi(\GE_m)$ and has a Parseval $C^*$-frame $\boldsymbol\psi$ obtain by applying the multiplication map to the tensor product of the Parseval $C^*$-frames for $\FS(\GE_0)$ and $\FS(\GH_m)$. Since $\Ei(m)$ is assumed to be Castelnuovo--Mumford regular we have that $\GE_m$ is the image of the multiplication map $\GH_m\otimes\GE_0\to\GE_m$. If $P_{E,m}$ is the projection of $\GH_m\otimes\GE_0$ onto $\GE_m$ then $\boldsymbol\psi$ is the image under $P_{E,m}$ of the Parseval $C^*$-frame for $\FS(\GH_m)\otimes\FS(\GE_0)$, which is a Parseval frame for $\GH_m\otimes\GE_0$. Hence $\boldsymbol\psi$ is a Parseval frame for $\GE_m$. 

Hence $\FS(\GH_m)\otimes\FS(\GE_0)$ has a Parseval $C^*$-frame given by a Parseval frame for $\GE_m$ (and this gives a natural representation of $\FS(\GH_m)\otimes\FS(\GE_0)$ as a matrix over $C^0(\M)$ of size $Nn^m$). Therefore \emph{any} Parseval frame for $\GE_m$ gives a Parseval $C^*$-frame $\FS(\GH_m)\otimes\FS(\GE_0)$, i.e. $\FS(\GH_m)\otimes\FS(\GE_0)$ coincides with $\FS(\GE_m)$.
\end{proof}
The following gives the geometric meaning of the Fock inner product $\GE_m$ on $H^0(\M;\Ei(m))$:
\begin{cor}
Let $\Ei$ be a holomorphic vector bundle over $\M$ and let $\GE_\N$ be the completion of $\bigoplus_{m\in\N_0}H^0(\M;\Ei(m))$ when represented as a quotient of $\Ai\otimes\C^N$ for some $N$. 
%with affine linear space $\E\subset\B\times\C^N$ and let $\GE_\N:=\cspan\{k_v\otimes\xi|\ (v,\xi)\in\E\}$ be its quotient module. 
Then for $l\geq m\gg 0$ we have
$$
\FS(\GE_l)=\FS(\GE_m)\otimes\FS(\GH_{l-m}).
$$
\end{cor}

\subsection{Nullstellensatz}
\begin{dfn}
Let $\Ei$ be a coherent analytic sheaf over $\M$. The \textbf{affine linear space} of $\Ei$ is the holomorphic linear space $\E$ of the Serre sheaf $\Ei_\V$ of the $\Ai$-module $E_\N:=\bigoplus_{m\in\N_0}H^0(\M;\Ei(m))$. Typically we consider only the restriction of $\E$ to $\B\subset\V$ and call this also the affine linear space of $\Ei$.
\end{dfn}
Recall that $\E$ and $\Ei_\B$ determine each other up to natural isomorphisms (see \cite{Fisc1}). %Recall that the association $\Ei_{\B\setminus\{0\}}\to\E$ is an anti-equivalence from the category of coherent $\Oi_{\B\setminus\{0\}}$-modules to the category holomorphic linear spaces over $\B\setminus\{0\}$. 
Since every coherent analytic sheaf over $\B$ is globally generated, $\E$ always appears as a holomorphic linear subspace
$$
\E\subset\B\times\C^N
$$
for some integer $N$. If $\Ei$ is generated by holomophic sections globally over $\M$ then the embedding $\E\subset\B\times\C^N$ is $\D^\times$-equivariant outside $0\in\B$ and hence descends to an embedding $\E\subset\M\times\C^N$ where $\E$ now denotes the holomorphic linear space of $\Ei$. 

\begin{Lemma}\label{gradlemma}
Let $\Ei$ be a coherent analytic sheaf over $\M$ and let $\E\subset\B\times\C^N$ be its affine linear space. 
Consider the submodule
$$
J(\E):=\{f\in\Ai\otimes\C^N|\ \bra f(v)|\xi\ket_{\C^N}=0\text{ for all }(v,\xi)\in\E\}
$$
and the quotient Hilbert module
$$
\GE:=\overline{\operatorname{span}}\{k_v\otimes\xi|\ (v,\xi)\in\E\}.
$$
Suppose that $\Ei$ is locally free and moreover that the Cowen--Douglas sheaf $\Ei_{\rm CD}$ of $\GE$ is locally free on $\B\setminus\{0\}$. Then the following are equivalent:
\begin{enumerate}[(a)]
\item{$J(\E)$ is a graded submodule.}
\item{$\GE$ is a graded quotient Hilbert module.}
\item{$\GE\cap(\GH_m\otimes\GE_0)=\operatorname{span}\{k_v^{(m)}\otimes\xi|\ (v,\xi)\in\E\}$ for all $m\in\N_0$.}
\item{$\Ei$ is globally generated (so we can take the embedding $\E\subset\B\times\C^N$ to be $\D^\times$-equivariant).}\label{equivbm}
\end{enumerate}
\end{Lemma}
We thus see that $\GE$ can be graded even if $E_\N:=\bigoplus_mH^0(\M;\Ei(m))$ does not fit into a short exact sequence $0\to I_\N\to\Ai\otimes\C^N\to E_\N\to 0$ of graded $\Ai$-modules (as the latter is slightly stronger than \eqref{equivbm}). 
\begin{proof}
Let $f\in\Ai\otimes\C^N$ and $(v,\xi)\in\B\times\C^N$. Observe that
$$
\bra f(v)|\xi\ket_{\C^N}=\bra k_v\otimes\xi|f\ket_{\GH_\N\otimes\C^N}.
$$
So we can describe $J(\E)$ as
$$
J(\E)=\{f\in\Ai\otimes\C^N|\ \bra k_v\otimes\xi|f\ket_{\GH_\N\otimes\C^N}=0\text{ for all }(v,\xi)\in\E\}.
$$
In other words,
$$
(\GH_\N\otimes\C^N)\ominus J(\E)=\GE:=\overline{\operatorname{span}}\{k_v\otimes\xi|\ (v,\xi)\in\E\}.
$$
Hence (a) is equivalent to (b). 

The quotient module $\GE$ is graded if and only if
$$
(p_m\otimes\bone_{\GE_0})\GE\subset\GE,\qquad\forall m\in\N_0.
$$
Since $(p_m\otimes\bone_{\GE_0})(k_v\otimes\xi)$ is a scalar multiple of $k_v^{(m)}\otimes\xi$ we can write this as
\begin{equation}\label{kvmsubsetGE}
\{k_v^{(m)}\otimes\xi|\ (v,\xi)\in\E\}\subset\GE,\qquad\forall m\in\N_0.
\end{equation}
%Or could it be that $X\cspan\{k_v\otimes\xi|\ (v,\xi)\in\E\}\ne\cspan\{X(k_v\otimes\xi)|\ (v,\xi)\in\E\}$ for a bounded operator $X$ on $\GH_\N\otimes\GE_0$? No, it cannot happen since $X$ is continuous in the norm topology on $\GH_\N\otimes\GE_0$.
%Clearly, if we define
%$$
%\GE_m:=\GE\cap(\GH_m\otimes\GE_0).
%$$
%then \eqref{kvmsubsetGE} is equivalent to
%$$
%\GE_m=\operatorname{span}\{k_v^{(m)}\otimes\xi|\ (v,\xi)\in\E\}.
%$$
Clearly $\cspan\{k_v^{(m)}\otimes\xi|\ (v,\xi)\in\E,\ m\in\N_0\}=\GE$. Therefore \eqref{kvmsubsetGE} and (c) are equivalent.  

Suppose that $\E(v)=\E([v])$. Then the Cowen--Douglas projection $\CD(\GE)=\CD(\GH_\N)\otimes\Pi^E$ satisfies $\Pi^E(v)=\Pi^E([v])$. So $\Pi^E$ is the pullback of its boundary limit $\varsigma(P_E)$ and $P_E$ must be over $\Ni$, i.e. the projection $P_E$ onto $\GE$ must preserve the grading. %This is equivalent to saying that the range $\GE$ of $P_E$ is graded. 
This gives the lemma. 

%\textbf{Not so fast}: If $\E\subset\E_{\rm CD}$ is proper then $\E(v)=\E([v])$ need not imply $\Pi^E(v)=\Pi^E([v])$. That is, $\Pi^E$ might fail to be $\D^\times$-equivariant and hence not equal the pullback of $\varsigma(P_E)$. But if $\Ei_{\rm CD}$ is really related to $\Gr(\Ei)$ then... what if $\Ei_{\rm CD}$ fails to be $\D^\times$-equivariant either due to singularities or to torsion? We need to recheck if $\varsigma(P_E)$ can be $C^0$ even if $\Pi^E$ is not $C^0$. It seems clear that $\Pi^E(r\ \cdot\ )$ is $C^0$ on $\Sb_r$ for $r$ close enough to $1$ and that it is $\Un(1)$-equivariant. But can $\Pi^E$ be discontinuous in the radial variable? 
\end{proof}

\begin{thm}[Nullstellensatz]\label{nullothm}
Let $\Ei$ be a coherent analytic sheaf over $\M=(\B\setminus\{0\})/\D^\times$ and let $\E\subset\B\times\C^N$ be its affine linear space
%Suppose that $E_\N=\bigoplus_mH^0(\M;\Ei(m))$ fits into a sequence $0\to I_\N\to\Ai\otimes\C^N\to E_\N\to 0$ of graded $\Ai$-modules. 
%Let $\tilde{E}_\N$ be the orthogonal complement of $I_\N$ w.r.t. the Fock inner product (thus $\tilde{E}_\N\cong E_\N$ as graded $\Ai$-module). 
Let
$$
J(\E):=\{f\in\Ai\otimes\C^N|\ \bra f(v)|\xi\ket_{\C^N}=0\text{ for all }(v,\xi)\in\E\}
$$
and for any subset $J\subset\Ai\otimes\C^N$ define
$$
\V(J):=\{(v,\xi)\in\B\otimes\C^N|\ \bra f(v)|\xi\ket_{\C^N}=0\text{ for all }f\in J\}.
$$
Let $I_\N$ be the kernel of the surjection of $\Ai\otimes\C^N$ onto $E_\N:=\bigoplus_mH^0(\M;\Ei(m))$. 
Finally let $[E_\N]$ be the closure of in the Fock inner product of $\GH_\N\otimes\C^N$, define
$$
\GE:=\overline{\operatorname{span}}\{k_v\otimes\xi|\ (v,\xi)\in\E\}.
$$
and let $\Ei_{\rm CD}$ be the Cowen--Douglas sheaf of $\GE$.
%and let $\Ei_{\rm CD}$ and $\tilde{\Ei}_{\rm CD}$ be the Cowen--Douglas sheaves of $\GE_\N$ and $[E_\N]$ respectively. 
Then for $m\gg 0$ we have
$$
E_m=\GE\cap(\GH_m\otimes\GE_0)\text{ and }I_m=J(\E)\cap(\GH_m\otimes\GE_0),
$$
%up to finite-dimensional subspaces, we have
%$$
%[E_\N]=\GE\text{ and }I_\N=J(\E),
%$$
and the following are equivalent: %(\textbf{But have we not shown that $\sing\Ei_\B=\sing\Ei_{\rm CD}$? Then (a) would always hold when $\Ei$ is locally free, whence the Nullstellensatz would hold for all locally free $\Ei$. Can it be that $\Ei_{\rm CD}\cong\Ei_\B$ even holds $\D^\times$-equivariantly analytically always, but $\varsigma(P_E)$ still only defines a Yang--Mills metric on $\Gr(\Ei)$? Yeah, since we obtain $\E=\E_{\rm CD}$ the sheaf $\Ei_{\rm CD}$ is $\Oi_{\rm CD}\otimes\Ei_\B$ and not $\Oi_{\rm CD}\otimes\Gr(\Ei)_\B$. A natural question now is if $\CD(\GE_m)$ and $\varsigma(P_E)$ necessarily are continuous when $\Ei_{\rm CD}$ is locally free. We say that $\Ran\breve{\varsigma}(P^E)$ need not be equal to $\GE_\N$, when $P^E:=\varsigma(P_E)$ is Yang--Mills on $\Gr(\Ei)$ and not on $\Ei$, since then the Serre sheaf of $\Ran\breve{\varsigma}(P^E)$ will be $\Gr(\Ei)$. This proves that $\CD(\GE_m)$ need not continuous.})
\begin{enumerate}[(a)]
\item{$\Ei_\B\cong\Ei_{\rm CD}$} %how to define $\E$ at 0? As $\Spec E_\N$? If $E_0=\C^N$ we could take $\E\cup(\{0\}\times\C^N)$ in the definition of $J(\E)$. Indeed we must, since otherwise we will have the constants in $J(\E)$. Well not necessarily all of $\bone\otimes\C^N$ so it is not a problem.
%\item{$\Ei_\B\cong\tilde{\Ei}_{\rm CD}$.}
%\item{$[E_\N]=\GE$ (up to finite-dimensional subspaces).}
%\item{$I=J(\E)$ (up to finite-dimensional subspaces).}
\item{$\V(J(\E))=\E$} %(\textbf{or should it just be $\V(J(\E))\cong \E$?}).
%\item{$\V(I)=\E$.}
\end{enumerate}
%In particular, if $\Ei$ is a holomorphic vector bundle over $\M$ then (b) is equivalent to:
%\begin{enumerate}[(a')]
%\item{$\Ei_{\rm CD}$ is locally free on $\B\setminus\{0\}$.}
%\end{enumerate}
If $\Ei$ is globally generated then $J(\E)=J_\N(\E)$ is a graded submodule and $\GE=\GE_\N$ is a graded quotient module.
Finally, suppose that the $\Ai$-module maps in the short exact sequence $0\to I_\N\to\Ai\otimes\C^N\to E_\N\to 0$ preserve the grading, and that $\Ei$ is locally free. Then (a) and (b) hold, and we have an identification
$$
[E_\N]=\GE_\N\text{ and }I_\N=J_\N(\E),
$$
of graded $\Ai$-modules.
\end{thm}
%Note that since we started on the geometric side with a coherent sheaf $\Ei$ and then formed the module $E_\N$ we do not need the ``saturation'' that appears in the usual Nullstellensatz for $N=1$.
\begin{proof}

We first show that the $\C$-linear span of the $k_x^{(m)}\otimes\xi$'s with $(x,\xi)\in\E$ is precisely the vector space $E_m$ for $m\gg 0$. That would give $[E_\N]=\GE$ (up to finite-dimensional vector spaces). Clearly $k_x^{(m)}\otimes\xi$ for $(x,\xi)\in\E$ is mapped to $E_m$ under the multiplication map $\Ai_m\otimes\C^N\to E_m$. In the proof of Lemma \ref{gradlemma} we saw that $\GE\cap(\GH_m\otimes\GE_0)$ equals $\operatorname{span}\{k_v^{(m)}\otimes\xi|\ (v,\xi)\in\E\}$ for $m\gg 0$. So for $m\gg 0$ the map $\Ai_m\otimes\C^N\to E_m$ is surjective and we have indeed $\GE_m=E_m$. 

We have
$$
\V(J(\E))=\E_{\rm CD}:=\{(v,\xi)\in\B\otimes\C^N|\ S_E^*(k_v\otimes\xi)=\bar{v}k_v\otimes\xi\},
$$
and $\E\subset\E_{\rm CD}$. Note that the fiber of $\Ei_{\rm CD}$ is given by $k_v\otimes\E_{\rm CD}$. So the two conditions $\V(J(\E))=\E$ and $\Ei_\B\cong\Ei_{\rm CD}$ are equivalent. 

That (a) holds when $\Ei$ is locally free follows from Proposition \ref{SerrevsCDthm}. The last statements follow from Lemma \ref{gradlemma}. 
\end{proof}
%\begin{Remark}
%If it turns out to be true that $\Ei_\B\cong\Ei_{\rm CD}$ always holds then $\V(J(\E))=\E$ will always hold. But even then continuity of $\CD(\GE_\N)$ does not imply that $\CD(\GE_m)$ is continuous for each $m$. So $P_E$ could fail to be over $\Ti_\GH^{(0)}$ even when $\Ei_{\rm CD}$ is locally free over $\B\setminus\{0\}$. We do not know. But $\varsigma(P_E)$ is always kind of Yang--Mills so if $\E=\E_{\rm CD}$ then $\Pi^E$ is $\D^\times$-equivariant and hence the pullback of something $\varsigma(P_E)$ defining $\Gr(\Ei)$. So no, when $\Gr(\Ei)$ is not locally free we do not expect $\Ei_\B\cong\Ei_{\rm CD}$.
%\end{Remark}

%Let us assume $H^0(\M;\Ei)=\C^N$ for simplicity and that $\Ei$ is 0-regular, so that the map $\Ai_m\otimes\C^N\to E_m$ is surjective. Then we get indeed $\GE_m=E_m$. 

\begin{Remark}[Grading and $\D^\times$-equivariance]
In general it is not clear if $\E_{\rm CD}$ is $\D^\times$-equivariant iff $\E$ is. If the inclusion $\E\subset\E_{\rm CD}$ is proper then $\E(v)=\E([v])$ need not imply $\Pi^E(v)=\Pi^E([v])$. That is, $\Pi^E$ might fail to be $\D^\times$-equivariant and hence not equal the pullback of $\varsigma(P_E)$. %When $\varsigma(P_E)$ is $C^0$ we know that $\Pi^E(r\ \cdot\ )$ is $C^0$ and $\Un(1)$-equivariant on $\Sb_r$ for $r$ close enough to $1$ (see Proposition \ref{closetoboundprop}). But can $\Pi^E$ be discontinuous in the radial variable? 
%If $\Ei$ is not torsionfree then we do not know what to expect $\varsigma(P_E)$ to define. And thus we cannot say if we expect $\Pi^E$ to be the pullback of $\varsigma(P_E)$ or not. 
\end{Remark}

The linear space $\E$ is by definition locally over an open subset $\U\subset\B$ the kernel of an $M$-tuple of holomorphic functions for some integer $M\geq 1$. But the Nullstellensatz above shows that if we take $M=+\infty$ then in this algebraic setting we can describe $\E_{\rm CD}$ globally as a zero-set when $\Ei$ is locally free:
\begin{cor}
Suppose that $\Ei_{\rm CD}=\Oi_{\rm CD}\otimes\Ei_{\B\setminus\{0\}}$ where $\Ei_{\B\setminus\{0\}}$ is the pullback of a holomorphic vector bundle $\Ei$ over $\M$ and let $\Theta_E:\B\times\ell^2(\N_0)\to\B\times\C^N$ be any multiplier with $\Ker M_{\Theta_E}^*=\cspan\{k_v\otimes\xi|\ (v,\xi)\in\E\}$. Then 
$$
\E(v)=\Ker\Theta_E^*(\bar{v}).
$$
\end{cor}

%\begin{prop}
%The affine linear space $\E$ is the vanishing set of $I_\N$,
%\begin{equation}\label{Eisvanofid}
%\E=\V(I_\N):=\{(v,\xi)\in\V\otimes\C^N|\ \bra f(v)|\xi\ket_{\C^N}=0\text{ for all }f\in I_\N\}.
%\end{equation}
%By definition of $J_\N(\E)$ this gives
%$$
%I_\N\subset J_\N(\E),
%$$
%and
%$$
%\E\subset\V(J_\N(\E)).
%$$
%\end{prop}
%\begin{proof}
%We have an exact sequence of sheaves
%$$
%0\to\Ii_\V\to\Oi_\V\otimes\C^N\to\Ei_\V\to 0.
%$$
%We want to obtain $\Ii_\V$ as the patchings of cokernel sheaves $\Coker\Theta_\U$ of locally defined morphisms $\Theta_\U:\Oi_\U\otimes\C^M\to\Oi_\U\otimes\C^N$ whose associated maps $\Theta_\U:\U\times\C^M\to\U\times\C^N$ define the linear space $\E$ via $\E_\U=\Coker\Theta_\U$. Then the fact that elements of $I_\N=H^0(\V;\Ii_\V)$ generate each $H^0(\U;\Ii_\U)=H^0(\U;\Coker\Theta_\U)$ gives that elements of $I_\N$ vanish on $\E$ and nowhere else. 
%So we see that a linear space $\E$ is always equal to the vanishing set of $H^0(\V;\Ii_\V)$ for some globally defined cokernel sheaf $\Ii_\V$. 
%\end{proof}

\section{The Toeplitz part of a superharmonic projection}

\subsection{Lifts of projections}\label{liftsec}

We are now going to investigate the possibility of quantizing smooth Hermitian vector bundles over $\M=\G/\K$. From the results of \cite{Hawk1} and \cite{Wang1, Wang2} we expect that a Hermitian metric on a vector bundle can be quantized in a stronger sense if the vector bundle is $\G$-equivariant or satisfies some kind of stability condition. 

Recall the Toeplitz short exact sequence
\begin{equation}\label{ToepSES}
0\to\Gamma_0\to\Ti_\GH^{(0)}\overset{\varsigma}{\to}C^0(\M)\to 0
\end{equation}
where $\Ti_\GH^{(0)}$ is the $C^*$-algebra of Toeplitz operators with symbol in $C^0(\M)$ acting on the Fock space $\GH_\N$ and 
$\Gamma_0$ is the ideal in $\Ti_\GH^{(0)}$ consisting of compact operators which preserve the grading on $\GH_\N$. By the Swan theorem, a $C^0$ vector bundle $\Ei$ over $\M$ is the same datum as an idempotent $P^E$ over $C^0(\M)$, which if taken selfadjoint also defines a Hermitian metric on $\Ei$. The symbol map $\varsigma$ was discussed extensively in the last two sections. The Toeplitz map $\breve{\varsigma}:C^0(\M)\to\Ti_\GH^{(0)}$ gives a positive linear splitting of \eqref{ToepSES}. Following ideas of noncommutative geometry, and in particular \cite{Hawk1, Hawk2}, a \textbf{quantization} of $P^E$ is a projection $P_E$ over $\Ti_\GH^{(0)}$ which is equal to $\breve{\varsigma}(P^E)$ modulo $\Gamma_0$, i.e. such that
$$
\varsigma(P_E)=P^E.
$$
As we have seen, there is a special class of projections over $\Ti_\GH^{(0)}$ whose ranges are quotient modules, namely the $\Psi$-superharmonic projections with symbol in $C^0(\M)$. A natural question is thus whether our given $P^E$ admits a $\Psi$-superharmonic quantization. We shall see that the answer is ``no'' in general. Yet for any $P^E$ there is a natural candidate to a $\Psi$-superharmonic quantization. Indeed, if $P^E$ belongs to $C^0(\M)\otimes\Mn_N(\C)$ then the Toeplitz operator $\breve{\varsigma}(P^E)$ acts on $\GH_\N\otimes\C^N$ and the range projection of $\breve{\varsigma}(P^E)$ is $\Psi$-superharmonic since $\breve{\varsigma}(P^E)$ is $\Psi$-superharmonic (indeed $\Psi$-harmonic). If we let $\Ran\breve{\varsigma}(P^E)$ denote the range projection of $\breve{\varsigma}(P^E)$, we would thus like to know if $\Ran\breve{\varsigma}(P^E)$ is a quantization of $P^E$. As mentioned, in general $\Ran\breve{\varsigma}(P^E)$ is not a quantization of $P^E$.

Let us look at the basic operator aspects of this lifting problem. Since $\varsigma(\breve{\varsigma}(P^E))=P^E$, we know that $\breve{\varsigma}(P^E)$ is a projection modulo $\Gamma_0$ or, what is the same since $\breve{\varsigma}(P^E)$ preserves the grading on $\GH_\N\otimes\C^N$, we know that $\breve{\varsigma}(P^E)$ is a projection modulo compact operators. By Brown--Douglas--Fillmore theory \cite[\S IX]{Davi2}, every projection modulo compacts is of the form normal plus compact. But recall that even more is true: every projection modulo compacts is projection plus compact (this is a particular case of \cite{Olse1}; see a quick proof in \cite[Prop. 3.1]{Weav1}). So there is a projection $Q_E$ acting on $\GH_\N\otimes\C^N$ such that
$$
Q_E=\breve{\varsigma}(P^E)+C_E
$$
with $C_E$ compact. Since $\Ti_\GH^{(0)}=\breve{\varsigma}(C^0(\M))+\Gamma_0$, we see that $Q_E$ belongs to $\Ti_\GH^{(0)}$. So there always exists a lift of $P^E$ to a projection over $\Ti_\GH^{(0)}$ of the same matrix size, as can also be shown by operator-algebraic reasoning applied to the sequence \eqref{ToepSES}. 

However, a projection modulo compacts is rarely equal to its range projection modulo compacts. For a simple example, let $K$ be a compact operator. Then $K$ is equal to zero modulo compacts (hence $K$ is a projection modulo compacts) but the range projection $P_K$ of $K$ need not be of finite rank (it is not unless $K$ is of finite rank, by definition) so $P_K$ is not zero modulo compact in general (a projection is compact iff it has finite rank). 

Thus, for the range projection of $\breve{\varsigma}(P^E)$, in general we have
$$
P^E=\varsigma(\breve{\varsigma}(P^E))\ne\varsigma(\Ran\breve{\varsigma}(P^E)).
$$
Still, since $\breve{\varsigma}(P^E)$ has entries in $\Ti_\GH^{(0)}\subset\Ni$ we know that $\Ran\breve{\varsigma}(P^E)$ has entries in $\Ni$, because $\Ni$ is a von Neumann algebra. In fact:
\begin{prop}
Let $P^E$ be a projection over $C^0(\M)$. Then $\Ran\breve{\varsigma}(P^E)$ has entries in $\Ti_\GH^{(0)}\subset\Ni$.
\end{prop}
\begin{proof}
As observed in \cite[Thm. 10.4]{DRS1}, the proof of \cite[Lemma 1.13]{Arv6c}% (see also Guo--Hu--Xu's calculation \cite[Example 1, \S2]{GHX1})
generalizes so as to show that the $C^*$-algebra $\Ti_\GH$ is the $C^*$-envelope of the operator system generated by $S$. This gives that $\Ti_\GH$ is an injective $C^*$-algebra, %hence monotone complete, 
hence an $AW^*$-algebra, hence every element of $\Ti_\GH$ has its range projection belonging to $\Ti_\GH$ (see \cite[Lemma 2.1.5]{SaWr1}).
\end{proof}
Thus the range projection of $\breve{\varsigma}(P^E)$ is a coinvariant projection over $\Ti^{(0)}_\GH$, hence with symbol in $C^0(\M)$, but in order to obtain $\Ran\breve{\varsigma}(P^E)$ from $\breve{\varsigma}(P^E)$ one may have to do more than just adding compacts.

%\begin{Remark}[Smooth sheaves with analytic structure]
%If the projection $\varsigma(P_E)$ defines a smooth vector bundle over $\M$ then  
%\end{Remark}
%\begin{prop}
%Let $P_E$ be a coinvariant projection over $\Ti_\GH^{(0)}$ and define $P^E:=\varsigma(P_E)$. Then $\varsigma(\Ran\breve{\varsigma}(P^E))=P^E$.
%\end{prop}
%\begin{proof}
%Let $\GH^E_\N$ denote the range of $P_E$. 

%We have $\breve{\varsigma}(P^E)\leq P_E$ so both $\breve{\varsigma}(P^E)$ and $C_E:=P_E-\breve{\varsigma}(P^E)$ preserves the subspace $\GE_\N\subset\GH_\N\otimes\C^N$ and acts by zero outside it. Since $C_E$ is compact, this gives that $\breve{\varsigma}(P^E)$ is Fredholm as an operator on $\GE_\N$ and thus we have equality $\Ran\breve{\varsigma}^{(m)}(P^E)=P_{E,m}$ for all large enough $m$. This gives the result. 
%\end{proof}

Given a coinvariant projection $P_E$, there is a unique projection $P^E$ over $L^\infty(\M)$ such that the Toeplitz operator $\breve{\varsigma}(P^E)$ is the $\Psi$-harmonic part of $P_E$. 
Indeed, it is given by $P^E=\varsigma(P_E)$. Let us now discuss uniqueness of a coinvariant lift:

\begin{prop}[Uniqueness of superharmonic lift]\label{uniqeliftprop}
Suppose that $P^E$ is a projection over $C^0(\M)$ with a coinvariant lift $P_E$. Then $\Ran\breve{\varsigma}(P^E)$ is also a coinvariant lift of $P_E$. In fact, $P_E$ equals $\Ran\breve{\varsigma}(P^E)$ up to finite-rank operators.
\end{prop}
\begin{proof}
The assumption $\varsigma(P_E)=P^E$ gives $P_E=\breve{\varsigma}(P^E)+C_E$ with a pure $\Psi$-superharmonic operator $C_E$. Moreover, $C_E$ is compact since $P_E$ is over $\Ti_\GH^{(0)}$.

We have $\breve{\varsigma}(P^E)=\lim_{q\to\infty}\Psi^q(P_E)\leq\Psi^p(P_E)\leq P_E$ for all $p\in\N$, so $\breve{\varsigma}(P^E)$ preserves the range of $P_E$, as does the pure superharmonic part $C_E$ of $P_E$. Since $\breve{\varsigma}(P^E)$ equals $P_E$ modulo compacts, this gives that $\breve{\varsigma}(P^E)$ restricts to a Fredholm operator on the range $\GH^E_\N$ of $P_E$. Thus $\breve{\varsigma}(P^E)$ is invertible modulo finite-rank operators as an operator on $\GH^E_\N$, so that $\breve{\varsigma}^{(m)}(P^E)$ is invertible as operator on $\GH^E_m$ for large enough $m$. Thus, up to finite-rank operators, $P_E$ equals the range projection of $\breve{\varsigma}^{(m)}(P^E)$. 
\end{proof}
%So in fact there can be (up to finite-rank operators) at most one superharmonic lift of a projection $P^E$ over $C^0(\M)$, and this is $\Ran\breve{\varsigma}(P^E)$. 
The proof of Proposition \ref{uniqeliftprop} breaks down if $C_E$ is not compact. 
\begin{Question}
Can we drop the continuity assumption in Proposition \ref{uniqeliftprop}? That is, if $P^E$ is a projection over $L^\infty(\M)$ with a coinvariant lift $P_E$, is then $\Ran\breve{\varsigma}(P^E)$ also a coinvariant lift of $P_E$?
\end{Question}
\begin{Remark}[Continuous symbol]\label{uniqueRemark}
As for the uniqueness of $\Ran\breve{\varsigma}(P^E)$ up to finite-rank operators as coinvariant lift of $P^E$, we shall see that the continuity of $P^E$ is a necessary assumption. Here it is very important to distinguish between $P^E$ being continuous and $P^E$ having entries in the embedded subalgebra $C^0(\M)\subset L^\infty(\M)$, since an element of the latter is just the almost everywhere equivalence class of a continuous functions. The map $\varsigma:\Ni\to L^\infty(\M)$ can take values in $C^0(\M)$ even on elements that do not belong to $\Ti_\GH^{(0)}$. Indeed, its values on $\breve{\varsigma}(C^0(\M))+\Gamma_\omega$ are in $C^0(\M)$. Therefore we can have a lift
$$
P_E\in(\breve{\varsigma}(C^0(\M))+\Gamma_\omega)\otimes\Bi(\GE_0),
$$
and coinvariance of $P_E$ just says that $P_E=\breve{\varsigma}(P^E)+C_E$ where $C_E$ is merely pure $\Psi$-superharmonic. So the $\Psi$-superharmonic lift would not be unique.  
\end{Remark}

 \subsubsection{Nonexisting lifts}

If $P^E$ is a projection over $C^\infty(\M)$ defining a smooth vector bundle $\Ei$ which is not holomorphic, what is the geometric meaning of the graded quotient module $\Ran\breve{\varsigma}(P^E)$? What is the coherent $\Oi_\M$-module $\Ei_{\rm CD}$ and how is it related to the Serre sheaf of the graded $\Ai$-module underlying $\Ran\breve{\varsigma}(P^E)$? Note that $\Ei_{\rm CD}$ cannot be the pullback of $\Ei$ even as smooth vector bundle because then $\Ei_{\rm CD}$ would have to be locally free, contradicting the assumption that $\Ei$ does not admit a holomorphic structure. Indeed, if $\Ei$ is a coherent $\Ci^\infty_\M$-module of the form $\Ei=\Ei^o\otimes_{\Oi_\M}\Ci^\infty_\M$ with $\Ei^o$ a coherent $\Oi_\M$-module then $\Ei$ is locally free if and only if $\Ei^o$ is locally free. To see this, note that we can define a $\bar{\pd}$-operator on $\Ei$ by setting $\bar{\pd}^E:=\bone\otimes\bar{\pd}$ on $\Ei^o\otimes_{\Oi_\M}\Ci^\infty_\M$ with $\bar{\pd}$ the operator on $C^\infty(\M)$ defined by the complex-analytic structure on $\M$. Then $\bar{\pd}^E$ is integrable, i.e. a holomorphic structure, since $\bar{\pd}$ is. The kernel of $\bar{\pd}^E$, which is precisely $\Ei^o$, is locally free by the Koszul--Malgrange theorem \cite{KoMa1}. 

This observation and Theorem \ref{bigCDversuslocallyfree} give us: 
\begin{cor}
Let $P^E$ be a projection over $C^\infty(\M)$ and suppose that the vector bundle defined by $P^E$ does not admit a holomorphic structure. Then
$$
\varsigma(\Ran\breve{\varsigma}(P^E))\ne P^E.
$$
and the vector bundles defined by $\varsigma(\Ran\breve{\varsigma}(P^E))$ and $P^E$ are not topologically isomorphic.  
\end{cor}

 \subsubsection{$\dim\Ran\breve{\varsigma}^{(m)}(P^E)$ versus $\chi(\Ei(m))$}\label{dimransec}

The projection onto any quotient module $\GE_\N$ is of the form $P_E=\breve{\varsigma}(P^E)+C_E$ where $P^E$ is a projection over $L^\infty(\M)$ and $\breve{\varsigma}(P^E)\leq P_E$. If we assume that $C_E$ is compact then the restriction of $\breve{\varsigma}(P^E)$ to $\GE_\N$ is invertible modulo compact, which is the same as being invertible modulo finite-rank operators. So when $C_E$ is compact 
we have for $m\gg 0$ that
$$
\dim\Ran\breve{\varsigma}^{(m)}(P^E)=\Tr(P_{E,m})=\chi(\Ei(m)),
$$
where $\Ei$ is the Serre sheaf of the graded $\Ai$-module underlying $\GE_\N$. %But $P_E$ need not be the lift of any projection over $C^\infty(\M)$ defining $\Ei$ as smooth vector bundle. 
  So $P^E$ has a lift with ``correct dimensions''. But this $P^E$ was special since it was the symbol of a superharmonic projection. In general we can ask:  
 \begin{Question}
Let $\Ei$ be a smooth vector bundle over $\M$ and let $P^E$ be a smooth Hermitian metric on $\Ei$. Does it follow in this generality that $\dim\Ran\breve{\varsigma}^{(m)}(P^E)$ equals $\chi(\Ei(m))$ for $m\gg 0$?
\end{Question}

We have a result in this direction:
\begin{prop}\label{propofrealalg}
 Let $\Ei$ be a holomorphic vector bundle over $\M$ and suppose that $P^E$ is a real-analytic Hermitian metric on $\Ei$ with a Parseval $C^*$-frame given by a basis for $H^0(\M;\Ei)$. Then $\dim\Ran\breve{\varsigma}^{(m)}(P^E)=\chi(\Ei(m))$ for $m\gg 0$. 
\end{prop}
\begin{proof}
%Let $N:=\dim H^0(\M;\Ei)$. 
For each $m\geq 0$ we have that $\FS(\GH_m)\otimes P^E$ has a Parseval $C^*$-frame given by a basis for $H^0(\M;\Ei(m))$.
Since $\breve{\varsigma}^{(m)}(P^E)=n_m(\omega\otimes\id)(\FS(\GH_m)\otimes P^E)$ (see \eqref{Toeplexplic}), this gives the result. 
\end{proof}

We saw in the last section (Theorem \ref{nullothm}) that we can associate a (graded) quotient module $\GE_\N$ to every (globally generated) holomorphic vector bundle $\Ei$ over $\M$ with $\dim\GE_m=\chi(\Ei(m))$ for $m\gg 0$. However, the symbol $\varsigma(P_E)$ of this quotient module $\GE_\N$ need not be continuous. 

Also, in the setting of Proposition \ref{propofrealalg} in general the projection $P_E:=\Ran\breve{\varsigma}(P^E)$ has symbol 
$$
\varsigma(P_E)\ne P^E,
$$
and therefore the metric $P^E$ is of ``less value'' for quantization purposes.

\subsection{Into Hardy space}

Let $P_E$ be a $\Psi$-superharmonic projection acting on $\GH_\N\otimes\C^N$. As before we write $P_E$ uniquely as
$$
P_E=\breve{\varsigma}(P^E)+C_E
$$
with $P^E$ a projection over $L^\infty(\M)$ and $\mathrm{SOT-}\lim_{m\to\infty}\Psi^m(C_E)=0$. If we let $T=(T_1,\dots,T_n)$ be the Kraus operators of $\Psi$, i.e. we have $\Psi(X)=\sum^n_{\alpha=1}T_\alpha^*XT_\alpha$ for all $X\in\Bi(\GH_\N\otimes\C^N)$, then $T$ has the same invariant and coinvariant subspaces as the shift $S$. Therefore, if $\GH^E_\N$ is the range of $P_E$ then the restriction 
$$
T_{E,\alpha}^*:=T_\alpha^*|_{\GE_\N}=(P_ET_\alpha|_{\GH^E_\N})^*,\qquad\forall\alpha=1,\dots,n
$$
preserves $\GE_\N$. %This defines a commutative operator tuple $T_E$ on the Hilbert space $\GE_\N$. 
From $\Psi(P_E)\leq P_E$ we get 
$$
\Psi_E(\bone):=\sum^n_{\alpha=1}T_{E,\alpha}^*T_{E,\alpha}\leq P_E
$$
where $P_E$ is now playing the role of identity operator on $\GH^E_\N$. That is, $T_E$ is a spherical contraction on $\GH^E_\N$. The limit
$$
A_{T_E}:=\mathrm{SOT-}\lim_{m\to\infty}\Psi_E^m(\bone)
$$
therefore exists, and indeed we see that
$$
A_{T_E}=\breve{\varsigma}(P^E)|_{\GE_\N}.
$$
Note that $A_{T_E}$ and $\breve{\varsigma}(P^E)$ are practically the same since $\breve{\varsigma}(P^E)$ acts by zero outside $\GH^E_\N$. If we assume that $C_E$ is compact then $\GH^E_\N$ coincides with $\Ran\breve{\varsigma}(P^E)$ up to finite-dimensional subspaces, and $A_{T_E}$ is a Fredholm operator in the sense that it is a positive operator with finite-dimensional kernel and cokernel. 

The ``asymptotic limit'' $A_{T_E}$ of a contraction $T_E$ has been widely studied in the case ($n=1$) of a single operator \cite{Gehe1, Gehe2, Kerc10, Kubr1} but occasionally also for tuples \cite{Pop7}. Using $A_{T_E}$ one changes the inner product on the Hilbert space to make $T_E$ an isometry, provided that $A_{T_E}$ is invertible. 
Here we shall use $A_{T_E}$ for this purpose and moreover find the geometric meaning of the new inner product.  
It turns out that, if $P^E$ defines a vector bundle, the rather nonstandard quantization $\GH^E_\bullet$ transforms via $A_{T_E}$ to the more familiar quantization taking place on the Hardy space of $P^E$.

\subsubsection{Subnormality with algebraic relations}

%Since any compression of $S_E$ to a coinvariant subspace satisfies the relations of $\Ai_\GH$, we want thus to show that these compressions $T_E$ give a unital CP-rep $\varphi_E$ of $\Oi_\GH$ with $Z^*Z\to T^*_ET_E$ and $Z\to T_E$ precisely when $T_E$ is a 1-isometry. And then we want that this gives precisely the minimal normal extension of $T_E$ by taking minimal Stinespring of $\varphi_E$. 

Recall that $\Sb\subset\Sb^{2n-1}$ is the principal $\Un(1)$-bundle over $\M\subset\C\Pb^{n-1}$ associated with the hyperplane bundle restricted to $\M$. Let $I_\M$ be the ideal in $\C[z_1,\dots,z_n]$ which defines $\GH_\bullet$. In other words, $I_\M$ is the ideal such that the homogeneous coordinate ring of $\M$ is given by $\Ai=\C[z_1,\dots,z_n]/I_\M$. 
\begin{dfn}
An $n$-tuple $T=(T_1,\dots,T_n)$ of commuting operators on a Hilbert space is $\GH_\bullet$-\textbf{subnormal} if $T$ is jointly subnormal with minimal normal extension $M=(M_1,\dots,M_n)$ satisfying the relations of the ideal $I_\M$ %in $\C[z_1,\dots,z_n]$ which defines $\GH_\bullet$ 
in the sense that
$$
f(M)=0,\qquad \forall f\in I_\M.
$$
In this case, $T$ is called an $\Sb$-\textbf{isometry} if 
$$
\sigma(M)\subset\Sb.
$$
\end{dfn}
By Athavale's theorem \cite[Prop. 2]{Atha3}, an $\Sb^{2n-1}$-isometry is the same as a commutative \textbf{spherical isometry}, i.e. an operator tuple $T$ of commuting operators with $\sum^n_{\alpha=1}T_\alpha^*T_\alpha=\bone$. So every $\Sb$-isometry is a spherical isometry. 

The following was inspired by \cite[Thm. 2.1]{Feld1} and \cite[\S2]{Pop7}:
\begin{Lemma}\label{GHsubnormallemma}
Let $T=(T_1,\dots,T_n)$ be a tuple of commuting operators on a Hilbert space $\Hi$. Then $T$ is an $\Sb$-isometry if and only if there exists a unital completely positive map 
$$
\varrho:C^0(\Sb)\to\Bi(\Hi)
$$
with $\varrho(Z_\alpha)=T_\alpha$ and 
\begin{equation}\label{subnormvarrho}
\varrho(Z_\alpha^*Z_\beta)=T_\alpha^*T_\beta,\qquad\forall \alpha,\beta\in\{1,\dots,n\}.
\end{equation}
%In this case, $T$ is an $\Sb$-isometry if and only if $\varrho$ is unital and completely positive. 
\end{Lemma}
\begin{proof}
Suppose that such a map $\varrho$ exists. As for any unital completely positive map, we have a Stinespring representation $\pi_\varrho:C^0(\Sb)\to\Bi(\Hi_\varrho)$ of $\varrho$, i.e. a $*$-homomorphism such that
$$
\varrho(f)=V_\varrho^*\pi_\varrho(f)V_\varrho,\qquad \forall f\in C^0(\Sb)
$$
with an isometry $V_\varrho:\Hi\to\Hi_\varrho$ into some Hilbert space $\Hi_\varrho$. 
Since $\pi_\varrho$ is a $*$-algebra homomorphism, the tuple $M=(M_1,\dots,M_n)$ defined by
$$
M_\alpha:=\pi_\varrho(Z_\alpha),\qquad\forall\alpha\in\{1,\dots,n\}
$$
consists of normal operators on $\Hi_\varrho$ satisfying the relations of the ideal in $\C[z_1,\dots,z_n]$ which defines $\GH_\bullet$. Since $\pi_\varrho$ is a $*$-algebra homomorphism we also have $\sigma(M)\subset\Sb$. If we can show that the subspace $V_\varrho(\Hi)\subset\Hi_\varrho$ is invariant under $M_1,\dots,M_n$ then the tuple $T$ will be an $\Sb$-isometry. For that we use the assumption \eqref{subnormvarrho}. For each $\alpha\in\{1,\dots,n\}$ it says 
$$
T_\alpha^*T_\alpha=V_\varrho^*\pi_\varrho(Z_\alpha^*Z_\alpha)V_\varrho=V_\varrho^*M_\alpha^*M_\alpha V_\varrho,
$$
and if we denote by $P$ the orthogonal projection of $\Hi_\varrho$ onto $V_\varrho(\Hi)$ then writing
\begin{align*}
V_\varrho^*M_\alpha^*M_\alpha V_\varrho&=PM_\alpha^*M_\alpha|_{V_\varrho(\Hi)}
\\&=T_\alpha^*T_\alpha+PM_\alpha^*(\bone-P)M_\alpha|_{V_\varrho(\Hi)}
\end{align*}
we conclude that
$$
T_\alpha^*T_\alpha=T_\alpha^*T_\alpha+PM_\alpha^*(\bone-P)M_\alpha|_{V_\varrho(\Hi)},
$$
so that $(\bone-P)M_\alpha|_{V_\varrho(\Hi)}=0$, i.e. $V_\varrho(\Hi)$ is invariant under $M_\alpha$, as desired. 

To prove the converse, note that $C^0(\Sb)$ is generated by a commuting tuple $Z=(Z_1,\dots,Z_n)$ of normal operators satisfying $\sum^n_{\alpha=1}Z_\alpha^*Z_\alpha=\bone$ (sphere condition) and the relations of $\GH_\bullet$ but no other relation. Therefore, if we suppose that $T$ is an $\Sb$-isometry on $\Hi$, with minimal normal extension $M$ thus satisfying the relations of $\GH_\bullet$ and having spectrum in $\Sb$, then there is a $*$-representation $\pi$ of $C^0(\Sb)$ with $\pi(Z_\alpha)=M_\alpha$. We let $V$ be the isometric embedding of $\Hi$ into the Hilbert space on which $M$ acts, and we define $\varrho(f):=V^*_M\pi(f)V_M$ for all $f\in C^0(\Sb)$. Then, since $\Hi$ is invariant under each $M_\alpha$, we have the property \eqref{subnormvarrho}. 
\end{proof}

\subsubsection{Similarity to a spherical isometry}
In \cite{An6} we showed that $C^0(\Sb)$ can be identified with the ``Cuntz--Pimsner algebra'' of the subproduct system $\GH_\bullet$, namely the $C^*$-algebra $\Oi_\GH$ defined as the quotient of the Toeplitz algebra $\Ti_\GH=C^*(S_1,\dots,S_n)$ by the ideal of compact operators. This fact can be useful as one can often adopt known constructions involving the Cuntz algebra $\Oi_n$ to a more general Cuntz--Pimsner algebra. Here is an example:
\begin{cor}
Let $\GH^E_\N$ be a graded quotient module such that the positive operator $A_{T_E}$ is invertible. % finite-dimensional subspace $\Ker A_{T_E}$ equals $\{0\}$. 
Define a commutative operator $n$-tuple $V_E$ by 
$$
V_{E,\alpha}:=A^{1/2}_{T_E}T_{E,\alpha}A^{-1/2}_{T_E},\qquad\forall\alpha\in\{1,\dots,n\}.
$$
Then $V_E$ is an $\Sb$-isometry.
\end{cor}
\begin{proof}
The tuple $V_E$ satisfies the same relations as does $T_E$, since they are similar. We have
$$
V_{E,\alpha}^* V_{E,\beta}=A^{-1/2}_{T_E}T_{E,\alpha}^*A_{T_E}T_{E,\beta}A^{-1/2}_{T_E}.
$$
Since $\Psi_E(A_{T_E}):=\sum_{\alpha=1}T_{E,\alpha}^*A_{T_E}T_{E,\alpha}=A_{T_E}$, it is clear that $V_E$ is a spherical isometry. We need to show that the minimal normal extension of $V_E$ also satisfies the relations of $\GH_\bullet$. But we can construct as in \cite[Thm. 2.3]{Pop7} a unital completely positive linear map $\varrho_E:\Oi_\GH\to\Bi(\GH^E_\N)$ with
$$
\varrho_E(Z_\alpha^*Z_\beta)=V_{E,\alpha}^*V_{E,\beta},\qquad\forall \alpha,\beta\in\{1,\dots,n\}.
$$
By Lemma \ref{GHsubnormallemma} this is precisely the statement that $V_E$ is an $\Sb$-isometry.
\end{proof}
Of course, after observing that $V_E$ is a spherical isometry an application of Athavale's theorem immediately gives that $V_E$ is subnormal. 
What is important is however that we obtain a unital completely positive linear map $\varrho_E:\Oi_\GH\to\Bi(\GH^E_\N)$ whose Stinespring dilation gives a $*$-representation $\pi_E$ of $C^0(\Sb)=\Oi_\GH$ on a Hilbert space containing $\GH^E_\N$ as a subspace invariant under $\pi_E(Z_1),\dots,\pi_E(Z_n)$. We shall see next that $V_E$ is unitarily equivalent to the multiplication tuple on the Hardy space associated to $\varsigma(P_E)$ and $\omega$, and that the normal dilation is the multiplication tuple on ambient $L^2$-space.

\begin{Remark}
The commutative spherical isometry $V_E$ given as above by $V_{E,\alpha}:=A^{1/2}_{T_E}T_{E,\alpha}A^{-1/2}_{T_E}$ is not the spherical isometry $W_E$ appearing in the polar decomposition $T_E=W_E|T_E|$ of the column operator $T_E:\GH^E_\N\to(\GH^E_\N)^{\oplus n}$. It is the distinction $|T_E|^2=\Psi_E(\bone)$ versus $A_{T_E}=\lim_m\Psi_E^m(\bone)$. We shall need to use $V_E$ because $A_{T_E}$ is a Toeplitz operator (i.e. $\Psi_E(A_{T_E})=A_{T_E}$) while $|T_E|^2$ is not. Moreover, $A_{T_E}$ is invertible (modulo finite-rank operators) under the natural assumption that $\varsigma(P_E)$ is $C^0$, while $|T_E|^2$ might not be so.
\end{Remark}
If $\GE$ is a Hilbert space and $A$ is a positive invertible operator on $\GE$, we denote by $A^{-1/2}\GE$ the vector space $\GE$ endowed with the inner product % (cf. \S3.1 in Suciu's Habiliation thesis or \cite{ACG1})
$$
\bra\phi|\psi\ket_{A^{-1/2}\GE}:=\bra A^{-1/2}\phi|A^{-1/2}\psi\ket_{\GE},\qquad\forall\phi,\psi\in \GE.
$$
We can then view $A^{1/2}$ as a unitary operator from $\GE$ to $A^{-1/2}\GE$,

%Check: If $X\in\Bi(\GE)$ is any operator, its adjoint as operator $X:\GE\to A^{-1/2}\GE$ satisfies
%$$
%\bra X^{*_A}\phi|\psi\ket_{\GE}=\bra \phi|X\psi\ket_{A^{-1/2}\GE}=\bra\phi A^{-1}X\psi\ket_{\GE}=\bra X^*A^{-1}\phi \psi\ket_{\GE}
%$$
%Taking $X=A^{1/2}$ one sees that $A^{1/2}$ is unitary as operator from $\GE$ to $A^{-1/2}\GE$.

When $\GH^E_\N$ is a graded quotient module such that $A_{T_E}$ is invertible, we write
$$
\GK^E_\N:=A_{T_E}^{-1/2}\GH^E_\N
$$
and dentote by $A_E^{1/2}:\GH^E_\N\to\GK^E_\N$ the associated unitary operator. The tuple $V_E$ will sometimes be identified with $A_E^{1/2}T_EA_E^{-1/2}$ acting on $\GK^E_\N$.

\subsubsection{Identification of $\GK^E_\N$}

In the following we denote by $\FS(\GH_1)^{\otimes (-m)}$ the transpose of $\FS(\GH_1)^{\otimes m}$ for each $m\in\N$. Thus 
$\FS(\GH_1)^{\otimes (-m)}$ is a projection over $C^\infty(\M)$ which defines the line bundle $\Oi_\M(-m)$. If $P^E$ is a projection over $C^\infty(\M)$ defining a smooth vector bundle $\Ei$ then the vector space $\Gamma^\infty(\Sb;\Ei_\Sb,P^E)$ of global sections of the pullback of $\Ei$ to $\Sb$ splits as $C^\infty(\Sb)$-module into
$$
\Gamma^\infty(\Sb;\Ei_\Sb,P^E)=\bigoplus_{k\in\Z}\Gamma^\infty(\M;\Ei(m),\FS(\GH_1)^{\otimes k}\otimes P^E).
$$
%Indeed, 

\begin{thm}\label{backtoHardythmgen}
Let $P^E$ be a projection over $C^0(\M)$ defining a smooth vector bundle $\Ei$ and assume that $\varsigma(\Ran\breve{\varsigma}(P^E))=P^E$. %Suppose also that there is a 0-regular holomorphic structure on $\Ei$ such that $P^E=\FS(\GE_0)$ for some inner product $\GE_0$ on $H^0(\M;\Ei)$. 
Then the limit operator $A_{T_E}$ of the spherical contraction $T_E$ on the quotient module $\GH^E_\N:=\Ran\breve{\varsigma}(P^E)$ can be used to map $\GH^E_\N$ into a subspace of the $L^2$-space of $P^E$ and $\omega$,
$$
\GK^E_\N:=A_{T_E}^{-1/2}\GH^E_\N\subset L^2(\Sb,\omega;P^E)=\bigoplus_{k\in\Z}L^2(\omega;\FS(\GH_1)^{\otimes k}\otimes P^E),
$$
and there is a holomorphic structure on $\Ei$ such that $\GK^E_\N$ identifies with the Hardy space of $P^E$ and $\omega$ (up to a finite-dimensional subspace),
$$
\GK^E_\N=H^0(\Sb,\omega;\FS(\GH_\N)\otimes P^E)=\bigoplus_{m\in\N_0}H^0(\M,\omega;\FS(\GH_m)\otimes P^E).
$$
So $\GK^E_\N$ is invariant under action of the generators $Z_1,\dots,Z_n\in C^0(\Sb)$ acting in the multiplication representation on $L^2(\Sb,\omega;\Ei,P^E)$ and the restriction of $Z=(Z_1,\dots,Z_n)$ to $\GK^E_\N$ identifies with the $\Sb$-isometry $V_E:=A_E^{1/2}T_EA_E^{-1/2}$. 
%identifies with the $\Sb$-isometry of multiplication operators by the coordinate functions on the Hardy space $H^0(\Sb,\omega;\Ei,\varsigma(P_E))$, with minimal normal dilation given by the multiplication tuple on $L^2(\Sb,\omega;\Ei,\varsigma(P_E))$.
\end{thm}
\begin{proof}
By assumption $P^E$ is an element of $C^0(\M)\otimes\Bi(\GE_0)$ for some finite-dimensional Hilbert space $\GE_0$. We may identify $\GE_0$ as a vector subspace of $\Gamma^0(\M;\Ei)$ via $\psi(x)=P^E(x)\psi$ for $\psi\in\GE_0$. Then $P^E$ has a Parseval $C^*$-frame consisting of an orthonormal basis for $\GE_0$, symbolically $P^E=\FS(\GE_0)$, but note that so far we have not shown that $\GE_0$ consists of holomorphic sections of $\Ei$ for any holomorphic structure on $\Ei$ (we did not assume that $\Ei$ admits a holomorphic structure). The projection $\FS(\GH_m)\otimes P^E$ has a Parseval $C^*$-frame $\boldsymbol\psi=(\psi_j^{(m)})_{j\in\J}$ given by elements of $\GH_m\otimes\GE_0$, and we denote by $\GE_m$ the $\C$-linear span in $\GH_m\otimes\GE_0$ of that $C^*$-frame $\boldsymbol\psi$. The tensor product of the Parseval $C^*$-frames for $\FS(\GH_m)$ and $P^E$ is a Parseval frame for $\GH_m\otimes\GE_0$. Thus, if $P_{E,m}$ is the projection of $\GH_m\otimes\GE_0$ onto $\GE_m$ then $\boldsymbol\psi$ is the image under $P_{E,m}$ of a Parseval frame for $\GH_m\otimes\GE_0$. Hence $\boldsymbol\psi$ is a Parseval frame for $\GE_m$ \cite[Example A]{HaLa1}. So we have 
\begin{equation}\label{FSGEmPE}
\FS(\GH_m)\otimes P^E=\FS(\GE_m).
\end{equation}
Since $\FS(\GE_m)(x)$ belongs to $\Bi(\GE_m)\subset\Bi(\GH_\N\otimes\GE_0)$ for each $x\in\M$ we have that 
$$
n_m\omega(\FS(\GE_m))=n_m\omega(\FS(\GH_m)\otimes P^E)=\breve{\varsigma}^{(m)}(P^E)
$$
is naturally an operator on $\GE_m$. Moreover, $\omega(\FS(\GE_m))$ is invertible on $\GE_m$ since the elements of $\boldsymbol\psi$ span $\GE_m$ by definition. Thus $\GH^E_m:=\Ran\breve{\varsigma}^{(m)}(P^E)=\GE_m$ for all $m$. %provided that we identify $\GE_m$ with a subspace of $\GH_m\otimes\GE_0$ via the multiplication map. 

The multiplication map $\GH_\N\otimes\GE_0\to\GE_\N$ makes $\GE_\N$ a graded quotient module, and this is the same as the action of $\Ai$ on $\GH^E_\N$ coming from compression of the shift on $\GH_\N\otimes\GE_0$.

Let $\tilde{\GE}_\N$ denote the graded vector space underlying $\GE_\N$ but endowed with the inner product of $L^2(\Sb,\omega;P^E)=\bigoplus_{k\in\Z}L^2(\omega;\FS(\GH_1)^{\otimes k}\otimes P^E)$. The actions of the generators $Z_1,\dots,Z_n$ of $C^0(\Sb)$ are the same on $\tilde{\GE}_\N\subset L^2(\Sb,\omega;P^E)$ as on $\GE_\N$, given just by multiplication on sections of $\Ei$ over $\Sb$ (the adjoints of the $Z_\alpha$'s act differently on $\tilde{\GE}_\N$ compared to $\GE_\N$ however). 

Therefore $\tilde{\GE}_\N$ is an $\Ai$-invariant subspace of $L^2(\Sb,\omega;P^E)$ and the multiplication tuple on $L^2(\Sb,\omega;P^E)$ is a normal extension of the shift tuple on $\tilde{\GE}_\N$. From the fact that $\GE_m$ generates $\Gamma^0(\M;\Ei(m))$ a $C^0(\M)$-module for each $m$ we obtain moreover that the normal extension is the minimal one.

The operator $c_{E,m}^{-1}\breve{\varsigma}^{(m)}(P^E)\in\Bi(\GE_m)$ is the Gram matrix of $\boldsymbol\psi$ as frame for $\tilde{\GE}_m$. This means that if we use the analysis operator of the frame $\boldsymbol\psi$ to identify $\GE_m$ and $\tilde{\GE}_m$ as vector spaces, the operator $c_{E,m}^{-1}\breve{\varsigma}^{(m)}(P^E)$ is the frame operator of $\boldsymbol\psi$ as frame for $\tilde{\GE}_m$ (see \cite[Example 3.1.1]{Bala1}). Therefore $c_{E,m}^{-1}\breve{\varsigma}^{(m)}(P^E)^{-1/2}\boldsymbol\psi$ will be a Parseval frame for $\tilde{\GE}_m$. We know that $\boldsymbol\psi$ is a Parseval frame for $\GE_m$. So the inner product on $\tilde{\GE}_m$ is obtained from that of $\GE_m$ by applying $\breve{\varsigma}^{(m)}(P^E)^{-1/2}$. This gives $\tilde{\GE}_m=\GK^E_m$. 

The assumption that $P^E$ is continuous and $\varsigma(\Ran\breve{\varsigma}(P^E))=P^E$ implies that $A_{T_E}$ is Fredholm. So the Hilbert space $\GK^E_\N$ is well-defined up to finite-dimensional subspaces. 

From Theorem \ref{bigCDversuslocallyfree} we have that $\GH^E_\N$ endows $\Ei$ with a canonical holomorphic structure. %Indeed, by assumption $\E$ has a $\D^\times$-equivariant topological structure defined by the pullback of $P^E$ and Theorem \ref{bigCDversuslocallyfree} gives that this structure is in fact analytic. 
The Cowen--Douglas projection $\CD(\GE_\N)=\CD(\GH_\N)\otimes P^E$ identifies $\GH^E_\N$ with a space of holomorphic sections of the Cowen--Douglas bundle. Since $\GK^E_\N$ is the same vector space as $\GH^E_\N$, the elements of $\GK^E_\N$ are also holomorphic sections. 
The proof is thus complete. 
 \end{proof}

\subsection{Hidden Szegö expansion}

When $P^E$ is a projection over $C^\infty(\M)$ we discussed the condition that $\Ran\breve{\varsigma}(P^E)$ is a lift of $P^E$. That is, the condition that $\Ran\breve{\varsigma}(P^E)$ and $\breve{\varsigma}(P^E)$ differ by a compact operator. Let us now investigate the geometric meaning of this compact operator which is the obstruction to the idempotency of the Toeplitz operator $\breve{\varsigma}(P^E)$.

\subsubsection{Geometric meaning of $[S^*_E,S_E]$}

Let $\GE_\N\subset\GH_\N\otimes\GE_0$ be a quotient module and suppose that the projection $P_E$ onto $\GE_\N$ has entries in $\Ti_\GH^{(0)}$. Then $P^E:=\varsigma(P_E)$ defines a holomorphic vector bundle $\Ei$ over $\M$ with Hilbert polynomial given by $\chi(\Ei(m))=\dim\GE_m$ for large enough $m$ (see \S\ref{dimransec}).

%Then $P^E:=\varsigma(P_E)$ defines a $C^0$ vector bundle $\Ei$ over $\M$ which is isomorphic to the Serre sheaf of $\GE_\N$. Therefore its Hilbert polynomial is given by $\chi(\Ei(m))=\dim\GE_m$ for large enough $m$. No, actually we only know that the pullback of $\Ei$ to $\B\setminus\{0\}$ is isomorphic to the Serre sheaf of $\GE_\N$.

%Using $\phi_m\circ\jmath_{l,m}=\phi_l$, or $\phi_m\circ\Psi^{l-m}=\phi_l$, we get from Hirzebruch--Riemann--Roch that
Using $\phi_m\circ\Psi^{l-m}=\phi_l$ we get from Hirzebruch--Riemann--Roch that
\begin{align}\label{phionidminusPsi}
(\Tr\otimes\Tr_{\GE_0})((\id-\Phi_*)(P_{E,m}))&=\chi(\Ei(m))-\chi(\Ei(m+1))\nonumber
\\&=m^{d-1}\int_\M\big(\tr_\omega\Theta^E+s_\omega/2\big)e^\omega+O(m^{d-2}),
\end{align}
where $\Theta^E$ is the curvature 2-form of the Chern connection of the metric $P^E$ and $\tr_\omega\Theta^E$ is its trace against the Kähler 2-form $\omega$. Using the coinvariance of $P_E$ we can write
\begin{equation}\label{idminusPsiPE}
(\id-\Phi_*)(P_E)=\sum^n_{\alpha=1}[P_E,S_\alpha]^*[P_E,S_\alpha]=[S_E^*,S_E]+[S^*,S].
\end{equation}
We saw in \S\ref{spexpsec} that 
$$
\phi_m([S^*,S])=\frac{m^{-1}}{\vol(\M,\Li)}\int_\M(s_\omega/2)e^\omega+O(m^{-2}),
$$
so from \eqref{phionidminusPsi} and \eqref{idminusPsiPE} we obtain
\begin{equation}\label{phionEcomm}
\phi_m([S^*_E,S_E])=\frac{m^{-1}}{\vol(\M,\Li)}\int_\M\tr_\omega\Theta^Ee^\omega+O(m^{-2}),
\end{equation}

When we discussed the spherical expansion $S$ in \S\ref{spexpsec} we observed that $\varsigma^{(m)}(m[S^*,S]p_m)$ approximates the (constant) scalar curvature. Now \eqref{phionEcomm} and the relation $\phi_m=\omega\circ\varsigma^{(m)}$ suggest that $\varsigma^{(m)}(m[S^*_E,S_E]p_m)$ should play the role of $\tr_\omega\Theta^E$. 

Since $\breve{\varsigma}(P^E)$ is the limit of the sequence $(\Psi^p(P_E))_{p\in\N_0}$, and since
$$
\Psi^p=\sum^p_{q=0}{p\choose q}(\id-\Psi)^q,
$$
we see that the operator in \eqref{idminusPsiPE} is a kind of first-order approximation to the compact operator $C_E:=P_E-\breve{\varsigma}(P^E)$. In the next subsection we discuss the geometric meaning of $C_E$, from which one could hint a relation between $[S_E^*,S_E]$ and $\tr_\omega\Theta^E$.

%When $\Ei$ is endowed with a holomorphic structure, the highest-order coefficient of the Szegö expansion for the Hardy space of the pullback of $\Ei$ to $\Sb$ is $\tr_\omega\Theta^E-\mu(\Ei)\bone_E$, where $\Theta^E$ is the Chern curvature of the metric $P^E$. However, $\GE_m$ does not coincide with $H^0(\omega,\FS(\GH_m)\otimes P^E)$ unless $P^E$ is $\G$-equivariant. Therefore the relation between the usual and the hidden Szegö expansions is not obvious. As far as we know there is no expansion on the classical side for the present quantization $\GE_\bullet$, and therefore we can only interpret $[S_E^*,S_E]$ as a quantum version of $\tr_\omega\Theta^E$. 
%We shall see in the next subsection that there is a clearer relation between $[S_E^*,S_E]$ and $\tr_\omega\Theta^E$. 

\subsubsection{Interpretation of the hidden Szegö expansion}

Let $\Ei$ be a globally generated holomorphic vector bundle over $\M$ and let $P^E$ be a Hermitian metric on $\Ei$. For each $m\in\N_0$ one can consider the ``Szegö endomorphism'' of the Hilbert space $H^0(\omega,\FS(\GH_m)\otimes P^E)$, which is the $C^0(\M)$-linear endomorphism $\Sigma^{E(m)}$ of $\Gamma^0(\M;\Ei(m),\FS(\GH_m)\otimes P^E)$ defined as follows. If $\boldsymbol\psi=(\psi_j)_{j=1,\dots,n^E_m}$ is an orthonormal basis for the Hilbert space $H^0(\omega,\FS(\GH_m)\otimes P^E)$ then one can view it as a sequence of elements in $\Gamma^0(\M;\Ei(m))$, and this sequence $\boldsymbol\psi$ is a $C^*$-frame for the Hilbert $C^0(\M)$-module $\Gamma^0(\M;\Ei(m),\FS(\GH_m)\otimes P^E)$ since we assume that $\Ei(m)$ is globally generated. The Szegö endomorphism $\Pi_m^E$ is defined to be the frame operator of the $C^*$-frame $\boldsymbol\psi$,
$$
\Sigma^{E(m)}\phi:=\sum^{n^E_m}_{j=1}(\psi_j|\phi)_{\FS(\GH_m)\otimes P^E}\psi_j,\qquad\forall \phi\in\Gamma^0(\M;\Ei(m)),
$$
where $(\psi_j(x)|\phi(x))_{\FS(\GH_m)\otimes P^E}(x)=(\psi_j(x)|\phi(x))_{\FS(\GH_m)\otimes P^E}$ is the inner product on the fiber $\Ei(x)$ obtained from the projection $\FS(\GH_m)\otimes P^E$. 

In \S\ref{balasec} we will discuss the notion of ``balanced'' metrics. By definition, the Hermitian metric $\FS(\GH_m)\otimes P^E$ is $\omega$-balanced if $\Sigma^{E(m)}$ is the identity endomorphism. That is, if an orthonormal basis for the Hilbert space $H^0(\omega,\FS(\GH_m)\otimes P^E)$ is at the same time a Parseval $C^*$-frame for the Hilbert $C^0(\M)$-module $\Gamma^0(\M;\Ei(m),\FS(\GH_m)\otimes P^E)$. 
The \textbf{Szegö expansion} of $P^E$ (and the reference metrics $P^L$ and $\omega$) is a large-$m$ expansion of $\Sigma^{E(m)}$ which shows how the obstruction to $\FS(\GH_m)\otimes P^E$ being balanced disappears as $m$ grows (see e.g. \cite{MaMa3}). 

We can also go in the opposite direction: we can start with a Parseval $C^*$-frame $\boldsymbol\psi$ for the Hilbert module $\Gamma^0(\M;\Ei(m),\FS(\GH_m)\otimes P^E)$ given by a basis for the vector space $H^0(\M;\Ei(m))$ and we can ask whether it is an orthonormal basis for $H^0(\omega,\FS(\GH_m)\otimes P^E)$ or not. The obstruction is the frame operator of $\boldsymbol\psi$ regarded as a sequence of vectors in $H^0(\omega,\FS(\GH_m)\otimes P^E)$. But the same information is contained in the Gram matrix of $\boldsymbol\psi$ regarded as a sequence of vectors in $H^0(\omega,\FS(\GH_m)\otimes P^E)$. More precisely, from \cite{Bala1} we have a $*$-algebra monomorphism from $\Bi(H^0(\omega,\FS(\GH_m)\otimes P^E))$ to $\Bi(\GH_\N\otimes\C^N)$ sending the frame operator of $\boldsymbol\psi$ to the operator $c_{E,m}^{-1}\breve{\varsigma}^{(m)}(P^E)$. 
%up to a $*$-isomorphism on the algebra of endomorphisms of $H^0(\omega,\FS(\GH_m)\otimes P^E)$ we can identify $\breve{\varsigma}^{(m)}(P^E)$ with  (cf. \cite{Bala1}).
Now, this Gram matrix is precisely
$$
\frac{\chi(\Ei(m))}{\rank\Ei}\omega((\boldsymbol\psi|\boldsymbol\psi)_{\FS(\GH_m)\otimes P^E})=c_{E,m}^{-1}\breve{\varsigma}^{(m)}(P^E).
$$
Thus a large-$m$ expansion of the operators $c_{E,m}^{-1}\breve{\varsigma}^{(m)}(P^E)$ will go under the name ``hidden Szegö expansion'' (``hidden'', as it has not been studied so far).

\subsubsection{$\varsigma^{E(m)}(A_{E,m}^{-1})$ expansion}
Under our assumption that $P_E$ is a superharmonic projection over $\Ti_\GH^{(0)}$, %$P_E$ equals $\breve{\varsigma}(P^E)$ modulo compacts. So 
the restriction $A_{T_E}$ of $\breve{\varsigma}(P^E)$ to the range $\GE_\N$ of $P_E$ is Fredholm. For present purposes we may assume $A_{T_E}$ is invertible. If we regard $A_{T_E}^{1/2}$ as a unitary operator
$$
A_E^{1/2}:\GE_\N\to\GK^E_\N
$$
onto the Hardy space $\GK^E_\N:=H^0(\Sb,\omega;\Ei,P^E)$ then the 1-isometry $V_E=A_E^{1/2}T_EA_E^{-1/2}$ is precisely the multiplication tuple on $\GK^E_\N$ (recall Theorem \ref{backtoHardythmgen}). Each endomorphism $f$ of $\Ei$ gives rise to an endomorphism of the pullback of $\Ei$ to $\Sb$, and then to a grading-preserving multiplication operator on $L^2(\Sb,\omega;\Ei,P^E)=\bigoplus_{k\in\Z}L^2(\M,\omega;\Ei(m),\FS(\GH_1)^{\otimes k}\otimes P^E)$. Restricting such multiplication operators to the Hardy space and following them by compression back to $\GK^E_\N$ gives us grading-preserving Toeplitz operators, which we denote by
$$
\breve{\varsigma}_{V_E}(f)=\sum_{m\in\N_0}\breve{\varsigma}^{(m)}_{V_E}(f)\in\Bi(\GK^E_\N)
$$
with $f\in\End\Gamma^0(\M;\Ei,P^E)$. Here $\breve{\varsigma}^{(m)}_{V_E}(f)$ is the component of $\breve{\varsigma}_{V_E}$ acting on the graded piece $\GK^E_m=H^0(\M,\omega;\Ei(m),\FS(\GH_m)\otimes P^E)$. 

\begin{Lemma}
The Toeplitz operators $\breve{\varsigma}_{V_E}(f)$ with $f\in\End\Gamma^0(\M;\Ei,P^E)$ are fixed-points of the unital map $\Psi_{V_E}(X):=\sum^n_{\alpha=1}V_{E,\alpha}^*XV_{E,\alpha}$ acting on $X\in\Bi(\GK_\N^E)$. Moreover, for each $f\in\End\Gamma^0(\M;\Ei,P^E)$ we have
\begin{equation}\label{VEToepl}
\breve{\varsigma}_{V_E}(f)=A_E^{-1/2}\breve{\varsigma}^{(m)}(f)A_E^{-1/2}.
\end{equation}
\end{Lemma}
\begin{proof}
Since $V_E$ is a commutative spherical isometry it is deduced as in \cite{Prun1} that the Toeplitz operators $\breve{\varsigma}_{V_E}(f)$ are the fixed points of the map $\Psi_{V_E}$. 

To deduce \eqref{VEToepl}, note that 
$$
\Psi_{V_E}(X):=\sum^n_{\alpha=1}V_{E,\alpha}^*XV_{E,\alpha}=A_E^{-1/2}\sum^n_{\alpha=1}T_{E,\alpha}^*A_E^{1/2}XA_E^{1/2}T_{E,\alpha}A_E^{-1/2}=A_E^{-1/2}\Psi_E(A_E^{1/2}XA_E^{1/2})A_E^{-1/2},
$$
where $\Psi_E(Y):=\sum^n_{\alpha=1}T_{E,\alpha}^*YT_{E,\alpha}$. An operator $Y$ is fixed under $\Psi_E$ if and only if $Y$ i a Toeplitz operator $\breve{\varsigma}^{(m)}(f)$ with symbol $f\in L^\infty(\M)\otimes\Bi(\GE_0)$ such that $Y$ is zero outside $\GE_\N$. So $X=A_E^{-1/2}\breve{\varsigma}^{(m)}(f)A_E^{-1/2}$ is a fixed point of $\Psi_{V_E}$ for each $f\in\End\Gamma^0(\M;\Ei,P^E)$. 
%and each fixed point of $\Psi_{V_E}$ arises like this since the $\breve{\varsigma}^{(m)}(f)$'s which are zero outside $\GE_\N$ are the fixed points of $\Psi_E$. 
\end{proof}

Define a state $\omega_E:\End\Gamma^0(\M;\Ei,P^E)\to\C$ by specifying it on rank-1 endomorphisms to be
$$
\omega_E(|\psi)(\phi|):=\frac{1}{\rank\Ei}\omega((\phi|\psi)_{P^E}),\qquad\forall \phi,\psi\in\Gamma^0(\M;\Ei),
$$
where $(\cdot|\cdot)_{P^E}$ is the $C^0(\M)$-valued inner product on $\Gamma^0(\M;\Ei,P^E)$ and $|\psi)(\phi|$ is the rank-1 endomorphism acting as $|\psi)(\phi|\psi':=(\phi|\psi')_{P^E}\psi$ for $\psi'\in\Gamma^0(\M;\Ei,P^E)$. 

Also, denote by $\phi_m^E:\Bi(\GK^E_m)\to\C$ the tracial state,
$$
\phi_m^E(X):=\frac{\Tr(X)}{\dim\GE_m}. %=\frac{n_m}{\chi(\Ei(m))}\phi_m(X). %=\frac{c_{E,m}}{\rank\Ei}\phi_m(X).
$$
Let $\varsigma_{V_E}^{(m)}$ be the adjoint of $\breve{\varsigma}^{(m)}_{V_E}$ with respect to $\phi_m^E$ and $\omega_E$.
That is, for $X\in\Bi(\GE_m)$ and $f\in\End\Gamma^0(\M;\Ei)$,
$$
\omega(\varsigma_{V_E}^{(m)}(X)f):=\phi_m^E(X\breve{\varsigma}_{V_E}^{(m)}(f)).
$$
From \eqref{VEToepl} we have
\begin{align*}
%\omega(\varsigma_{V_E}^{(m)}(X)f)&:=
\phi_m^E(X\breve{\varsigma}_{V_E}^{(m)}(f))&=\phi_m^E(A_E^{-1/2}XA_E^{-1/2}\breve{\varsigma}^{(m)}(f))
\\&=\frac{c_{E,m}}{\rank\Ei}\phi_m(A_E^{-1/2}XA_E^{-1/2}\breve{\varsigma}^{(m)}(f))
\\&=c_{E,m}\omega_E(\varsigma^{(m)}(A_E^{-1/2}XA_E^{-1/2})f)
\end{align*}
so that
\begin{equation}\label{VEvsTEsymb}
\varsigma_{V_E}^{(m)}(X)=c_{E,m}\varsigma^{(m)}(A_E^{-1/2}XA_E^{-1/2}).
\end{equation}
%\begin{Remark}
%We obtain $\varsigma^{(m)}_{V_E}(P_{E,m})$ from applying $\Phi_{V_E}$ infinitely many times on $P_{E,m}$, instead of using $\Phi_E$ as in $\varsigma^{E(m)}(P_{E,m})=P^E$,
%$$
%\varsigma^{(m)}_{V_E}(X)=(\Phi_{V_E}^r(X))_{r\in\N_0}+\Gamma_0,\qquad\forall X\in\Bi(\GE_m).
%$$
%Normally ordered $V_E$'s are not constant under $\Phi_{V_E}$. Instead the matrix elements $V_{E,\alpha}A_EV_{E,\beta}$ are constant under $\Phi_{V_E}$. And 
%$$
%V_EA_EV_E^*=A_E^{1/2}T_EA_E^{-1/2}A_EA_E^{-1/2}T_E^*A_E^{1/2}=A_E.
%$$
%This gives $\varsigma^{(m)}_{V_E}(A_{E,m})=P^E$. More generally, we have
%$$
%A_E^{1/2}\Phi_E^r(A_E^{-1/2}XA_E^{-1/2})A_E^{1/2}=\Phi_{V_E}^r(X).
%$$
%Using $A^{1/2}=\bone+\Gamma_0$ we obtain for all $X\in\Bi(\GE_m)$ that
%\begin{align}\label{VEvsTEsymb}
%\varsigma^{(m)}_{V_E}(X)&=A_E^{1/2}(\Phi_E^r(A_E^{-1/2}XA_E^{-1/2}))_{r\in\N_0}A_E^{1/2}+\Gamma_0\nonumber
%\\&=(\Phi_E^r(A_{E,m}^{-1/2}XA_{E,m}^{-1/2}))_{r\in\N_0}+\Gamma_0\nonumber
%\\&=\varsigma^{E(m)}(A_{E,m}^{-1/2}XA_{E,m}^{-1/2}).
%\end{align}
%Weird: Should we not have a factor $c_{E,m}$ somewhere?
%\end{Remark}

Finally, let $\Theta^E$ be the curvature of the Chern connection of the metric $P^E$ on the holomorphic vector bundle $\Ei$, and let $\Delta^E_\omega$ be the Bochner Laplacian acting on $\End\Gamma^\infty(\M;\Ei)$ (see \cite[Eq. (1.3.19)]{MaMa3}). 
\begin{Lemma}\label{vectorexplemma}
Let $\breve{\varsigma}_{V_E}^{(m)}$ be the $m$th component of the Toeplitz map on $H^0(\Sb,\omega;\Ei,P^E)$ and let  $\varsigma_{V_E}^{(m)}$ be its adjoint with respect to $\phi_m^E$ and $\omega_E$. Then for all $f\in\End\Gamma^0(\M;\Ei)$ we have the Berezin transform
$$
\varsigma_{V_E}^{(m)}(\breve{\varsigma}_{V_E}^{(m)}(f))=f+m^{-1}(\{\tr_\omega\Theta^E,f\}/2-\mu(\Ei)f+\Delta^E_\omega f)+O(m^{-2})
$$
where $\{f,g\}:=fg+gf$. In particular,
$$
\varsigma_{V_E}^{(m)}(\breve{\varsigma}_{V_E}^{(m)}(\bone_E))=\bone_E+m^{-1}(\tr_\omega\Theta^E-\mu(\Ei)\bone_E)+O(m^{-2}).
$$
\end{Lemma}
If the scalar curvature $s_\omega$ were not constant we would have
$$
\varsigma_{V_E}^{(m)}(\breve{\varsigma}_{V_E}^{(m)}(\bone_E))=\bone_E+m^{-1}(\tr_\omega\Theta^E-\mu(\Ei)-s_\omega/2+\underline{s}_\omega)\bone_E+O(m^{-2})
$$
where $\underline{s}_\omega:=\omega(s_\omega)$.
\begin{proof}
We are going to show that 
$$
\Sigma^{E(m)}:=c_{E,m}^{-1}\varsigma_{V_E}^{(m)}(\breve{\varsigma}_{V_E}^{(m)}(\bone_E))
$$
is the Szegö kernel of $\GK^E_m:=H^0(\M,\omega;\Ei(m),\FS(\GH_m)\otimes P^E)$. Then we will obtain the desired expansion from a rescaling of that in \cite[Thm. 5.2]{Wang2} %\cite[\S3.1]{KMS1}.

Let $(\psi_j)_{j\in\J}$ be an orthonormal basis for $\GK^E_m$ and let
$$
\FS^\dagger(\GK^E_m)\in\End\Gamma^0(\M;\Ei(m),\FS(\GK^E_m))\otimes\Bi(\GK^E_m)
$$
be the matrix whose $(j,k)$th entry is the rank-1 endomorphism $|\psi_j)(\psi_k|$. Using the isomorphism $\End\Gamma^0(\M;\Ei(m))\cong\End\Gamma^0(\M;\Ei)$ we shall always regard $\FS^\dagger(\GK^E_m)$ as an element of $\End\Gamma^0(\M;\Ei)\otimes\Bi(\GK^E_m)$. We want to prove the formula
\begin{equation}\label{varsigmaVE}
\varsigma_{V_E}^{(m)}(X)=c_{E,m}(\id\otimes\Tr)((\bone\otimes X)\FS^\dagger(\GK^E_m)).
\end{equation}
Then $c_{E,m}^{-1}\varsigma_{V_E}^{(m)}(P_{E,m})$ would be the sum of the diagonal elements of $\FS^\dagger(\GK^E_m)$, which is indeed the Szegö kernel of $\GK^E_m$.

%Let $\tilde{A}_{E,m}:=c_{E,m}A_{E,m}$. 
Since $A_{E,m}^{-1/2}\GE_m=\GK^E_m$ we have
$$
A_{E,m}^{-1/2}\FS^\dagger(\GE_m)A_{E,m}^{-1/2}=\sum_j|A_{E,m}^{-1/2}\phi_j)(A_{E,m}^{-1/2}\phi_j|=\sum_j|\psi_j)(\psi_j|
$$
where $(\phi_j)_{j\in\J}$ is an orthonormal basis for $\GE_m$, where $(\psi_j)_{j\in\J}=(A_{E,m}^{-1/2}\phi_j)_{j\in\J}$ is an orthonormal basis for $\GK^E_m$ and where $|\psi_j)(\psi_j|$ is the associated rank-1 operator on $\Gamma^0(\M;\Ei(m),\FS(\GH_m)\otimes P^E)$. That is,
$$
A_{E,m}^{-1/2}\FS^\dagger(\GE_m)A_{E,m}^{-1/2}=\FS^\dagger(\GK^E_m).
$$
From \eqref{VEvsTEsymb} we have
\begin{align*}
\varsigma_{V_E}^{(m)}(X)&=c_{E,m}\varsigma^{(m)}(A_{E,m}^{-1/2}XA_{E,m}^{-1/2})=c_{E,m}(\id\otimes\Tr)((\bone\otimes A_{E,m}^{-1/2}XA_{E,m}^{-1/2})\FS^\dagger(\GE_m))
\\&=c_{E,m}(\id\otimes\Tr)((\bone\otimes X)A_{E,m}^{-1/2}\FS^\dagger(\GE_m)A_{E,m}^{-1/2}))
\end{align*}
and so \eqref{varsigmaVE} holds. %The fact that $A_{E,m}^{-1/2}\FS(\GE_m)A_{E,m}^{-1/2}$ is not idempotent reflects precisely the nonunitality of $\varsigma_{V_E}^{(m)}$ or equivalently it reflects the unbalance of $\FS(\GH_m)\otimes P^E$.

From \cite[Thm. 5.2]{Wang2} we have 
$$
\frac{\chi(\Ei(m))}{\vol(\M,\Li)\rank\Ei}\Sigma^{E(m)}=m^d\bone_E+m^{d-1}(\tr_\omega\Theta^E-s_\omega/2)\bone_E+O(m^{d-2}),
$$
from which we obtain the result for $f=\bone_E$ using
$$
\frac{\chi(\Ei(m))}{\vol(\M,\Li)\rank\Ei}=m^d+m^{d-1}(\mu(\Ei)-\underline{s}_\omega/2)+\cdots.
$$
The same argument gives that the expansion of $\varsigma_{V_E}^{(m)}(\breve{\varsigma}_{V_E}^{(m)}(f))$ for general $f$ follows from the expansion in \cite[Prop. 3.6]{KMS1}. %(\textbf{cite original work instead}).

\end{proof}
%\begin{Remark}[Details]
%The above rescaling of the usual Szegö expansion follows from 
%\begin{align*}
%\frac{m^d+a_1m^{d-1}+O(m^{d-2})}{m^d+b_1m^{d-1}+O(m^{d-2})}&=\frac{1+a_1m^{-1}+O(m^{-2})}{1+b_1m^{-1}+O(m^{-2})}
%\\&=(1+a_1m^{-1})\sum_{p=0}(-b_1m^{-1})^p+O(m^{-2})
%\\&=1+(a_1-b_1)m^{-1}+O(m^{-2}).
%\end{align*}
%\end{Remark}

%Note that $\omega_E(\varsigma^{(m)}_{V_E}(P_{E,m}-A_{T_E,m}))=\omega_E(\varsigma^{(m)}(P_{E,m}-A_{T_E,m}))=0$. 
The following gives a geometric meaning to the compact operator $P_E-A_{T_E}$: %hidden Szegö expansion:
\begin{thm}\label{VEsymbprop}
%Let $A_{E,m}$ and $\tilde{A}_{E,m}=c_{E,m}A_{E,m}$ be the scaled and unscaled limit operators. Then
$$
\varsigma^{(m)}_{V_E}(P_{E,m}-A_{T_E,m})=m^{-1}(\tr_\omega\Theta^E-\mu(\Ei)\bone_E)+O(m^{-2}).
$$
%And with $\tilde{A}_{E,m}:=c_{E,m}A_{E,m}$ (should it not be $\tilde{A}_{E,m}:=c_{E,m}^{-1}A_{E,m}$?) we have
%$$
%\varsigma^{(m)}_{V_E}(P_{E,m}-\tilde{A}_{E,m})=-m^{-1}\tr_\omega\Theta^E+O(m^{-2}).%=m^{-1}(\tr_\omega\Theta^E-s_\omega/2)\bone_E+O(m^{-2}).
%$$
%So  while (for $m\gg 0$)
%$$
%\omega_E(\varsigma^{(m)}_{V_E}(P_{E,m}-\tilde{A}_{E,m}))=1-\frac{n_m\rank\Ei}{\chi(\Ei(m))}.
%$$
%In particular,
%$$
%\varsigma(N[V_E^*,V_E])=\lim_{m\to\infty}\varsigma^{(m)}_{V_E}(m[V_E^*,V_E])=\tr_\omega\Theta^E.
%$$
\end{thm}

\begin{proof}

Taking $X=P_{E,m}$ in \eqref{VEvsTEsymb} yields $\varsigma^{(m)}_{V_E}(P_{E,m})=\varsigma^{E(m)}(A_{T_E,m}^{-1})$ and $\varsigma^{(m)}_{V_E}(A_E)=\varsigma^{E(m)}(P_{E,m})=P^E$. We obtain
\begin{align*}
\varsigma^{(m)}_{V_E}(P_{E,m}-A_{T_E,m})&=\varsigma^{(m)}_{V_E}(P_{E,m})-P^E
\\&=\varsigma^{(m)}_{V_E}(\breve{\varsigma}^{(m)}_{V_E}(P^E)-P^E
\\&=m^{-1}\tr_\omega\Theta^E+O(m^{-2}).
\end{align*}
\end{proof}
%For $m\gg 0$ we calculate
%\begin{align*}
%1-\frac{n_m\rank\Ei}{\chi(\Ei(m))}&=1-c_{E,m}
%\\&=\frac{n_m}{\chi(\Ei(m))}(\chi(\Ei(m))/n_m-\rank\Ei)
%\\&=\frac{n_m}{\chi(\Ei(m))}\phi_m(P_{E,m}-\tilde{A}_{E,m})
%\\&=\phi_m^E(P_{E,m}-\tilde{A}_{E,m})
%\\&=\omega_E(\varsigma^{(m)}_{V_E}(P_{E,m}-\tilde{A}_{E,m})).
%\end{align*}
%\begin{cor}
%$$
%\varsigma^{(m)}_{V_E}(\breve{\varsigma}^{(m)}_{V_E}(f))=\varsigma^{E(m)}(A_{E,m}^{-1}\breve{\varsigma}^{E(m)}(f)A_{E,m}^{-1})
%$$
%so
%$$
%\varsigma^{(m)}_{V_E}(\breve{\varsigma}^{(m)}_{V_E}(\bone_E))=\varsigma^{E(m)}(A_{E,m}^{-2}).
%$$
%\end{cor}
%\begin{cor}\label{Szegoscalarcor}
%The positive operator $P_{E,m}-\tilde{A}_{E,m}$ is a scalar multiple of the identity modulo $O(m^{-2})$ if and only if $\tr_\omega\Theta^E=\mu(\Ei)\bone_E$.
%\end{cor}
%\begin{proof}
%Since $\varsigma^{(m)}_{V_E}$ is unital and injective, the endomorphism $\varsigma^{(m)}_{V_E}(X)$ is a scalar multiple of $\bone_E$ iff $X$ is a scalar multiple of the identity operator $P_{E,m}$. Theorem \ref{VEsymbprop} now gives the result. 
%\end{proof}

\section{Lifts of Yang--Mills metrics}

\subsection{Balanced metrics}\label{balasec}
The notion of balanced metrics on holomorphic vector bundles was introduced in \cite{Wang1}. In this section we shall reformulate it in terms of Hilbert modules and frames. %First a preliminary remark:

Reall from \S\ref{versussec} that if $\GE_m$ is an inner product on $H^0(\M;\Ei(m))$ then $\FS(\GE_m)$ denotes the projection over $C^\infty(\M)$ with a Parseval $C^*$-frame consisting of an orthonormal basis for $\GE_m$. 
\begin{dfn}\label{dfnofbalance}
Let $\Ei$ be a holomorphic vector bundle over $\M$ with $\Aut(\Ei)=\C\bone_E$. Let $m\in\N$ be large enough so that $\dim H^0(\M;\Ei)=\chi(\Ei(m))$. A Hermitian metric $B^{E(m)}\in C^0(\M)\otimes\Bi(H^0(\omega,B^{E(m)}))$ on $\Ei(m)$ is $\omega$-\textbf{balanced} if 
$$
B^{E(m)}=\FS(H^0(\omega,B^{E(m)})).
$$
In other words, $B^{E(m)}$ is $\omega$-balanced if
\begin{enumerate}[(i)]
\item{$B^{E(m)}$ has a Parseval $C^*$-frame consisting of a basis for the vector space $H^0(\M;\Ei(m))$ and}
\item{if $P_{E,m}$ denotes the identity operator on the vector space $H^0(\M;\Ei(m))$ then
\begin{equation}\label{balanceeq}
\omega(B^{E(m)})=\frac{\rank\Ei}{\chi(\Ei(m))}P_{E,m}.
\end{equation}}
\end{enumerate}
If $\Aut(\Ei)$ is nontrivial then a Hermitian metric $B^{E(m)}$ on $\Ei(m)$ is \textbf{weakly $\omega$-balanced} if $B^{E(m)}$ is the direct sum of $\omega$-balanced metrics on the summands in some decomposition of $\Ei(m)$ into a direct sum of simple holomorphic vector bundles. If the summands have the same reduced Hilbert polynomials then $B^{E(m)}$ is $\omega$-\textbf{balanced}.
\end{dfn}
In other words, $B^{E(m)}$ is weakly $\omega$-balanced if it has a Parseval $C^*$-frame consisting of an orthonormal basis for the Hilbert space $\GK^E_m:=H^0(\M,\omega;\Ei(m),B^{E(m)})$: in symbols $B^{E(m)}=\FS(H^0(\omega,B^{E(m)}))$. The defining properties (i) and (ii) of an $\omega$-balanced metric are thus the same no matter what $\Aut(\Ei)$ is, since (ii) requires the summands in any decomposition of $\Ei$ to have the same reduced Hilbert polynomial, viz. $\chi(\Ei(m))/\rank\Ei$. 

By Remark \ref{frameomegaremark}, our notion of balanced metric coincides with that of \cite{Wang1}. Indeed, (i) in Definition \ref{dfnofbalance} says that $B^{E(m)}=\FS(\GK^E_m)$ for some inner product $\GK^E_m$ on $H^0(\M;\Ei(m))$ while (ii) says that $\GK^E_m=H^0(\omega,B^{E(m)})$. 

When $B^{E(m)}$ is balanced we also say that the inner product (or Hilbert space) $H^0(\omega,B^{E(m)})$ is balanced. Thus, an inner product $\GE_m$ on $H^0(\M;\Ei(m))$ is balanced iff the $L^2$-inner product of the Hermitian metric $\FS(\GE_m)$ coincides with $\GE_m$. 

\begin{Remark}
A $\G$-equivariant vector bundle need not admit an $\omega$-balanced metric, since the irreducible summands need not have the same reduced Hilbert polynomials, but it always admis a weakly $\omega$-balanced metric in the above sense (see \cite{Moss4}).  
\end{Remark}

\begin{Remark}[The normalization constant]
Since $B^{E(m)}$ must have fiber trace equal to $\rank\Ei$, the only inner product on $H^0(\M;\Ei(m))$ which can support a Parseval $C^*$-frame for $B^{E(m)}$ as orthonormal basis is the one in which the natural $L^2(B^{E(m)},\omega)$ inner product. If this inner product is scaled by $\chi(\Ei(m))/\rank\Ei$ then the frame bound is scaled by $\chi(\Ei(m))/\rank\Ei$ and the balancedness condition becomes to existence of a tight frame with frame constant $\chi(\Ei(m))/\rank\Ei$. 

\end{Remark}
The balance condition reads $B^{E(m)}=\FS(H^0(\omega,B^{E(m)}))$.
Recall from \S\ref{versussec} that this means that there is a basis $(\psi_j)_{j\in\J}$ for the vector space $H^0(\M;\Ei(m))$ such that
$$
B^{E(m)}(x)=\sum_{j,k\in\J}\bra\psi_j(x)|\psi_k(x)\ket_{H^0(\omega,B^{E(m)})}|\psi_j\ket\bra\psi_k|\in\End(H^0(\M;\Ei(m))).
$$
If we have either a surjection $\Ai\otimes\C^N\to E_\N\to 0$ or an embedding $E_\N\hookrightarrow\Ai\otimes\C^N$ of graded $\Ai$-modules then we could regard $B^{E(m)}(x)$ as an element
$$
B^{E(m)}(x)\in\End(\Ai_m\otimes\C^N)=\Bi(\GH_m\otimes\C^N).
$$
The assumption that the surjection (or embedding) is a map of graded $\Ai$-modules ensures that we still have $B^{E(m)}=\FS(H^0(\omega,B^{E(m)}))$ for the same holomorphic vector bundle $\Ei(m)$, up to holomorphic isomorphism. 

%We have $\psi(x)=\FS(H^0(\omega,B^{E(m)}))(x)\psi$ for $\psi\in\GE_m$ for any $\FS$-type metric so one really needs that 
%$$
%B^{E(m)}(x)=\sum_{j,k\in\J}\bra\psi_j(x)|\psi_k(x)\ket_{H^0(\omega,B^{E(m)})}|\psi_j\ket\bra\psi_k|\in\End(H^0(\M;\Ei(m)))
%$$
%with $(\psi_j)_{j\in\J}$ a basis for the vector space $E_m=H^0(\M;\Ei(m))$ (or more generally a Parseval frame for the Hilbert space $H^0(\omega,B^{E(m)})$). %So if we want to regard $B^{E(m)}(x)$ as an operator on $\GH_m\otimes\C^N$ by identifying $E_m$ with a subspace of $\Ai_m\otimes\C^N$ for some $N$ then this inclusion $E_m\subset\Ai_m\otimes\C^N$ has to be part of either a surjection $\Ai\otimes\C^N\to E_\N\to 0$ or an embedding $E_\N\hookrightarrow\Ai\otimes\C^N$ of graded $\Ai$-modules. Indeed, only then 

Let us now investigate the consequences of this representation of the balanced metrics, first in the case of a surjection $\Ai\otimes\C^N\to E_\N\to 0$. 

Given the holomorphic structure on $\Ei$ there is an $m_0\in\N_0$ such that $\Ei(m)$ is Castelnuovo--Mumford regular for all $m\geq m_0$. Thus, the graded vector space $E_{\geq m_0}:=\bigoplus_{m\geq m_0}H^0(\M;\Ei(m))$ is a finitely generated graded $\Ai$-module and isomorphic to a graded quotient of $\Ai_{\geq m_0}\otimes H^0(\M;\Ei(m_0))$. Idenfity $E_{\geq m_0}$ with a vector subspace of $\Ai_{\geq m_0}\otimes H^0(\M;\Ei(m_0))$. If we let $\GE_{m_0}$ be $H^0(\M;\Ei(m_0))$ endowed with some given inner product then the closure $\GE_{\geq m_0}$ of $E_{\geq m_0}$ in $\GH_{\geq m_0}\otimes\GE_{m_0}$ is a quotient module in the sense of Hilbert modules, if we endow it with the action given by compressing the $\Ai$-action on $\GH_{\geq m_0}\otimes\GE_{m_0}$, 
and the underlying graded $\Ai$-module is isomorphic to $E_{\geq m_0}$. 

Now $B^{E(m)}$ can be viewed as a positive idempotent element of $C^0(\M)\otimes\Bi(\GH_m\otimes\GE_{m_0})$. The operator $P_{E,m}$ in \eqref{balanceeq} is then the projection of $\GH_m\otimes\GE_{m_0}$ onto the Hilbert subspace $\GE_m$. Define a projection $B^E_m\in C^0(\M)\otimes\Bi(\GE_{m_0})$ by writing
$$
B^{E(m)}=\FS(\GH_m)\otimes B^E_m.
$$
Then $B^E_m$ defines $\Ei$ as smooth vector bundle for each $m$. 

%For $m\geq m_0$ we know that the tensor product of a Parseval frame of $\GH_{m-m_0}$ with one for $\GE_{m_0}$ is a Parseval frame for the Hilbert space $\GE_m$. Since $B^{E(m)}=\FS(\GE_m)$, this means that for $m\geq m_0$ we have a 

Using the explicit formula \eqref{Toeplexplic} for the Toeplitz map $\breve{\varsigma}^{(m)}$ we can rewrite \eqref{balanceeq} as  
%
%$$
%B^{E(m)}=B^E_m\otimes\FS(\GH_{m-m_0}).
%$$
%Then $B^E_m$ defines $\Ei(m_0)$ as smooth vector bundle. Recalling the definition of the Toeplitz map $\breve{\varsigma}^{(m)}$ we can rewrite \eqref{balanceeq} as  
\begin{equation}\label{balanceeqToepl}
\breve{\varsigma}^{(m)}(B^E_m)=c_{E,m}P_{E,m}
\end{equation}
with the constant 
$$
c_{E,m}:=\frac{n_m\rank\Ei}{\chi(\Ei(m))}. %=1+O(m^{-1}).
$$
If we have a balanced metric $B^{E(m)}=\FS(\GH_m)\otimes B^E_m$ on $\Ei(m)$ for each $m\geq m_0$ then 
$$
\sum_{m\geq m_0}c_{E,m}^{-1}\breve{\varsigma}^{(m)}(B^E_m)=P_E,
$$
where $P_E$ is the orthogonal projection of $\GH_{\geq m_0}\otimes\GE_{m_0}$ onto $\GE_{\geq m_0}$. 
%Here we are not really using that the $C^*$-frame for $B^{E(m)}$ consists of holomorphic sections. But one needs this correlation to have both the frame and the Gram matrix $P_{E,m}$ involve the same space $E_m$. Otherwise we are kind of attempting to make $\varsigma(P_E)$ define $\Ei$ using whatever $\bar{\pd}^E$ needed. 

%(\textbf{How do we know that we get exactly $P_E$; we just know that $c_{E,m}^{-1}\breve{\varsigma}^{(m)}(B^E_m)$ is a projection with the correct dimensions. Or is it actually part of the balance condition that we get $P_E$, since we have fixed $E_\N$ and we require the Parseval $C^*$-frames to be in the $E_m$'s?})

\subsection{Superharmonic lifts}

We have seen hat if $P^E=\FS(\GE_0)$ is a projection in $C^0(\M)\otimes\Mn_N(\C)$ which defines a holomorphic subbundle $\Ei\subset\Oi_\M\subset\C^N$ then there is a natural associated quotient module $\GE_\N$ with $\FS(\GH_m)\otimes P^E=\FS(\GE_m)$ for all $m$ (Proposition \ref{geomintprop}). If the projection $P_E$ onto $\GE_\N$ has continuous symbol $\varsigma(P_E)$ then we have shown that $\varsigma(P_E)=P^E$ (Theorem \ref{bigCDversuslocallyfree}). But we shall see later that $\varsigma(P_E)$ is rarely continuous (Theorem \ref{GiesSOTlimprop}). 

Also, as we discussed, for any $N\times N$-projection $P^E$ over $L^\infty(\M)$ the Toeplitz range $\Ran\breve{\varsigma}(P^E)$ is a projection onto a graded module quotient of $\GH_\N\otimes\C^N$. And $\varsigma(\Ran\breve{\varsigma}(P^E))$ is continuous when $P^E$ is. But even if $P^E$ is over $C^\infty(\M)$ and even if the vector bundle $\Ei$ defined by $P^E$ is holomorphic it cannot hold that $\Ran\breve{\varsigma}(P^E)$ lifts $P^E$ unless possibly if $\Ei$ is globally generated. 
However, there is a special class of Hermitian metrics which do have coinvariant lifts:
\begin{Lemma}\label{coinvYMlemma}
%Let $\Ei$ be a smooth vector bundle over $\M=\G/\K$. Suppose that $\Ei$ admits a slope-stable holomorphic structure $\bar{\pd}^E$, so that the holomorphic vector bundle $(\Ei,\bar{\pd}^E)$ admits a Yang--Mills metric $P^E$. Then the $\Psi$-superharmonic projection $\Ran\breve{\varsigma}(P^E)$ is a lift of $P^E$:
Let $\Ei$ be a slope-stable holomorphic vector bundle over $\M=\G/\K$, and let $P^E$ be the Yang--Mills metric on $\Ei$. Then the $\Psi$-superharmonic projection $\Ran\breve{\varsigma}(P^E)$ is a lift of $P^E$:
$$
\varsigma(\Ran\breve{\varsigma}(P^E))=P^E.  
$$
\end{Lemma}
\begin{proof}
%For the proof we may assume that $\Ei$ itself is slope-stable. 
From \cite{Wang2} we know that there exists a sequence of metrics $(B^E_m)_{m\in\N_0}$ on $\Ei$ such that $\FS(\GH_m)\otimes B^E_m$ is an $\omega$-balanced metric on $\Ei(m)$ for large enough $m$ and
$$
\lim_{m\to\infty}B^E_m=P^E
$$
in $C^0$ (actually in $C^\infty$). 

We can represent $E_\N$ as a quotient of $\Ai\otimes\GE_0$ which is a graded quotient up to a finite-dimensional subspace. Therefore we can embed $E_\N$ as a vector subspace of $\GH_\N\otimes\GE_0$ whose closure is coinvariant up to a finite-dimensional subspace. We may choose the dimension of $\GE_0$ large enough so that $P^E$ is an element of $C^0(\M)\otimes\Bi(\GE_0)$. The balanced metric $B^E_m\otimes\FS(\GH_m)$ on $\Ei(m)$ has a Parseval $C^*$-frame given by elements of $E_m$ and hence can be regarded as an element of $C^0(\M)\otimes\Bi(\GH_m\otimes\GE_0)$ as determined by our chosen embedding $E_m\subset\GH_m\otimes\GE_0$. The projection $B^E_m$ is an element of $C^0(\M)\otimes\Bi(\GE_0)$. The balance condition then says that $c_{E,m}^{-1}\breve{\varsigma}^{(m)}(B^E_m)$ equals the projection $P_{E,m}$ of $\GH_m\otimes\GE_0$ onto the embedded copy of $E_m$. 

Now $\|\breve{\varsigma}^{(m)}(f)\|\leq\|f\|$ for all $f\in C^0(\M)$, so
$$
\lim_{m\to\infty}\|\breve{\varsigma}^{(m)}(B_m^E-P^E)\|\leq\lim_{m\to\infty}\|B_m^E-P^E\|=0,
$$
i.e. $\breve{\varsigma}(P^E)$ and $\sum_m\breve{\varsigma}^{(m)}(B^E_m)$ differ by a compact operator. In turn this gives that $\breve{\varsigma}(P^E)$ and $P_E=\sum_mc_{E,m}^{-1}\breve{\varsigma}^{(m)}(B^E_m)$ also differ by a compact, since $c_{E,m}^{-1}=1+O(m^{-1})$. So $P_E$ is a lift of $P^E$, i.e. $\varsigma(P_E)=P^E$. In particular $P_E$ is a projection over $\Ti_\GH^{(0)}$. We want to show that $P_E$ differs from the range projection of $\breve{\varsigma}(P^E)$ only by a finite-rank operator. But $P_E$ is $\Psi$-superharmonic (up to finite-rank operators), so $\varsigma(P_E)=P^E$ gives its $\Psi$-harmonic part as $\breve{\varsigma}(P^E)$, so we are done by the uniquness of superharmonic lifts (Proposition \ref{uniqeliftprop}).

%(up to the finite-dimensional error that $P_E$ is only coinvariant up to finite-rank operators since $E_\N$ is just coinvariant up to finite-dimensional subspaces). 
\end{proof}
%\begin{Remark}
%We have $\varsigma(P_E)=P^E$ and from this we conclude $P_E=\breve{\varsigma}(P^E)+C_E$ with $C_E$ compact, and this does not rely on $P_E$ being $\Psi$-superharmonic but only that $\Ti_\GH^{(0)}=\breve{\varsigma}(C^0(\M))+\Gamma_0$. There could be many projections $P_E$ over $\Ti_\GH^{(0)}$ of the form $P_E=\breve{\varsigma}(P^E)+C_E$ for our given $P^E$. However, any $\Psi$-superharmonic lift $P_E$ of $P^E$ is given by $\Ran\breve{\varsigma}(P^E)$, up to finite-rank operators.
%\end{Remark}
%\begin{Remark}
%When we say that $P^E$ is Yang--Mills with respect to $\bar{\pd}^E$ then we mean that the Chern curvature of $\bar{\pd}^E$ and $P^E$ satisfies the Yang--Mills equation. This does not require $\bar{\pd}^E$ to be the Levi-Civita dbar-operator of $P^E$.
%\end{Remark}

%We saw in the proof of Lemma \ref{coinvYMlemma} that $\GH^E_\N$ is (up to finite-dimensional subspaces) a completion of $E_\N$. From Theorem \ref{backtoHardythmgen} we then obtain:
\begin{thm}\label{YMCDthm}
Let $P^E$ be %as in Lemma \ref{coinvYMlemma},
a projection over $C^\infty(\M)$ defining a Hermitian Yang--Mills vector bundle $\Ei$ over $\M$ and set $\GH^E_\N:=\Ran\breve{\varsigma}(P^E)$. Then the Cowen--Douglas sheaf $\Ei_{\rm CD}$ of $\GH^E_\N$ is analytically isomorphic over $\B\setminus\{0\}$ to the pullback $\Ei_{\B\setminus\{0\}}$ of $\Ei$; as Hermitian holomorphic vector bundles we have
\begin{equation}\label{Hermfactor}
\Ei_{\rm CD}=\Oi_{\rm CD}\otimes\Ei_{\B\setminus\{0\}}
\end{equation}
where $\Ei_{\B\setminus\{0\}}$ is endowed with the Hermitian metric given by pullback of $P^E$. Moreover, the subnormal tuple on the space $\GK^E_\N:=\breve{\varsigma}(P^E)^{-1/2}\GH^E_\N$ is unitarily equivalent (up to finite-rank operators) to the multiplication tuple on the Hardy space $H^0(\Sb,\omega,\FS(\GH_\N)\otimes P^E)$. 
\end{thm}
\begin{proof}
We saw in the proof of Lemma \ref{coinvYMlemma} that $\GH^E_\N$ is (up to finite-dimensional subspaces) a completion of $E_\N$, and that the symbol $\varsigma(P_E)$ of the projection $P_E:=\Ran\breve{\varsigma}(P^E)$ onto $\GH^E_\N$ coincides with $P^E$. So $P_E$ is a projection over $\Ti_\GH^{(0)}$. Theorem \ref{bigCDversuslocallyfree} gives that the Cowen--Douglas metric on $\Ei_{\rm CD}$ is given by $\CD(\GE_\N)=\CD(\GH_\N)\otimes P^E$, where $P^E$ is pulled back to a $\D^\times$-equivariant function on $\B\times\{0\}$. This gives \eqref{Hermfactor} and in particular the Cowen--Douglas bundle of $\GH^E_\N$ is isomorphic to the pullback of $\Ei$. 

In the proof of Theorem \ref{backtoHardythmgen} we saw that $\GH^E_m$ identifies via $\FS(\GH_m)\otimes P^E$ with the $\C$-linear span $\GE_m$ of a $C^*$-frame $\boldsymbol\psi$ for $\FS(\GH_m)\otimes P^E$. The space $\GE_\N$ is thus a graded $\Ai$-module whose algebraic part is precisely $E_\N$. (We stress that $\FS(\GH_m)\otimes P^E$ acting on $\C\bone\otimes\otimes\GE_m\subset C^0(\M)\otimes\GE_m$ does not give holomorphic sections of $\Ei(m)$ but just some other $C^*$-frame for $\FS(\GH_m)\otimes P^E$. Thus, when $\FS(\GH_\N)\otimes P^E$ acts on $C^0(\Sb)\otimes\GE_0$ it does not projects $\GH^E_\N$ to the copy of $E_\N$ sitting in $\Gamma^0(\Sb;\Ei)$.) 

As in the proof of Theorem \ref{backtoHardythmgen} we denote by $\tilde{\GE}_m$ the space $\GE_m^E\subset\Gamma^0(\M;\Ei(m),\FS(\GH_m)\otimes P^E)$ endowed with the inner product of $L^2(\omega,\FS(\GH_m)\otimes P^E)$.
We saw that the subnormal tuple on the space $\GK^E_\N$ is unitarily equivalent (up to finite-rank operators) to the multiplication tuple on the subspace $\tilde{\GE}_\N$ of $L^2(\Sb,\omega;\FS(\GH_\N)\otimes P^E)$.
So we only need to show that multiplication tuples on $H^0(\Sb,\omega;\FS(\GH_\N)\otimes P^E)$ and $\tilde{\GE}_\N$ are unitarily equivalent. %(since we already know that $\tilde{\GE}_\N\cong\GK^E_\N$). 
But we now that the underlying graded $\Ai$-modules are isomorphic, and since $H^0(\Sb,\omega,\FS(\GH_\N)\otimes P^E)$ and $\tilde{\GE}_\N$ sits as Hilbert subspaces of the same Hilbert space we can extend this isomorphism to a unitary operator from $H^0(\Sb,\omega,\FS(\GH_\N)\otimes P^E)$ to $\tilde{\GE}_\N$.

%The graded $\Ai$-module underlying $\GE_\N=\GH^E_\N$ is just $E_\N$ since we used $E_\N$ to obtain the balanced metrics. The graded $\Ai$-module underlying $H^0(\Sb,\omega;\FS(\GH_\N)\otimes P^E)$ is isomorphic to $E_\N$ since it is defined by the same holomorphic structure. 

\end{proof}

\begin{cor}
Let $\Ei$ be a slope-stable vector bundle over $\M$, and denote its associated graded $\Ai$-module by $E_\N:=\bigoplus_{m\in\N_0}H^0(\M;\Ei(m))$. Represent $E_\N$ as a quotient $\Ai$-module where the surjection $\Ai\otimes\GE_0\to E_\N$ is grading-presering up to finite-dimensional vector spaces, and let $[E_\N]$ be its closure of $E_\N$ in $\GH_\N\otimes\GE_0$. Then the Cowen-Douglas sheaf $\Ei_{\rm CD}$ of $[E_\N]$ is locally free boundary limit of the Cowen--Douglas projection $\Pi^E$ for $[E_\N]$ is the Yang--Mills metric on $\Ei$.
\end{cor}

\begin{Remark}[Balanced metric versus $\CD(\GE_m)$]
Let $\Ei$ be a slope-stable holomorphic vector bundle over $\M=\G/\K$ and let $\GE_\N:=\Ran\breve{\varsigma}(P^E))$ be the quotient module from Lemma \ref{coinvYMlemma}. Recall that $\CD(\GE_m)=\CD(\GH_m)\otimes P^E_m$ is the limit of the powers of the positive operator $\CD(\GH_m)\otimes\varsigma^{(m)}(P_{E,m})$. If we use that same notation $B^E_m$ as in the proof of Lemma \ref{coinvYMlemma} for the metrics on $\Ei$ converging to the Yang--Mills metric $P^E$ then we have $c_{E,m}^{-1}\varsigma^{(m)}(\breve{\varsigma}^{(m)}(B^E_m))=\varsigma^{(m)}(P_{E,m})$ and hence 
\begin{equation}\label{Beretrasfopowlimit}
\lim_{p\to\infty}(c_{E,m}^{-1}\varsigma^{(m)}(\breve{\varsigma}^{(m)}(B^E_m)))^p=P^E_m.
\end{equation}
But $\varsigma^{(m)}(\breve{\varsigma}^{(m)}(B^E_m))$ is the Berezin transform of $B^E_m$, so we have $\varsigma^{(m)}(\breve{\varsigma}^{(m)}(B^E_m))=B^E_m+O(m^{-1})$. The constant $c_{E,m}^{-1}$ is of the form $1+O(m^{-1})$. Therefore
$$
c_{E,m}^{-1}\varsigma^{(m)}(\breve{\varsigma}^{(m)}(B^E_m))=B_m^E+K^E_m
$$
with $\lim_m\|K^E_m\|=0$. We also know that, for all $m\gg 0$,
$$
P^E_m=\lim_p\varsigma^{(m)}(P_{E,m})^p=P^E.
$$
The coinvariance of $P_E$ gives that, for $m\gg 0$,
\begin{equation}\label{varsigmadirectsum}
\varsigma^{(m)}(P_{E,m})=P^E\oplus C^E_m\text{  on }\Ran P^E\oplus\Ker P^E_m
\end{equation}
with $C^E_m\geq 0$. In contrast, $K^E_m$ need not be positive and need not be zero on the range of $B^E_m$ or the range of $P^E$. For large $m$ we know that $\|P^E-B_m^E\|<1$ in the norm of $C^0(\M)$ so that $\varsigma(P_E)(x)$ and $B^E_m(x)$ are unitarily equivalent uniformly in $x$ \cite[Prop. 5.2.6]{Ols}. The analytic embeddings of $\Ei(m)$ into $\Oi_\M\otimes\GE_0$ given by $\FS(\GH_m)\otimes P^E$ and $B^E_m$ are thus arbitrarily close as $m$ gets large but they need not coincide for any finite $m$. 
\end{Remark}

\subsection{Subharmonic lifts}

When the dual $\Ei^*$ of $\Ei$ is globally generated we have an embedding $\Ei\hookrightarrow\Oi_\M\otimes H^0(\M;\Ei^*)$ of holomorphic vector bundles, and hence an inclusion $E_\N\subset\Ai\otimes H^0(\M;\Ei^*)$ of graded $\Ai$-modules. Still there is an $m_0\in\N$ such that $\Ei(m)$ is regular for all $m\geq m_0$. So we can represent $E_{\geq m_0}$ as a graded quotient of $\Ai\otimes H^0(\M;\Ei(m_0))$ as before and we obtain a graded quotient $\GE_{\geq m_0}$ of $\GH_\N\otimes H^0(\M;\Ei(m_0))$ by identifying $E_{\geq m_0}$ with a vector subspace of $\Ai\otimes H^0(\M;\Ei(m_0))$ and completing it in the inner product of the Fock space. The coinvariant subspace $\GE_{\geq m_0}$ of $\GH_\N\otimes H^0(\M;\Ei(m_0))$ will have the same $\Ai$-action as the completion of $E_{\geq m_0}$ in $\Ai\otimes H^0(\M;\Ei^*)$ up to graded $\Ai$-module isomorphism but as Hilbert $\Ai$-modules they are very different: one is a quotient module and one is a submodule and this is a very important difference for Hilbert modules. For instance, if $P_E$ is the projection onto a quotient module then $(\id-\Phi)(P_E)$ is a finite-rank operator, while if $P_E$ projects out a submodule $(\id-\Phi)(P_E)$ is not of finite rank except in trivial cases \cite{Guo3}.

Suppose now that we have a balanced metric $B^{E(m)}=B^E_m\otimes\FS(\GH_{m-m_0})$ on $\Ei(m)$ for each $m\geq m_0$. 
%We have a natural Parseval $C^*$-frame for $B^{E(m)}$ given as the tensor product of the frames for $B^E_m$ and $\FS(\GH_{m-m_0})$. When we evaluate $B^{E(m)}$ at points we can replace the tensor product with pointwise multiplication and then have the $\Ai$-module structure on $E_\N$ coming into play. 
Since $[E_\N]\subset\GH_\N\otimes H^0(\M;\Ei^*)$ is an embedding of graded $\Ai$-modules we can, up to graded $\Ai$-module isomorphism (or equivalently without changing the isomorphism class of the holomorphic vector bundle $\Ei$) regard $B^{E(m)}$ as an element of $C^0(\M)\otimes\Bi(\GH_m\otimes H^0(\M;\Ei^*))$. The balance condition then becomes  
\begin{equation}\label{balanceeqToeplsubmod}
\breve{\varsigma}^{(m)}(B^E_m)=c_{E,m}I_{E,m},
\end{equation}
where $I_E=\sum_mI_{E,m}$ is the projection onto the graded submodule $[E_\N]\subset\GH_\N\otimes H^0(\M;\Ei^*)$. Note that $I_E$ is $\Psi$-subharmonic, in contrast to the $\Psi$-superharmonic $P_E$ obtained from the presentation of $E_\N$ as quotient module. 

We shall see that, in case the sequence $(B^E_m)_{m\geq m_0}$ converge, both $\varsigma(I_E)$ and $\varsigma(P_E)$ are Yang--Mills metrics on holomorphic vector bundles isomorphic to $\Ei(m_0)$, and for all practical purposes they coincide if we allow an analytic isomorphism to act on $\Ei$ so that they are metrics on the same vector bundle. 

%What about if $\Ei^*$ is not globally generated, how do we obtain a submodule? Well we only need it for our subsheaf $\Gi\subset\Ei$, that we obtain $0\to G_\N\to E_\N$ and we do not need global generation. 

%We could also use $m_0$ for an integer such that $\Ei(m)^*$ is globally generated for all $m\geq m_0$, and then $l_0\geq m_0$ for an integer such that $\Ei(m)$ is $l_0$-regular.

%More generally, instead of assuming that $\Ei^*$ is globally generated we can assume that $\Ei$ is a subbundle $\Ei\subset\Fi$ of some 0-regular holomorphic vector bundle $\Fi$. Then we obtain a $\Psi_F$-subharmonic projection $I_E$ acting on the quotient module $\GF_\N$, where $\Psi_F$ is defined as $\Psi$ but using the shift on $\GF_\N$ instead of that on $\GH_\N$. In the same way as in the proof of Lemma \ref{coinvYMlemma} we obtain that $I_E$ equals $\Ran\breve{\varsigma}(P^F)$ modulo compacts, so $\varsigma(I_E)=P^F$. %So even though $I_E$ is not a $\Psi$-subharmonic lift of $P^F$ it is at least $\Psi_E$-subharmonic and its image $\GE_\N$ is a coinvariant subspace of $\GH_\N\otimes\GF_0$ with Cowen--Douglas sheaf equal to $\Ei$. 
%Let us record these facts:
\begin{Lemma}\label{invYMlemma}
Let $\Gi$ be a smooth vector bundle over $\M=\G/\K$ admitting a slope-stable holomorphic structure $G_\N=\bigoplus_{m\in\N_0}H^0(\M;\Gi(m))$. Suppose that $\Gi$ is a subbundle $\Gi\subset\Ei$ of some Castelnuovo--Mumford regular holomorphic vector bundle $\Ei$ and choose an embedding $G_\N\subset E_\N$ as graded $\Ai$-module. Let $I_G$ be the projection onto the $S_E$-invariant subspace $[G_\N]\subset [E_\N]$. Then the symbol of the $\Psi_E$-subharmonic projection $I_G$ defines $\Gi$ as smooth vector bundle as well as a Hermitian metric on $\Gi$ which is Yang--Mills with respect to a holomorphic structure isomorphic to $G_\N$.   
\end{Lemma}
\begin{proof}
Given our embedding $G_\N\subset E_\N$ we obtain the projection $I_G$ onto $[G_\N]\subset[E_\N]\subset\GH_\N\otimes\GE_0$ as
\begin{equation}\label{balanceeqToeplsubmodG}
\breve{\varsigma}^{(m)}(B^G_m)=c_{E,m}I_{G,m},\qquad\forall m\gg 0
\end{equation}
for the unique $\omega$-balancing metrics $\FS(\GH_m)\otimes B^G_m$ on $\Gi(m)$ with respect to the holomorphic structure $G_\N\subset E_\N$. Let $P^G:=\lim_{m\to\infty}B^G_m$ be the associated Yang--Mills metric on $\Gi$. In the same way as in the proof of Lemma \ref{coinvYMlemma} we see that 
$$
\lim_{m\to\infty}\|\breve{\varsigma}^{(m)}(B_m^G-P^G)\|\leq\lim_{m\to\infty}\|B_m^G-P^G\|=0,
$$
i.e. $\breve{\varsigma}(P^G)$ and $\sum_m\breve{\varsigma}^{(m)}(B^G_m)$ differ by a compact operator. In turn this gives that $\breve{\varsigma}(P^G)$ and $I_G=\sum_mc_{G,m}^{-1}\breve{\varsigma}^{(m)}(B^G_m)$ differ by a compact, since $c_{G,m}-1=O(m^{-1})$. So $I_G$ is a lift of $P^G$, i.e. $\varsigma(I_G)=P^G$. In particular $I_G$ is a projection over $\Ti_\GH^{(0)}$.

%$I_G$ equals $\breve{\varsigma}(P^G)$ modulo compacts. So $\varsigma(P_G)=P^G$ and we are done.
\end{proof}
Note that the subspace $[G_\N]\subset[E_\N]\subset\GH_\N\otimes\GE_0$ is merely semi-invariant under the shift $S$ on $\GH_\N\otimes\GE_0$. 
\begin{cor}
Let $\Gi$ be a smooth vector bundle over $\M$ with a slope-stable holomorphic structure $G_\N=\bigoplus_{m\in\N_0}H^0(\M;\Gi(m))$. Suppose that $\Gi^*$ is globally generated. Then the Yang--Mills metric $P^G$ admits a $\Psi$-subharmonic lift.
\end{cor}
\begin{proof}
Since $\Gi^*$ is globally generated is globally generated we can take $\Ei$ in Lemma \ref{invYMlemma} to be a trivial holomorphic vector bundle. This means that $[G_\N]$ is a submodule of $\GH_\N\otimes\GG_0$ for some Hilbert space $\GG_0$ and hence $I_G$ will be $\Psi$-subharmonic. 
\end{proof}

%\begin{Question}
%Do we have a subharmonic version of Theorem \ref{YMthm}?
%\end{Question}

\subsection{Direct sums} % of Yang--Mills metrics

If $\GE=\GF\oplus\mathfrak{G}$ is a direct sum of Hilbert $\Ai$-modules then for the Serre sheaves we have \cite[Prop. 7.14(3)]{GoWe1}
$$
\Ei\cong\Fi\oplus\Gi
$$
as $\Oi_\M$-modules. Also, as in Remark \ref{reducrem}, $\GF_\N\subset\GE_\N$ is a reducing submodule if and only if for the Cowen--Douglas sheaves we have
$$
\Ei_{\rm CD}\cong\Fi_{\rm CD}\oplus\Gi_{\rm CD}
$$
as Hermitian holomorphic vector bundles. In other words, iff
$$
\CD(\GF\oplus\mathfrak{G})\cong\CD(\GF)\oplus\CD(\mathfrak{G}).
$$
%Conversely, if we have $\CD(\GE)=\CD(\GF)\oplus\CD(\mathfrak{G})$ then we obtain $\CD(\GE)\cong\CD(\GF\oplus\mathfrak{G})$. This implies that $\GE$ equals $\GF\oplus\mathfrak{G}$ as Hilbert $\Ai$-module, up to $\Ai$-module isomorphism. We cannot ensure that this $\Ai$-module isomorphism preserves the grading, so $\GE$ is not isomorphic to $\GF\oplus\mathfrak{G}$ as graded Hilbert $\Ai$-module. Still that is just a finite-dimensional error. Indeed, $\coh(\M)\cong\qgr(\Ai)$ is an isomorphism of Abelian categories, so it preserves direct sums. 
\begin{prop}[Direct sums of Yang--Mills metrics]\label{dirsumYMprop}
Suppose that $P^E$ is a projection over $C^\infty(\M)$ defining a holomorphic vector bundle $\Ei=\Fi\oplus\Gi$ with slope-stable summands $\Fi$ and $\Gi$, and that $P^E=P^F+P^G$ with $P^F$ and $P^G$ Yang--Mills metrics on $\Fi$ and $\Gi$. Then 
$$
\varsigma(\Ran\breve{\varsigma}(P^E))=P^E.  
$$
\end{prop}
\begin{proof}
From Lemma \ref{coinvYMlemma} we get $\varsigma(P_F)=P^F$ and $\varsigma(P_G)=P^G$ for superharmonic projections $P_F$ and $P_G$ onto quotients $\GF_\N$ and $\mathfrak{G}_\N$, and the Serre sheaves of these quotients are isomorphic to $\Fi$ and $\Gi$ as holomorphic vector bundles. Setting $\GE_\N:=\GF_\N\oplus\mathfrak{G}_\N$ we obtain $\Ei$ as the Serre sheaf of $\GE_\N$. 

We can present $\GE_\N$ as a quotient of $\GH_\N\otimes\GE_0$ where $\GE_0=\GF_0\oplus\GG_0$. 
The ranges of $\breve{\varsigma}(P^F)$ and $\breve{\varsigma}(P^G)$ are orthogonal so, $\Ran\breve{\varsigma}(P^E)=\Ran\breve{\varsigma}(P^F+P^G)$ is the same as the projection $P_E:=P_F+P_G$ onto $\GE_\N$. 

Thus $P^E=\varsigma(P_F)+\varsigma(P_G)=\varsigma(P_F+P_G)=\varsigma(P_E)$ and the proposition holds. 
\end{proof}

\subsection{The nature of $\varsigma(P_E)$}

If $\Ei$ is a torsionfree $\Oi_\M$-module, denote by $\Gr(\Ei)$ the torsionfree $\Oi_\M$-module obtained by summing the successive quotients in the Harder--Narasimhan--Seshadri filtration of $\Ei$ (see \cite[\S2.1]{Jaco2} for details). Thus $\Gr(\Ei)$ is a direct sum of slope-stable torsionfree sheaves on $\M$, and the summands have the same slope iff $\Ei$ is slope-semistable.

 %\begin{Remark}
If $\Ei$ is a holomorphic vector bundle over $\M$, a locally free subsheaf $\Gi\subset\Ei$ is a subbundle iff the quotient $\Ei/\Gi$ is locally free. So if $\Gr(\Ei)$ is locally free then all the subsheaves in the Harder--Narasimhan--Seshadri filtration are by subbundles and this splits smoothly. Hence for an arbitrary holomorphic vector bundle $\Ei$ we have that $\Gr(\Ei)\cong\Ei$ smoothly if and only if $\Gr(\Ei)$ is locally free.
%Example: The structure sheaf $\Oi_\M$ can have locally free subsheaves $\Gi$ which are nontrivial (given by some divisor in the case of a curve) and hence cannot be subbundles. The quotient $\Oi_\M/\Gi$ is torsion and in particular not locally free. 
% \end{Remark}
% \begin{Remark}[Locally free $\Gr(\Ei)$]
%Every S-equivalence class of semistable sheaves contains exactly one polystable sheaf up to isomorphism. Therefore, for every polystable vector bundle $\Ei$ there is a whole family of semistable sheaves $\Ei_{\rm CD}$ with $\Gr(\Ei_{\rm CD})\cong\Ei$. Probably the fact that $\Ei$ is a vector bundle forces each $\Ei_{\rm CD}$ to be locally free. Anyway we see that it often happens that $\Gr(\Ei_{\rm CD})$ is locally free when $\Ei$ is semistable. For instance it always happens if $\dim\M=1$. And in $\dim_\C\M=1$ one has $\Gr(\Ei_{\rm CD})$ locally free for any vector bundle $\Ei_{\rm CD}$. 

%When $\Ei:=\Gr(\Ei_{\rm CD})\cong\Ei_{\rm CD}$ as $C^0$-bundles, we have $\Ei\cong\Ei_{\rm CD}$ as holomorphic vector bundles over $\B\setminus\{0\}$, or at least the Serre sheaves over $\B$ are analytically isomorphic on $\B\setminus\{0\}$, but not with a $\D^\times$-equivariant map. This means that their quotients modules $\GH^E_\N$ and $\GQ^E_\N$ are isomorphic as $\Ai$-modules but not as graded $\Ai$-modules. They are the same as graded vector spaces as well since the Hilbert polynomials coincide (by the topological isomorphism $\Ei\cong\Ei_{\rm CD}$). 
%\end{Remark}

\begin{thm}\label{GrvarsigmaPEthm}
%Let $\Ei$ be a holomorphic vector bundle over $\M=\G/\K$ and let $P_E$ be the projection of $\GH_\N\otimes\GE_0$ onto a quotient module $\GE_\N$ with Cowen--Douglas sheaf $\Ei_{\rm CD}$ equal to $\Ei$ on $\M$. Suppose that $\Gr(\Ei)$ is locally free and let $Q^E$ be the direct sum of the Yang--Mills metrics on the stable summands of $\Gr(\Ei)$. 
Let $\Ei$ be a holomorphic vector bundle over $\M=\G/\K$, represent $E_\N:=\bigoplus_{m\in\N_0}H^0(\M;\Ei(m))$ as a quotient of $\Ai\otimes\GE_0$ for some Hilbert space $\GE_0$, and let $\GE_\N$ be the completion of $E_\N$ in $\GH_\N\otimes\GE_0$. Let $P_E$ be the projection of $\GH_\N\otimes\GE_0$ onto a quotient module $\GE_\N$. Suppose that $\Gr(\Ei)$ is locally free and let $Q^E$ be the direct sum of the Yang--Mills metrics on the stable summands of $\Gr(\Ei)$.
Then $\varsigma(P_E)=Q^E$. In particular $\varsigma(P_E)$ is a matrix over $C^0(\M)\subset L^\infty(\M)$.% as elements of $L^\infty(\M)\otimes\Bi(\GE_0)$.
%$\varsigma(P_E)=\varsigma(Q_E)$ as elements of $L^\infty(\M)\otimes\Bi(\GE_0)$ where $\varsigma(Q_E)$ defines $\Gr(\Ei)$ as $\Ci^\infty_\M$-module, and if $\Gr(\Ei)$ is locally free then $\varsigma(Q_E)$ is a real-analytic Yang--Mills metric. 
\end{thm}  
\begin{proof}
We consider first the case when $\Ei$ is slope-semistable. Assume that $\Ei$ has a filtration %there are only two members in the Seshadri filtration 
$0\to\Gi\to \Ei$ with $\Gi$ a stable subbundle (a filtration with several subbundles can be treated by the same argument). The direct sum $\Gr(\Ei)$ of stable quotients is then 
$$
%\Qi=\Qi_1\oplus\Qi_2:=
\Gr(\Ei)=(\Ei/\Gi)\oplus\Gi. 
$$
The stable vector bundles $\Gi$ and $\Fi:=\Ei/\Gi$ admit Yang--Mills metrics $\varsigma(P_G)$ and $\varsigma(P_F)$ obtained from $\Psi$-superharmonic projections $P_G$ and $P_F$ over $\Ti_\GH^{(0)}$. These are the projections onto the completions of $G_\N:=\bigoplus_mH^0(\M;\Gi(m))$ and $F_\N:=\bigoplus_mH^0(\M;\Fi(m))$ realized as quotient modules (recall Lemma \ref{coinvYMlemma}).

The image of the projection $Q:=P_G+P_F$ is the completion of the direct sum $Q_\N=G_\N\oplus F_\N$ in a Fock inner product and $\varsigma(Q)$ is a Yang--Mills metric on $\Gr(\Ei)$ as we saw in Proposition \ref{dirsumYMprop}. Here $Q_\N=G_\N\oplus F_\N$ is a direct sum of graded $\Ai$-modules by definition.

%We want to show that $\varsigma(P_E)$ defines $\Qi$ ($\cong\Ei$). 

We note that $G_\N\oplus F_\N$ equals $E_\N$ as graded vector space (up to finite-dimensional subspaces), but $G_\N$ is merely a submodule of $E_\N$ and need not have an $\Ai$-module complement. Therefore $Q_\N=(E_\N/G_\N)\oplus G_\N$ need not be $\Ai$-isomorphic to $E_\N$ as graded $\Ai$-module (under the present local freness assumption $\Gr(\Ei)\cong\Ei$ we know that the pullbacks of $\Gr(\Ei)$ and $\Ei$ are analytically isomorphic as vector bundles over $\B\setminus\{0\}$ but this need not be $\D^\times$-equivariantly). %So let $G_\N^\perp$ be the orthogonal complement of $G_\N$ in $E_\N$. 

We realize $E_\N$ as a quotient module $\GE_\N$ and denote by $I_G$ the projection of $\GE_\N$ onto the submodule $\GG_\N:=[G_\N]$. Let also $P_F$ denote the projection of $\GE_\N$ onto the quotient module $\GE_\N\ominus\GG_\N=[F_\N]$.
Then
$$
\varsigma(P_E)=\varsigma(I_G+P_F)=\varsigma(I_G)+\varsigma(P_F)
$$
is a direct sum of Yang--Mills metrics on the holomorphic direct sum $\Gi\oplus\Fi$. Indeed, $\varsigma(I_G)$ can be identified with $\varsigma(P_G)$ by Lemma \ref{invYMlemma}.

%\begin{Remark}
%If we have one more member in the Jordan--Hölder filtration then it is of the form $0\to\Gi\to\Fi\to\Ei$ and  $\Ei\cong(\Ei/\Fi)\oplus(\Fi/\Gi)\oplus\Gi$ as smooth vector bundles if we assume that $\Gr(\Ei)$ is locally free. Here we have 
%$$
%\Qi_1=\Ei/\Fi,\quad\Qi_2=\Fi/\Gi,\quad\Qi_3=\Gi.
%$$
%Note that $\mu(\Qi_1)=\mu(\Qi_2)=\mu(\Qi_3)$ and there is nothing like a maximal destabilizing subsheaf in the Jordan--Hölder filtration. If we instead have a Harder--Narasimhan filtration of an arbitary vector bundle $\Ei$,
%$$
%0\to\Gi\to\Fi\to\Ei,
%$$
%then the quotient $\Qi_1=\Ei/\Fi$ is called the \textbf{minimal destabilizing quotient} sheaf while $\Gi$ is the \textbf{maximal destabilizing subsheaf}.  
%\end{Remark}
%By the same argument as before we obtain that $\varsigma(P_E)=\varsigma(P_{E/F}+P_{F/G}+P_G)$ is a direct sum of Yang--Mills metrics on the holomorphic direct sum $(\Ei/\Fi)\oplus\Fi/\Gi\oplus\Gi$. And similarly for an arbitrary number of elements in the Jordan--Hölder filtration.

Now if $\Ei$ is not semistable we have a Harder--Narasimhan filtration by semistable subsheaves (assumed to be subbundles here). For the proof of the theorem we may assume it is given by $0\to\Gi\to\Ei$. Set $\Fi:=\Ei/\Gi$. We have $P_E=P_G+P_F$ and since $\Fi$ is semistable we know that $\varsigma(P_F)$ is the direct sum of Yang--Mills metrics on $\Gr(\Fi)$, and similarly for $\Gi$. This gives the result.% Indeed we get $\varsigma(P_F)=\varsigma(P_{F/G})+\varsigma(P_G)$ a direct sum of Yang--Mills metrics and then $\varsigma(P_F)=\varsigma(P_{E/F})+\varsigma(P_F)$ also.
\end{proof}
%\begin{Remark}
%Even if $\Gi\subset\Ei$ is a locally free subsheaf it need not be a subbundle so we do not know if $0\to\Gi\to\Ei\to\Ei/\Gi\to 0$ splits smoothly and hence $\Qi\cong\Ei$ might fail also topologically.
%\end{Remark}
\begin{cor}
Let $\GE_\N$ be a range of a $\Psi$-superharmonic projection $P_E$ over $\Ti_\GH^{(0)}$ 
 % so that $\Ei_{\rm CD}$ is locally free on $\B\setminus\{0\}$ 
and let $\Ei_{\rm CD,\M}$ denote the holomorphic vector bundle over $\M$ such that $\Ei_{\rm CD}=\Oi_{\rm CD}\otimes\Ei_{\rm CD,\M}$. Then $\varsigma(P_E)$ is a Yang--Mills metric on $\Ei_{\rm CD,\M}$.
\end{cor}
\begin{proof}
Since $\Ei_{\rm CD}$ is locally free (see Theorem \ref{bigCDversuslocallyfree}), the Serre sheaf $\Ei$ of $\GE_\N$ is locally free and $\varsigma(P_E)$ is a real-analytic metric on $\Ei_{\rm CD,\M}$. We have a factorization $\CD(\GE_\N)=\CD(\GH_\N)\otimes\varsigma(P_E)$ so $\varsigma(P_E)$ defines $\Ei_{\rm CD,\M}$ and its holomorphic structure. By Theorem \ref{GrvarsigmaPEthm} we have that $\varsigma(P_E)$ is a Yang--Mills metric on $\Gr(\Ei)$, so $\Gr(\Ei)$ is analytically isomorphic to $\Ei_{\rm CD,\M}$ and $\varsigma(P_E)$ is a Yang--Mills metric on $\Ei_{\rm CD,\M}$. %\textbf{One seems to need something like real-analytic projections defining holomorphic vector bundles in Swan fashion. This would be true of course if this real-analyticity is the same ``holomorphicity'' condition of a projection $P^E$ that is used by Popovici and by Treil--Wick and others. And that should be equivalent to having a Parseval $C^*$-frame of holomorphic sections for the Levi--Civita $\bar{\pd}$-operator. So the latter has to be integrable, and that is again the same condition on $P^E$.}
\end{proof}

The next step would be to obtain a generalization Wang's theorem (Lemma \ref{Wanglemma}), namely that every torsionfree slope-stable sheaf has a ``singular Yang--Mills metric'' coming from a sequence of ``singular balanced metrics''. That would prove:
\begin{Conjecture}
Theorem \ref{GrvarsigmaPEthm} is true without the assumption that $\Gr(\Ei)$ is locally free, i.e. $\varsigma(P_E)$ is always a metric on $\Gr(\Ei)$ which is the direct sum of singular Yang--Mills metrics on the simple summands of $\Gr(\Ei)$. 
\end{Conjecture}
If this conjecture is true then we have:
\begin{cor}
For arbitrary holomorphic vector bundle $\Ei$, letting $P_E$ be the projection onto a Fock completion of $E_\N:=\bigoplus_{m\in\N_0}H^0(\M;\Ei(m))$ the symbol $\varsigma(P_E)$ has entries $C^0(\M)\subset L^\infty(\M)$ if and only if $\Gr(\Ei)$ is locally free. 
\end{cor}

\subsection{Gieseker-stability and superharmonic lifts}

With $L^\infty(\M)$ acting as multiplication operators on $L^2(\M,\omega)$ one can consider limits of sequences in $L^\infty(\M)\otimes\Mn_N(\C)$ in the strong operator topology. 
\begin{prop}\label{GiesSOTlimprop}
Let $\Ei$ be a Gieseker-stable vector bundle and let $(B^E_m)_{m\gg 0}$ be its sequence of balanced metrics. Then the projection
$$
P^E:=\mathrm{SOT-}\lim_{m\to\infty}B^E_m
$$
exists as a matrix over $L^\infty(\M)$ and it coincides with $\varsigma(P_E)$, where $P_E$ is the projection onto the completion of $E_\N:=\bigoplus_mH^0(\M;\Ei(m))$ in Fock space.
 \end{prop}
 \begin{proof}
%We want to show that $P^E$ coincides with $\varsigma(P_E)$, where $P_E$ is the projection onto the quotient module $\GE_\N$ associated with $\Ei$. 
% For that, note that 
We have
\begin{align*}
 \mathrm{SOT-}\lim_{m\to\infty}B^E_m&
= \mathrm{SOT-}\lim_{m\to\infty}\varsigma^{(m)}(\breve{\varsigma}^{(m)}(B^E_m))
 \\&=\mathrm{SOT-}\lim_{m\to\infty}c_{E,m}^{-1}\varsigma^{(m)}(\breve{\varsigma}^{(m)}(B^E_m))
 \\&=\mathrm{SOT-}\lim_{m\to\infty}\varsigma^{(m)}(P_{E,m})
  \\&=\varsigma(P_E),
 \end{align*}
where we used the balance of each $B^E_m$ for $m\gg 0$ in the penultimate line.
 \end{proof}
 For a vector bundle $\Ei$ which is Gieseker-stable but not slope-stable the limit $\lim_{m\to\infty}B^E_m$ does not exist in $C^\infty$ by \cite{Wang2}. If it also fails to exist in $C^0$ then the projection $P_E$ onto $[E_\N]$ must be of the form $P_E=\breve{\varsigma}(P^E)+C_E$ with $C_E$ noncompact. 
\begin{Conjecture}
Let $\Ei$ be a Gieseker-stable vector bundle and let $(P^E_m)_{m\gg 0}$ be its sequence of balanced metrics. Then the projection
$$
P^E:=\mathrm{SOT-}\lim_{m\to\infty}P^E_m
$$
defines $\Gr(\Ei)$ and a singular Yang--Mills metric on $\Gr(\Ei)$.
 \end{Conjecture}
 %This conjecture generalizes to vector bundles admitting singular balanced metrics.  

%\begin{cor}
%If $\Gr(\Ei)$ is locally free then $\Ei$ is Gieseker-stable iff $\Ei$ is slope-stable. \textbf{Maybe not if smooth convergence requires more than compactness of $C_E$}
%\end{cor}
%\begin{proof}
%When $\Gr(\Ei)$ is locally free we obtained that $P_E=\breve{\varsigma}(P^E)+C_E$ is over $\Ti_\GH^{(0)}$, in fact over $\breve{\varsigma}(C^\infty(\M))+\Gamma_0$. But if $\Ei$ is Gieseker-stable then we obtain $\GE_\N=[E_\N]$ as the graded $\Ai$-module giving rise to $\Ei$ using the balanced metrics, and $\CD(\GE_m)=B^E_m\otimes\CD(\GH_m)$ is balanced for $m\gg 0$. Since $C_E$ is compact we have norm convergence $\varsigma(P_E)=\lim_mB^E_m$, so $\varsigma(P_E)$ is a smooth metric which is the $C^0$-limit of the sequence of balanced metrics. We would have to show that the convergence is in $C^\infty$ (or at leat $C^2$) if we want $\Ei$ to be slope-stable. Otherwise it would be the case that the smooth vector bundle $\Ei$ has another holomorphic structure (namely that of $\Gr(\Ei)$) for which the balanced metrics converge $C^\infty$ instead. 
%But if the balanced metrics do not converge then $\varsigma(P_E)$ cannot be smooth, or at least the convergence cannot be smooth. Thus the balanced metrics must converge smoothly, so by Wang's theorem $\Ei$ is slope-stable.
%\end{proof}

\subsection{From balance to Gieseker-stability}
We now give a more direct proof of one implication in the main theorem of \cite{Wang1}:
\begin{thm}\label{HalfWangthm}
Let $\Ei$ be a holomorphic vector bundle over $\M$ and suppose that $\Ei(m)$ admits an $\omega$-balanced Hermitian metric for each $m\gg 0$. Then $\Ei$ is Gieseker-polystable.
\end{thm}
\begin{proof}
Fix $m\gg 0$. By assumption we have a balanced inner product $\GH^E_m$ on $H^0(\M;\Ei(m))$. Let $(\psi_j)_{j\in\J}$ be an orthonormal basis for $\GH^E_m$. Let $(\cdot|\cdot)$ be the Hilbert $C^*$-structure on $\Gamma^\infty(\M;\Ei(m))$ obtained from the projection $\FS(\GH^E_m)$. Recall that $\FS(\GH^E_m)\in C^\infty(\M)\otimes\Bi(\GH^E_m)$ can be represented by the matrix whose $(j,k)$th entry is the function $(\psi_j|\psi_k)\in C^\infty(\M)$. Let
$$
\FS^\dagger(\GH^E_m)\in\End\Gamma^\infty(\M;\Ei(m),\FS(\GH^E_m))\otimes\Bi(\GH^E_m)
$$
be the matrix of rank-1 operators $|\psi_j)(\psi_k|$ on $\Gamma^\infty(\M;\Ei(m),\FS(\GH^E_m))$. Using $\End\Gamma^\infty(\M;\Li^m)\cong C^\infty(\M)$ we have the isomorphism 
\begin{equation}\label{Endosiso}
\End\Gamma^\infty(\M;\Ei(m))\cong\End\Gamma^\infty(\M;\Ei),
\end{equation}
and we shall identify $\FS^\dagger(\GH^E_m)$ with an endomorphism of $\Gamma^\infty(\M;\Ei)$. 

Recall that the state $\omega_E:\End\Gamma^\infty(\M;\Ei)\to\C$ is defined on rank-1 endomorphisms by
$$
\omega_E(|\psi)(\phi|):=\frac{1}{\rank\Ei}\omega((\phi|\psi)).
$$
So balance of $\GH^E_m$ says that
\begin{equation}\label{omegaEbalance}
\omega_E(|\psi_j)(\psi_k|)=\frac{1}{\chi(\Ei(m))}
\end{equation}
for all $j,k=1,\dots,\dim\GH^E_m=\chi(\Ei(m))$. Recall also that balance says that $(\psi_j)_{j\in\J}$ is a Parseval frame for $\Gamma^\infty(\M;\Ei(m),\FS(\GH^E_m))$, which under the isomorphism \eqref{Endosiso} means that
\begin{equation}\label{framebalance}
\sum_j|\psi_j)(\psi_j|=\bone_E,
\end{equation}
where $\bone_E$ is the identity in $\End\Gamma^\infty(\M;\Ei)$. 

For $X\in\Bi(\GH^E_m)$, define an endomorphism of $\Gamma^\infty(\M;\Ei)$ by
\begin{align*}
\varsigma_m^E(X)&:=(\id\otimes\Tr_{\GH^E_m})((\bone\otimes X)\FS^\dagger(\GH^E_m))
\\&=\chi(\Ei(m))(\id\otimes\phi_m^E)((\bone\otimes X)\FS^\dagger(\GH^E_m)).
\end{align*}
Then $\varsigma_m^E(\bone)$ is the frame operator of $(\psi_j)_j$ regarded as an element of $\End\Gamma^\infty(\M;\Ei)$, which equals $\bone_E$ by balance \eqref{framebalance}.  

We also define, for each $f\in\End\Gamma^0(\M;\Ei)$, a Toeplitz-type operator on $\GH^E_m$ by
%\textbf{No, we want $\varsigma_m^E(\bone)$ to act on $\Gamma^0(\M;\Ei)$, not on $\Gamma^0(\M;\Ei(m))$.}
\begin{align*}
\breve{\varsigma}^{E(m)}(f)%&=\frac{\chi(\Ei(m))}{\rank\Ei}(\omega\otimes\id)((f\otimes \bone)\FS^\dagger(\GH^E_m))
%\\
&=\chi(\Ei(m))(\omega_E\otimes\id)((f\otimes\bone_{\GH^E_m})\FS^\dagger(\GH^E_m)).
\end{align*}
The resulting map $\breve{\varsigma}^{E(m)}:\End\Gamma^0(\M;\Ei)\to\Bi(\GH^E_m)$ is then unital by balance \eqref{omegaEbalance}. 
%\textbf{No, we cannot apply $\omega_E$ to an element of $\End\Gamma^0(\M;\Ei(m))$}. Or we could under the isomorphism 
%$$
%\End\Gamma^0(\M;\Ei(m))\cong\End\Gamma^0(\M;\Ei). 
%$$

The maps $\varsigma^{E(m)}$ and $\breve{\varsigma}^{E(m)}$ are adjoints with respect to $\omega_E$ and $\phi^E_m$,
\begin{align*}
\phi^E_m(\breve{\varsigma}^{E(m)}(f)X)&=(\omega_E\otimes\Tr_{\GH^E_m})((f\otimes X)\FS^\dagger(\GH^E_m))
\\&=\omega_E(f(\id\otimes\Tr_{\GH^E_m})((\bone\otimes X)\FS^\dagger(\GH^E_m))),
\end{align*}
and since they are both unital they therefore both intertwine the states $\omega_E$ and $\phi^E_m$,
$$
\phi^E_m(\breve{\varsigma}^{E(m)}(f))=\omega_E(f)\text{ and }\omega_E(\varsigma_m^E(X))=\phi_m^E(X).
$$  

Let $\Gi\subset\Ei$ be an analytic subsheaf. Then $G_\N:=\bigoplus_mH^0(\M;\Gi(m))$ is a graded submodule of $E_\N:=\bigoplus_mH^0(\M;\Ei(m))$, and we denote by $P_G$ the projection of $[E_\N]$ onto $[G_\N]$. %Note that
%$$
%\lim_{m\to\infty}\frac{\phi_m(P_{G,m}))}{\rank\Ei}=\frac{\rank\Gi}{\rank\Ei}.
%$$

If $(\psi_j)_{j\in\J}$ is our balanced frame as before then there is a subset $\J_G\subset\J$ such that each $\psi_j^G:=\psi_j$ with $j\in\J_G$ will be in $G_m:=H^0(\M;\Gi(m))$ and these $\psi_j^G$'s form a basis for $G_m$. So if $P^G$ is the projection acting on $\Gamma^\infty(\M;\Ei)$ with image $\Gamma^\infty(\M;\Gi)$ then $P^G\psi^G_j=\psi^G_j$ for all $j\in\J_G$. So for $j,k\in\J_G$ we obtain 
$$
\chi(\Ei(m))\omega_E(|P^G\psi_j)(P^G\psi_k|)=\frac{\chi(\Ei(m))}{\rank\Ei}\omega((\psi_j|\psi_k))=\delta_{j,k}
$$
by balance. Let $F_\N$ be the quotient $E_\N/G_\N$. Since $(\psi_j^G|\psi_k)$ can be nonzero even for $k\notin\J_G$ but $\breve{\varsigma}^{E(m)}(P^G)$ is always positive there is a positive operator $C_{G,m}$ on $F_m$ such that
$$
\breve{\varsigma}^{E(m)}(P^G)=P_{G,m}\oplus C_{G,m}\text{ on }E_m=G_m\oplus F_m.
$$
So we obtain
$$
\breve{\varsigma}^{E(m)}(P^G)\geq P_{G,m}.
$$
This gives, for $m$ large enough so that $\dim E_m=\chi(\Ei(m))$ and $\dim G_m=\chi(\Gi(m))$, that
\begin{align*}
\frac{\rank\Gi}{\rank\Ei}&=\omega_E(P^G)
\\&=\phi^E_m(\breve{\varsigma}^{E(m)}(P^G))
\\&\geq\phi^E_m(P_{G,m})
\\&=\frac{\chi(\Gi(m))}{\chi(\Ei(m))}.
\end{align*}
Equality $\chi(\Ei(m))/\rank\Ei=\chi(\Gi(m))/\rank\Gi$ occurs iff $\breve{\varsigma}^{E(m)}(P^G)=P_{G,m}$ and this happens iff $\FS(\GH^E_m)$ splits into $P^F\oplus P^G$ with $P^G$ having $(\psi^G_j)_{j\in\J_G}$ as Parseval $C^*$-frame, so that $\Ei$ splits holomorphically as $\Ei=\Fi\oplus\Gi$. This gives the statement. 
\end{proof}

\section{Equivariant vector bundles}
In this section we specialize to the case of equivariant vector bundles over $\M=\G/\K$. The quantization of these were initiated in \cite{Hawk1}, which we now build on. A new thing here is that we add new data: inner products on the quantum side and Hermitian metrics on the classical side. It turns out that these go very well with the quantization because of the equivariance, and in fact the existence of a Hermitian metric' which``quantizes perfectly ' is likely to chararacterize equivariant vector bundles, in a sense made more precise below. 

Given any $\K$-representation $\GK$ we can form the vector bundle 
$$
\Ei=\G\times_\K\GK
$$
associated with the principal $\K$-bundle $\G\to\M\to 0$ and the represetation $\GK$. Such a vector bundle $\Ei$ is called $\G$-\textbf{equivariant} \cite{Snow1}. Every $\G$-equivariant vector bundle comes with a natural holomorphic structure \cite[\S3.2]{Rama1}. The space $H^0(\M;\Ei)$ of global holomorphic sections then carries a representation of $\G$ which is said to be \textbf{induced} from the representation $\GK$ of the subgroup $\K$ \cite{Bott1}. A $\G$-equivariant vector bundle $\Ei=\G\times_\K\GK$ is \textbf{irreducible} if $\GK$ is an irreducible $\K$-representation. Then $H^0(\M;\Ei(m))$ is an irreducible $\G$-representation for all $m\geq 0$. 

%Let $\Ei$ be a $\G$-equivariant vector bundle over $\M$. By replacing $\Ei$ with $\Ei(m_0)$ for $m_0$ large enough if necessary, we may suppose that $E_\N:=\bigoplus_mH^0(\M;\Ei(m))$ appears as a graded quotient of $\Ai\otimes E_0$, where the surjections $\Ai_m\otimes E_0\to E_m$ are the multiplication maps on sections. In this way $E_m$ sits as a vector subspace of $\Ai_m\otimes E_0$ which is invariant under the standard $\G$-representation on $\Ai_m\otimes E_0$. 

%Subproduct.... unique $\G$-invariant inner product $\GE_0$ on $H^0(\M;\Ei)$ and $\GH_m\otimes H^0(\M;\Ei)$. Then also on the subspace $H^0(\M;\Ei(m))$. We obtain a quotient module $\GE_\N$ of $\GH_\N\otimes\GE_0$. 

\subsection{Quotient modules with equivariant Cowen--Douglas sheaves}

\subsubsection{Equivariant Cowen--Douglas metrics are balanced}
If $\Ei$ is an irreducible $\G$-equivariant vector bundle over $\M=\G/\K$ then each fiber $\Ei(x)$ carries an irreducible $\K$-representation and a unique $\K$-invariant inner product. This gives us a unique $\G$-equivariant Hermitian metric $P^E$ on $\Ei$. We also have a unique $\G$-invariant inner product $\GE_0$ on the $\G$-representation $H^0(\M;\Ei)$. The evaluation $H^0(\M;\Ei)\ni\psi\to\psi(x)\in\Ei(x)$ is $\G$-equivariant so if it is surjective then $\Ei(x)$ sits as a Hilbert subspace of $\GE_0$. The $\G$-invariant Hermitian metric on $\Ei$ is then given by $P^E=\FS(\GE_0)$. It is known that $P^E$ is $\omega$-Yang--Mills \cite{Koba4}.

\begin{prop}\label{balancedeqprop}
Let $P^E$ be a projection over $C^0(\M)$ defining a globally generated irreducible $\G$-equivariant Hermitian vector bundle $\Ei$, so that $P^E=\FS(\GE_0)$ for the $\G$-invariant inner product $\GE_0$ on $H^0(\M;\Ei)$. Then $\FS(\GE_m)=\FS(\GH_m)\otimes P^E$ is balanced on $\Ei(m)$ for each $m\in\N_0$. Moreover, the Cowen--Douglas metric of the quotient module $\GE_\N:=\Ran\breve{\varsigma}(P^E)$ is of the form
\begin{equation}\label{equivCD}
\CD(\GE_\N)(v)=\CD(\GH_\N)(v)\otimes P^E([v]),\qquad\forall v\in\B\setminus\{0\},
\end{equation}
and the graded $\Ai$-module underlying $\GE_\N$ is precisely $E_\N:=\bigoplus_mH^0(\M;\Ei(m))$. 
\end{prop}
\begin{proof}
As noticed in Proposition \ref{geomintprop}, we have $\FS(\GE_m)=\FS(\GH_m)\otimes P^E$ for all $m$. 

The Haar orthogonality relations (see \S\ref{Harrelsec}) say precisely that
$$
\omega(P^E_{\mu,\nu})=\frac{1}{\dim\GE_0}\delta_{\mu,\nu},\qquad\forall \mu,\nu=1,\dots,\dim\GE_0
$$
i.e. the Hilbert space $H^0(\omega,\FS(\GE_0))$ coincides with $\GE_0$. In other words, the metric $P^E=\FS(\GE_0)$ is $\omega$-balanced. 

More generally, for each $m\in\N_0$, the metric $\FS(\GE_m)=\FS(\GH_m)\otimes P^E$ on $\Ei(m)$ is $\G$-equivariant and hence balanced by the above argument. We thus have
$$
\breve{\varsigma}^{(m)}(P^E)=c_{E,m}P_{E,m},\qquad\forall m\in\N_0.
$$
It follows that $P_E$ is the range projection of $\breve{\varsigma}(P^E)$ and that $\breve{\varsigma}(P^E)$ equals $P_E$ modulo compacts. Therefore \eqref{equivCD} holds by Theorem \ref{bigCDversuslocallyfree}. 

As in the proof of Theorem \ref{backtoHardythmgen} the vector space $\GE_m:=\Ran\breve{\varsigma}^{(m)}(P^E)$ coincides with the $\C$-linear span $E_m$ of a Parseval frame for $\FS(\GH_m)\otimes P^E$. 
%We need to show that $\CD(\GE_m)=\FS(\GE_m)$ for all $m$. We want that $\varsigma^{(m)}(P_{E,m})$ is just a scalar multiple of $P^E$. Then we would have that $B^E_m=P^E$. 

\end{proof}

\subsubsection{$\jmath_{l,m}^E$ is unital}

Recall our unital completely positive maps $\iota_{m,l}:\Bi(\GH_m)\to\Bi(\GH_l)$ defined by
\begin{equation}\label{iota}
\iota_{m,l}(A):=V^*_{m,l}(A\otimes\bone_{l-m})V_{m,l},\qquad\forall A\in\Bi(\GH_m)
\end{equation} 
where $V_{m,l}$ is the isometric embedding of $\GH_l$ into $\GH_m\otimes\GH_{l-m}$. These can be generalized to vector-valued quotient modules. To simplify the notation we write $\iota_{m,l}$ also for the induced map from $\Bi(\GH_m\otimes\C^N)$ to $\Bi(\GH_l\otimes\C^N)$, where we identify operators on $\GH_\N\otimes\C^N$ with $N\times N$-matrices of operators on $\GH_\N$ as before.

Let now $P_E$ be the projection of $\GH_\N\otimes\C^N$ onto a coinvariant subspace $\GE_\N$. Coinvariance gives $\iota_{m,l}^E(P_{E,m})\geq P_{E,l}$. The operator $P_{E,l}\iota_{m,l}^E(P_{E,m})P_{E,l}$ is thus an operator on $\GE_l$ which is $\geq P_{E,l}$. Since $\iota_{m,l}^E$ is unital it is contractive, so that we must have
$$
\iota_{m,l}^E(P_{E,m})P_{E,l}=P_{E,l}=P_{E,l}\iota_{m,l}^E(P_{E,m}).
$$
The map $\iota_{m,l}^E:\Bi(\GE_m)\to\Bi(\GE_l)$ defined by
$$
\iota_{m,l}^E(A):=P_{E,l}\iota_{m,l}(A)|_{\GE_l},\qquad\forall A\in\Bi(\GE_m)
$$
is therefore unital. From \eqref{iota} we obtain the formula
\begin{equation}\label{iotaE}
\iota_{m,l}^E(A)=V^{E*}_{m,l}(A\otimes\bone_{l-m})V_{m,l}^E,\qquad\forall A\in\Bi(\GE_m),
\end{equation} 
where $V_{m,l}^E:=V_{m,l}|_{\GE_l}$. Since $\iota_{m,l}^E$ is unital, $V_{m,l}^E$ is an isometric embedding of $\GE_l$ into $\GE_m\otimes\GH_{l-m}$. The adjoint is given by $V^{E*}_{m,l}=P_{E,l}V_{m,l}^*|_{\GE_m\otimes\GH_{l-m}}$. 
%where $V_{m,l}^E:=(P_{E,m}\otimes\bone_{l-m})V_{m,l}|_{\GE_m}$. Since $\iota_{m,l}^E$ is unital, $V_{m,l}^E$ is an isometric embedding of $\GE_l$ into $\GE_m\otimes\GH_{l-m}$. 

%In the same ways as for the ambient map $\iota_{m,l}$ (see \cite{An6}) we obtain the explicit formula
%\begin{equation}\label{iotaE}
%\iota_{m,l}^E(A)=V^{E*}_{m,l}(A\otimes\bone_{l-m})V_{m,l}^E,\qquad\forall A\in\Bi(\GE_m),
%\end{equation} 
%where $V_{m,l}^E$ is an isometric embedding of $\GE_l$ into $\GE_m\otimes\GH_{l-m}$. 

The Hilbert space $\GH_m$ carries an irreducible unitary representation of $\G$. Endow $\GH_m\otimes\C^N$ with the unitary $\G$-representation where $\G$ trivially on the factor $\C^N$.
\begin{prop}
Suppose that each Hilbert space $\GE_m\subset\GH_m\otimes\GE_0$ is invariant under the $\G$-action on $\GH_m\otimes\GE_0$. Then the isometries $V^E_{m,l}:\GE_l\to\GE_m\otimes\GH_l$ are intertwiners of $\G$-representations.  
\end{prop}
\begin{proof}
As already noticed above, coinvariance ensures that $V^E_{m,l}$ is the restriction of $V_{m,l}$ to $\GE_m$. Since $\GE_m$ and $\GE_l$ are $\G$-invariant they carry $\G$-representations by restriction of the $\G$-representations on $\GH_m\otimes\C^N$ and $\GH_l\otimes\C^N$ respectively. So we obtain the proposition from the fact that $V_{m,l}$ intertwines the $\G$-representation on $\GH_m\otimes\C^N$ with that on $\GH_l\otimes\C^N$. 
\end{proof}
Let $\phi^E_m$ be the normalized trace on $\Bi(\GE_m)$,
$$
\phi^E_m(A):=(\dim\GE_m)^{-1}\Tr(A),\qquad\forall A\in\Bi(\GE_m),
$$
and define the completely positive map
$$
\jmath_{l,m}^E:\Bi(\GE_l)\to\Bi(\GE_m)
$$
to be the adjoint of $\iota_{m,l}^E$ with respect to $\phi_m^E$ and $\phi_l^E$. Basically by definition this means that
$$
\jmath_{l,m}^E(B):=(\phi_{l-m}^E\otimes\id)(V_{m,l}BV_{m,l}^*),\qquad\forall B\in\Bi(\GE_l).
$$
\begin{prop}\label{unitajmathEprop}
Suppose that the Cowen--Douglas sheaf of $\GE_\N$ is of the form $\Ei_{\rm CD}=\Oi_{\rm CD}\otimes\Ei_{\B\setminus\{0\}}$ where $\Ei_{\B\setminus\{0\}}$ is the pullback to $\B\setminus\{0\}$ of a $\G$-equivariant vector bundle $\Ei$ over $\M$. 
%descends to an irreducible $\G$-equivariant vector bundle $\Ei$ on $\M=\G/\K$. %and that $H^0(\M;\Ei)$ is an irreducible $\G$-representation.
Then $\jmath_{l,m}^E$ is unital.
\end{prop}
\begin{proof}
By assumption \eqref{equivCD} holds and then, as before, the graded $\Ai$-module underlying $\GE_\N$ is precisely $E_\N:=\bigoplus_mH^0(\M;\Ei(m))$. Thus each vector space $\GE_m$ is a $\G$-representation. 

As explained in \cite[\S C]{Hawk1}, for all $l\geq m\geq 0$ the representation $\GE_m$ is contained as a vector space in the tensor product $\GE_l\otimes\GH_{l-m}^*$. The $\G$-invariant inner product on $\GE_l\otimes\GH_{l-m}^*$ is obtained by averaging over the group $\G$ using the Haar state $\omega$, and equals the Fock inner product on $\GE_l\otimes\GH_{l-m}^*$ up to a scalar factor on each irreducible direct summand. Therefore there is an isometric embedding of $\GE_m$ into $\GE_l\otimes\GH_{l-m}^*$,$$
W_{l,m}^E:\GE_m\to\GE_l\otimes\GH_{l-m}^*.
$$  
By assumption the $\G$-representation on $\GE_0$ is irreducible and this gives that the $\G$-representation on $\GE_m=H^0(\M;\Ei(m))$ is irreducible for each $m$. %No, only the $\K$-rep inducing $\GE_0$ is assumed irreducible.
Then $W_{l,m}^E$ and $V_{m,l}^E$ are unique and from \cite[Eq. (5.9)]{Hawk1} we have
$$
\jmath_{l,m}^E(B)=(\phi_{l-m}^E\otimes\id)(V_{m,l}BV_{m,l}^*)=W_{l,m}^{E*}(B\otimes\bone_{\GH_{l-m}^*})W_{l,m}^E,.
$$ 
The rightmost formula shows that $\jmath_{l,m}^E$ is unital, as asserted. 
\end{proof}
%An alternative proof of Proposition \ref{unitajmathEprop} is given in the proof of Proposition \ref{balancedeqprop}. 
\begin{Remark}[Reversed time evolution]\label{timeremark}
Recall that the \textbf{conjugate} of an irreducible representation $u$ of a compact group (or more generally, compact quantum group) $\G$ is an irreducible representation $\bar{u}$ such that $u\otimes\bar{u}$ contains the trivial representation. The existence of a conjugate to each irreducible representation is one of the structural properties that characterize representation categories of compact quantum groups. Therefore we do not expect isometries $W_{l,m}^E:\GE_m\to\GE_l\otimes\GH_{l-m}^*$ to exist unless the $\GE_m$'s are representations of some compact quantum group. In \cite{An4} we interpreted the existence of the backward maps $\jmath_{l,m}^E$ as a ``quantum symmetry'' in case the Hilbert spaces $\GE_m$ models the environment of some physical quantum system.
\end{Remark}
If $\GE_m$ is not irreducible then we again let $\phi_m^E$ be the trace on $\Bi(\GE_m)$ corresponding to the Haar state, i.e. the direct sum of the normalized traces on the irreducible direct summands. Then the adjoint $\jmath_{l,m}^E$ of $\iota_{m,l}^E$ with respect to $\phi_m^E$ and $\phi_l^E$ is again unital. If all irreducible summands of $\Ei$ have the same reduced Hilbert polynomial $\chi(\Ei(m))/\rank\Ei$ then these two definitions of $\jmath_{l,m}^E$ coincide.

%We shall show that this gives the weak balanced condition for any homogeneous vector bundle, and that it proves Guo-polystability iff the irreducible components have the same polytail. 

%\begin{Remark}
%The unitality of $\jmath_{l,m}^E$ gives balance directly. Indeed, in the setting of Proposition \ref{balancedeqprop}, the Cowen--Douglas metric is just the pullback of $P^E$ and since $\breve{\varsigma}^{(m)}(P^E)=c_{E,m}P_{E,m}$, we are done. What??
%\end{Remark}

\subsubsection{The $(d+1)$-isometry $S_E$}
Recall that if $\GE_\N\subset\GH_\N\otimes\GE_0$ is a graded subspace then we denote by $S_E=(S_{E,1},\dots,S_{E,n})$ the shift compressed to $\GE_\N$ and consider the grading-preserving positive operator $S_E^*S_E:=\sum_{\alpha=1}^nS_{E,\alpha}^*S_{E,\alpha}$. We let $|S_E|$ be the positive square root of $S_E^*S_E$. In this section we obtain a generalization of some results in \S\ref{multsection} (which are recovered by taking $\Ei$ to be the trivial line bundle). First of all, Proposition \ref{unitajmathEprop} gives:
\begin{cor}\label{SEsquarecor}
Let $S_E$ be the shift on $\GE_\N$. Suppose that the Cowen--Douglas sheaf of $\GE_\N$ descends to a $\G$-equivariant vector bundle $\Ei$ on $\M=\G/\K$ and that all irreducible summands of $\Ei$ have the same reduced Hilbert polynomial $\chi(\Ei(m))/\rank\Ei$. Then $|S_E|P_{E,m}$ is a scalar for each $m\in\N_0$, viz. 
$$
|S_E|P_{E,m}=\sqrt{\frac{\chi(\Ei(m+1))}{\chi(\Ei(m))}}P_{E,m}.
$$
%and $S_E$ is a $(d+1)$-isometry. 
\end{cor}
\begin{proof}
The operator
$$
\jmath_{m,m+1}^E(P_{E,m+1})=\frac{\chi(\Ei(m))}{\chi(\Ei(m+1))}S_E^*S_EP_{E,m}
$$ 
equals the identify $P_{E,m}$ by Proposition \ref{unitajmathEprop}. %The proof that $S_E$ is a $(d+1)$-isometry is then just a copy of that in \S\ref{multsection}. 
\end{proof}

Recall the maps $\Psi$ and $\Phi_*$ on $\Bi(\GH_\N)$ defined by $\Psi(X):=\sum_{\alpha=1}^nT_\alpha^*XT_\alpha$ and $\Phi_*(X):=\sum_{\alpha=1}^nS_\alpha^*XS_\alpha$, where $S=|S|T$ is the shift on $\GH_\N$. 
\begin{prop}\label{diffoptraceprop}
For any $X= \sum_mX_m\in\Gamma_b=\prod_m\Bi(\GH_m)$ we have
$$
\phi_m((\id-\Psi)^p(X))=\sum^p_{r=0}(-1)^r{p\choose r}\phi_{m+r}(X)
$$
and
$$
\Tr((\id-\Phi_*)^p(X_m))=\sum^p_{r=0}(-1)^r{p\choose r}\Tr(X_{m+r})
$$
for all $m,p\in\N_0$. 
\end{prop}
\begin{proof}
We have
$$
(\id-\Psi)^p=\sum^p_{r=0}(-1)^r{p\choose r}\Psi^r
$$
and $\phi_m\circ\Psi^r=\phi_{m+r}$. Hence the first result. The second follows from $(\id-\Phi_*)^p=\sum^p_{r=0}(-1)^r{p\choose r}\Phi^r_*$ and $\Tr(\Phi^r_*(X)p_m)=\Tr(X\Phi^r(p_m))=\Tr(Xp_{m+r})$. 
\end{proof}
If $\GE_\N$ is a quotient module then we have the shift tuple $S_E$ acting on $\GE_\N$, and as in \S\ref{disombacksec} we can consider the operators 
$$
B_p(S_E):=(\id-\Phi_{E,*})^p(\bone),
$$
where $\Phi_{E,*}(X):=\sum_{\alpha=1}^nS_{E,\alpha}^*XS_{E,\alpha}$ for all $X\in\Bi(\GE_\N)$. Since the backward shift $S^*$ on $\GH_\N\otimes\GE_0$ preserves the subspace $\GE_\N$ we have 
$$
B_p(S_E):=(\id-\Phi_*)^p(P_E),
$$
where as usual $\Phi_*(X):=\sum_{\alpha=1}^nS_\alpha^*XS_\alpha$. 
\begin{prop}\label{zerotraceprop}
For any graded quotient module $\GE_\N$ and each $p\in\N_0$ the operator $B_p(S_E)$ satisfies
$$
\Tr(B_p(S_E)p_m)%=\Tr((\id-\Phi_*)^p(P_{E,m}))
=\sum^p_{r=0}(-1)^r{p\choose r}\Tr(P_{E,m+r})
$$
for all $m\in\N_0$. For $p\geq d+1$ we have
$$
\Tr(B_p(S_E))=0.
$$
\end{prop}
\begin{proof}
From Proposition \ref{diffoptraceprop} we have $\Tr((\id-\Phi_*)^p(P_{E,m}))=\sum^p_{r=0}(-1)^r{p\choose r}\Tr(P_{E,m+r})$. Now the Serre sheaf $\Ei$ of $\GE_\N$ is Castelnuovo--Mumford regular, so 
$$
\Tr(P_{E,m+r})=\dim\GE_{m+r}=\dim H^0(\M;\Ei(m+r))=\chi(\Ei(m+r))
$$
for all $m+r\in\N_0$. Let $\delta^p$ is the $p$th iterate of the difference operator $\delta$ acting on sequences $a=(a(m))_{m\in\N_0}$ as
$$
(\delta a)(m):=a(m)-a(m+1).
$$
We have
$$
\sum^p_{r=0}(-1)^r{p\choose r}\Tr(P_{E,m+r})=(\delta^p\chi_E)(m)
$$
where $\chi_E(m):=\chi(\Ei(m))$. The Euler characteristic $\chi(\Ei(m))$ is a polynomial in $m$ of degree $d$, which is the same as saying that $\delta^p\chi_E=0$ for $p\geq d+1$. So for $p\geq d+1$ we obtain
$$
\Tr(B_p(S_E)p_m)=0,\qquad\forall m\in\N_0.
$$
\end{proof}
\begin{Remark}
We could also consider the operators $B_p(T_E)$ of the rescaled tuple $T_E$; these are given by $B_p(T_E)=(\id-\Psi)^p(P_E)$. 
The normalized trace $\phi_m(P_E)=\chi(\Ei(m))/n_m$ is not a polynomial in $m$. Therefore Proposition \ref{diffoptraceprop} does not imply that $\Tr(B_p(T_E)p_m)$ is zero for any finite $p$. Therefore $T_E$ is not a $p$-isometry for any $p$ and $\breve{\varsigma}(P^E)=\Psi^\infty(P_E)$ cannot be expressed as a finite sum $\sum^p_{q=0}{p\choose q}(\id-\Psi)^q(P_E)$ for any $p$. 
\end{Remark}

\begin{prop}\label{SEisomprop}
Let $\GE_\N$ be a quotient module such that $\Ei_{\rm CD}=\Oi_{\rm CD}\otimes\Ei_{\B\setminus\{0\}}$ where $\Ei_{\B\setminus\{0\}}$ is the pullback to $\B\setminus\{0\}$ of a $\G$-equivariant vector bundle $\Ei$ over $\M$. Then $S_E$ is a strict $(d+1)$-isometry: $B_d(S_E)\ne 0$ and
$$
B_p(S_E)=0,\qquad\forall p\geq d+1.
$$
Moreover, if $\Ei$ is irreducible then $B_p(S_E)$ acts as a scalar on each graded piece $\GE_m\subset\GE_\N$ for each $p=0,\dots,d$,
$$
B_p(S_E)=\sum_{m\in\N_0}\phi_m^E(B_p(S_E))P_{E,m},
$$
and we have the Scatten-class estimate
$$
B_p(S_E)\in\Li^q\iff q>d+2-p.
$$
%Finally, the hidden Szegö expansion has only $d+1$ terms, (\textbf{no, see the proof})
%$$
%\breve{\varsigma}(P^E)=\Psi^\infty(P_E)=\sum^{d+1}_{q=0}{p\choose q}(\id-\Psi)^q(P_E).
%$$
\end{prop}
\begin{proof}
Suppose first that $\Ei$ is irreducible. Then %$\Phi_*(P_E)=|S_E|^2$ belongs to the center of the von Neumann algebra $\Gamma_b\otimes\Bi(\GE_0)$ (really? It acts by zero outside $\GE_\N$), i.e. 
$\Phi_{E,*}(P_E)=|S_E|^2$ acts as a scalar on each graded piece $\GE_m$. It follows that $B_1(S_E)P_{E,m}=(\id-\Phi_{E,*})(P_E)P_{E,m}$ is a scalar, viz. $\phi_m^E(B_1(S_E))P_{E,m}$. Hence also 
$$
B_2(S_E)P_{E,m}=(\id-\Phi_{E,*})(B_1(S_E))P_{E,m}
$$
is a scalar, and so on. From this and Proposition \ref{zerotraceprop} we see directly that $B_p(S_E)=0$ for $p\geq d+1$ when $\Ei$ is irreducible. For arbitrary equivariant $\Ei$ we have that $\GE_\N$ is a direct sum of $S_E$-reducing subspaces corresponding to the irreducible summands of $\Ei$, so $S_E$ is a $(d+1)$-isometry for all equivariant vector bundles $\Ei$.  

From Proposition \ref{zerotraceprop} we have the estimate
$$
\phi_m^E(B_p(S_E))=O(m^{-p})
$$
from which the stated Schatten-class estimate on $B_p(S_E)$ follows. 

%Similarly, $\Psi(P_E)=|T_E|^2$ is central, so $B_1(T_E)=(\id-\Psi)(P_E)$ is central and then also $B_2(T_E)=(\id-\Psi)(B_1(T_E))$ is central, and so on. But does it follow that $B_p(T_E)=0$ for $p\geq d+1$? No. 

%Since $(\id-\Psi)^p(P_E)p_m$ is just a scalar multiple $n_m/n_{m+p}$ of $(\id-\Phi_*)^p(P_E)p_m$ we obtain that $(\id-\Psi)^p(P_E)$ is central as well. 
\end{proof}
We see that $B_p(S_E)$ is zero precisely when $B_p(S_E)$ is trace-class.

%The Cauchy dual $S^\natural_E:=|S_E|^{-2}S_E$ of $S_E$ is in the equivariant case unitarily equivalent to the subnormal tuple on $\GK^E_\N$ and they can be regarded as the same. That shows directly that the Cauchy dual is jointly subnormal when $|S_E|$ is central. Or is $|S_E|^2$ really equal to $\breve{\varsigma}(P^E)$? No. So why would the Cauchy dual be unitarily equivalent to the tuple on $\GK^E_\N$? 
%\begin{Remark}
%As in Theorem \ref{Cauchydualthm} it follows that, in the setting of Proposition \ref{SEisomprop}, $S_E$ is completely hyperexpansive while its Cauchy dual $S_E^\natural:=|S_E|^{-2}S_E$ is subnormal. But note that $|S_E|^2=\Phi_*(P_E)$ and $\breve{\varsigma}(P^E)$ are not the same, so $S_E^\natural$ need not be unitarily equivalent to the subnormal tuple on the space $\GK^E_\N$ obtained by changing the inner product on $\GE_\N$ using $\breve{\varsigma}(P^E)^{1/2}$. 
%\end{Remark}

\subsubsection{Characterization of equivariance}

\begin{prop}\label{chareqprop}
Let $\GE_\N$ be a quotient module with continuous symbol $P^E=\varsigma(P_E)$. Assume that the shift $S_E$ on $\GE_\N$ has no reducing subspaces. Then the following are equivalent:
\begin{enumerate}[(a)]
\item{$\breve{\varsigma}^{(m)}(P^E)=c_{E,m}P_{E,m}$ for all $m\gg 0$.}
\item{$\GE_m=c_{E,m}H^0(\omega,\FS(\GH_m)\otimes P^E)$ for all $m\gg 0$.}
\item{$S_E^*S_Ep_m=\dim\GE_{m+1}/\dim\GE_m$ for all $m\gg 0$.}
\item{$\FS(H^0(\omega,\FS(\GH_m)\otimes P^E))=\FS(\GH_m)\otimes P^E$ is $\omega$-balanced for all $m\gg 0$.}
\end{enumerate}
%In particular, each of these conditions implies that the Cowen--Douglas sheaf $\Ei_{\rm CD}$ of $\GE_\N$ is locally free over $\B\setminus\{0\}$. 
\end{prop}
\begin{proof}
The operator $\breve{\varsigma}^{(m)}(P^E)$ compares the inner products $\GE_m$ and $H^0(\omega,\FS(\GH_m)\otimes P^E)$ (cf. Theorem \ref{backtoHardythmgen} and its proof). So (a) is equivalent to (b). Clearly (a)$\iff$(d). 

If the $\jmath_{l,m}^E$'s are unital for all $l\geq m$ then so is their limit $c_{E,m}^{-1}\breve{\varsigma}^{(m)}:\Gamma^\infty(\M;\Ei)\to\Bi(\GE_m)$ as $l$ goes to infinity. So (c) implies (a). 

If (b) holds then the fact that the multiplication tuple on $H^0(\Sb,\omega;P^E)$ is a spherical isometry ensures that (c) holds (cf. the proof of Lemma \ref{squareofSlemma}). 
This gives the proposition. 

%\textbf{We want} to show that local freeness is automatic. Clearly $P_E$ is over $\breve{\varsigma}(L^\infty(\M))+\Gamma_0$ when (a) holds. Moreover, if (b) holds then we indeed have $\FS(H^0(\omega,\FS(\GH_m)\otimes P^E))=\FS(\GH_m)\otimes\FS(H^0(\omega,P^E))$ for all $m$ because of the subproduct condition on $\GE_\bullet$. Does this imply that $P^E=\FS(H^0(\omega,P^E)$? Since we have (a) that would give that $P^E=B^E_m$ for all $m$ and $\Pi^E=P^E$ would be $C^0$ and $\Ei_{\rm CD}$ locally free. 

\end{proof}
We know that all conditions in Proposition \ref{chareqprop} hold when $\GE_\N$ comes from an equivariant vector bundle $\Ei$. 
We expect (c) in Proposition \ref{chareqprop} to hold only if the vector bundle defined by $P^E$ is $\G$-equivariant (cf. Remark \ref{timeremark}). Therefore the equivalent conditions in Proposition \ref{chareqprop} are likely to characterize the quotient modules which give rise to $\G$-equivariant vector bundles over $\M$.

\subsection{Guo-stability}
In this section we show, as asserted in the Introduction, that every $\G$-equivariant vector bundle over $\G/\K$ is a direct sum of vector bundles satisfying a stability condition (which we call Guo-stability) that is stronger than Gieseker-stability. This seems to be a new result. 

\begin{thm}
Let $\Ei$ be a $\G$-equivariant vector bundles over $\M$ and suppose that all irreducible summands of $\Ei$ have the same reduced Hilbert polynomial $\chi(\Ei(m))/\rank\Ei$. Then for every quotient sheaf $\Ei\to\Fi\to 0$ and all $m\gg 0$ we have 
\begin{equation}\label{finitelevelGies}
\frac{\chi(\Fi(m))}{\chi(\Fi(l))}\geq\frac{\chi(\Ei(m))}{\chi(\Ei(l))},\qquad\forall l\geq m
\end{equation}
with equality iff $\Fi$ is a subbundle (in which case $\Fi$ is a direct summand of $\Ei$). In particular, $\Ei$ is Gieseker-polystable.
\end{thm}
\begin{proof}
By replacing $\Ei$ with $\Ei(m_0)$ for large enough $m_0$ if necessary we may assume that $E_\N:=\bigoplus_{m\in\N_0}H^0(\M;\Ei(m))$ is a graded quotient of $\Ai\otimes\C^N$. Similarly, whenever $\Ei\to\Fi\to 0$ is a quotient sheaf we may assume that $F_\N:=\bigoplus_{m\in\N_0}H^0(\M;\Fi(m))$ is a graded quotient of $E_\N$. Let $\GE_\N$ and $\GF_\N$ be the completions of $E_\N$ and $F_\N$ in the inner product of $\GH_\N\otimes\C^N$. Then $\GF_\N$ is an invariant subspace for the backward shift $S_E^*$ on $\GE_\N$. We have thus encoded the quotient sheaf $\Fi$ as an $S_E^*$-invariant subspace $\GF_\N$ of $\GE_\N$, and conversely every $S_E^*$-invariant subspace gives rise to a quotient of $\Ei$ (see \S\ref{alggradmodsec}). 

Consider the unital completely positive map map
$$
\Psi_E(X):=(S_E^*S_E)^{-1/2}\sum^n_{\alpha=1}S_{E,\alpha}^*XS_{E,\alpha}(S_E^*S_E)^{-1/2},\qquad\forall X\in\Gamma_b.
$$
By Corollary \ref{SEsquarecor} $\Psi_E$ restricts to $\jmath_{l,m}^E:\Bi(\GE_{m+1})\to\Bi(\GE_m)$: For $B\in\Bi(\GE_{m+1})$ we have
$$
\Psi_E(B):=\frac{\chi(\Ei(m+1))}{\chi(\Ei(m))}\sum^n_{\alpha=1}S_{E,\alpha}^*BS_{E,\alpha}.
$$
We have seen that $S_E^*S_E$ commutes with every grading-preserving operator on $\GE_\N$, and this ensures that $S_{E,\alpha}^*$ and $S_{E,\alpha}^*(S_E^*S_E)^{-1/2}$ have the same invariant graded subspaces for each $\alpha\in\{1,\dots,n\}$. 
Since $\Psi_E$ is unital, the assumption that $\GF_\N$ is invariant under $S_E^*$ is thus equivalent to $\Psi_E(P_F)\leq P_F$ with equality iff $\GF_\N$ is reducing. For $X\in\Bi(\GE_l)$ we have $\Psi^{l-m}_E(X)=\jmath_{l,m}(X)$ for all $m\leq l$. 
In terms of the maps $\jmath_{l,m}^E$ the inequality $\Psi_E(P_F)\leq P_F$ reads $\jmath_{l,m}^E(P_{F,l})\leq P_{F,m}$ for all $l\geq m$. Using $\phi_m^E\circ\jmath^E_{l,m}=\phi_l^E$ we then obtain
$$
\frac{\chi(\Fi(l))}{\chi(\Ei(l))}=\phi_l^E(P_{F,l})=\phi_m^E(\jmath_{l,m}^E(P_{F,l}))\leq\phi_m^E(P_{F,m})=\frac{\chi(\Fi(m))}{\chi(\Ei(m))}
$$
for all $l\geq m$, with equality iff $\GF_{\geq m}$ is reducing. 
\end{proof}

Most likely the Guo-stability condition \eqref{finitelevelGies} %implies and 
is stronger than even slope-stability, and is a characteristic of $\G$-equivariance (cf. Proposition \ref{chareqprop}). %But:
%\begin{Question}
%Can slope-stability be deduced directly from \eqref{finitelevelGies}? 
%\end{Question}

\end{document}